\documentclass[aos]{imsart}

\RequirePackage{amsthm,amsmath,amsfonts,amssymb}
\RequirePackage[numbers]{natbib}
\RequirePackage[colorlinks,citecolor=blue,urlcolor=blue]{hyperref}
\RequirePackage{graphicx}

\startlocaldefs
\theoremstyle{plain}

\newtheorem{theorem}{Theorem}[section]
\newtheorem{lemma}[theorem]{Lemma}

\newtheorem{theos}{Theorem}[section]
\newtheorem{lems}[theorem]{Lemma}

\newtheorem{props}{Proposition}
\newtheorem{cors}{Corollary}
\newtheorem{assumption}{Assumption}
\theoremstyle{remark}

\usepackage{bm}
\usepackage{float}

\usepackage{epsf}
\usepackage{fancyheadings}
\usepackage{graphics}
\usepackage{graphicx}
\usepackage{psfrag}
\usepackage{hyperref}

\usepackage{color}

\usepackage{amsthm}
\usepackage{amsfonts}
\usepackage{amsmath}
\usepackage{amssymb}

\usepackage[utf8]{inputenc}
\usepackage{textgreek}
\usepackage{amssymb}
\usepackage{tcolorbox}

\usepackage{algorithmic}
\usepackage{algorithm}
\usepackage{enumitem}

\usepackage[noabbrev,nameinlink]{cleveref}

\newtheorem{lems}{Lemma}
\newtheorem{cors}{Corollary}

\newtheorem{theorem}{Theorem}

\newlength{\widebarargwidth}
\newlength{\widebarargheight}
\newlength{\widebarargdepth}
\DeclareRobustCommand{\widebar}[1]{%
  \settowidth{\widebarargwidth}{\ensuremath{#1}}%
  \settoheight{\widebarargheight}{\ensuremath{#1}}%
  \settodepth{\widebarargdepth}{\ensuremath{#1}}%
  \addtolength{\widebarargwidth}{-0.3\widebarargheight}%
  \addtolength{\widebarargwidth}{-0.3\widebarargdepth}%
  \makebox[0pt][l]{\hspace{0.3\widebarargheight}%
    \hspace{0.3\widebarargdepth}%
    \addtolength{\widebarargheight}{0.3ex}%
    \rule[\widebarargheight]{0.95\widebarargwidth}{0.1ex}}%
  {#1}}

\makeatletter
\long\def\@makecaption#1#2{
        \vskip 0.8ex
        \setbox\@tempboxa\hbox{\small {\bf #1:} #2}
        \parindent 1.5em  
        \dimen0=\hsize
        \advance\dimen0 by -3em
        \ifdim \wd\@tempboxa >\dimen0
                \hbox to \hsize{
                        \parindent 0em
                        \hfil 
                        \parbox{\dimen0}{\def\baselinestretch{0.96}\small
                                {\bf #1.} #2
                                } 
                        \hfil}
        \else \hbox to \hsize{\hfil \box\@tempboxa \hfil}
        \fi
        }
\makeatother

\long\def\comment#1{}

\newcommand{\mns}{\mkern-2mu}  
\newcommand{\matsnorm}[2]{\ensuremath{|\mns|\mns| #1 |\mns|\mns|_{{#2}}}}

\newcommand{\vecnorm}[2]{\| #1\|_{#2}}

\newcommand{\inprod}[2]{\ensuremath{\langle #1 , \, #2 \rangle}}

\newcommand{\Exs}{\ensuremath{{\mathbb{E}}}}
\newcommand{\Prob}{\ensuremath{{\mathbb{P}}}}

\newcommand{\numobs}{\ensuremath{n}}

\DeclareMathOperator{\diag}{diag}

\newcommand{\Xspace}{\ensuremath{\mathcal{X}}}

\newcommand{\R}{\mathbb{R}}
\newcommand{\real}{\R}
\newcommand{\E}{\mathbb{E}}

\newcommand{\tr}{\operatorname{tr}}
\newcommand{\F}{\mathcal{F}}

\newcommand{\eps}{\varepsilon}

\newcommand{\argmax}{\operatorname{argmax}}

\newcommand{\widgraph}[2]{\includegraphics[keepaspectratio,width=#1]{#2}}

\newcommand{\stdBasis}{e} 
\newcommand{\zerovec}{0} 
\newcommand{\linkfun}{g}
\newcommand{\Id}{\ensuremath{\mathbf{I}}}

\newcommand{\Regressor}{{x}}             
\newcommand{\Nuisance}{{z}}      
\newcommand{\Noise}{\varepsilon}       
\newcommand{\Response}{y}       
\newcommand{\TargetDim}{\ensuremath{d_{T}}}      
\newcommand{\NuisanceDim}{\ensuremath{d_{N}}}    
\newcommand{\LinearNoiseVar}{\sigma^2}    
\newcommand{\EstLinearNoiseVar}{{\widehat{\sigma}^2}}    

\newcommand{\TrueTargetPar}{\theta^*} 
\newcommand{\TrueNuisancePar}{{\eta}^{*}}  
\newcommand{\TrueFunPar}{h^{*}}        
\newcommand{\EstFunPar}{\widehat{h}}    
\newcommand{\FunPar}{h}             

\newcommand{\TargetPar}{\theta}       
\newcommand{\NuisancePar}{{\eta}}        
\newcommand{\EstTargetPar}{{\widehat{\theta}}}  
\newcommand{\EstNuisancePar}{{\widehat{{\eta}}}}    
\newcommand{\SoluTargetPar}{{\widetilde\theta}}   
\newcommand{\AuxiNuisancePar}{{\widebar\theta}}  
\newcommand{\AllNuisancePar}{\omega}             
\newcommand{\TrueAllNuisancePar}{\omega^{*}}     
\newcommand{\EstAllNuisancePar}{\widehat\omega}  

\newcommand{\TargetSpace}{\Theta}     
\newcommand{\NuisanceSpace}{\mathcal{H}}   

\newcommand{\Covar}{\SigMat}   
\newcommand{\InvSqrtCovar}{\OmegaMat}  

\newcommand{\History}{\F}    

\newcommand{\indep}{\perp \!\!\!\! \perp}

\newcommand{\iRegressor}{{x}_{i}}  
\newcommand{\iNuisance}{{z}_{i}}   
\newcommand{\iResponse}{y_{i}}   

\newcommand{\jRegressor}{{x}_{j}}  
\newcommand{\jNuisance}{{z}_{j}}   
\newcommand{\jResponse}{y_{j}}   

\newcommand{\iNoise}{\varepsilon_{i}} 
\newcommand{\iHistory}{\F_{i-1}}      

\newcommand{\mymatrix}[1]{\ensuremath{\mathbf{#1}}}

\newcommand{\IdMat}{\mymatrix{I}}

\newcommand{\SigMat}{\ensuremath{\mymatrix{\Sigma}}}
\newcommand{\CovMat}{\SigMat}
\newcommand{\OmegaMat}{\ensuremath{\mymatrix{\Omega}}}
\newcommand{\GamMat}{\ensuremath{\mymatrix{\Gamma}}}

\newcommand{\iCovar}{{\CovMat}_{i}}   
\newcommand{\iEstCovar}{{\widehat\CovMat}_{i}}   

\newcommand{\LinScaleVec}{v}    
\newcommand{\EstLinScaleVec}{{\widehat{v}}}    

\newcommand{\LinScaleVecOneArm}{w}  

\newcommand{\ScalePreCondVec}{A}   
\newcommand{\SelectProb}{p}
\newcommand{\EstSelectProb}{{\widehat{p}}}

\newcommand{\GlmMeanVec}{m}     

\newcommand{\DiagProbMatrix}{{\bf D}_\SelectProb}   
\newcommand{\GlmVar}{\nu^2}         
\newcommand{\GlmVarDeri}{(\GlmVar)'}  
\newcommand{\InvCompOne}{{\bf C}}  
\newcommand{\InvCompTwo}{{\bf \Delta}}  
\newcommand{\InvCompDiag}{{\bf B}}  
\newcommand{\InvCompRow}{{\bf K}}  
\newcommand{\HighOrdMa}{{\bf H}}  

\newcommand{\IndexOne}{I_1}    
\newcommand{\IndexTwo}{I_2}     
\newcommand{\NumIndexOne}{n_1}   
\newcommand{\NumIndexTwo}{n_2}   
\newcommand{\Numobs}{n}          

\newcommand{\EmpMean}{\widehat\E_{\NumIndexTwo}}     
\newcommand{\CondMean}{\widetilde\E_{\NumIndexTwo}}  
\newcommand{\EmpMeanAll}{\widehat\E_{\Numobs}}     
\newcommand{\ScoreFun}{\phi}  

\newcommand{\iMartdiff}{d_i}          
\newcommand{\iMartdiffa}{d_{ai}}        
\newcommand{\iMartdiffb}{d_{bi}}      
\newcommand{\BernVarb}{\sigma^2_{b}}      
\newcommand{\BernBoundb}{b_{b}}      

\newcommand{\OneDirect}{u}
\newcommand{\TrueOneDirectPar}{\TrueTargetPar_{\OneDirect}}
\newcommand{\OneDirectPar}{\TargetPar_{\OneDirect}}
\newcommand{\EstOneDirectPar}{\EstTargetPar_{\OneDirect}}
\newcommand{\SoluOneDirectPar}{\SoluTargetPar_{\OneDirect}}
\newcommand{\SuppDirect}{S_{\OneDirect}}

\newcommand{\iCovariate}{Q_i}
\newcommand{\CovariateMa}{\ensuremath{{\bf Q}}}
\newcommand{\HatMa}{\ensuremath{{\bf H}}}

\newcommand{\SparseLevel}{s}
\newcommand{\SparseSet}{S}
\newcommand{\RegressorMa}{\ensuremath{{\bf{X}}}}
\newcommand{\NuisanceMa}{\ensuremath{{\bf{Z}}}}
\newcommand{\AllSelecProbMa}{\ensuremath{{\bf{P}}}}

\newcommand{\AllParameter}{\beta}

\newcommand{\GlmMeanVecwhat}{\ensuremath{\widehat \GlmMeanVec}}
\newcommand{\InvSqrtCovarwhat}{\ensuremath{\widehat \InvSqrtCovar}}
\newcommand{\ScalePreCondVecwhat}{\ensuremath{\widehat\ScalePreCondVec}}
\newcommand{\TargetParwhat}{\ensuremath{\widehat\TargetPar}}
\newcommand{\NuisanceParwhat}{\ensuremath{\widehat\NuisancePar} }
\newcommand{\AllParameterwhat}{\widehat\AllParameter}
\newcommand{\FunParwbar}{\widebar\FunPar} 
\newcommand{\ScoreFunwtil}{\widetilde\ScoreFun}
\newcommand{\InvSqrtCovarbar}{\bar\InvSqrtCovar}
\newcommand{\GlmMeanVecbar}{\bar \GlmMeanVec}
\newcommand{\FunParbar}{\bar \FunPar}
\newcommand{\cwtil}{\widetilde c}
\newcommand{\Bernbwtil}{\widetilde \BernBoundb}
\newcommand{\epsbar}{\bar\eps}
\newcommand{\MatFun}{{\bf F}}

\newcommand{\EstCovBound}{{B_{\Covar}}}

\newcommand{\powercoef}{{\alpha}}

\newcommand{\frobnorm}[1]{\matsnorm{#1}{F}}
\newcommand{\opnorm}[1]{\matsnorm{#1}{\mathrm{op}}}

\newcommand{\vtwonorm}[1]{\vecnorm{#1}{2}}

\newcommand{\PreCondMean}{\ensuremath{\widetilde\E_{\NumIndexOne}}}

\newcommand{\defn}{\ensuremath{:=}}
\newcommand{\Normal}{\ensuremath{\mathcal{N}}}
\newcommand{\linkFun}{g}

\newcommand{\convdist}{\ensuremath{\stackrel{d}{\rightarrow}}}

\newcommand{\liloh}{\ensuremath{o}}
\newcommand{\bigoh}{\ensuremath{\mathcal{O}}}

\newcommand{\Ctil}{\ensuremath{{\bf\tilde{C}}}}
\newcommand{\Qwtil}{\widetilde Q}
\newcommand{\CoverSet}{\ensuremath{\mathcal{C}}}
\newcommand{\Term}{\ensuremath{T}}
\newcommand{\dummyRV}{v}

\newcommand{\PlainPar}{\ensuremath{\theta}}

\newcommand{\myasslabel}[1]{{\bf{#1}}}
\usepackage{scalerel}

\newcommand{\myassumption}[3]{
  \begin{enumerate}[label={\scaleto{\mbox{(#1)}}{10pt}}, align = left]
  \item \label{#2} {#3}
  \end{enumerate}
}

\newcommand{\subgauss}{\ensuremath{\nu}}

\newcommand{\Pclass}{\ensuremath{\mathcal{P}}}

\newcommand{\DelTil}{\ensuremath{{\bf \tilde{\Delta}}}}

\newcommand{\czerotil}{\ensuremath{\tilde{c}_0}}

\newcommand{\prob}{\Prob}

\newcommand{\state}{\ensuremath{x}}

\newcommand{\EarlyScoreFun}{\ensuremath{\psi}}
\newcommand{\OLS}{\operatorname{OLS}}
\newcommand{\lasso}{\operatorname{Lasso}}

\newcommand{\algtext}[1]{{\texttt{#1}}}

\newcommand{\AlgShortPL}{{\algtext{AdapTZ-PL}}}
\newcommand{\AlgShortGLM}{{\algtext{AdapTZ-GLM}}}
\newcommand{\AlgLong}{{\algtext{AdapTZ}}$\;$}
\newcommand{\AlgLongPL}{{\algtext{AdapTZ-PL}}$\;$}
\newcommand{\AlgLongGLM}{{\algtext{AdapTZ-GLM}}$\;$}

\newcommand{\GLMlasso}{\operatorname{GLMlasso}}
\newcommand{\glmloss}{f}

\endlocaldefs

\begin{document}

\begin{frontmatter}
\title{Semi-parametric inference based on \\ adaptively collected
      data}

\begin{aug}
\author[A]{\fnms{Licong}~\snm{Lin}\ead[label=e1]{liconglin@berkeley.edu}},
\author[B]{\fnms{Koulik}~\snm{Khamaru}\ead[label=e2]{kk1241@stat.rutgers.edu}}
\and
\author[C]{\fnms{Martin}~\snm{J. Wainwright}\ead[label=e3]{wainwrigwork@gmail.com}}
\address[A]{Department of Statistics,
UC Berkeley\printead[presep={,\ }]{e1}}

\address[B]{Department of Statistics,
Rutgers University\printead[presep={,\ }]{e2}}

\address[C]{EECS and Mathematics, MIT \\ Lab for Information and Decision
  Systems, Statistics and Data Science Center
  \printead[presep={,\ }]{e3}}
\end{aug}

\begin{abstract}
Many standard estimators, when applied to adaptively collected data,
fail to be asymptotically normal, thereby complicating the
construction of confidence intervals.  We address this challenge in a
semi-parametric context: estimating the parameter vector of a
generalized linear regression model contaminated by a non-parametric
nuisance component.  We construct suitably weighted estimating
equations that account for adaptivity in data collection, and provide
conditions under which the associated estimates are asymptotically
normal.  Our results characterize the degree of ``explorability''
required for asymptotic normality to hold. For the simpler problem of
estimating a linear functional, we provide similar guarantees under
much weaker assumptions.  We illustrate our general theory with
concrete consequences for various problems, including standard linear
bandits and sparse generalized bandits, and compare with other methods
via simulation studies.
\end{abstract}

\begin{keyword}[class=MSC]       
\kwd[Primary ]{62E20}
\kwd[; secondary ]{6208}
\end{keyword}

\begin{keyword}
\kwd{Semi-parametric inference}
\kwd{Two stage weighted Z-estimator}
\kwd{Adaptive data collection}
\kwd{Generalized linear model}
\end{keyword}

\end{frontmatter}


\section{Introduction} 

A canonical problem in semi-parametric statistics is to estimate a
low-dimensional parameter in the presence of a high-dimensional or
non-parametric nuisance component. A standard goal is to obtain
estimators that are both $\sqrt{\numobs}$-consistent and
asymptotically normal; these properties streamline the task of
designing asymptotically valid confidence intervals and hypothesis
tests.  There is now a rich literature on this topic
(e.g.,~\cite{bickel1982adaptive,bickel1993efficient,robinson1988root,andrews1994asymptotics,robins1995semiparametric,ai2003efficient,van2011targeted,chernozhukov2018double});
however, the bulk of these findings involve datasets consisting of
i.i.d. (or weakly dependent) samples, in which case standard
asymptotic results such as the central limit theorem are in force.

Of interest to us in this paper are settings in which such assumptions
no longer hold.  In particular, we consider a model that allows for
the dataset to have been collected in an \emph{adaptive manner}; in
particular, the distribution of the $(i+1)$-th data point is allowed
to depend on the preceding $i$ samples. Such adaptively collected
datasets arise in various applications, among them bandit
experiments~\cite{lattimore2020bandit}, active
learning~\cite{fontaine2021online}, time series
modeling~\cite{box2015time}, adaptive stochastic approximation
schemes~\cite{deshpande2018accurate,lai1982least}, and dynamic
treatment schemes.

The main contribution of this paper is to propose and analyze a family
of estimators for which asymptotic normality holds even for a data
collection model that allows for fairly general sequential dependence.
We do so within the semi-parametric framework of generalized partial
linear regression.  In such models, a scalar response variable
$\Response$ is linked to a covariate vector $\Regressor \in
\real^{\TargetDim}$ and an auxiliary vector $\Nuisance \in
\real^{\NuisanceDim}$ via the equation
\begin{align}
\label{eqn:general-setup_0}
\iResponse & = g \big(\inprod{\iRegressor}{\TrueTargetPar} +
\TrueFunPar(\iNuisance) \big) + \iNoise.
\end{align}
Here $\{ \iNoise \}_{i \geq 1}$ is an i.i.d. noise sequence; the
function $g: \real \rightarrow \real$ is known as the inverse link;
the vector $\TrueTargetPar \in \R^{\TargetDim}$ is the \emph{target
parameter} of interest; and $\TrueFunPar: \real^{\NuisanceDim}
\rightarrow \real$ is a high-dimensional (or nonparametric) nuisance
component. We assume that the covariate-auxiliary pair $(\iRegressor,
\iNuisance)$ at round $i$ can depend on the set of previous
observations $\big \{ \big(\jRegressor, \jNuisance, \jResponse \big)
\big\}_{j = 1}^{i - 1}$.

As one illustrative example, the partial linear regression model---as
a special case of the general
set-up~\eqref{eqn:general-setup_0}---arises in the treatment
assignment problem
(e.g.,~\cite{tewari2017ads,yom2017encouraging,figueroa2021adaptive,wang2021reinforcement,trella2022reward}).
Given a collection of $\TargetDim$ drugs, the goal is to determine the
most effective one.  In order to do so, we undertake a sequential
experiment involving a collection of $\numobs$ patients, in which our
decision at each round is to either assign one of the $\TargetDim$
drugs, or to provide no treatment (which might correspond to a control
group).  For a given patient index $i \in [\numobs] \defn \{ 1,
\ldots, \numobs \}$, the decision to assign drug $k \in [\TargetDim]$
is encoded by setting the regression vector $\iRegressor = e_k$, the
binary indicator vector with a single one in position $k$. On the
other hand, assignment to the control group is coded by setting
$\iRegressor = {\mathbf{0}}$, corresponding to the all-zeros vector.
With these choices, the response $\iResponse$ is a noisy version of
$\TrueTargetPar_k$ if we assign the drug $k$, or pure noise if we
assign the control group.  Within this set-up, various adaptive
procedures for choosing the covariate vectors are natural.  For
instance, a doctor might decide the treatment of a patient $i$ based
on their personal information $\iNuisance$, and the historical data
from previous patients $\{ (\jRegressor, \jNuisance, \jResponse) \}_{j
  = 1}^{i - 1}$.

\subsection{Visualizing breakdown under adaptivity}

In order to motivate our proposed methodology, it is useful to
visualize how classical guarantees, valid under i.i.d. sampling, can
break down when the data points are collected adaptively.  A simple
example suffices to illustrate this phenomenon: more specifically, let
us consider the linear model
\begin{align}
\label{EqnPartialLinear}
\iResponse & = \inprod{\iRegressor}{\TrueTargetPar} +
\inprod{\iNuisance}{\TrueNuisancePar} + \Noise_i,
\end{align}
involving a target parameter $\TrueTargetPar \in \R^{\TargetDim}$, and
a nuisance parameter $\NuisancePar \in \R^{\NuisanceDim}$.  This is a
special case of our general set-up with the link function $\linkFun(x)
= x$ and the nuisance function \mbox{$\FunPar(\Nuisance) =
  \inprod{\Nuisance}{\TrueNuisancePar}$}.  Given an estimate
$\widehat{\NuisancePar}$ of the nuisance vector $\TrueNuisancePar$, a
standard $Z$-estimate $\SoluTargetPar$ of the target parameter can be
obtained by defining the score function
\begin{subequations}
  \begin{align}
    \label{EqnEarlyScore}
\EarlyScoreFun_i(\iResponse, \iRegressor, \iNuisance, \TargetPar,
{\NuisancePar}) & \defn (\iRegressor - \SelectProb_i) \big \{
\iResponse - \inprod{\iRegressor}{ \TargetPar} -
\inprod{\iNuisance}{{\NuisancePar}} \big \},
  \end{align}
and then solving the estimating equations
\begin{align}
\label{EqnEEIntro}
\sum_{i=1}^{\Numobs} \EarlyScoreFun_i(\iResponse, \iRegressor,
\iNuisance, \SoluTargetPar, \widehat{\NuisancePar}) & = 0.
\end{align}
\end{subequations}
\begin{figure}[ht!]
  \begin{center}
\begin{tabular}{ccc}    
\widgraph{0.45 \linewidth}{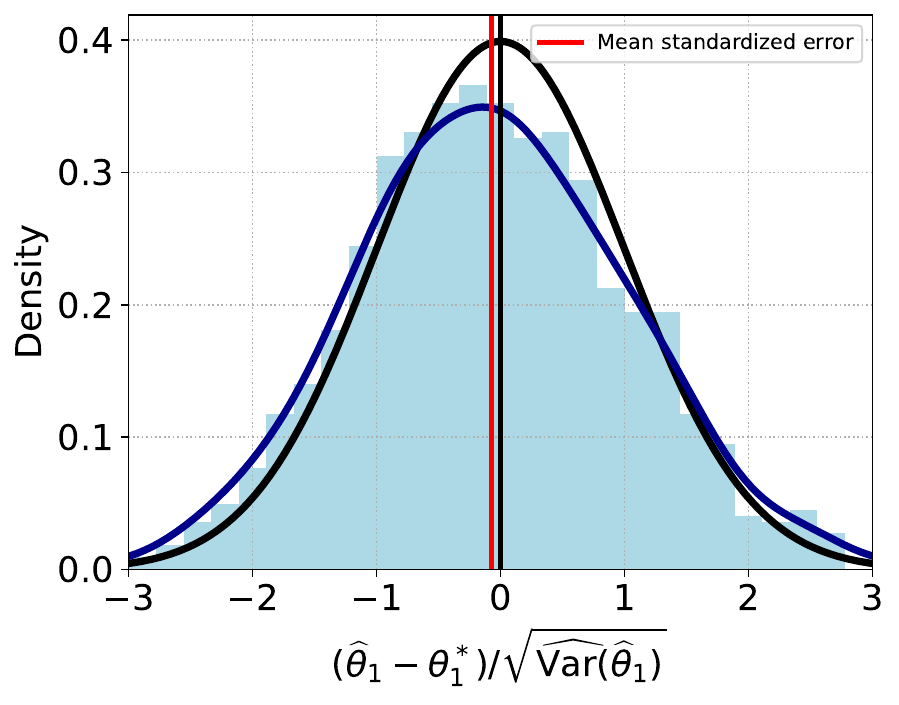} &&
\widgraph{0.47 \linewidth}{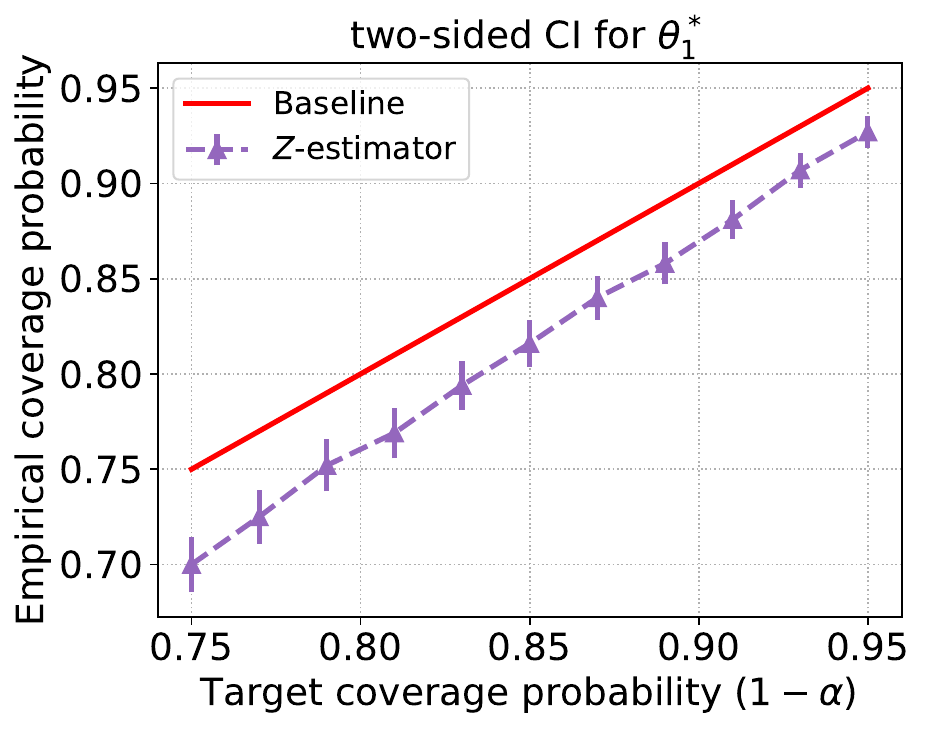} \\
(a) && (b)
\end{tabular}
\caption{(a): Standardized estimation error of the
  $Z$-estimator~\eqref{EqnEEIntro} for the first coordinate
  $\TrueTargetPar_1$; shown is a histogram based on $1000$ trials.
  (b): Empirical coverage probability of two-sided confidence
  interval for $\TrueTargetPar_1$ for a simulation for with parameters
  $(\TargetDim, \NuisanceDim, \Numobs) = (2, 1000, 950)$.
  See Section~\ref{sec:exp_linear} for details.}
\label{fig:high_dim_linear_demo}
\end{center}
\end{figure}
In the definition~\eqref{EqnEarlyScore} of the score function
$\EarlyScoreFun_i$, the vector $p_i$ is the conditional mean of
$\iRegressor$ given the past data points.  This $Z$-estimator is a a
well-studied procedure~\cite{robinson1988root}; we refer readers to
Section~\ref{SecBackground} and equation~\eqref{eq:naive_dml_score}
for more details.  When the data points are i.i.d., it can be
shown~\cite{chernozhukov2018double} that the estimate $\SoluTargetPar$
is $\sqrt{\numobs}$-consistent and asymptotically normal.

However, when the data is collected in an adaptive manner, these
attractive guarantees may fail to hold.  To illustrate such a
breakdown, we performed experiments on a linear
model~\eqref{EqnPartialLinear} with \mbox{$(\TargetDim,
  \NuisanceDim)=(2, 1000)$,} and in order to apply the LASSO bandit
algorithm~\cite{oh2021sparsity}, we assumed that the nuisance vector
$\NuisancePar \in \real^{1000}$ was $4$-sparse.  We generated a path
of $\Numobs = 950$ samples using the LASSO bandit procedure to select
the covariates in an adaptive fashion, as applied to the target vector
$\TrueTargetPar = [2, \; 2 ]^T \in \real^2$.

Panel (a) of Figure~\ref{fig:high_dim_linear_demo} shows that the
standardized estimation error associated with $\widehat{\theta}_1$ is
\emph{not} standard Gaussian; instead, the distribution has a downward
\emph{bias}, as reflected by the negative mean $-0.07$ of the
standardized errors.  Thus, we see that asymptotic normality may fail
to hold with adaptively collected data.  Panel (b)
of Figure~\ref{fig:high_dim_linear_demo} shows that confidence intervals
constructed from the unweighted $Z$-estimator fail to provide the
desired target coverage; in particular, the fraction of times that
they cover the true parameter is consistently below the target
coverage. This under-coverage is to be expected given the deviations
of the standardized error from Gaussianity.

To be clear, such distributional anomalies are a wide-spread
phenomenon: they are specific to \emph{neither} the particular
$Z$-estimator \emph{nor} the LASSO bandit algorithm that we have
simulated here.  Similar types of breakdown are well-documented in the
time series and forecasting literature, dating back to the classical work
of Dickey and Fuller~\cite{dickey1979distribution},
White~\cite{white1959limiting}, and Lai and
Wei~\cite{lai1982least}. More recent
work~\cite{deshpande2018accurate,zhang2020inference,khamaru2021near}
has highlighted a similar phenomenon in multi-armed bandit problems
with popular selection algorithms like Thompson sampling, upper
confidence bound (UCB), and $\epsilon$-greedy selection.


\subsection{Related work}

In this section, we survey existing literature on inference using
adaptively collected data and semi-parametric inference that are
relevant to our problem.

\subsubsection{Inference using adaptively collected data}

In their seminal work, Lai and
Wei~\cite{lai1982least,lai1994asymptotic} studied various regression
models in which the covariate-response pairs are collected in an
adaptive fashion.  Among other results, they provided conditions under
which the ordinary least squares (OLS) estimate is asymptotically
normal.  However, their results require a stability condition on the
covariate matrix.  This stability condition fails to hold in various
settings, among them certain types of autoregressive
models~\cite{dickey1979distribution,white1959limiting,lai1982least},
the UCB and related online procedures for
bandits~\cite{lattimore2020bandit}, as well as offline procedures for
multi-armed bandit problems with adaptively collected data
(e.g.,~\cite{deshpande2018accurate,zhang2020inference}).

In order to address these challenges, Hadad et
al.~\cite{hadad2021confidence} proposed an adaptively weighted version
of the augmented inverse propensity-weighted
(AIPW,~\cite{mannor2004sample}) estimator for multi-armed
bandits. They suggested certain choices of the adaptive weights that
ensure the variance stabilization necessary to apply martingale
central limit theory. Subsequent work by Zhan et
al.~\cite{zhan2021off} and Bibaut et al.~\cite{bibaut2021post} extend
this approach to develop asymptotically normal estimators for
contextual bandits. Zhang et al.~\cite{zhang2021statistical} analyzes
a weighted $M$-estimator for contextual bandit problems.  Syrgkanis et
al.~\cite{syrgkanis2023post} proposes a weighted $Z$-estimator for
estimating the structural parameters in a structural mean nested
model.  All of these works on bandit problems all assume the data
collection algorithm is known, and therefore enables the construction
of weighted estimators based on the selection probability of each
arm. Alternatively, when the bandit algorithm is unknown, Deshpande et
al.~\cite{deshpande2018accurate} and Khamaru et
al.~\cite{khamaru2021near} propose online-debiasing procedures that
lead to asymptotically normal behavior.


\subsubsection{Neyman orthogonality in semi-parametric inference}

Semi-parametric statistics addresses how to estimate low-dimensional
parameters in the presence of high-dimensional or nonparametric
nuisance parameters; it is associated with a rich and evolving
literature
(e.g.,~\cite{bickel1982adaptive,pfanzagl2012contributions,bickel1993efficient,robinson1988root,andrews1994asymptotics,robins1995semiparametric,ai2003efficient,van2011targeted,chernozhukov2018double}).
A key concept is that of Neyman orthogonality of the score
function~\cite{neyman1959optimal}, which formalizes the first-order
effect of perturbations in the nuisance terms on the target
estimator. Neyman orthogonality has played an important role in
semi-parametric
estimation~\cite{andrews1994asymptotics,newey1994asymptotic}; targeted
learning~\cite{van2011targeted}; as well as inference for
high-dimensional linear
models~\cite{zhang2014confidence,belloni2014pivotal,belloni2016post,javanmard2014confidence}.
Sample splitting methods, in which different portions of the dataset
are used to estimate the non-parametric and parametric components, are
also commonly used in the literature
(e.g.,~\cite{bickel1982adaptive,schick1986asymptotically,fan2012variance,kato2021adaptive}).

Chernozhukov et al.~\cite{chernozhukov2018double} combined the notion
of Neyman orthogonality with sample splitting to construct
$Z$-estimators that are asymptotically normal; they referred to this
approach as double/debiased machine learning (DML).  Sample splitting
weakens the requirement of Donsker class conditions on the nuisance
estimators, thereby allowing for the use of more sophisticated
non-parametric procedures.  Other procedures that build upon or are
closely related to the DML approach have been developed for estimating
heterogeneous treatment
effects~\cite{nie2020quasi,kennedy2020towards,fan2022estimation,knaus2022double,semenova2021debiased};
continuous treatment
effects~\cite{colangelo2020double,semenova2021debiased}; tree-based
methods~\cite{wager2018estimation,athey2019generalized,oprescu2019orthogonal};
statistical learning with nuisance
parameters~\cite{foster2019orthogonal}; as well as dynamical treatment
effects~\cite{lewis2021double,bodory2022evaluating,chernozhukov2022automatic}.
Some of this work goes beyond the i.i.d. setting in allowing for
samples drawn from stable Markov chains, but do not address the
general adaptive setting of interest in this paper.

In the i.i.d. setting, Belloni et al.~\cite{belloni2016post} studied
inference in generalized linear models with nuisance parameters,
developing a general framework for inference of a one-dimensional
parameter in the presence of high-dimensional nuisance.  Liu et
al.~\cite{liu2021double} propose an estimator for partially logistic
regression models. Both works exploit Neyman orthogonality, and their
methods involve solving a certain estimating equation, as in this
paper.  In this paper, we focus on a similar problem setting, but
mainly as a vehicle to study the effect of adaptive data collection.


\subsubsection{Non-asymptotic confidence intervals}

As opposed to asymptotic guarantees, an alternative approach is to
exploit concentration inequalities to construct non-asymptotic
confidence regions that are valid uniformly in time.  For instance,
Abbasi et al.~\cite{abbasi2011improved} prove an any-time
self-normalized concentration inequality for bandit problems. These
bounds were further developed for multi-armed
bandits~\cite{jamieson2014lil,kaufmann2016complexity} and for general
sequential experiments~\cite{howard2021time}.  On one hand, these
methods are equipped with non-asymptotic guarantees, and remain
relatively robust to model mis-specification. On the flip side,
however, there are many settings in which these procedures lead to
confidence intervals that are overly conservative relative to those
constructed based on asymptotically normal estimators; for instance,
see Figure 2 in the paper~\cite{hadad2021confidence} for a comparison
of this type.


\subsection{Our contributions and paper organization}

In this paper, we study how to estimate a target parameter
$\TrueTargetPar$ associated with a generalized linear regression model
in presence of both (possibly nonparametric) nuisance components, and
a general model for adaptive data collection.  Due to the sequential
dependence induced by adaptive data collection, many standard
$Z$-estimators may exhibit non-normal asymptotic behavior, and our
main contribution is to rectify this issue.  In order to do so, we
propose and analyze a family of estimators for $\TrueTargetPar$ and
show that under mild conditions these estimators are asymptotically
unbiased and asymptotically normal. These procedures are based on an
adaptive re-weighting of two-stage $Z$-estimators, so that we refer to
them as \AlgLong methods.  In Theorem~\ref{thm:linear_new1}, we
discuss the \AlgLongPL procedure that is tailored to the partial
linear model, whereas Theorem~\ref{thm:glm_new1_temp} provides
guarantees on a more general procedure (\AlgShortGLM) that applies to
generalized linear models. Under certain regularity conditions, both
of these theorems yield an asymptotically valid confidence region for
the parameter vector $\TrueTargetPar$.  Next, we consider the problem
of estimating a linear functional of the form $\OneDirect^\top
\TrueTargetPar$, where $\OneDirect$ is any fixed unit vector in
$\R^{\TargetDim}$. In Theorem~\ref{thm:linear_new3}
and~\ref{thm:glm_new1_temp_one_arm}, we show that, for this simpler
problem, it is possible to obtain asymptotic normality under much
weaker conditions compared to~Theorem~\ref{thm:linear_new1}
and~\ref{thm:glm_new1_temp}. Finally, in
Section~\ref{SecApplications}, we demonstrate the usefulness of our
general theory by developing its consequences for some concrete
classes of semi-parametric models.

\paragraph*{Notation} For any
numbers $\Numobs, \NumIndexOne, \NumIndexTwo \geq 1$ such that
$\Numobs = \NumIndexOne + \NumIndexTwo$ and a sequence of random
variables $\{W_i\}_{i=1}^{\Numobs}$, we use the shorthand
\begin{align*}
\EmpMean f_i(W_i)  = \frac{1}{\NumIndexTwo} \sum_{i = \NumIndexOne +
  1}^{\Numobs} f_i(W_i) \quad \mbox{and} \quad
\CondMean f_i(W_i)  = \frac{1}{\NumIndexTwo} \sum_{i = \NumIndexOne +
  1}^{\Numobs} \Exs f_i(W_i \mid \History_{i - 1})
\end{align*}
We use $\vtwonorm{\cdot}$ to denote the $2$-norm for a vector; for
matrices, we use $\opnorm{\cdot}$ and $\frobnorm{\cdot}$ to denote
their operator and Frobenius norms, respectively.  For vectors $a, b
\in \real^d$, we use $\inprod{a}{b} = \sum_{j=1}^d a_j b_j$ as a
shorthand for their Euclidean inner product.


\section{Main results}
\label{SecMain}

In this section, we first set up the class of problems to be studied
in this paper.  Our focus is asymptotic guarantees for the parameters
of a generalized linear regression model in the presence of a
non-parametric nuisance component.  Our main results are analyses of
two algorithms for estimation in the adaptive generalized
model~\eqref{eqn:general-setup} with nuisance parameters.  We derive
several asymptotic normality guarantees on parameters of interest when
these procedures are applied.  Our first algorithm (\AlgShortPL) is
designed for the partial linear model (i.e., the special case
$\linkfun(x) = x$), whereas the second one (\AlgShortGLM) applies to
more general non-linear link functions $\linkfun$.


\subsection{Problem set-up}
\label{SecBackground}

Suppose that a scalar response variable $\Response$ is linked to a
covariate vector $\Regressor \in \real^{\TargetDim}$ and auxiliary
vector $\Nuisance \in \real^{\NuisanceDim}$ via the equation
\begin{align}
\label{eqn:general-setup}
\Response & = g \big(\inprod{\Regressor}{\TrueTargetPar} +
\TrueFunPar(\Nuisance) \big) + \Noise,
\end{align}
where $\Noise$ is a zero-mean noise variable.  Here $g: \real
\rightarrow \real$ is a known link function, whereas
\mbox{$\TrueTargetPar \in \TargetSpace \subset \real^{\TargetDim}$} is
an unknown \emph{target parameter}, and the function $\TrueFunPar:
\R^{\NuisanceDim} \rightarrow \R$ is also unknown.  We assume that the
target parameter space $\TargetSpace$ is a bounded open subset of
$\R^{\TargetDim}$, whereas $\TrueFunPar$ belongs to some class
$\NuisanceSpace$ of functions that are uniformly bounded in the
supremum norm.

The model~\eqref{eqn:general-setup} is a particular instantiation of a
\emph{semi-parametric model}, as it contains both a parametric and a
non-parametric component.  Of primary interest is the parametric
component $\TrueTargetPar$: our goal is to develop point estimates as
well as confidence sets associated with these estimates.  In this
context, the unknown function $\TrueFunPar$ plays the role of a
\emph{nuisance parameter}.  It needs to be controlled to obtain a good
estimate of $\TrueTargetPar$, but is not of intrinsic interest in its
own right.


\subsubsection{Allowed forms of adaptive data collection}

In order to estimate the target parameter $\TrueTargetPar$, we observe
a collection of $\Numobs$ samples, each of the form $(\iRegressor,
\iResponse, \iNuisance)$ for $i = 1, \ldots, \Numobs$.  We allow the
data collection to be sequentially dependent in the following way.
The samples define a nested sequence of $\sigma$-fields with
$\History_0 = \emptyset$, and
\begin{align}
\iHistory & = \sigma \Big( \{\Regressor_j, \Response_j, \Nuisance_j
\}_{j=1}^{i-1} \Big) \qquad \mbox{for each $i = 2, \ldots, \Numobs$.}
\end{align}
Let $\Pclass$ be a family of distributions on $\R^{\NuisanceDim}$.\footnote{For example, $\Pclass$ can be the set of all distributions on $[0,1]^{\NuisanceDim}$.}
At stage $i = 1, \ldots, \Numobs$, we assume that:
\begin{itemize}
  \item the distribution of the nuisance vector $\iNuisance$ conditioned on $\iHistory$ belongs to  $\Pclass$.
\item the choice of regressor $\iRegressor$ is determined according to
  a known selection function that maps pairs \mbox{$(\iNuisance,
    \iHistory)$} to probabilities \mbox{$\SelectProb_i(\iNuisance,
    \iHistory) \in [0,1]$.}
\end{itemize}
With a slight abuse of notation, we often adopt the shorthand
$\SelectProb_i$ for the function value $\SelectProb_i(\iNuisance,
\iHistory)$.  Throughout this paper, we assume that the selection
functions are known to us; for example, these functions could
correspond to policies in the setting of a contextual bandit.

\paragraph*{Structural assumptions}  Our analysis involves
some structural assumptions on the link function $g$, as well as the
space $\Xspace \subset \real^{\TargetDim}$ in which the covariates
lie.
\begin{itemize}
\item[(a)] In Theorem~\ref{thm:linear_new1} and Theorem~\ref{thm:linear_new3}, we
  provide guarantees for $g(x) = x$, in which case our general
  set-up~\eqref{eqn:general-setup} reduces to the setting of partial
  linear regression.
\item[(b)] In Theorem~\ref{thm:glm_new1_temp}
  and Theorem~\ref{thm:glm_new1_temp_one_arm}, we allow the function $g$
  to be non-linear, requiring only certain smoothness and
  identifiability conditions.
\item[(c)] Throughout the paper, we assume that the
  $\TargetDim$-dimensional regressor vector $\Regressor$ takes values
  in a discrete set that consists of an orthonormal basis of
  $\real^{\TargetDim}$, along with the all-zeros vector.  Without loss
  of generality---rotating as needed---we can assume that the
  orthonormal basis is the standard one $\{\stdBasis_1, \ldots,
  \stdBasis_{\TargetDim}\}$, where $\stdBasis_j \in
  \real^{\TargetDim}$ is the vector with a single one in coordinate
  $j$ (and zeros elsewhere).  This particular setting arises naturally
  for multi-armed bandits and treatment assignment problems.
\end{itemize}

Given the assumed structure of the covariates, the selection functions
are naturally viewed as \emph{selection probabilities}---that is, for
each $i = 1, \ldots, \numobs$ and $j = 1, \ldots, \TargetDim$
\begin{align}
\SelectProb_{ij} & \defn \E \big[ \Regressor_{ij} \mid \iHistory,
  \iNuisance \big]
\end{align}
is the conditional probability that $\iRegressor = e_j$. Thus, the
conditional probability of $\iRegressor = 0$ is given by
$\SelectProb_{i0} \defn 1 - \sum_{j=1}^{\TargetDim}\SelectProb_{ij}$.


\subsection{Guarantees for the partial linear model}
\label{sec:partial_linear}

This section is devoted to a special case of the general set-up:
choosing $g(\Regressor) = \Regressor$ leads to the \emph{partial
linear regression model}
\begin{align}
\label{eqn:partial-linear-model}
\iResponse & = \inprod{\iRegressor}{\TrueTargetPar}  + 
\TrueFunPar(\iNuisance)  +  \iNoise.
\end{align}
We assume that the target parameter $\TrueTargetPar$ lies in a bounded
 open subset $\TargetSpace \subset \real^{\TargetDim}$, whereas the
nuisance function $\TrueFunPar$ belongs to a function class
$\NuisanceSpace$ with bounded $\ell_\infty$-norm.


\subsubsection{Estimating the target parameter $\TrueTargetPar$}

Our procedure is a particular type of \mbox{$Z$-estimator,} in that we
compute the solution to a set of equations based on a
$\TargetDim$-dimensional score function.  Let us introduce some
notation required to define this score function.  The conditional
covariance of the regression vector $\iRegressor$, when conditioned
upon the pair $(\iHistory, \iNuisance)$, is given by
\begin{align}
\label{EqnDefnIcovar}  
\iCovar & \defn \E \big[ (\iRegressor - \SelectProb_i) (\iRegressor -
  \SelectProb_i)^\top \mid \iHistory, \iNuisance \big],
\end{align}
where $\SelectProb \in \real^{\TargetDim}$ is the vector of selection
probabilities previously defined.  Note that this matrix can be
computed at each time $i$, since the selection mechanism is known.
Using this random matrix, we then construct the \emph{score
function}\footnote{Strictly speaking, this score function
$\ScoreFun_i$ also depends on the quadruple $(\iResponse, \iRegressor,
\iNuisance, \iHistory)$, but we omit this dependence for notational
simplicity.}
\begin{align}
\label{eq:linear_score}
\ScoreFun_i(\TargetPar, \FunPar) & \defn \iCovar^{-1/2} (\iRegressor -
\SelectProb_i) \big \{ \iResponse - \inprod{\iRegressor}{\TargetPar} -
\FunPar(\iNuisance) \big \},
\end{align}
An important property of $\ScoreFun_i$ is that it is conditionally
mean zero---viz.
\begin{subequations}
\begin{align}
\label{eqn:score-zero-conditional-mean}
    \E \big[ \ScoreFun_i(\TrueTargetPar, \TrueFunPar) \mid \iHistory
      \big] & = 0.
\end{align}
Moreover, it satisfies the Neyman orthogonality condition,
\begin{align}
\label{eqn:score-nuisance-gradient}     
\E \big[ \partial_{\FunPar} \ScoreFun_i(\TrueTargetPar, \TrueFunPar)
  \{\FunPar - \TrueFunPar\} \mid \iHistory \big] & = 0 \qquad
\mbox{for any $h \in \NuisanceSpace$,}
\end{align}
\end{subequations}
where $\partial_{\FunPar}\ScoreFun_i$ is the Gateaux derivative.  See
the Supplement for more details on this derivative and the associated
orthogonality condition.

The conditional mean property~\eqref{eqn:score-zero-conditional-mean}
is needed to ensure consistency at the population level, whereas the
orthogonality condition~\eqref{eqn:score-nuisance-gradient} guarantees
that---again at the population level---the first-order effect of
perturbing the nuisance parameter vanishes.  With this intuition in
place, we introduce the \AlgShortPL~algorithm, a shorthand for
\emph{adaptive two-stage $Z$-estimation for the partially linear
model.}
\begin{algorithm}[ht!]
\caption{ $\;\;\;\;\;$ \AlgShortPL: partial linear
  model}\label{algo:DML-linear}
\begin{algorithmic}[1]
\STATE{Given $\Numobs$ samples $\big \{ (\iRegressor,
  \iNuisance, \iResponse) \big \}_{i = 1}^{\Numobs}$ from the
  partial linear model~\eqref{eqn:partial-linear-model}}.
\vspace{5pt}
\STATE{Define the index sets $\IndexOne \defn
  \{1,2, \ldots, \NumIndexOne\}$ and $\IndexTwo \defn \{
  \NumIndexOne + 1, \ldots, \Numobs \}$, and set $\NumIndexTwo \defn
  \Numobs - \NumIndexOne$.}
\STATE{ Compute an estimate $\EstFunPar$ of the nuisance function
  $\TrueFunPar$ based on the samples $\{(\iResponse,
  \iRegressor, \iNuisance)\}_{i \in \IndexOne}$.}
\STATE{Based on the samples $\{(\iResponse, \iRegressor,
  \iNuisance)\}_{i \in \IndexTwo}$, form the estimating equations
\begin{align}
\label{eqn:linear-estimating-eqn}
\frac{1}{\NumIndexTwo} \sum \limits_{i \in \IndexTwo}
\ScoreFun_i(\PlainPar, \EstFunPar) & = 0,
\end{align}
and compute a solution $\SoluTargetPar$.  }
\end{algorithmic} 
%
\begin{flushleft}
\noindent {\bf{Note:}} By the definition~\eqref{eq:linear_score} of the
score functions $\ScoreFun_i$, the estimating
equations~\eqref{eqn:linear-estimating-eqn} are linear in the
parameter $\PlainPar$; moreover, our analysis in proving
Theorem~\ref{thm:linear_new1} establishes that this linear system has
a unique solution $\SoluTargetPar$ with probability tending to one as
$\Numobs$ increases.
\noindent
\end{flushleft}
\end{algorithm}


\subsubsection{Asymptotic normality}

The main result of this section is an asymptotic normality guarantee
for the vector $\SoluTargetPar$ computed using the \AlgLongPL
algorithm.  We begin by stating our assumptions and discussing their
role in the theorem.

\myassumption{\mbox{{\bf{NOI}}$(\subgauss,
    \LinearNoiseVar)$}}{assn-lin-noise}{Conditioned upon $(\iRegressor,\iNuisance,
  \iHistory)$, each element of the zero-mean noise sequence
  $\{\iNoise\}_{i = 1}^\Numobs$ is sub-Gaussian with parameter
  $\subgauss$, and has conditional variance \mbox{${\LinearNoiseVar
      \defn \E[\iNoise^2 \mid \iRegressor, \iNuisance, \iHistory]}$.}  } 
\myassumption{{\bf{SEL}}$(t)$}{assn-lin-selection-prob}{ The selection
  probabilities $\SelectProb_{ij}$ at each round $i$ satisfy the lower
  bound
  \begin{align}
\label{eq:a3b}    
  \SelectProb_{ij} & \geq \frac{c_0}{i^{2t}} \quad \mbox{for all $j =
    0, 1, \ldots, \TargetDim$ and $i = 1, 2, \ldots$,}
\end{align}
  for some constant $c_0 > 0$ and exponent $t \in [0, \tfrac{1}{2})$.
  }
  
\myassumption{{\bf{\small{NUI}}}}{assn-lin-nuisance-est}{ Let
  $\Pclass$ be a family of distributions sufficiently rich to contain
  all possible distributions of $\iNuisance$ conditioned on
  $\iHistory$, for all $i \geq 1$.  The estimator $\EstFunPar$
  obtained from Step 3 of the \AlgLongPL procedure satisfies
\begin{align}
\label{eq:ass_nuisance_est}
\sup_{P \in \Pclass}(\E_{\dummyRV\sim
  P}|\EstFunPar(\dummyRV) - \TrueFunPar(\dummyRV)|^2)^{1/2}=o_p(1).
\end{align}
}

Let us clarify the meaning and significance of these assumptions.  The
\emph{noise condition}~\ref{assn-lin-noise} allows us to control the
tail behavior of the noise, and is relatively standard though can be
relaxed\footnote{See Supplement for more details.}.  More interesting
is the \emph{selection condition}~\ref{assn-lin-selection-prob}, which
allows the minimum selection probability to decrease as fast as
$\Numobs^{-2t}$ for some $t \in [0, 1/2)$. This is slightly more
  relaxed than those in some past works, such as requiring that the
  selection probabilities be uniformly bounded away from
  zero~\cite{zhang2021statistical}; or converge to some non-random
  limit~\cite{hadad2021confidence,zhan2021off}.  Finally, the
  \emph{nuisance condition}~\ref{assn-lin-nuisance-est} guarantees
  that the estimate $\EstFunPar$ based on the hold-out set is a
  weakly-consistent estimator for the true nuisance function
  $\TrueFunPar$.  {In practice, one can use various procedures to
    estimate $\TrueFunPar$ (e.g., $k$-nearest neighbor estimators,
    random forests, boosting, kernel methods and neural networks). }
  \medskip
  
\noindent With this set-up, we now state our first main result:
\begin{theorem}
\label{thm:linear_new1} 
Suppose that Assumptions~\ref{assn-lin-noise},
~\ref{assn-lin-selection-prob} and~\ref{assn-lin-nuisance-est} are in
force. Then the estimate~$\SoluTargetPar$ obtained from \AlgLongPL
(Algorithm~\ref{algo:DML-linear}) satisfies
\begin{align}
\label{eq:linear_new1_result}
(\sqrt{\NumIndexTwo} \EmpMean \iCovar^{1/2})
(\SoluTargetPar - \TrueTargetPar) &  \convdist 
\Normal(0, \LinearNoiseVar \Id_{\TargetDim}).
\end{align}  
\end{theorem}
\noindent See Section~\ref{proof:thm:linear_new1} in the Supplement for the proof. \\

\noindent A few comments regarding this claim are in order.
\vspace{5pt}

\paragraph*{IID nuisance} Finding a suitable choice of $\Pclass$
for verifying the condition~\eqref{eq:ass_nuisance_est} is non-trivial
in general.  However, when the nuisances $\iNuisance$ are i.i.d. and
independent of $\iHistory$, this condition reduces to $(\E
|\EstFunPar(\iNuisance) - \TrueFunPar(\iNuisance) |^2)^{1/2} =
o_p(1)$, and so is concrete and explicit.

\paragraph*{Linear nuisance function} Suppose that the
nuisance function is linear in $\Nuisance$---that is, say ${
  \TrueFunPar(\Nuisance) = \inprod{\Nuisance}{\TrueNuisancePar}}$ for
some $\TrueNuisancePar \in \R^{\NuisanceDim}$---and that $\E
\|\iNuisance\|_2^2\leq M_\Nuisance < \infty$ for all $i \geq 1$.
Under these conditions, given an estimate $\EstNuisancePar$ with
\mbox{$\vtwonorm{\EstNuisancePar - \TrueNuisancePar} =o_p(1)$,} it
follows that Assumption~\ref{assn-lin-nuisance-est} holds with
$\EstFunPar(\Nuisance) = \inprod{ \Nuisance}{\EstNuisancePar}$, and
$\mathcal P$ given by the set of all distributions with second moment
at most $M_\Nuisance$.


{%
\paragraph*{Extension to continuous regressors} 
As stated, Theorem~\ref{thm:linear_new1} applies to regressors
$\iRegressor$ taking values in the finite cardinality set $\{
\zerovec, e_1, \ldots, e_{\TargetDim}\}$. However, an analogous result can be proved for continuous-valued
regressors as well.  Concretely, suppose that the regressors take values in
the $\ell_2$-ball $\{\Regressor \in \R^{\TargetDim} \, \mid \,
\vecnorm{\Regressor}{2} \leq 1 \}$ according to some known probability
density.  Recalling that $\iCovar$ denotes the conditional covariance
matrix of $x_i$ from equation~\eqref{EqnDefnIcovar}, say that
Assumption~\ref{assn-lin-selection-prob} is replaced by the condition
that $\iCovar\succeq c_0 i^{-2t}$ for all $i$ for some exponent $t \in
[0, 1/2)$ and pre-factor $c_0 > 0$.  Under these conditions, the claim
  of Theorem~\ref{thm:linear_new1} remains valid.  We refer the reader
  to Section~\ref{proof:thm:linear_new1} in the Supplement for a more in-depth
  discussion.  }



\paragraph*{Computational complexity} Note that the matrix
$\iCovar$ is the covariance of a multinomial distribution, and a
Cholesky decomposition of such matrices can be carried out in
$\bigoh(\TargetDim^2)$ time.  Therefore, the time complexity of
setting up the estimating equations~\eqref{eqn:linear-estimating-eqn}
scales $\bigoh(\Numobs\TargetDim^2)$. Solving the system of linear
equations requires at most $\bigoh(\TargetDim^2)$ time.


\paragraph*{Inference for the target parameter} 

From Theorem~\ref{thm:linear_new1}, we can construct a confidence region for
the whole parameter vector (e.g., by a $\chi^2$-test). In addition, if the
sequence of random matrices $\EmpMean [\iCovar^{1/2}]$ converge to
some non-random and invertible matrix---say $\GamMat^{1/2}$---then
equation~\eqref{eq:linear_new1_result} implies that
$\sqrt{\NumIndexTwo}(\EstTargetPar - \TrueTargetPar)$ is
asymptotically normal with covariance $\LinearNoiseVar\GamMat^{-1}$.

{\paragraph*{Estimation of the variance $\LinearNoiseVar$} When
  $\LinearNoiseVar$ is unknown, it needs to be estimated.  If the
  sample sizes satisfy the lower bound \mbox{$\NumIndexTwo\geq c
    \Numobs$} for some constant $c > 0$, a consistent estimate is
  given by the plug-in
  \begin{align}
  \label{eq:consistent_var}
    \EstLinearNoiseVar & \defn \EmpMean(\iResponse-\iRegressor^\top
    \SoluTargetPar-\EstFunPar(\iNuisance))^2.
\end{align}
  More precisely, we  have
  $\EstLinearNoiseVar{\to}\LinearNoiseVar$ in probability whenever, in
  addition to the conditions in Theorem~\ref{thm:linear_new1}, the
  fourth moments $\sup_{P \in \Pclass}\E_{\dummyRV\sim
    P}|\EstFunPar(\dummyRV) - \TrueFunPar(\dummyRV)|^4$ and
  \mbox{$\E[\iNoise^4 \mid \iRegressor, \iNuisance, \iHistory]$} are
  bounded by some constant.  See the end of
  Section~\ref{proof:thm:linear_new1} in the Supplement for the proof of this claim.}


{\paragraph*{Adaptive estimation of the nuisance function} The
  procedure described here is based on sample splitting, with the
  first $\NumIndexOne$ samples used to estimate the nuisance
  $\TrueFunPar$.  An alternative
  approach is to sequentially update the estimate $\EstFunPar$ so as
  to achieve better sample efficiency. Namely, instead of solving
  equation~\eqref{eqn:linear-estimating-eqn}, we find $\SoluTargetPar$
  by solving
\begin{align*}
\frac{1}{\Numobs} \sum \limits_{i=1}^\Numobs \ScoreFun_i(\PlainPar,
\EstFunPar_i) & = 0,
\end{align*}
where $\EstFunPar_i$ are nuisance estimates using samples
$\{(\Response_{j},\Regressor_j,\Nuisance_j)\}_{j=1}^{i-1}$.  It can
shown that, under the conditions of Theorem~\ref{thm:linear_new1} and
when the sequence of nuisance estimates satisfy the limiting relation
$\sum_{i=1}^\Numobs\E_{\EstFunPar_i,\Nuisance_i}
(\EstFunPar_i(\iNuisance) - \TrueFunPar(\iNuisance))^2/n \to 0$, then
we have
\begin{align*}
(\sqrt{\Numobs} {\widehat\E_{\Numobs}} \iCovar^{1/2}) (\SoluTargetPar
  - \TrueTargetPar) & \convdist \Normal(0, \LinearNoiseVar
  \Id_{\TargetDim}).
\end{align*}
Intuitively, one might expect good empirical behavior for this
approach since the variance of $\SoluTargetPar$ scales as $1/\Numobs$;
on the flip side, it could be computationally more expensive.  We
refer readers to Section~\ref{SecAdaNuisance} in the Supplement for
more details.  }

{\paragraph*{Inference with unknown selection probabilities}
  Theorem~\ref{thm:linear_new1} can also be generalized to the
  scenario where 
   the exact values of the selection
  probabilities $\SelectProb_i$ are unknown, but  only 
  consistent estimates $\EstSelectProb_i$  are available.  See
  Section~\ref{SecEstSelectProb} in the Supplement for details.  }

\subsection{Fixed direction inference for the partial linear model}
\label{sec:One-coordinate-inference}
In many applications, one is only interested in estimating linear
functionals of the target parameter vector.  Concretely, given a
unit-norm vector $\OneDirect\in\R^{\TargetDim}$, consider the problem
of providing confidence intervals for the scalar target
$\TrueOneDirectPar \defn \inprod{\OneDirect}{\TrueTargetPar}$;
standard examples include the first coordinate $\TrueTargetPar_{1}$,
or the difference between two coordinates $\TrueTargetPar_1 -
\TrueTargetPar_2$.  We will show that inferential
guarantees for such scalar quantities can be obtained under much weaker
conditions than Theorem~\ref{thm:linear_new1}. Namely, we only require
Assumption~\ref{assn-lin-selection-prob} to hold for coordinates $j$
for which $\OneDirect_j$ is non-zero.


\subsubsection{Constructing the score function}

Suppose that we use the dataset $\big \{ (\iRegressor, \iNuisance,
\iResponse) \big\}_{i=1}^{\NumIndexOne}$ to compute an initial pair of
``crude'' estimates $\EstTargetPar$ and $\EstFunPar$.  Recalling that
$\inprod{\cdot}{\cdot}$ denotes the Euclidean inner product, we
consider the one-dimensional score function
\begin{subequations}
\begin{align}
\label{eq:one-arm-score-fun}    
\ScoreFun_{i1}(\OneDirectPar, \EstTargetPar, \EstFunPar) = \inprod{
  \ScalePreCondVec_{i1}}{\iRegressor - \SelectProb_i} \big \{
\iResponse - \inprod{\iRegressor}{\OneDirect} \OneDirectPar -
\iRegressor^\top (\IdMat_{\TargetDim} - \OneDirect \OneDirect^\top)
\EstTargetPar - \EstFunPar(\iNuisance) \big \},
\end{align}
where the vector $\ScalePreCondVec_{i1}$ is given by
\begin{align*}
\ScalePreCondVec_{i1} = \big( \Covar_{i}^{-1} \OneDirect \big)
\tfrac{1}{ \sqrt{\OneDirect^\top{\Covar^{-1}_{i}} \OneDirect}},
\end{align*}
and the inverse covariance matrix $\iCovar^{-1}$ admits the explicit
expression
\begin{align}
\label{eq:cov_inverse}  
\iCovar^{-1}(\iNuisance, \iHistory) & = \begin{pmatrix}
  \frac{1}{\SelectProb_{i1}} + \gamma_i & \gamma_i &\gamma_i & \cdots
  & \gamma_i \\
\gamma_i & \frac{1}{\SelectProb_{i2}} + \gamma_i &\gamma_i &\cdots &
\gamma_i \\
\gamma_i & \gamma_i&\frac{1}{\SelectProb_{i3}} + \gamma_i &\cdots &
\gamma_i \\
\vdots & \vdots & \vdots & \vdots & \vdots \\ \gamma_i & \gamma_i &
\cdots & \gamma_i & \frac{1}{\SelectProb_{i\TargetDim}} + \gamma_i
\end{pmatrix} \qquad \mbox{where $\gamma_i = 1/\SelectProb_{i0}$.}
\end{align}
This choice of $\ScalePreCondVec_{i1}$ allows us to stabilize the
variance of the score function: concretely, we have $\E
|\inprod{\ScalePreCondVec_{i1}}{\iRegressor - \SelectProb_i}|^2 = 1$.
Our next step is to find $\SoluOneDirectPar$ by solving the linear
system
\begin{align}
\label{eq:one-arm-est-eq2}  
\frac{1}{\NumIndexTwo } \sum_{i=\NumIndexOne + 1}^{\NumIndexTwo }
\ScoreFun_{i1}(\SoluOneDirectPar, \EstTargetPar, \EstFunPar ) = 0
\end{align}
\end{subequations}

\subsubsection{Guarantee of asymptotic normality}

We are now ready to establish a guarantee for the estimate
$\TrueOneDirectPar$.  We do so under the following weaker variant of
our earlier selection condition~\ref{assn-lin-selection-prob}:
\myassumption{{\bf{SEL}}$^\ast(t,u,
  \SuppDirect)$}{assn-lin-selection-prob-weak}{ For some \mbox{$t \in
    [0, \tfrac{1}{2})$,} the selection probabilities are lower bounded
    as
\begin{align}
\label{eq:a4}      
\SelectProb_{ij} \succeq \frac{c_0}{i^{2t}} \quad \mbox{for all $j \in
  \SuppDirect \cup \{0\}$, and for all $i = 1, 2 \ldots$,}
\end{align}
where $\SuppDirect \defn \{ j \mid \OneDirect_j \neq 0\}$ is the
support set of $\OneDirect$.}

\vspace{8pt}

Compared to condition~\ref{assn-lin-selection-prob},
Assumption~\ref{assn-lin-selection-prob-weak} is weaker in the sense
that the lower bound condition is imposed \emph{only} on the support
set of the vector $\OneDirect$, along with the reference point (the
all-zeroes vector).  This difference is significant, for example, when
our goal is to estimate a single coordinate, or the difference of two
coordinates.
\begin{theorem}
\label{thm:linear_new3}
Suppose that Assumptions~\ref{assn-lin-noise},
~\ref{assn-lin-selection-prob-weak} and~\ref{assn-lin-nuisance-est}
are in force.  Then the $Z$-estimate $\SoluOneDirectPar$ computed
from~\eqref{eq:one-arm-est-eq2} using any consistent estimate
$\EstTargetPar$ of $\TargetPar$ satisfies
\begin{align*}
\Big( \EmpMean \tfrac{1}{\sqrt{\OneDirect^\top \iCovar^{-1}
    \OneDirect}} \Big) (\SoluOneDirectPar - \TrueOneDirectPar) \convdist
\Normal(0, \LinearNoiseVar ).
\end{align*}
\end{theorem}
\noindent See Section~\ref{proof:thm:linear_new3} in the Supplement for the proof. \\

A few comments regarding Theorem~\ref{thm:linear_new3} are in order.
First, its guarantees hold under a weaker assumption on the selection
probability, albeit at the expense of assuming the \emph{a priori}
existence of a consistent estimator of $\EstTargetPar$. However, since
typically we estimate $\TrueTargetPar$ and $\TrueFunPar$
simultaneously in the partial linear model, we would also obtain a
consistent estimator of $\TrueTargetPar$ if we can find a consistent
estimator of $\TrueFunPar$
(cf. condition~\eqref{eq:ass_nuisance_est}).

Second, suppose that the nuisance function is linear---i.e.,
$\TrueFunPar(\Nuisance) = \inprod{\Nuisance}{\TrueNuisancePar}$ for
some \mbox{$\TrueNuisancePar \in \R^{\NuisanceDim}$.} Similar
to Theorem~\ref{thm:linear_new1}, let $\EstNuisancePar$ be an estimator of
$\TrueNuisancePar$ with \mbox{$\|\EstNuisancePar -
  \TrueNuisancePar\|_2 = o_p(1)$} and assume that $\sup_{i}\E
\|\iNuisance\|_2^2\leq M_\Nuisance<\infty$, then
Assumption~\ref{assn-lin-nuisance-est} is satisfied with
$\EstFunPar(\Nuisance) = \inprod{\Nuisance}{\EstNuisancePar}$ and
$\Pclass$ be the set of distributions with the second moment less than
$M_\Nuisance$.

Observe that Theorem~\ref{thm:linear_new3} allows us to construct an
asymptotically valid level-$\alpha$ confidence interval for
$\TrueOneDirectPar$. Specifically, we have
\begin{align*}
\lim_{\Numobs \to \infty} \Prob \left[\SoluOneDirectPar - \tfrac{q_{1
      - \alpha/2} \sigma}{\sqrt{\NumIndexTwo}} \Big( \EmpMean
  \tfrac{1}{\sqrt{\OneDirect^\top \iCovar^{-1} \OneDirect}} \Big)^{-1}
  \leq \TrueOneDirectPar \leq \SoluOneDirectPar + \tfrac{q_{1 -
      \alpha/2} \sigma}{\sqrt{\NumIndexTwo}} \Big( \EmpMean \tfrac{1}{
    \sqrt{\OneDirect^\top \iCovar^{-1} \OneDirect}} \Big)^{-1} \right]
= 1 - \alpha,
\end{align*}
where $q_{1 - \alpha/2}$ is the $1 - \alpha/2$ quantile of the
standard normal distribution. In particular, if we are interested in
the first co-ordinate $\TrueTargetPar_1$, then setting
$\OneDirect=e_1$ and applying Theorem~\ref{thm:linear_new3} yields
\begin{align*}
\Big( \EmpMean \sqrt{\frac{\SelectProb_{i0}
    \SelectProb_{i1}}{\SelectProb_{i0} + \SelectProb_{i1}}}
\Big)(\SoluTargetPar_1 - \TrueTargetPar_1) \convdist \Normal(0,
\LinearNoiseVar).
\end{align*}

Third, although we have stated the result with $\EstTargetPar$ assumed
to be consistent for the full vector $\TargetPar$, in fact, we require
only that that is consistent for any direction that is orthogonal to
$\OneDirect$, i.e., it suffices to have the slightly weaker
consistency condition $(\IdMat_{\TargetDim} -
\OneDirect\OneDirect^\top)(\EstTargetPar - \TrueTargetPar)
\overset{p}{\to} 0$.

{ Finally, Theorem~\ref{thm:linear_new3} can also be generalized
  to the continuous regressors case as follows.  Suppose that the
  regressors take values in the $\ell_2$-ball $\{\Regressor \in
  \mathbb{R}^{\TargetDim} \mid \vecnorm{\Regressor}{2} \leq 1 \}$
  according to some known probability density.  Then the same
  guarantee holds if we replace
  Assumption~\ref{assn-lin-selection-prob-weak} with the condition
  that there is some exponent $t \in [0, 1/2)$ and pre-factor $c_0 >
    0$ such that ${\vecnorm{v}{2}}/{\sqrt{v^\top \iCovar^{-1} v}}
    \succeq c_0 i^{-t}$ for $v \in \{\OneDirect, \iCovar^{-1/2}
    \OneDirect\}$ for $i = 1, 2, \ldots$.  See
    Section~\ref{proof:thm:linear_new3} in the Supplement for a more detailed
    discussion. }

\subsection{Generalized linear model} 
We now return to the general setting, in which we have a model of the
form
\begin{align}
\label{eq:generalized-linear-model}  
\iResponse & = g \big( \inprod{\iRegressor}{\TrueTargetPar}  + 
\TrueFunPar(\iNuisance) \big)  +  \iNoise,
\end{align}
for a general inverse link function $g$.  We assume that the parameter
$(\TrueTargetPar, \TrueFunPar) \in \TargetSpace
\times \NuisanceSpace$, where the parameter space $\TargetSpace$ is a
bounded open set in $\R^{\TargetDim}$ and $\NuisanceSpace$ is a set of
functions with bounded $\ell_\infty$-norm.

\subsubsection{Estimating the target parameter $\TrueTargetPar$}

We start by constructing a different score function.

introduce an auxiliary nuisance vector $\AuxiNuisancePar$, and define
the score function
\begin{align}
\label{eq:glm_score}  
\ScoreFun_i(\TargetPar, \AuxiNuisancePar, \FunPar) \equiv
\InvSqrtCovar_{i}(\iRegressor - \GlmMeanVec_i) \; \Big \{ \iResponse -
g \big(\inprod{\iRegressor}{\TargetPar}  +  \FunPar(\iNuisance) \big)
\Big \}
\end{align} where 
\begin{subequations}
\begin{align}
\GlmMeanVec_i &\equiv \E((\iRegressor
g'\big(\inprod{\iRegressor}{\AuxiNuisancePar} +
\FunPar(\iNuisance)\big)|\iNuisance, \iHistory)[\E(
  g'\big(\inprod{\iRegressor}{\AuxiNuisancePar} +
  \FunPar(\iNuisance)\big)|\iNuisance, \iHistory)]^{-1},
\label{eqn:glm-mean-vector}
\\
\InvSqrtCovar_i &\equiv {} [\E(\iNoise^2(\iRegressor -
  \GlmMeanVec_i)(\iRegressor - \GlmMeanVec_i)^\top|\iNuisance,
  \iHistory)]^{-1/2} \notag \\ &= [\E(\GlmVar
  \big(g\big(\inprod{\iRegressor}{\AuxiNuisancePar} +
  \FunPar(\iNuisance)\big)\big)(\iRegressor -
  \GlmMeanVec_i)(\iRegressor - \GlmMeanVec_i)^\top|\iNuisance,
  \iHistory)]^{-1/2}
\label{eqn:glm-covmat}
\end{align} 
\end{subequations}
and $\GlmVar(x)\equiv\E
(\iNoise^2|g\big(\inprod{\iRegressor}{\TrueTargetPar} +
\TrueFunPar(\iNuisance)\big)=x)$ is the conditional variance of the
noise $\iNoise$.  When $\AuxiNuisancePar=\TrueTargetPar
(\text{or\,}\EstTargetPar) $ and $\FunPar=\TrueFunPar
(\text{or\,}\EstFunPar) $, we denote the corresponding $\GlmMeanVec_i$
and $\InvSqrtCovar_i$ by
$\GlmMeanVec^*_i~(\text{or\,}\GlmMeanVecwhat_i)$ and
$\InvSqrtCovar^*_i~(\text{or\,}\InvSqrtCovarwhat_i)$
respectively. Intuitively speaking, the vector $\GlmMeanVec_i$ can be
viewed as a weighted conditional expectation of the regressor
$\iRegressor$, while the matrix $\InvSqrtCovar_i$ can be viewed the
inverse square root of a weighted conditional covariance matrix of
$\iRegressor$. When $g(x)=x$, $\GlmMeanVec, \InvSqrtCovar$ does not
depend on $\AuxiNuisancePar$ and the score function in
equation~\eqref{eq:glm_score} reduces to the early one in
equation~\eqref{eq:linear_score} for the partial linear model.

Similar to the partial linear model case, this score function
satisfies a version of the Neyman orthogonality condition. More
specifically, we have
\begin{subequations}
\begin{align}
    \E(\ScoreFun(\TrueTargetPar, \TrueTargetPar, \TrueFunPar
    )|\iHistory) & = 0\label{eq:score_exp},
    \\
\E(\partial_{\AuxiNuisancePar}\ScoreFun_i(\TrueTargetPar ,
\TrueTargetPar, \TrueFunPar )|\iHistory)& = 0,
\text{and} \label{eq:no_theta01} \\
\E(\partial_{\FunPar}\ScoreFun_i(\TrueTargetPar, \TrueTargetPar,
\TrueFunPar )[ \FunPar - \TrueFunPar]|\iHistory) &=0 \text{~~for
  all~~} \FunPar\in\NuisanceSpace. \label{eq:no_theta012}
\end{align}
\end{subequations}
We defer the proof of these equations to Supplement.

With this set-up, we estimate the target parameter $\TrueTargetPar$
using the following \AlgLongGLM procedure, or \emph{adaptive two-stage
$Z$-estimation for the generalized linear model}.

\begin{algorithm}[ht!]
\caption{ $\;\;\;\;\;$ \AlgShortGLM: generalized linear
  model}\label{algo:DML-glm}
\begin{algorithmic}[1]
\STATE{Given $\Numobs$ samples $\big \{
  (\iRegressor, \iNuisance, \iResponse) \big \}_{i = 1}^{\Numobs}$ from
  the partial linear model~\eqref{eq:generalized-linear-model}}.
\vspace{5pt}
\STATE{Define the index sets $\IndexOne \defn \{1,2, \ldots,
  \NumIndexOne\}$ and $\IndexTwo \defn \{ \NumIndexOne + 1, \ldots,
  \Numobs \}$, and define $\NumIndexTwo = \Numobs - \NumIndexOne$.}
\STATE{Use samples $\{(\iResponse, \iRegressor, \iNuisance)_{i\in
    \IndexOne}\}$ to obtain an estimate $\EstFunPar$ for the nuisance
  function $\TrueFunPar$ and $\EstTargetPar$ for the target parameter
  $\TrueTargetPar$.}
\STATE{Find $\SoluTargetPar$ by solving the equation
\begin{align}
\label{eqn:non-linear-estimating-eqn}
\frac{1}{\NumIndexTwo} \sum \limits_{i \in \IndexTwo}
\ScoreFun_i(\SoluTargetPar, \EstTargetPar, \EstFunPar)=0
\end{align}
}.
\end{algorithmic}

\end{algorithm}

\subsubsection{Asymptotic normality}

We now turn to a result on the asymptotic normality of the estimator
$\SoluTargetPar$ computed using the \AlgLongGLM procedure described
as Algorithm~\ref{algo:DML-glm}.  Let us begin with the underlying assumptions.

\myassumption{{\bf{SEL}}$^\prime(t, \delta)$}{assn-glm-prob}{For some
  $t \in [0, \tfrac{1}{4}]$ and $\delta > 0$, the selection
  probabilities $\SelectProb_{ik}$ are lower bounded as
\begin{align}
\label{eq:glm_assn_prob}    
\SelectProb_{ik} \geq c_i \defn \frac{c_0}{i^{2(t - \delta)}} \quad
\mbox{for all $k = 1, \ldots, \TargetDim$, and $i = 1, 2, \ldots$.}
\end{align}
In addition, the probability of selecting the zero vector satisfies
$\SelectProb_{i0}\geq \czerotil$ for some $\czerotil>0$.  }


\myassumption{{\bf{NUI}}$^\prime$}{assn-glm-nuisance-est}{ Suppose
  that all distributions in $\Pclass$ are supported on a set
  $\mathrm{dom}(\Pclass) \subseteq \R^{\NuisanceDim}$. The estimators
  $\EstTargetPar, \EstFunPar$ obtained in Step 3 of
  Algorithm~\ref{algo:DML-glm} satisfy \mbox{$\|\EstTargetPar -
    \TrueTargetPar\|_2=o_p(\Numobs^{-1/4})$}, and $\sup
  \limits_{\dummyRV\in\mathrm{dom}(\Pclass)}|\EstFunPar(\dummyRV) -
  \TrueFunPar(\dummyRV)|=o_p(\Numobs^{-1/4})$.  }

\myassumption{{\bf{IDE}}}{assn-glm-identifiability}{The model is
  identifiable under our choice of the score function, concretely,
\begin{align*}
 \|\CondMean (\ScoreFun_i(\TargetPar, \TrueTargetPar, \TrueFunPar
 ) - \ScoreFun_i(\TrueTargetPar, \TrueTargetPar, \TrueFunPar ))\|_2 & \geq
c_\ScoreFun \|\CondMean \partial_{\TargetPar}
 \ScoreFun_i(\TrueTargetPar, \TrueTargetPar, \TrueFunPar
 )(\TargetPar - \TrueTargetPar )\|_2\wedge c_\ScoreFun \Numobs^{-1/4}
\end{align*}
almost surely for any $\TargetPar \in \R^{\TargetDim}$ and some constant
$c_\ScoreFun > 0$ .  }

\myassumption{{\bf{\small{EIG}}}}{assn-glm-minimum-eig}{ The minimum singular value
  of the gradient $\CondMean \partial_{\TargetPar}
  \ScoreFun_i(
  \TrueTargetPar, \TrueTargetPar, \TrueFunPar )$ is not too small,
  namely,
\begin{align*}
\lim_{\Numobs\to\infty} \prob(\sigma_{\min}(\CondMean
\InvSqrtCovar^*_{i}(\iRegressor - \GlmMeanVec^*_i)
g'\big(\inprod{\iRegressor}{ \TrueTargetPar} + \TrueFunPar(\iNuisance)
\big)(\iRegressor - \GlmMeanVec^*_i)^\top) \geq
m_{\ScoreFun}\Numobs^{\delta-t}) & \to 1
\end{align*}
for some constant \mbox{$m_{\ScoreFun} > 0$.}  }


\subsubsection{Other standard GLM assumptions}
\label{sec:standard-glm-assumption}
In addition to the above four assumptions, we make the additional
assumptions on our generalized linear model.
\begin{itemize}
\item There exist constants $M_\TargetPar$ and $D_x > 0$ such that
  $\sup_{\TargetPar \in \TargetSpace} \|\TargetPar\|_2 \leq
  M_\TargetPar$, $\|\iRegressor\|_2 \leq D_x$, and $\TrueFunPar$
  satisfies $\|\TrueFunPar(\iNuisance)\|_\infty\leq M_\FunPar$ for
  some $M_\FunPar > 0$.
  %
  \item The conditional variance $\GlmVar(x)$ is three-times
    differentiable, $\GlmVar(x), \GlmVarDeri(x)$ are
    $L_\eps,L_{\eps'}$-Lipschitz respectively for $|x|\leq M_\FunPar +
    D_\Regressor M_\TargetPar$, and there exist some $M_\eps,m_\eps>0$
    such that $M_\eps\geq \GlmVar(x)\geq m_\eps$ for $|x|\leq
    M_\FunPar + D_\Regressor M_\TargetPar$. Furthermore, we assume
    that the zero mean noise $\iNoise$ is sub-Gaussian with parameter
    $\subgauss$ conditioned $\iRegressor, \iNuisance, \iHistory$.
  \item The inverse link function $g$ is three-times differentiable,
    monotone and the functions $g, g', g^{''}$ are $L_g, L_{g'},
    L_{g^{''}}$-Lipschitz continuous, respectively. Moreover,
    \\ \mbox{$\inf_{|x|\leq M_\FunPar + D_\Regressor
        M_\TargetPar}|g'(x)|\geq l_{g}$} for some $l_g > 0$.
\end{itemize}

Assumption~\ref{assn-glm-prob} is slightly stronger than
Assumption~\ref{assn-lin-selection-prob} in the sense that we need to
replace $t$ by $t - \delta$ for some small constant $\delta$ and
restrict $t\in[0,1/4]$.  Assumption~\ref{assn-glm-nuisance-est} is
made on the performance of the pilot estimators. The reason we assume
$\EstTargetPar - \TrueTargetPar, \sup|\EstFunPar(\dummyRV) -
\TrueFunPar(\dummyRV)|=\liloh_p(\Numobs^{-1/4})$ is to ensure
second-order terms in the Taylor expansion of the inverse link $g$
vanish.  
{In contrast, in partial linear models,
  Assumption~\ref{assn-lin-nuisance-est} only requires the nuisance
  estimator to be consistent. Since the first-order Taylor
  approximation of $g$ is exact in the linear case, no assumptions on
  convergence speed are needed to eliminate the approximation error
  terms in Taylor expansion.  }

The conditions~\ref{assn-glm-identifiability}
and~\ref{assn-glm-minimum-eig} ensure that the expectation of score
function is sufficiently away from zero when $\TargetPar$ is away from
$\TrueTargetPar$. In the simple scenario where $g(\state) = \state$,
Assumption~\ref{assn-glm-identifiability}
and~\ref{assn-glm-minimum-eig} are implied by the rest
assumptions. Also, it can be shown that in logistic regression
Assumption~\ref{assn-glm-minimum-eig} holds, and Assumption~\ref{assn-glm-identifiability} holds when $\TargetDim=1$
(see Supplement
for detailed
derivations). However, due to the adaptive nature of the collected
data, it is in general hard to verify these two assumptions. To
address this issue, in practice, we suggest verifying them with all
$\CondMean$ replaced by $\EmpMean$ instead. Since the empirical mean
concentrates around the conditional expectation, the empirical version
of Assumption~\ref{assn-glm-identifiability}
and~\ref{assn-glm-minimum-eig} hold with high probability when the
assumptions themselves are true. Therefore, we may use the empirical
version as a surrogate for the original assumptions.

Finally, in Section~\ref{sec:standard-glm-assumption} we enlist some
standard assumptions on the GLM. The first condition assumes
boundedness condition in the regressors, the parameter space, and the
true nonlinear function. The second and third condition respectively
puts some smoothness condition on the conditional variance
functional~$\GlmVar(\cdot)$ and the link function~\linkFun.

\begin{theorem}
  \label{thm:glm_new1_temp}
Suppose that
Assumptions~\mbox{\ref{assn-glm-prob}---\ref{assn-glm-minimum-eig}}
and the standard GLM assumptions
from Section~\ref{sec:standard-glm-assumption} are in force.  Then the
estimate $\SoluTargetPar$ obtained from \AlgLongGLM\\
(cf. Algorithm~\ref{algo:DML-glm}) satisfies
 \begin{align}
\label{eq:glm_asym_norm0_temp}   
(\EmpMean \InvSqrtCovarwhat_{i}(\iRegressor - \GlmMeanVecwhat_i)
g'\big( \inprod{\iRegressor}{\EstTargetPar} + \EstFunPar(\iNuisance)
\big) (\iRegressor - \GlmMeanVecwhat_i)^\top) \sqrt{\NumIndexTwo}
(\SoluTargetPar - \TrueTargetPar) \convdist \Normal(0,
\IdMat_{\TargetDim}).
\end{align}
\end{theorem}
\noindent See Section~\ref{proof:thm:glm_new1_temp} in the Supplement for the proof.

\vspace{8pt}

A simple case is when the nuisance function is linear, i.e.,
$\TrueFunPar(\Nuisance)=\inprod{\Nuisance}{\TrueNuisancePar}$ for some
$\TrueNuisancePar\in\NuisancePar$, and $\NuisanceSpace$ is a bounded
set in $\R^{\NuisanceDim}$. In this case, the assumptions from
Section~\ref{sec:standard-glm-assumption} hold if there exist
$M_{\NuisancePar},D_{x}>0$ such that
$\sup_{\NuisancePar\in\NuisanceSpace}\|\NuisancePar\|_2\leq
M_{\NuisancePar}$ and $\|(\iRegressor^\top,
\iNuisance^\top)^\top\|_2\leq D_{x}$. Moreover,
Assumption~\ref{assn-glm-nuisance-est} is satisfied if in addition we
have an estimator $\EstNuisancePar\in\NuisanceSpace$ such that
$\|\EstNuisancePar - \TrueNuisancePar\|_2=\liloh_p(\Numobs^{-1/4})$.

Similar to the partial linear model discussed
in Section~\ref{sec:partial_linear}, we can adaptively estimate the nuisance function $\TrueFunPar$ to achieve better sample efficiency. Also,   Theorem~\ref{thm:glm_new1_temp} allows us to
construct a confidence region for the parameter vector
$\TrueTargetPar$ via a $\chi^2-$ test. Moreover,  if the weighted matrix on
the L.H.S. of equation~\eqref{eq:glm_asym_norm0_temp} converges, 
 then
$\sqrt{\NumIndexTwo}(\SoluTargetPar - \TrueTargetPar)$ is
asymptotically normal, and we can construct a confidence region
(interval) for any subset of the parameter vector $\TrueTargetPar$. In
absence of such convergence, it becomes challenging to construct confidence regions
for fixed directions of $\TrueTargetPar$, i.e., $ \inprod{\OneDirect}{
  \TrueTargetPar}$ in general, without relying on any additional assumption (e.g., a strong Gaussian approximation version of~equation~\ref{eq:glm_asym_norm0_temp}); see Section 3.2.2 in the
paper~\cite{khamaru2021near} for a detailed argument.

Nonetheless, we can provide an asymptotically normal estimate for $
\inprod{\OneDirect}{\TrueTargetPar}$ using a variant of the estimator
discussed in this section. Interestingly, when we are interested only
in confidence intervals for $ \inprod{\OneDirect}{\TrueTargetPar}$ for
a fixed direction $\OneDirect$, we can weaken the conditions
of Theorem~\ref{thm:glm_new1_temp}. We discuss the conditions in details in
our next section.




\subsection{Fixed direction inference for the GLM}

For any direction $\OneDirect\in\R^{\TargetDim}$ such that
$\|\OneDirect\|_2=1$, we can construct a one-dimensional score
function and obtain an asymptotically normal estimator for
$\TrueOneDirectPar \defn \inprod{\OneDirect}{\TrueTargetPar}$. Our
construction follows the same idea as in
equation~\eqref{eq:one-arm-score-fun}. Specifically, we consider a one
dimensional score function
\begin{align}
\ScoreFun_{i1}(\OneDirectPar, \AuxiNuisancePar, \FunPar) \equiv
\ScalePreCondVec_{i1}(\iRegressor - \GlmMeanVec_i)(\iResponse - g \big(
\inprod{\iRegressor}{\OneDirect} \OneDirectPar +
\iRegressor^\top(\IdMat_{\TargetDim} - \OneDirect\OneDirect^\top)\AuxiNuisancePar
+ \FunPar(\iNuisance))), \label{eq:glm_score_one_arm}
\end{align}
where,
\begin{align*}
\ScalePreCondVec_{i1} & \defn \OneDirect^\top
\InvSqrtCovar_i^2/\sqrt{\OneDirect^\top \InvSqrtCovar_i^2 \OneDirect}.
\end{align*} 
Similarly, we can define $\ScalePreCondVecwhat_{i1} \text{\,(or\,}
\ScalePreCondVec^*_{i1})$ by plugging in $(\AuxiNuisancePar,
\FunPar)=(\EstTargetPar, \EstFunPar) \text{\,(or $(\TrueTargetPar,
  \TrueFunPar)$)}$. We point out that,
\mbox{$\E|\ScalePreCondVec_{i1}(\iRegressor - \GlmMeanVec_i)|^2=1$}
which will be useful in the later sections. With these definitions in
hand we estimate the parameter $\OneDirectPar$ using
Algorithm~\ref{algo:DML-glm} but with step 4 replaced by finding
$\SoluOneDirectPar$ that solves
\begin{align}
\label{eq:glm_one_arm_estimating_equation}
\frac{1}{\NumIndexTwo} \sum \limits_{i \in \IndexTwo}
\ScoreFun_{i1}(\SoluOneDirectPar, \EstTargetPar, \EstFunPar)=0.
\end{align}
Likewise, we have asymptotic guarantee for $\SoluOneDirectPar$ under
the following variants of
Assumption~\ref{assn-glm-prob},~\ref{assn-glm-identifiability}
and~\ref{assn-glm-minimum-eig}.

\myassumption{{\bf{SEL}}$^{\prime\ast}(t, \delta, \SuppDirect)$}{assn-glm-prob-weak}{ The selection probabilities $\SelectProb_{ik}$ at each
  round satisfy the lower bound
\begin{align}
\label{eq:glm_assn_prob_one_arm}    
\SelectProb_{ik} \geq c_i \defn \frac{c_0}{i^{2(t - \delta)}} \quad
\text{for all} \;\; i\geq 1, k \in\SuppDirect
\end{align}
for some constant $c_0>0$ and $t \in [0, \tfrac{1}{4}]$ and
$\delta>0$. In addition, the probability of selecting $\zerovec$
satisfies $\SelectProb_{i0}\geq \czerotil$ for some $\czerotil>0$.}
\myassumption{{\bf{IDE}}$^\ast$}{assn-glm-identifiability-weak}{The
  model is identifiable under our choice of the score function,
  concretely,
\begin{align*}\|\CondMean (\ScoreFun_{i1}(\OneDirectPar, \TrueTargetPar, \TrueFunPar ) - \ScoreFun_{i1}(\TrueOneDirectPar, \TrueTargetPar, \TrueFunPar ))\|
&\geq {c_\ScoreFun}\|\CondMean
  \partial_{\OneDirectPar}\ScoreFun_{i1}(\TrueOneDirectPar, \TrueTargetPar,
  \TrueFunPar )(\OneDirectPar - \TrueOneDirectPar)\|_2\wedge c_\ScoreFun
  \Numobs^{-1/4}
 \end{align*} almost surely for any $\TargetPar \in\R^{\TargetDim}$ and some $c_{\ScoreFun}>0$.}

\myassumption{{\bf EIG}$^\ast$}{assn-glm-minimum-eig-weak}{ The gradient $\CondMean
  \partial_{\OneDirectPar}\ScoreFun_{i1}(\TrueOneDirectPar,
  \TrueTargetPar, \TrueFunPar )$ is not too small, namely,
  $\lim_{\Numobs\to\infty}\prob(|\CondMean
  \ScalePreCondVec^*_{i1}(\iRegressor - \GlmMeanVec^*_i)g'\big(\inprod{\iRegressor}{\TrueTargetPar}
  + \TrueFunPar(\iNuisance)
  \big)(\iRegressor - \GlmMeanVec^*_{i})^\top\OneDirect|\geq
  m_{\ScoreFun}\Numobs^{\delta-t})\to 1$ for some constant
  $m_{\ScoreFun}>0$.
}

A few comments regarding the assumptions are in order. 
Assumption~\ref{assn-glm-prob-weak} is weaker than
Assumption~\ref{assn-glm-prob} since we do not have assumptions on the
selection probability of coordinates that are not on the support of the vector 
$\OneDirect$. Assumption~\ref{assn-glm-identifiability-weak}
and~\ref{assn-glm-minimum-eig-weak} are adaptations of
Assumption~\ref{assn-glm-identifiability}
and~\ref{assn-glm-minimum-eig} with the score function $\ScoreFun_i$
replaced by $\ScoreFun_{i1}$. Similarly, both Assumption~\ref{assn-glm-identifiability-weak}
and~\ref{assn-glm-minimum-eig-weak} are implied
by the rest assumptions on GLM when $g(x)=x$. Moreover, 
Assumption~\ref{assn-glm-identifiability-weak}~and~\ref{assn-glm-minimum-eig-weak} can be verified  when $g(x)$ is the logit function (see Supplement
for details).


\begin{theorem}
\label{thm:glm_new1_temp_one_arm}
In addition to the standard GLM conditions
from Section~\ref{sec:standard-glm-assumption}, suppose that
Assumptions~\ref{assn-glm-nuisance-est},
~\ref{assn-glm-prob-weak},~\ref{assn-glm-identifiability-weak}
and~\ref{assn-glm-minimum-eig-weak} are in force.  Then the estimate
$\SoluOneDirectPar$ from
equation~\eqref{eq:glm_one_arm_estimating_equation} satisfies
\begin{align}
(\EmpMean \ScalePreCondVecwhat_{i1}(\iRegressor - \GlmMeanVecwhat_i)
g'(\inprod{\iRegressor}{\EstTargetPar } + \EstFunPar(\iNuisance)
)(\iRegressor - \GlmMeanVecwhat_{i})^\top\OneDirect)\sqrt{\NumIndexTwo
}(\SoluOneDirectPar - \TrueOneDirectPar) \convdist
\Normal(0,1)\label{eq:glm_asym_norm0_temp_one_arm}.
\end{align}
\end{theorem}

\noindent See Section~\ref{proof:thm:glm_new1_temp_one_arm} in the Supplement for the proof. \\

\vspace{5pt}

 Theorem~\ref{thm:glm_new1_temp_one_arm} allows us to construct 
asymptotically valid level $\alpha$ confidence interval for
$\TrueOneDirectPar$. Denote $(\EmpMean
\ScalePreCondVecwhat_{i1}(\iRegressor-
\GlmMeanVecwhat_i)g'\big(\inprod{\iRegressor}{\EstTargetPar} +
\EstFunPar(\iNuisance) \big)(\iRegressor-
\GlmMeanVecwhat_{i})^\top\OneDirect)$ by $v_{cov}$. Concretely, we
have
\begin{align*}
\lim_{\Numobs\to\infty} \Prob \Big[\SoluOneDirectPar
   - \frac{q_{1 - \alpha/2}\sigma}{\sqrt{\NumIndexTwo}v_{{cov}}} \leq
  \TrueOneDirectPar \leq \SoluOneDirectPar +
  \frac{q_{1 - \alpha/2}\sigma}{\sqrt{\NumIndexTwo}v_{{cov}}} \Big] &
=1 - \alpha,
\end{align*}
where $q_{1 - \alpha/2}$ is the $1 - \alpha/2$ quantile of standard normal
distribution.


\section{Some consequences for specific models}
\label{SecApplications}

In this section, we provide several examples in which we can construct
a suitable pilot estimator for the nuisance (and target) parameters.
By making use of such estimates with the \AlgLongPL or \AlgLongGLM
algorithms, we can develop explicit and computationally efficient
procedures that enjoy the guarantees stated in
Theorem~\ref{thm:linear_new1}
through~\ref{thm:glm_new1_temp_one_arm}. Throughout this section, we assume $\NumIndexOne=\Numobs/K$ for some $K\geq 2$.


\subsection{Partitioned linear model with adaptive data collection}
\label{subsec:lowdim_linear}

We begin with the simplest of settings, namely
a partitioned linear model of form
\begin{align*}
\iResponse = \inprod{ \iRegressor}{\TrueTargetPar} +
\inprod{\iNuisance}{\TrueNuisancePar} + \iNoise,
\end{align*}
where $\TrueTargetPar \in \R^{\TargetDim}$ and $\TrueNuisancePar \in
\R^{\NuisanceDim}$.  Suppose that the covariate vectors are collected
in an adaptive fashion, taking values in the set $\{e_1, \ldots,
e_{\TargetDim}, \zerovec \}$, with the selection probability vector
\mbox{$p_i \in \real^{\TargetDim + 1}$} at round $i$ allowed to be a
function of the pair $(\iNuisance, \iHistory)$. Under this set-up, Lai
and Wei~\cite{lai1982least} showed that the ordinary least squares
estimator {$(\TargetParwhat_{\OLS}, \NuisanceParwhat_{\OLS})$} is
consistent even without a stability condition on the design matrix.
Therefore, we can construct an asymptotically normal estimator of
$\TrueTargetPar$ using~\AlgLongPL with the OLS estimator as the pilot
estimator for $\TrueNuisancePar$.  Concretely, we assume that
\begin{align}
\label{eqn:lin-model-extra-assump}
\inf_{P\in\Pclass}\sigma_{\min}(\E_{\iNuisance\sim
  \Pclass}\iNuisance\iNuisance^\top) > 0 \quad \mbox{and} \quad
\|\iNuisance\|_2\leq B \quad \mbox{for some constant $B >
  0$.}
\end{align}
Finally, recalling that $\iCovar$ to denote the conditional covariance
of the regressor at step $i$, we deduce the following corollary from
Theorem~\ref{thm:linear_new1}.
\begin{cors}
\label{corr:linear_model_exm}
Suppose that
Assumptions~\ref{assn-lin-noise}--\ref{assn-lin-nuisance-est} holds
for some {$t\in[0,1/2)$}, and moreover
  condition~\eqref{eqn:lin-model-extra-assump} holds. Then the
  estimate $\SoluTargetPar$, obtained from~\AlgLongPL with
  $(\TargetParwhat_{\OLS}, \NuisanceParwhat_{\OLS})$ as pilot
  estimators, satisfies
\begin{align}
\label{eqn:low-dim-lin-model-cor}
(\sqrt{\NumIndexTwo} \EmpMean \iCovar^{1/2}) (\SoluTargetPar -
\TrueTargetPar) & \convdist \Normal(0, \LinearNoiseVar
\Id_{\TargetDim}).
\end{align} 
\end{cors}
\noindent See Supplement
 for the proof.


\subsection{Sparse high-dimensional linear model}
\label{subsec:highdim_linear}

Next we consider the high-dimensional linear regression problem
\begin{align*}
\iResponse = \inprod{\iRegressor}{\TrueTargetPar} +
\inprod{\iNuisance}{\TrueNuisancePar} + \iNoise, 
\end{align*}
where $\TrueTargetPar\in\R^{\TargetDim}$ and
$\TrueNuisancePar\in\R^{\NuisanceDim}$.  We allow for a partially
high-dimensional form of asymptotics, in which the target dimension
$\TargetDim$ stays fixed while the nuisance dimension $\NuisanceDim$
is allowed to grow to infinity as $\Numobs \to \infty$. We assume that
the noise variable $\iNoise's$ are sub-Gaussian with parameter
$\subgauss$. We also assume the nuisance vector $\TrueNuisancePar$ is
sparse with $|\{\TrueNuisancePar_i\neq 0\}|=\SparseLevel$. Note that
our theorems allow the nuisance component to vary as long as an
accurate pilot estimator is attainable. We use the $\lasso$ estimates
as pilot estimators---that is
\begin{align}
\label{eqn:lasso}
(\TargetParwhat_{\lasso}, \NuisanceParwhat_{\lasso}) \defn \arg
\min_{\TargetPar, \NuisancePar} \Big \{
\tfrac{1}{2\NumIndexOne}\sum_{i=1}^{\NumIndexOne}(\iResponse -
\inprod{\iRegressor}{\TargetPar} -
\inprod{\iNuisance}{\NuisancePar})^2 +
\lambda_{\NumIndexOne}(\|\TargetPar\|_1 + \|\NuisancePar\|_1) \Big \},
\quad \mbox{where} \\
\notag
\lambda_{\NumIndexOne} \defn 2\subgauss(B' + 1)\sqrt{\tfrac{2[\log
      (\tfrac{2}{\delta_{\NumIndexOne}}) + \log(\TargetDim +
      \NuisanceDim)]}{\NumIndexOne}},  \quad
\delta_{\NumIndexOne} \defn \min\{(\SparseLevel + \TargetDim)
\NumIndexOne^{2t-1/2}, \tfrac{1}{\TargetDim + \NuisanceDim}\}
\end{align} for some constant $B'>0.$
In our result, we assume that the sparsity level is bounded as
\begin{align}
\label{eqn:sparsity-condition}
(\SparseLevel + \TargetDim)\sqrt{\log (\TargetDim +
  \NuisanceDim)}=\liloh_p(\NumIndexOne^{1/2-2t}) \qquad \mbox{for some
  exponent $t \in [0, 1/4)$.}
\end{align}
Moreover, assume the nuisance component satisfies
\begin{align}
\label{eqn:lin-model-extra-assump-2}
\inf_{P\in\Pclass}\sigma_{\min}(\E_{\iNuisance\sim
  \Pclass}\iNuisance\iNuisance^\top) > 0 \quad \mbox{and} \quad
\vecnorm{\iNuisance}{\infty}\leq B' \quad \mbox{for some constant $B'
  > 0$.}
\end{align}
Given this set-up, we can apply Theorem~\ref{thm:linear_new1} so as to
derive the following corollary:
\begin{cors}
\label{corr:linear_model_exm_2}
Suppose Assumptions~\ref{assn-lin-noise}--\ref{assn-lin-nuisance-est}
and the sparsity condition~\eqref{eqn:sparsity-condition} holds for
some $ t\in[0,1/4)$, and Assumption~\eqref{eqn:lin-model-extra-assump-2}
  is in force. Then the estimate $\SoluTargetPar$, obtained
  from~\AlgLongPL with $(\TargetParwhat_{\lasso},
  \NuisanceParwhat_{\lasso})$ as pilot estimators, satisfies
\begin{align}
\label{eqn:sparse-lin-model-cor}
(\sqrt{\NumIndexTwo} \EmpMean \iCovar^{1/2}) (\SoluTargetPar -
\TrueTargetPar) & \convdist \Normal(0, \LinearNoiseVar
\Id_{\TargetDim}).
\end{align} 
\end{cors}
\noindent See the Supplement for the proof.

\vspace{8pt}

\noindent In general, it is non-trivial to develop an asymptotically
valid confidence region for both the target and the nuisance
parameters; however, Corollary~\ref{corr:linear_model_exm_2} illustrates how
many nuisance parameters we are able to tolerate in order to have
valid inference for a fixed number of target parameters.


\subsection{Sparse generalized linear model}
\label{subsec:highdim_glm}

We now consider an extension of the sparse linear model.  Suppose that
we observe triples $(\iRegressor, \iNuisance, \iResponse)$ related via
the model
\begin{align*}
\iResponse = g\big(\inprod{\iRegressor}{\TrueTargetPar} +
\inprod{\iNuisance}{\TrueNuisancePar}\big) + \iNoise.
\end{align*}
We assume that the link function $g$ arises in the usual exponential
family way, so that there is a function $G$ such that $G'(t) = g(t)$.
Thus, the negative log likelihood associated with this model takes the
form
\begin{align*}
\glmloss(\TargetPar, \NuisancePar ; \iRegressor, \iNuisance,
\iResponse) & = G\big(\inprod{\iRegressor}{\TargetPar} +
\inprod{\iNuisance}{\NuisancePar}\big) -
\iResponse(\inprod{\iRegressor}{\TargetPar} +
\inprod{\iNuisance}{\NuisancePar}).
\end{align*}
As a pilot estimator for the \AlgLongGLM procedure (cf.
Algorithm~\ref{algo:DML-glm}), we compute the $\ell_{1}$-regularized
estimate
\begin{align}
\label{eqn-GLM-lasso}
    (\TargetParwhat_{\GLMlasso}, \NuisanceParwhat_{\GLMlasso}) &\defn
\arg\min \limits_{\TargetPar, \NuisancePar} \big \{
\frac{1}{\NumIndexOne}\sum_{i=1}^{\NumIndexOne} \glmloss(\TargetPar,
\NuisancePar ; \iRegressor, \iNuisance, \iResponse) +
\lambda_{\NumIndexOne}(\|\TargetPar\|_1 + \|\NuisancePar\|_1) \big \}.
\end{align}
with the choices
\begin{align*}
\lambda_{\NumIndexOne} \defn 2\subgauss D_x\sqrt{\frac{2[\log
      (2/\delta_{\NumIndexOne}) + \log(\TargetDim +
      \NuisanceDim)]}{\NumIndexOne}},\text{ and }
\delta_{\NumIndexOne} \defn \min\{(\SparseLevel + \TargetDim)
\NumIndexOne^{2t-1/4},\frac{1}{\TargetDim + \NuisanceDim}\},
\end{align*}
where $D_x$ is an upper bound on $\vecnorm{\begin{pmatrix}\iRegressor^\top &\iNuisance^\top\end{pmatrix}}{2}$ for all $i.$

In our analysis, we assume that target dimension $\TargetDim$ is fixed
while the nuisance dimension $\NuisanceDim$ is allowed to go to
infinity as $\Numobs\to\infty$. Again, we assume that the noise
variables $\{\iNoise\}_{i=1}^\numobs$ are independent, each
sub-Gaussian with parameter at most $\subgauss$, and the true nuisance
vector $\TrueNuisancePar$ is sparse with $|\{\TrueNuisancePar_i\neq
0\}|=\SparseLevel$.  Moreover, we assume that the sparsity level
$\SparseLevel$ satisfies the condition
\begin{align}
\label{eqn:sparsity-condition-glm}
(\SparseLevel + \TargetDim)\sqrt{\log
  (\TargetDim +
  \NuisanceDim)}=\liloh_p(\NumIndexOne^{1/4-2t}).
\end{align}
With this set-up, we can apply Theorem~\ref{thm:glm_new1_temp} so as to
obtain the following guarantee:
\begin{cors}
\label{corr:linear_model_exm_3}
Suppose that
Assumptions~\ref{assn-glm-prob}---\ref{assn-glm-minimum-eig} hold for
some $t \in [0, 1/4)$ and
  Assumptions~\eqref{eqn:lin-model-extra-assump}
  and~\eqref{eqn:sparsity-condition-glm} are in force.  Then the
  estimate $\SoluTargetPar$, computing using \AlgLongGLM with the
  pilot estimators~\eqref{eqn-GLM-lasso}, satisfies
\begin{align}
\label{eqn:sparse-GLM-cor}
(\EmpMean \InvSqrtCovarwhat_{i}(\iRegressor - \GlmMeanVecwhat_i)
g'\big( \inprod{\iRegressor}{\EstTargetPar} + \EstFunPar(\iNuisance)
\big) (\iRegressor - \GlmMeanVecwhat_i)^\top) \sqrt{\NumIndexTwo}
(\SoluTargetPar - \TrueTargetPar) \convdist \Normal(0,
\IdMat_{\TargetDim}). 
\end{align} 
\end{cors}

\noindent See Supplement
for the proof
of Corollary~\ref{corr:linear_model_exm_3}.

\vspace{8pt}

\noindent Compared with Corollary~\ref{corr:linear_model_exm_2} for
high-dimensional linear models, here we need a stronger assumption on
the sparsity level (i.e., $(\SparseLevel + \TargetDim)\sqrt{\log
  (\TargetDim + \NuisanceDim)}=\liloh_p(\NumIndexOne^{1/4-2t})$) and
restrict $t\in[0,1/8)$. This is due to the need of a
  $\liloh_p(\Numobs^{-1/4})$-consistent pilot estimator. We remark
  that our assumption on the sparsity level is probably not sharp and
  can be improved under stronger assumptions (e.g. when the data are
  i.i.d. collected~\cite{belloni2016post}).

{
\subsection{Partial linear model with nonparametric nuisance}Lastly, we consider a  case where the nuisance component is nonparametric, namely a   partial linear model  given by
\begin{align*}
\iResponse = \inprod{ \iRegressor}{\TrueTargetPar} +
\TrueFunPar(\iNuisance) + \iNoise,
\end{align*}
where $\TrueTargetPar \in \R^{\TargetDim}$, $\iNuisance \in
\R^{\NuisanceDim}$ and $\TrueFunPar: \R^{\NuisanceDim}\mapsto\R$ is some nonparametric function.  Similar to Section~\ref{subsec:lowdim_linear}, suppose the covariate vectors $\iRegressor$ take values in the set $\{e_1, \ldots,
e_{\TargetDim}, \zerovec \}$ with probabilities given by the selection probability vector $\SelectProb_{i}$. Additionally,   assume that
\begin{subequations}
\begin{align}
\label{eqn:lin-model-extra-assump-iid}
&\iNuisance \overset{i.i.d.}{\sim}P~~\text{for some distribution $P$ on $[0,1]^{\NuisanceDim}$ ~and~ } \SelectProb_i\indep\iNuisance\mid \iHistory,
\\
\label{eqn:lin-model-extra-assump-iid-2}
&\iNoise~~\text{is independent of } (\iRegressor,\iNuisance,\iHistory)
\text{~~and~~} \iNoise\overset{i.i.d.}{\sim}Q~~\text{for some
  distribution $Q,$ }
\end{align}
and $\TrueFunPar$ is Lipschitz continuous with parameter $L>0$, i.e., 
\begin{align}
\label{eqn:lin-model-extra-assump-lip}
|\TrueFunPar(v_1)-\TrueFunPar(v_2)|\leq L \vecnorm{v_1-v_2}{2} 
~~~~\text{ for all $v_1,v_2\in[0,1]^{\NuisanceDim}$.}
\end{align}
Under these assumptions, various non-parametric procedures---for
example, a $k$-nearest neighbor
estimate~\cite{gyorfi2002distribution,Tsybakov2008IntroductionTN}---
can be used to find a consistent pilot estimator $\EstFunPar$ for
$\TrueFunPar$.  Given such a pilot estimator, applying
Theorem~\ref{thm:linear_new1} yields:
\end{subequations}
\begin{cors}
\label{corr:linear_model_exm_4}
Suppose that
Assumption~\ref{assn-lin-noise}~and~\ref{assn-lin-selection-prob} hold
for some {$t\in[0,1/2)$}, as well as
  conditions~\eqref{eqn:lin-model-extra-assump-iid}~and~\eqref{eqn:lin-model-extra-assump-lip}. Then
  the estimate $\SoluTargetPar$, obtained from~\AlgLongPL\\ with the
  $k$-nearest neighbor estimate $\EstFunPar$ as the pilot estimator,
  satisfies
\begin{align}
\label{eqn:low-dim-lin-model-cor-2}
(\sqrt{\NumIndexTwo} \EmpMean \iCovar^{1/2}) (\SoluTargetPar -
\TrueTargetPar) & \convdist \Normal(0, \LinearNoiseVar
\Id_{\TargetDim}).
\end{align} 
\end{cors}
See the Supplement
for details of the $k$-nearest neighbor estimate and  the proof
of Corollary~\ref{corr:linear_model_exm_4}.
}
{ Note that in equation~\eqref{eqn:lin-model-extra-assump-iid} we
  require the selection probability to only depend on the history
  $\iHistory$ but not on the nuisance $\iNuisance$. In practice, this
  reflects the scenario where the treatment assignment scheme
  (determining the selection probabilities $\SelectProb_i$) needed to
  be determined before observing the individual context vector
  $\iNuisance$. We have imposed this condition to simplify analysis,
  but note that it can be removed as long as a consistent pilot
  estimator $\EstFunPar$ for $\TrueFunPar$ can be devised.}

  
\section{Numerical results}
  
In this section, we illustrate our theoretical guarantees with a
selection of numerical studies.  We provide results for both an
adaptive linear model as well as an adaptive logistic model.  In
addition to showing results based on our proposed algorithms, we
compare with other existing methods including' (a) maximum likelihood
estimators; (b) methods based on concentration inequalities; and (c)
an existing $Z$-estimator procedure derived from the double machine
learning (DML) approach~\cite{chernozhukov2018double}.


\subsection{Adaptive linear model}
\label{sec:exp_linear}

In this section, we study the semi-parametric problem in a
(potentially) high-dimensional linear model. As in the applications in
Sections~\ref{subsec:lowdim_linear} and~\ref{subsec:highdim_linear},
we consider the linear model
\begin{align*}
\iResponse = \inprod{\iRegressor}{\TrueTargetPar} +
\inprod{\iNuisance}{\TrueNuisancePar} + \iNoise,
\end{align*}
where the triples $(\iResponse, \iRegressor, \iNuisance)$ are
adaptively collected in the following way:
\begin{enumerate}
\item[(1)] The nuisance component $\iNuisance$ has
  i.i.d. $\Normal(0,1)$ entries and is independent of the history
  $\iHistory$.
\item[(2)] Given $(\Response_1, \Regressor_1, \Nuisance_1), \ldots,
  (\Response_{i-1}, \Regressor_{i-1}, \Nuisance_{i-1})$, we solve a
  LASSO (or OLS) problem so as to obtain the estimates
  $\EstTargetPar^{i}$ and $\EstNuisancePar^{i}$.
\item[(3)] From the estimator $\EstTargetPar^{i}$, the algorithm
  selects an arm \mbox{$k_i \defn \argmax_{k}\{\EstTargetPar^{i}_k +
    \sqrt{\frac{C\log \Numobs}{\Numobs^i_{k}}}\}$}, where $C > 0$ is
  some constant and $\Numobs^i_{k}$ is the number of times the arm $k$
  has been chosen up to time $i - 1$.
\item[(4)] Finally, the regressor $\iRegressor$ is chosen according to
  the arm selection probability $\SelectProb_i\in\R^{\TargetDim}$
  (i.e., $\prob(\iRegressor=e_k)=\SelectProb_{ik}$ and we define $e_0
  \defn \zerovec$), where we set
  \begin{align*}
 \SelectProb_{i0}=0.2, \qquad p_{ik}=\min\{\frac{1}{2i^{2t}},
 \frac{0.4}{\TargetDim}\} \quad \mbox{for $k \neq k_i$,} \quad \text{
   and } \quad \SelectProb_{i{k_i}} = 1 - \sum_{0\leq k\leq
   \TargetDim,k\neq k_i} \SelectProb_{ij}.
  \end{align*}
We make observations $\iResponse =
\inprod{\iRegressor}{\TrueTargetPar} +
\inprod{\iNuisance}{\TrueNuisancePar} + \iNoise$ contaminated by noise
$\iNoise \stackrel{i.i.d.}{\sim} \Normal(0,1)$. When $i=1$, we set
$\SelectProb_{i0}=0.2$, $\SelectProb_{ik}=0.8/\TargetDim$ for all
$k\in[\TargetDim]$.
\end{enumerate}

After collecting the data, we apply the \AlgLongPL method
(Algorithm~\ref{algo:DML-linear} with the score function defined in
equation~\eqref{eq:one-arm-score-fun}) to perform inference on the
first coordinate of the target parameter $\TrueTargetPar_1$.

In our first experiment, we choose the target dimension
$\TargetDim=2$, the nuisance dimension $\NuisanceDim=5$ and the number
of samples $n=500$. We consider the no-margin scenario where
$\TrueTargetPar_1=\TrueTargetPar_2=2$. Under these conditions, Zhang
et al.~\cite{zhang2020inference} show that the selection probability
$\SelectProb$ may not converge, so that the stability condition can be
violated. Moreover, we assume the nuisance parameter vector
$\TrueNuisancePar$ is a fixed vector generated from $\mathcal N(0,
\IdMat_{\NuisanceDim})$; we choose the OLS estimator as the pilot
estimator for $\TrueTargetPar, \TrueNuisancePar$ using
$\NumIndexOne=\Numobs/4=125$ samples. The results are shown in
Figure~\ref{fig:low_dim_linear1}.
\begin{figure}[htb]
  \begin{center}
\begin{tabular}{ccc}
\widgraph{0.47\linewidth}{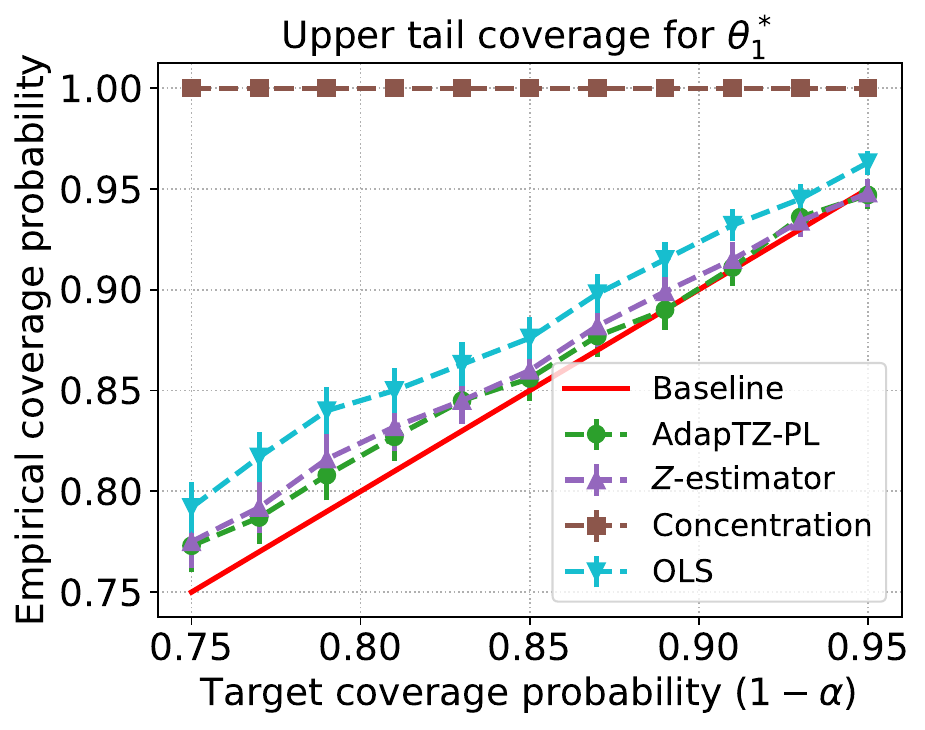} &&
\widgraph{0.47\linewidth}{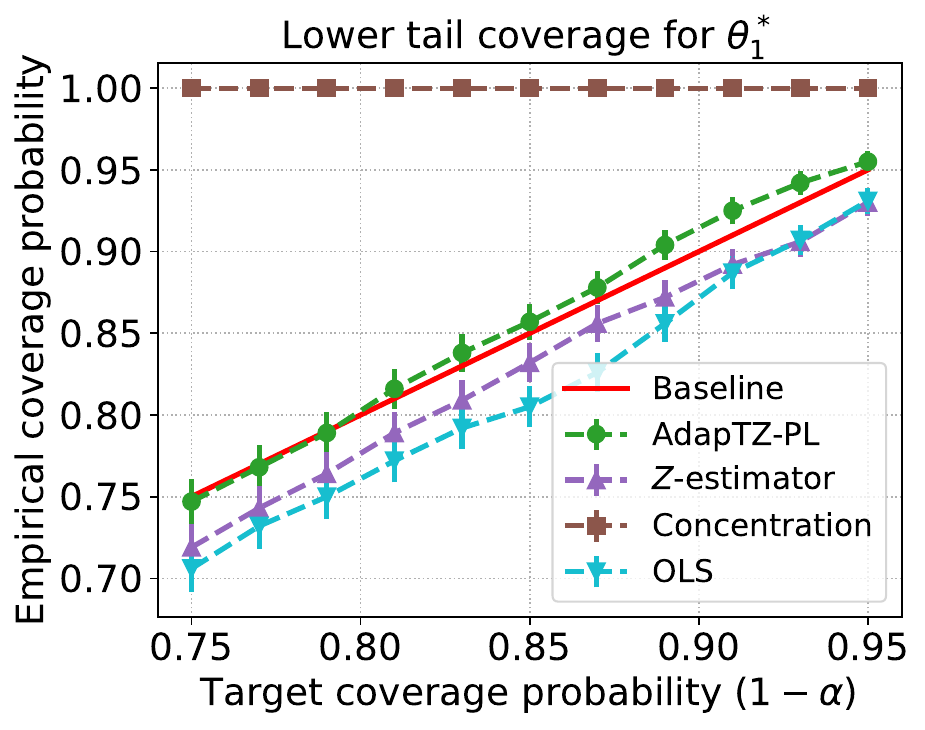} \\
(a) && (b)
\end{tabular}
\begin{tabular}{c}
  \widgraph{0.47\linewidth}{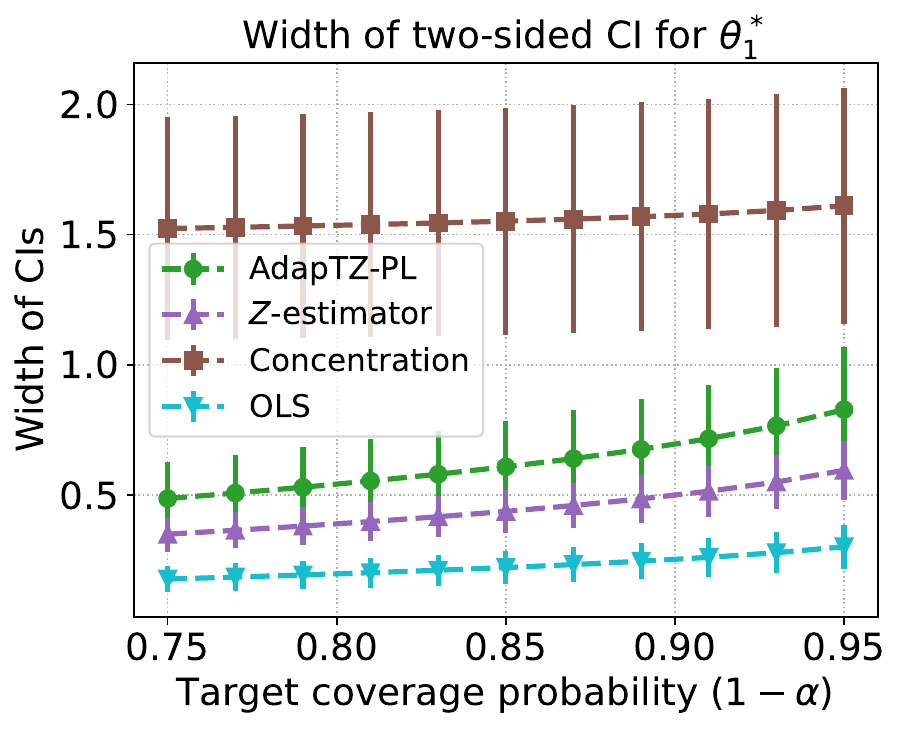}
  \\ (c)
\end{tabular}
\caption{Average coverage and width of confidence intervals for
  $\TrueTargetPar_1$ over $T = 1000$ repetitions of an adaptive linear
  model. The error bars denote $\pm 1$ standard error. Parameters:
  $\TargetDim = 2, \NuisanceDim = 5, \Numobs = 500, \NumIndexOne =
  125$, $C = 2$ and $t = 0.2$. (a) and (b): Coverage of level $1 -
  \alpha$ one-sided confidence intervals for $\TrueTargetPar_1$. (c):
  Width of level $1 - \alpha$ two-sided confidence intervals for
  $\TrueTargetPar_1$.}
\label{fig:low_dim_linear1}
\end{center}
\end{figure}

We compare the \AlgLongPL estimator to three other procedures: (i)
ordinary least squares; (ii) a DML $Z$-estimator based on the
unweighted score function
\begin{align}
\label{eq:naive_dml_score}  
\ScoreFun_{i}(\TargetPar, \FunPar) \defn (\iRegressor -
\SelectProb_i)(\iResponse - \inprod{\iRegressor}{ \TargetPar} -
\FunPar(\iNuisance));
\end{align}
and (iii) a confidence interval derived from a standard concentration
inequality (cf. Theorem 2 in Abbasi-Yadkori et
al.~\cite{abbasi2011improved}.)  Figure~\ref{fig:low_dim_linear1}
shows the empirical coverage probability and width of confidence
intervals obtained from each method. We observe that the \AlgLongPL
method provides appropriate coverage for all confidence
levels. However, while the ordinary least squares estimator and the
$Z$-estimator provide valid upper tail coverage and have shorter
confidence intervals, they are both downward biased~\cite{nie2018why}
and fail to achieve proper lower tail coverage.

\begin{figure}[ht!]
  \begin{center}
\begin{tabular}{ccc}
  \widgraph{0.47\linewidth}{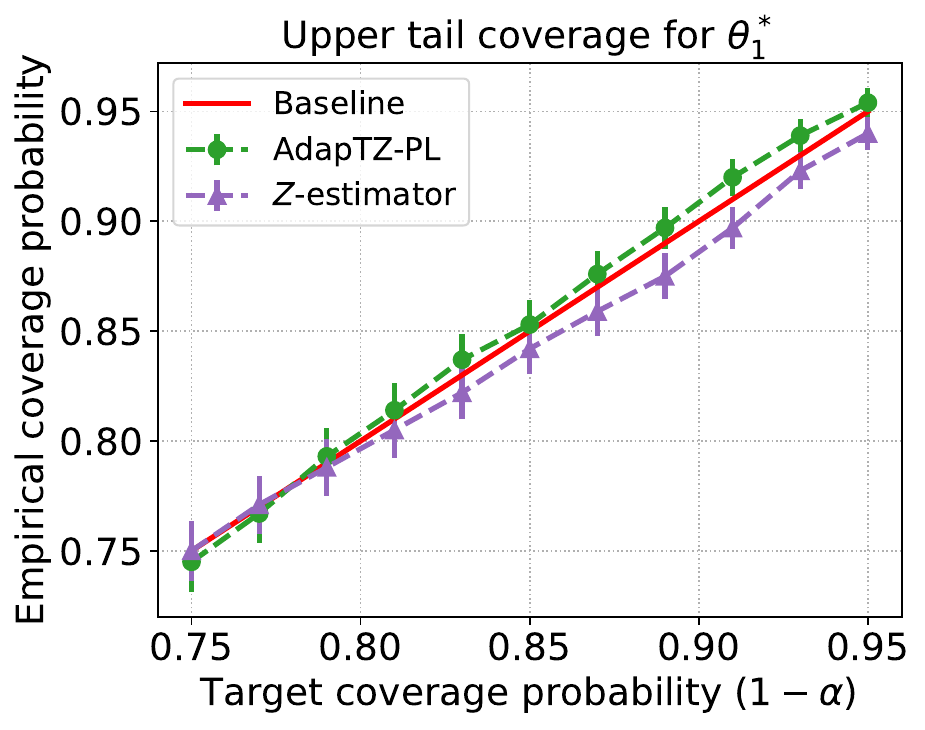}
  &&
  \widgraph{0.47\linewidth}{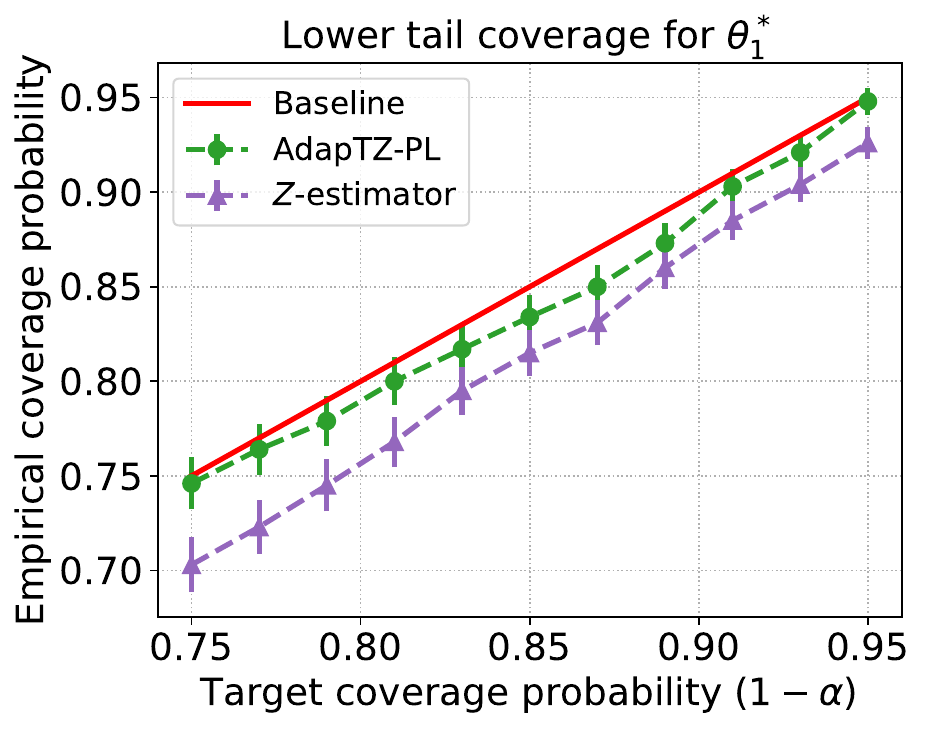}
  \\ (a) && (b)
\end{tabular}
\begin{tabular}{c}
  \widgraph{0.47\linewidth}{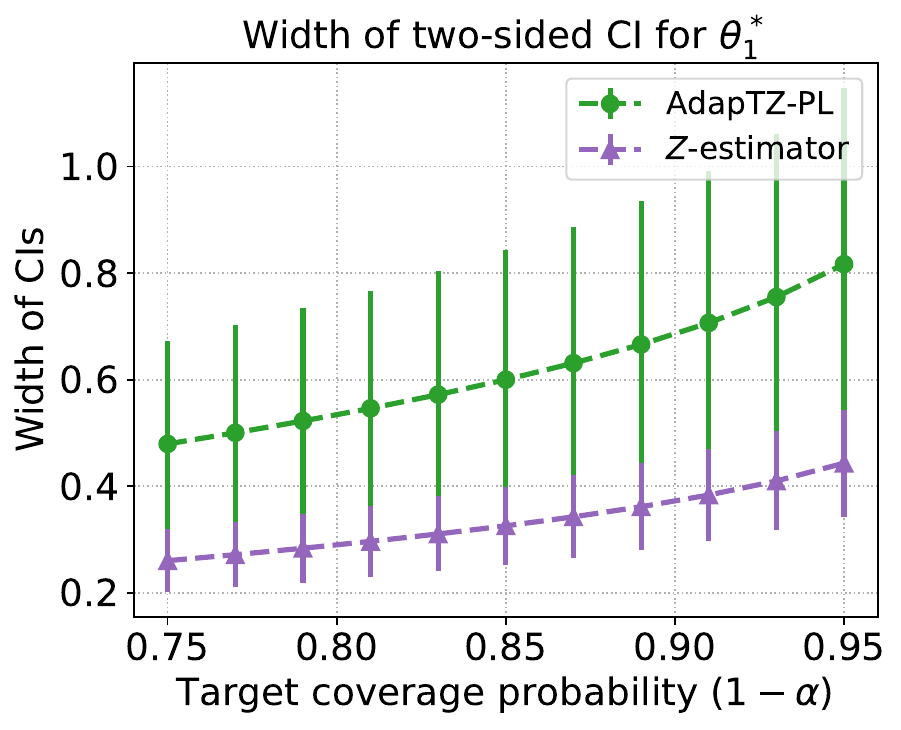}
  \\ (c)
\end{tabular}
\caption{Average coverage and width of confidence intervals for
  $\TrueTargetPar_1$ over $1000$ repetitions of an adaptive linear
  model. The error bars are $\pm 1$ standard error. Parameters:
  $\TargetDim=2, \NuisanceDim=1000, \Numobs=950, \NumIndexOne=475$,
  $C=16$ and $t=0.2$. (a) and (b): coverage of level $1 - \alpha$
  one-sided confidence intervals for $\TrueTargetPar_1$. (c): width of
  level $1 - \alpha$ two-sided confidence intervals for
  $\TrueTargetPar_1$.}\label{fig:high_dim_linear1}
  \end{center}
\end{figure}

In the second experiment, we consider a linear model with
high-dimensional nuisance. Namely, with the choice $(\TargetDim,
\NuisanceDim) = (2, 1000)$, we generate $\Numobs=950$ samples. Similar
to the first experiment, we consider the no margin scenario where
$\TrueTargetPar_1 = \TrueTargetPar_2 = 2$. We also assume the linear
model is sparse, in the sense that $\TrueNuisancePar_i=0$ for $i > 2$
and the first two coordinates of the nuisance parameter vector
$\TrueNuisancePar_1, \TrueNuisancePar_2$ are generated from
$\Normal(0, \IdMat_2)$. We generate the samples in the same way as the
first experiment, but use the LASSO estimator (cf. equation~\ref{eqn:lasso}) in both the data
generating process and to obtain pilot estimates $\EstTargetPar,
\EstNuisancePar$ of the parameters. We choose the Lasso regularization $\lambda=0.05,0.15$ for data generation and the pilot estimate, respectively.

Figure~\ref{fig:high_dim_linear1} compares the coverage probability of
\AlgLong and the standard DML $Z$-estimator. We see that the \AlgLong
procedure achieves proper empirical coverage probability at most
levels.  Similar to the low dimensional case, while the $Z$-estimator
has lower variance, it is downward biased and does not have proper
coverage. We do not provide here the confidence interval derived from
the concentration inequalities in Abbasi-Yadkori et
al.~\cite{abbasi2011improved} since the interval is too wide due to a
$\sqrt{d/n}$ factor inside the bound.

\subsection{Adaptive logistic model}
\label{sec:exp_logistic}
We then demonstrate the usage of \AlgLongGLM method when applied to a
logistic regression model with adaptively collected data. We generate
the data via the procedure described in Section~\ref{sec:exp_linear},
with the following changes:
\begin{enumerate}
\item[(1)] The variables $\iNuisance$ are generated from an
  autoregressive process \mbox{$\iNuisance=\gamma\Nuisance_{i-1} +
    W_i$}, where $\Nuisance_{0} \defn \zerovec,\gamma=0.5$ and $W_i$
  are i.i.d. random variables following $\Normal(0,
  \IdMat_{\NuisanceDim})$.
\item[(2)] As pilot estimators, we compute the maximum likelihood
  estimates $\EstTargetPar$ and $\EstNuisancePar$ of the unknown
  parameters $\TrueTargetPar$ and $\TrueNuisancePar$, respectively.
\item[(3)] The responses $\iResponse$ are Bernoulli random variables
  with mean $g\big(\inprod{\iRegressor}{\TrueTargetPar} +
  \inprod{\iNuisance}{\TrueNuisancePar}\big)$ given $\iHistory,
  \iRegressor, \iNuisance$.
\end{enumerate}
In this example, we investigate a low-dimensional instance with
dimensions $\TargetDim=2$ and $\NuisanceDim=20$, along wit the sample
size $\Numobs = 2000$. Again, we set $\TrueTargetPar_1 =
\TrueTargetPar_2 = 2$, and let $\TrueNuisancePar$ be a fixed vector
generated from $\Normal(0, \IdMat_{\NuisanceDim})$. Moreover, we use
the MLE to generate pilot estimates for $\TrueTargetPar$ and
$\TrueNuisancePar$ based on $\NumIndexOne = \Numobs/2 = 1000$ samples.

From Figure~\ref{fig:low_dim_glm1}, we observe that both \AlgLong and
$Z$-estimator have upper tail coverage over the prespecified
level. However, the $Z$-estimator as well as the MLE fail to achieve
appropriate lower tail coverage. This is consistent with our previous
observations in the linear model. Additionally, it should be noted
that the empirical coverage probability of \AlgLong is not perfectly
aligned with the baseline, likely due to the relatively small sample
size.
\begin{figure}
  \begin{center}
    \begin{tabular}{ccc}
\widgraph{0.47\linewidth}{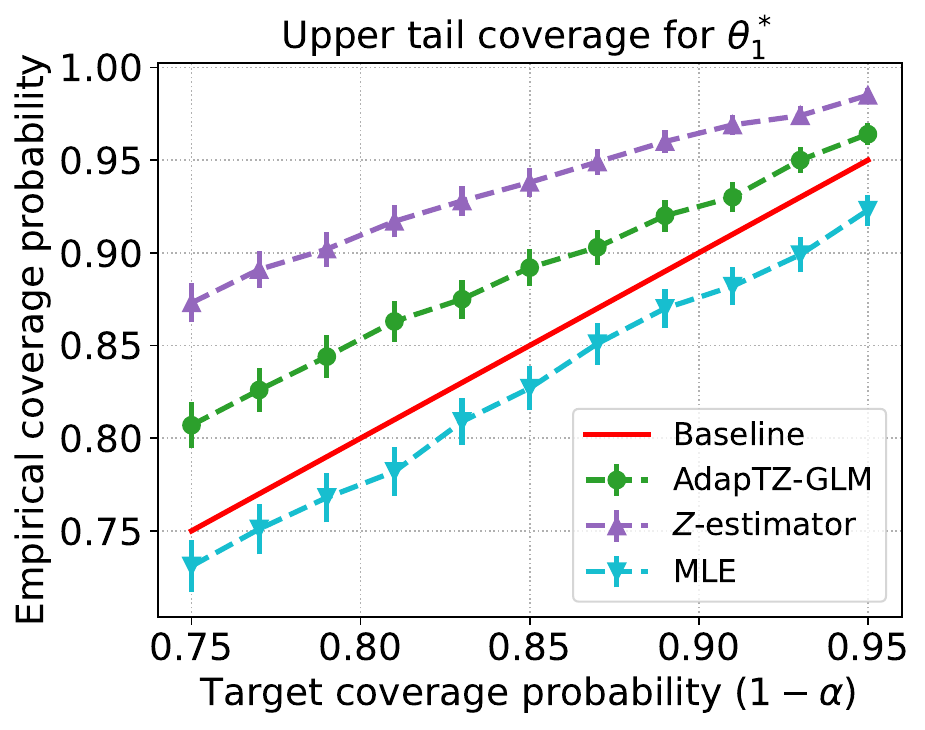} &&
\widgraph{0.47\linewidth}{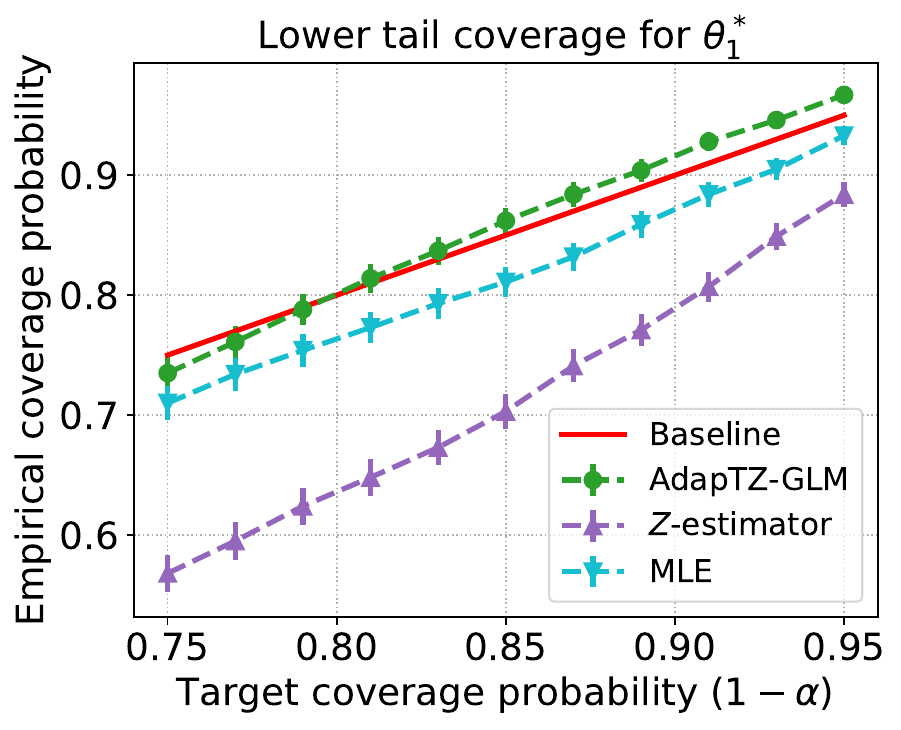} \\
(a) && (b)
    \end{tabular}
\begin{tabular}{c}    
  \widgraph{0.47\linewidth}{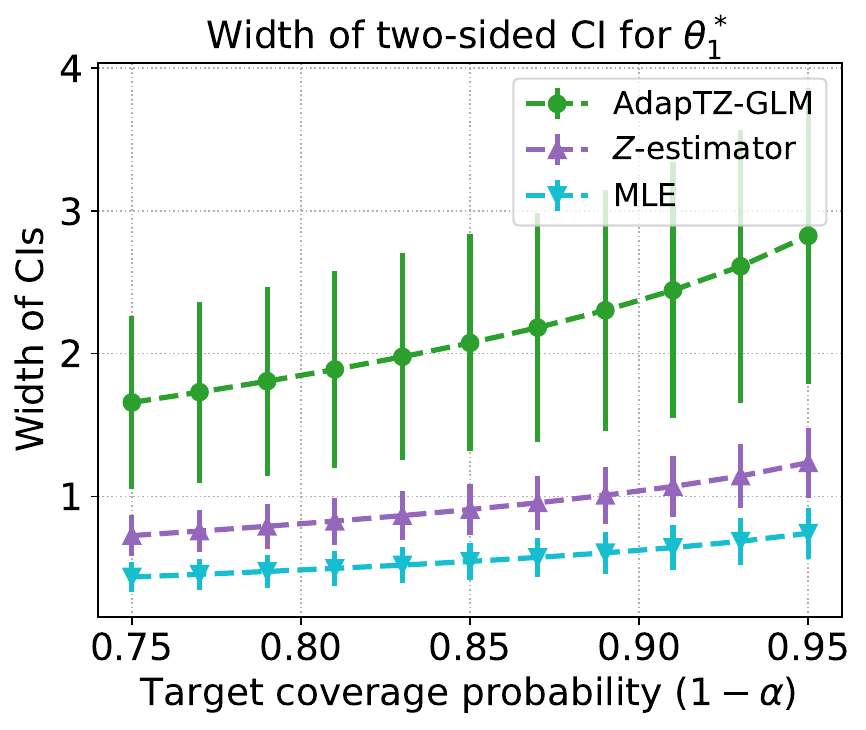}\\
  (c)
\end{tabular}
\caption{Average coverage and width of confidence intervals for
  $\TrueTargetPar_1$ over $1000$ repetitions of an adaptive logistic
  model. The error bars denote $\pm 1$ standard error. Parameters:
  $\TargetDim=2, \NuisanceDim=20, \Numobs=2000, \NumIndexOne=1000$,
  $C=8$ and $t=0.1$. Panles (a) and (b) give coverage of level $1 -
  \alpha$ one-sided confidence intervals for $\TrueTargetPar_1$. Panel
  (c) shows the width of level $1 - \alpha$ two-sided confidence
  intervals for $\TrueTargetPar_1$.} \label{fig:low_dim_glm1}
  \end{center}
\end{figure}

 Finally, we also experiment with a logistic regression
model with a sparse high-dimensional nuisance component. Concretely,
we generate $\Numobs=950$ samples with the choice
$(\TargetDim,\NuisanceDim)=(2,1000)$. We consider the no margin
scenario where $\TrueTargetPar_i=\TrueTargetPar_2=2$ as in previous
experiments. We assume the logistic regression model is sparse, in the
sense that $\TrueNuisancePar_i=0$ for $i>5$ and
$\TrueNuisancePar_{1:5}=\mathbf{1}_5$. We assume the data are
generated via the same procedure as in the first experiment for the
logistic regression model, but use the LASSO estimator for logistic
regression (cf. equation~\ref{eqn-GLM-lasso}) with penalty $\lambda=0.0025$ in both data generation and to obtain pilot estimates
$\EstTargetPar,\EstNuisancePar$. Moreover, we assume the random
vectors $W_i\in\R^{\NuisanceDim}$ in the autoregressive process are
instead generated in the following way: $W_{ik}\overset{i.i.d.}{\sim}
\Normal(0,\IdMat_5)$ for $k\leq 5$ and $W_{ik}\overset{i.i.d.}{\sim}
\Normal(0,1/\NuisanceDim)$ for $6\leq k\leq \NuisanceDim$.  The
heterogeneous variance of $W_i$ is selected to ensure the norm of the
nuisance component is of order $\bigoh_p(1)$, and the non-sparse
nuisance contribute a non-vanishing and detectable signal to the
response $\iResponse$.

In Figure~\ref{fig:high_dim_glm1}, we see that $\AlgShortGLM$ achieves
upper and lower tail coverage over the prespecified level, while the
naive $Z$-estimator---while it has small variance---exhibits a
downward bias and fails to have proper lower tail coverage. Again,
$\AlgShortGLM$ is not fully aligned with the baseline probably due to
the relatively small sample size in the logistic regression problem.
\begin{figure}[htb]
  \begin{center}
 \begin{tabular}{ccc}
\widgraph{0.47\linewidth}{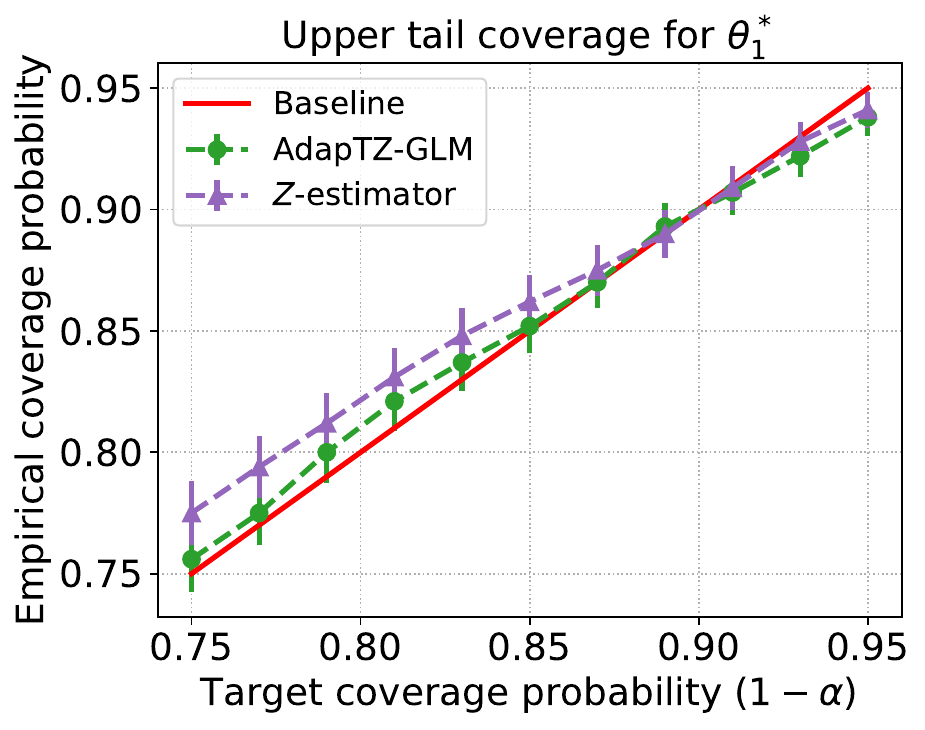}
&&
\widgraph{0.47\linewidth}{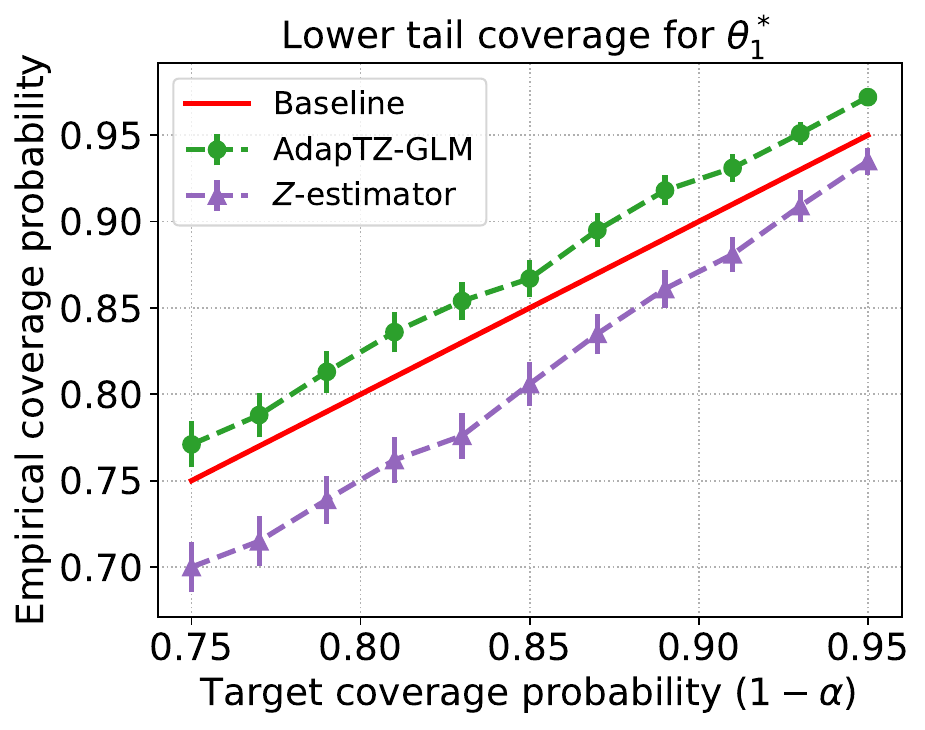}
\\ (a) && (b)
    \end{tabular}
\begin{tabular}{c}    
  \widgraph{0.47\linewidth}{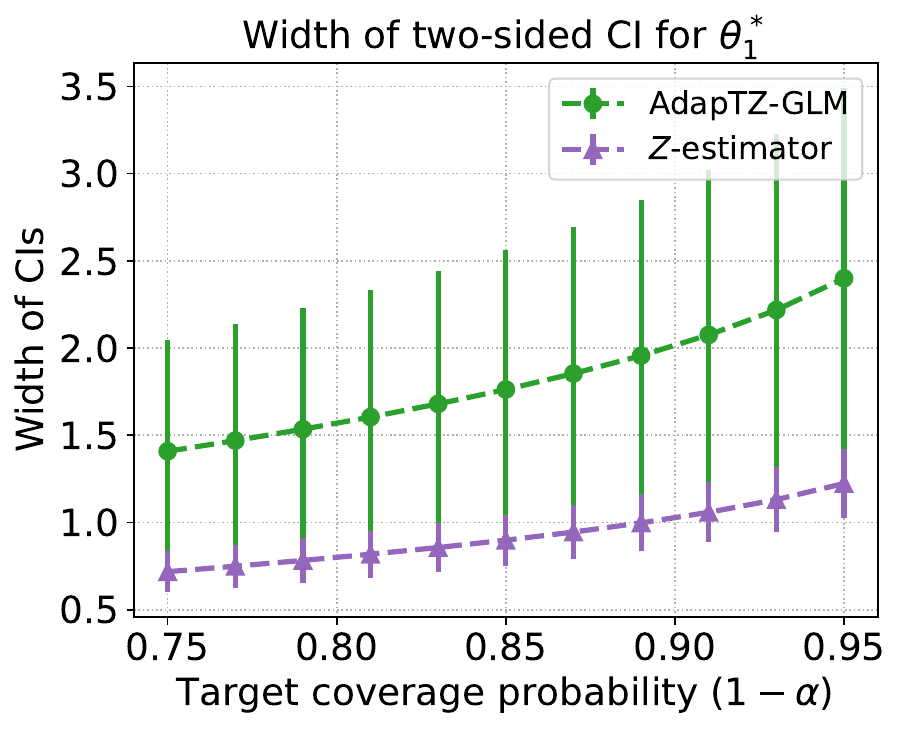}\\ (c)
\end{tabular}
\caption{
{
Average coverage and width of confidence
    intervals for $\TrueTargetPar_1$ over $1000$ repetitions of an
    adaptive logistic model. The error bars denote $\pm 1$ standard
    error. Parameters: $\TargetDim=2, \NuisanceDim=1000, \Numobs=950,
    \NumIndexOne=475$, $C=100$ and $t=0.1$. (a) and (b): coverage of
    level $1 - \alpha$ one-sided confidence intervals for
    $\TrueTargetPar_1$. (c): Width of level $1 - \alpha$ two-sided
    confidence intervals for
    $\TrueTargetPar_1$.}}\label{fig:high_dim_glm1}
  \end{center}
\end{figure}

\color{black}
%

%


\section{Discussion}

In this paper, we studied the problem of constructing confidence
intervals for a low-dimensional target parameters in presence of
high-dimensional or non-parametric nuisance components.  The main
novelty in our work is tackling the challenge of doing so when the
data has been adaptively collected.  We proposed a class of
procedures, known as \AlgLong methods, that are based on adaptive
reweighting of two-stage $Z$-estimators.  We developed versions of
these procedures for the partially linear model, as well as the more
general class of generalized linear models with semi-parametric
nuisances.  Our main results guarantee that, under certain regularity
conditions, there are versions of such estimators that enjoy
asymptotic normality.  Notable features of our analysis include the
fact that (a) we assume only mild ``explorability'' conditions on the
adaptive data collection procedure; and (b) in contrast to prior
state-of-the art~\cite{lai1994asymptotic,lai1982least}, we do not
require any sort of stability condition.

Our work suggests a number of directions for future work.  First, the
results in this paper provide inferential guarantees for a parameter
vector of fixed dimension within a semi-parametric model (in which the
nuisance quantities may be high-dimensional or non-parametric).  It
would interesting to extend our results so as to \emph{also allow} for
the target parameter to be high-dimensional, or more generally to
targets with a non-parametric flavor.  Second, we have provided
asymptotic normality guarantees with certain variances that depend on
the problem instance.  In the semi-parametric literature with
i.i.d. data, there are instance-dependent notions of optimality---in
terms of the smallest variance for $\sqrt{\numobs}$-consistent
estimators---that have been characterized
(e.g.,~\cite{newey1990semiparametric,hahn98role}).  In the more
challenging setting of adaptive data considered here, these notions of
optimality are not well-understood.  It would be interesting to derive
sharp lower bounds for the adaptive models studied here, and to
propose estimators that achieve these bounds.

Third, the construction of our adaptively weighted Z-estimator relies
on knowing the selection probabilities at each round. In some
applications, including experimental design and in bandit experiments,
this assumption is reasonable.  However, for various of observational
studies, this assumption is less realistic, so that designing optimal
procedures that can operate without such knowledge is an important
direction.

\begin{funding}
{\small{This work was partially supported by Office of Naval Research
    Grant ONR-N00014-21-1-2842 and National Science Foundation grant
    DMS-2311072 to MJW, and funding from the Howard Friesen Chair in
    Engineering at UC Berkeley.}}
\end{funding}


\begin{supplement}
The supplementary material contains all technical proofs and additional discussions.
\end{supplement}

\bibliographystyle{imsart-number} 
\bibliography{reference}       

\newpage

\begin{supplement}

\stitle{Supplement for semi-parametric inference based on adaptively
  collected data} \sdescription{ }

\end{supplement}


\section{Proofs of the theorems}
\label{SecProofs}

We give the proofs of our four general results, with
Sections~\ref{proof:thm:linear_new1}
through~\ref{proof:thm:glm_new1_temp_one_arm} devoted to the proofs of
Theorem~\ref{thm:linear_new1} through
Theorem~\ref{thm:glm_new1_temp_one_arm} respectively.


\subsection{Proof of Theorem~\ref{thm:linear_new1}}
\label{proof:thm:linear_new1}

Recalling the definition~\eqref{eq:linear_score} of the score function
$\ScoreFun_i$, note that it is linear in the parameter vectors
$\TargetPar$ and $\FunPar$, and that we have the convenient
decomposition
\begin{align*}
\EmpMean\ScoreFun_i(\TargetPar, \FunPar) = (\EmpMean \LinScaleVec_i
\iRegressor^\top)(\TargetPar - \TrueTargetPar) + \EmpMean
\LinScaleVec_i \iNoise - \EmpMean \LinScaleVec_i (\FunPar(\iNuisance)
- \TrueFunPar(\iNuisance) ),
\end{align*}
where $\LinScaleVec_i \defn \iCovar^{-1/2}(\iRegressor -
\SelectProb_i)$. By the definition of our $Z$-estimator, the pair
$(\SoluTargetPar, \EstFunPar)$ satisfies the condition
$\EmpMean\ScoreFun_i(\SoluTargetPar, \EstFunPar) = 0$.  Re-arranging
this equality and multiplying both sides by $\sqrt{\NumIndexTwo }$
yields
\begin{align}
\label{linear_asymp1}
\sqrt{\NumIndexTwo } (\EmpMean \LinScaleVec_i
\iRegressor^\top)(\SoluTargetPar - \TrueTargetPar) & =
\sqrt{\NumIndexTwo } \big \{ \EmpMean \LinScaleVec_i \iNoise -
\EmpMean \LinScaleVec_i (\EstFunPar(\iNuisance) -
\TrueFunPar(\iNuisance)) \big \}.
\end{align}
We next analyze equation~\eqref{linear_asymp1} via the following three
results which we prove in
Lemma~\ref{lm:asym_1},~\ref{lm:asym_2}\\~and~\ref{lm:asym_3} (see
details in the Supplement.)
\begin{subequations}
\begin{align}
&\EmpMean \sqrt{\NumIndexTwo }\LinScaleVec_i \iNoise \convdist
  \Normal(0, \LinearNoiseVar \IdMat_{\TargetDim})
\label{eqn:asymp-normality} \\
&\EmpMean \sqrt{\NumIndexTwo }\LinScaleVec_i (\EstFunPar(\iNuisance) -
\TrueFunPar(\iNuisance) )\overset{p}{\to}0
\label{eqn:zero-nuisance-term} \\
&\opnorm{\EmpMean \LinScaleVec_i \iRegressor^\top - \EmpMean
  \iCovar^{1/2}}=\liloh_p(\sigma_{\min}(\EmpMean \iCovar^{1/2}))
\label{eqn:changing-the-variance}
\end{align}
\end{subequations}
With these three results at hand, the rest of the proof is
straightforward. Indeed, substituting the
conditions~\eqref{eqn:asymp-normality} and
~\eqref{eqn:zero-nuisance-term} into equation~\eqref{linear_asymp1}
and applying Slutsky's theorem yields
\begin{align}
\label{eqn:asymp_norm_intermediate}  
(\EmpMean \LinScaleVec_i
\iRegressor^\top)\sqrt{\NumIndexTwo}(\SoluTargetPar - \TrueTargetPar)
\convdist \Normal(0, \LinearNoiseVar \IdMat_{\TargetDim}).
\end{align}
This distributional convergence also implies that $\|(\EmpMean
\LinScaleVec_i \iRegressor^\top )\sqrt{\NumIndexTwo }(\SoluTargetPar -
\TrueTargetPar)\|_2=\bigoh_p(1)$ by continuous mapping theorem,
definition of weak convergence and boundedness in
probability. Combining the last implication with the
bound~\eqref{eqn:changing-the-variance} yields
\begin{align*}
  \| (\EmpMean \iCovar^{1/2} \sqrt{\NumIndexTwo }(\SoluTargetPar -
  \TrueTargetPar)\|_2 &\leq \|(\EmpMean \LinScaleVec_i
  \iRegressor^\top)\sqrt{\NumIndexTwo }(\SoluTargetPar -
  \TrueTargetPar)\|_2 + \|(\EmpMean (\LinScaleVec_i \iRegressor^\top -
  \iCovar^{1/2}) \sqrt{\NumIndexTwo }(\SoluTargetPar -
  \TrueTargetPar)\|_2\\ &= \bigoh_p(1) + \liloh_p(\| (\EmpMean
  \iCovar^{1/2}\sqrt{\NumIndexTwo } (\SoluTargetPar -
  \TrueTargetPar)\|_2 ).
\end{align*}
Putting together the pieces we have ${\| (\EmpMean
  \iCovar^{1/2}\sqrt{\NumIndexTwo } (\SoluTargetPar -
  \TrueTargetPar)\|_2=\bigoh_p(1)}$, and consequently we deduce
\begin{align*}
    \|(\EmpMean (\LinScaleVec_i \iRegressor^\top - \iCovar^{1/2})
    \sqrt{\NumIndexTwo }(\SoluTargetPar - \TrueTargetPar)\|_2
    =\liloh_p(1)
\end{align*}
Finally, combining the last result with the convergence
statement~\eqref{eqn:asymp_norm_intermediate} and applying the
Slutsky's theorem, we conclude that
\begin{align*}
(\EmpMean \iCovar^{1/2})\sqrt{\NumIndexTwo }(\SoluTargetPar -
  \TrueTargetPar) \convdist \Normal(0, \LinearNoiseVar
  \IdMat_{\TargetDim}).
\end{align*}
This completes the proof of Theorem~\ref{thm:linear_new1}.

{\paragraph*{Generalization to continuous regressors:} We remark that
  versions of the three key
  Lemmas~\ref{lm:asym_1},~\ref{lm:asym_2}~and~\ref{lm:asym_3} can also
  be established when the regressors $\iRegressor$ take continuous
  values; consequently, there is a generalization of
  Theorem~\ref{thm:linear_new1} to this continuous
  setting. Concretely, an essential component in the proof of the
  lemmas is to use an lower bound condition on the covariance
  $\iCovar$---cf. in particular
  Assumption~\ref{assn-lin-selection-prob-equi} in
  Section~\ref{SecTechnicalLemmas} of the Supplement) to control the
  $\ell_2$ norm of certain auxiliary quantities. While we show that
  Assumption~\ref{assn-lin-selection-prob}~and~\ref{assn-lin-selection-prob-equi}
  are equivalent in the case of discrete regressors,
  Assumption~\ref{assn-lin-selection-prob-equi} is already assumed in
  the case of continuous regressors as we stated before. Thus, all
  derivations follow from the same arguments. See
  Section~\ref{SecAuxLemmas} in the supplement for more details on
  these arguments.}

{\paragraph*{Consistency of $\EstLinearNoiseVar$ in
    equation~\eqref{eq:consistent_var} }
Here, we prove that the estimator $\EstLinearNoiseVar$ in
  equation~\eqref{eq:consistent_var} is a consistent estimate of the
  noise variance $\LinearNoiseVar$.  Note that
\begin{align*}
|\EstLinearNoiseVar- \LinearNoiseVar| &= |\EmpMean(\iResponse-
\iRegressor^\top \SoluTargetPar- \EstFunPar(\iNuisance))^2-
\LinearNoiseVar | \\
&=
|\EmpMean(\iNoise+\iRegressor^\top\TrueTargetPar+\TrueFunPar(\iNuisance)-
\iRegressor^\top \SoluTargetPar- \EstFunPar(\iNuisance))^2-
\LinearNoiseVar| \\
& \leq |\EmpMean\iNoise^2+\EmpMean(\iRegressor^\top(\TrueTargetPar-
\SoluTargetPar)+(\TrueFunPar(\iRegressor)-
\EstFunPar(\iNuisance)))^2- \LinearNoiseVar| \\
& \qquad
+2\sqrt{\EmpMean\iNoise^2}\cdot\sqrt{\EmpMean(\iRegressor^\top(\TrueTargetPar-
  \SoluTargetPar) +(\TrueFunPar(\iNuisance)-
  \EstFunPar(\iNuisance)))^2},
\end{align*}
where the third line follows from the Cauchy--Schwarz inequality.

We claim that
\begin{align}
\label{EqnAux}
\EmpMean \iNoise^2 \overset{p}{\to}
\LinearNoiseVar,~~~\text{and}~~~\EmpMean(\iRegressor^\top(\TrueTargetPar-
\SoluTargetPar) + (\TrueFunPar(\iNuisance)- \EstFunPar(\iNuisance)))^2
\overset{p}{\to}0.
\end{align}
Equation~\eqref{eq:consistent_var} follows immediately from the
previous bound and these two auxiliary claims.

To prove the first claim in equation~\eqref{EqnAux}, observe that
\begin{align*}
\EmpMean\iNoise^2 = \CondMean\iNoise^2+ (\EmpMean- \CondMean)\iNoise^2
= \LinearNoiseVar+ \liloh_p(1)\overset{p}{\to}\LinearNoiseVar,
\end{align*}
where the second inequality uses
Assumption~\ref{assn-lin-nuisance-est} and the finite fourth moment
condition of $\iNoise$.  To prove the second claim, note that
\begin{align*}
    & \quad\EmpMean(\iRegressor^\top(\TrueTargetPar- \SoluTargetPar) +
  (\TrueFunPar(\iNuisance)- \EstFunPar(\iNuisance)))^2 \\
  & \leq 2 \EmpMean \big[\iRegressor^\top(\TrueTargetPar-
    \SoluTargetPar)^2 \big] + 2 \EmpMean \big[
    (\TrueFunPar(\iNuisance)- \EstFunPar(\iNuisance))^2 \big]\\
  & \leq 2 \vecnorm{\TrueTargetPar- \SoluTargetPar}{2}^2 + 2 \CondMean
  (\TrueFunPar(\iNuisance)- \EstFunPar(\iNuisance))^2 + 2 (\EmpMean-
  \CondMean) (\TrueFunPar(\iNuisance)- \EstFunPar(\iNuisance))^2 \\
  & = 2 \vecnorm{\TrueTargetPar- \SoluTargetPar}{2}^2 + \liloh_p(1) \\
  & = \liloh_p(1),
\end{align*}
where the third line uses $\vecnorm{\iRegressor}{2}\leq 1$, the fourth
line follows from Assumption~\ref{assn-lin-nuisance-est},
Lemma~\ref{lm:tech_md} and the finite fourth moment condition on
$\TrueFunPar(\iNuisance)- \EstFunPar(\iNuisance)$; the last line uses
equation~\eqref{linear_asymp1}--- \eqref{eqn:changing-the-variance},
the fact that $\EmpMean\iCovar^{1/2}\gtrsim\sqrt{\NumIndexTwo}
\Numobs^{-t}\gtrsim \NumIndexTwo^{1-t}\to\infty$ for some $t<1/2,$ and
noting that $\EmpMean \sqrt{\NumIndexTwo }\LinScaleVec_i \iNoise $ in
equation~\eqref{eqn:asymp-normality} has finite variance.  }

\color{black}

\subsection{Proof of Theorem~\ref{thm:linear_new3}}
\label{proof:thm:linear_new3}

The proof of this theorem is essentially the same as that
of Theorem~\ref{thm:linear_new1} but with a different weighting vector.
Let us introduce the shorthand
\begin{align*}
\LinScaleVecOneArm_{i1} & \defn
\inprod{\ScalePreCondVec_{i1}}{\iRegressor - \SelectProb_i(\iNuisance,
  \iHistory)} \in \real.
\end{align*}
Following a decomposition similar to equation~\eqref{linear_asymp1},
we have
\begin{align}
\label{linear_asymp_one_coord_2}
&(\EmpMean \LinScaleVecOneArm_{i1} \inprod{\iRegressor}{\OneDirect})
\sqrt{\NumIndexTwo }(\SoluOneDirectPar - \TrueOneDirectPar) +
     [\EmpMean \LinScaleVecOneArm_{i1} \iRegressor^\top
       (\IdMat_{\TargetDim} - \OneDirect
       \OneDirect^\top)] \sqrt{\NumIndexTwo }(\EstTargetPar -
     \TrueTargetPar) \notag \\
& = [\EmpMean \sqrt{\NumIndexTwo} \LinScaleVecOneArm_{i1} \iNoise -
       \EmpMean \sqrt{\NumIndexTwo }\LinScaleVecOneArm_{i1}
       (\EstFunPar(\iNuisance) - \TrueFunPar(\iNuisance) )].
\end{align}

We prove the following three results
in Lemma~\ref{lm:asym_21},~\ref{lm:asym_22}~and~\ref{lm:asym_23}
respectively (see details in Supplement),
 which
analyze the three terms in the last decomposition.
\begin{subequations}
\begin{align}
&\EmpMean \sqrt{\NumIndexTwo }\LinScaleVecOneArm_{i1} \iNoise
  \convdist \Normal(0, \LinearNoiseVar )
\label{eqn:asymp-normality2} \\
&\EmpMean \sqrt{\NumIndexTwo }\LinScaleVecOneArm_{i1}
(\EstFunPar(\iNuisance) - \TrueFunPar(\iNuisance)
)\overset{p}{\to}0
\label{eqn:zero-nuisance-term2} \\
&\Big|\EmpMean \LinScaleVecOneArm_{i1}
\iRegressor^\top\OneDirect - \EmpMean \frac{1
}{\sqrt{\OneDirect^\top\iCovar^{-1}\OneDirect}}\Big|=\liloh_p\Big(\Big|\EmpMean
\frac{1
}{\sqrt{\OneDirect^\top\iCovar^{-1}\OneDirect}}\Big|\Big).\label{eqn:changing-the-variance2}
\end{align}
\end{subequations}

Assuming these three results are given at the moment, plugging
equation~\eqref{eqn:asymp-normality2},~\eqref{eqn:zero-nuisance-term2}
into~\eqref{linear_asymp_one_coord_2} we deduce
\begin{align}
\label{linear_asymp_one_coord_3}
(\EmpMean \LinScaleVecOneArm_{i1} \iRegressor^\top
\OneDirect)\sqrt{\NumIndexTwo }(\SoluOneDirectPar - \TrueOneDirectPar) +
          [\EmpMean \LinScaleVecOneArm_{i1} \iRegressor^\top
            (\IdMat_{\TargetDim} - \OneDirect\OneDirect^\top)]\sqrt{\NumIndexTwo
          }(\EstTargetPar - \TrueTargetPar) \convdist
          \Normal(0, \LinearNoiseVar ).
\end{align}

Moreover, invoking the bound~\eqref{eqn:changing-the-variance2} we
have
\begin{align}
\label{linear_asymp_one_coord_3_prior}
    \Big(\EmpMean \LinScaleVecOneArm_{i1}
    \iRegressor^\top\OneDirect - \EmpMean \frac{1
    }{\sqrt{\OneDirect^\top\iCovar^{-1}\OneDirect}}\Big)\sqrt{\NumIndexTwo
    }(\SoluTargetPar_{1} - \TrueTargetPar_{1})&=\liloh_p\Big(\EmpMean
    \frac{1
    }{\sqrt{\OneDirect^\top\iCovar^{-1}\OneDirect}}\sqrt{\NumIndexTwo
    }(\SoluTargetPar_{1} - \TrueTargetPar_{1})\Big)
\end{align}
Putting
equation~\eqref{linear_asymp_one_coord_3},~\eqref{linear_asymp_one_coord_3_prior}
together, it remains to show
\begin{align}
\EmpMean \LinScaleVecOneArm_{i1} \iRegressor^\top
(\IdMat_{\TargetDim} - \OneDirect\OneDirect^\top)\sqrt{\NumIndexTwo
}(\EstTargetPar - \TrueTargetPar)=\liloh_p(1).
\label{linear_asymp_one_coord_4_prior} 
\end{align} 
\paragraph*{Proof of equation~\eqref{linear_asymp_one_coord_4_prior}}

Observe that $\E (\LinScaleVecOneArm_{i1} \iRegressor^\top \mid
\iHistory)=0$, which implies that $\{ \LinScaleVecOneArm_{i1}
\iRegressor^\top \}_{i \geq 1}$ forms a martingale difference
sequence. We have
\begin{align*}
\E \|\sqrt{\NumIndexTwo } \EmpMean \LinScaleVecOneArm_{i1}
\iRegressor^\top\|_2^2 & = \frac{1}{\NumIndexTwo }\sum_{i=\NumIndexOne
  + 1}^{\Numobs}\E \|\LinScaleVecOneArm_{i1} \iRegressor^\top\|_2^2 \;
\leq \; \frac{1}{\NumIndexTwo} \sum_{i=\NumIndexOne + 1}^{\Numobs}\E
\LinScaleVecOneArm_{i1}^2=1,
\end{align*}
where the last line follows from the bound $\|\iRegressor^\top\|_2\leq
1$ and noting that $\E \LinScaleVecOneArm_{i1}^2=1$.  Thus, we
conclude that ${\sqrt{\NumIndexTwo }\EmpMean \LinScaleVecOneArm_{i1}
  \iRegressor^\top = \bigoh_p(1)}$. Combining this fact with the
assumption that $(\IdMat_{\TargetDim} -
\OneDirect\OneDirect^\top)(\EstTargetPar -
\TrueTargetPar)\overset{p}{\to}0$ yields the claim.

{\paragraph*{Generalizing to continuous regressors} Similarly,
  Lemma~\ref{lm:asym_21},~\ref{lm:asym_22}~and~\ref{lm:asym_23} can be
  established when the regressors $\iRegressor$ take continuous
  values, and therefore Theorem~\ref{thm:linear_new1} can be
  generalized to this setting. Notably, a key component in the proof
  of the lemmas is to obtain lower bounds on the quantities
  $\OneDirect^\top\iCovar^{-1}\OneDirect/\OneDirect^\top\iCovar^{-2}\OneDirect$
  and $1/\sqrt{\OneDirect^\top\iCovar^{-1}\OneDirect}$;
  cf. equations~\eqref{eq:vi_bound_b}~and~\eqref{eq:linear-compare-scale}. While
  we bound these terms through direct calculation using
  equation~\eqref{eq:cov_inverse} in the discrete regressors case,
  here we explicitly assume they are bounded from below.  }


\subsection{Proof of Theorem~\ref{thm:glm_new1_temp}}
\label{proof:thm:glm_new1_temp}

For notational simplicity, we only prove the result when the nuisance
function is linear, i.e., $\TrueFunPar(\Nuisance) =
\inprod{\Nuisance}{\TrueNuisancePar}$ for some $\TrueNuisancePar
\in \NuisanceSpace\in\R^{\NuisanceDim}$. The proof general nuisance
function $\FunPar$ is essentially the same with
$\inprod{\Nuisance}{\NuisancePar}$ replaced by
$\FunPar(\Nuisance)$. See Section~\ref{sec:gen-nuisance-function-proof} for
more details.

For the linear nuisance function $\TrueFunPar(\Nuisance) =
\inprod{\Nuisance}{\TrueNuisancePar}$, the GLM assumptions and
Assumption~\ref{assn-glm-nuisance-est} can be replaced by the
following two simplified versions:
\begin{itemize}
\item[(a)] (Bounded covariates and nuisance) There exist
  $M_{\NuisancePar},M_{\AllNuisancePar}>0$ such that
  $\sup_{\NuisancePar\in\NuisanceSpace}\|\NuisancePar\|_2\leq
  M_{\NuisancePar}$ and $\|(\iRegressor^\top, \iNuisance^\top)\|_2\leq
  D_{x}$;
\item[(b)] (Accuracy of pilot estimates) The pilot estimator
  $\EstNuisancePar\in\NuisanceSpace$ from Step 3 of~\AlgLongGLM\\
  satisfies $\|\EstNuisancePar - \TrueNuisancePar\|_2 =
  \liloh_p(\Numobs^{-1/4})$.
\end{itemize}


\subsubsection{Main argument}

Substituting the definition of $\ScoreFun_i$ from
equation~\eqref{eq:glm_score} into the estimating
equation~\eqref{eqn:non-linear-estimating-eqn}, we find that
\begin{subequations}
\begin{align}
\label{eq:glm_estimating1}    
\quad &\sqrt{\NumIndexTwo }\EmpMean \InvSqrtCovarwhat_i(\iRegressor -
\GlmMeanVecwhat_i)\iNoise \\
\label{eq:glm_estimating2}
&=\sqrt{\NumIndexTwo }\EmpMean \InvSqrtCovarwhat_i(\iRegressor -
\GlmMeanVecwhat_i)(g(\iRegressor^\top\SoluTargetPar +
\iNuisance^\top\EstNuisancePar )-g(\iRegressor^\top\TrueTargetPar +
\iNuisance^\top\TrueNuisancePar )).
  \end{align}
\end{subequations}
Focusing on the preceding equation, we now perform a second order
Taylor series expansion of $g$ around the point $(\EstTargetPar,
\EstNuisancePar )$, thereby we obtain
\begin{align}
\quad &\sqrt{\NumIndexTwo } \EmpMean \InvSqrtCovarwhat_i(\iRegressor -
\GlmMeanVecwhat_i) \iNoise \notag\\
&= \sqrt{\NumIndexTwo }\EmpMean \InvSqrtCovarwhat_i(\iRegressor -
\GlmMeanVecwhat_i) g'\big(\inprod{\iRegressor}{\EstTargetPar} +
\inprod{\iNuisance}{\EstNuisancePar} \big)(\iRegressor -
\GlmMeanVecwhat_i)^\top(\SoluTargetPar - \TrueTargetPar )
\notag \\
\label{eq:glm_taylor_est}
& \quad + \sqrt{\NumIndexTwo }\EmpMean \InvSqrtCovarwhat_i(\iRegressor
- \GlmMeanVecwhat_i)\Big[Q_1 + Q_2 + Q_3 + Q_4\Big],
\end{align}
where
\begin{subequations}
\begin{align*}
Q_1 & \defn g'(\inprod{\iRegressor}{\EstTargetPar} +
\inprod{\iNuisance}{\EstNuisancePar}
)\inprod{\iNuisance}{\EstNuisancePar - \TrueNuisancePar } \\
Q_2 & \defn g'\big(\inprod{\iRegressor}{\EstTargetPar} +
\inprod{\iNuisance}{\EstNuisancePar}
\big)\GlmMeanVecwhat_i^\top(\SoluTargetPar - \TrueTargetPar ) \\
Q_3 & \defn \frac{1}{2}\int^1_{0}\int^1_{0}
g''\big(\inprod{\iRegressor}{\EstTargetPar + r_1r_2(\SoluTargetPar -
  \EstTargetPar )} + \inprod{\iNuisance}{\EstNuisancePar}
\big)|\inprod{\iRegressor}{\SoluTargetPar - \EstTargetPar }|^2 dr_1d
r_2, \quad \mbox{and} \\
Q_4 & \defn - \frac{1}{2}
\int^1_{0}\int^1_{0}g''\big(\inprod{\iRegressor}{\EstTargetPar + r_1
  r_2 (\TrueTargetPar - \EstTargetPar )} +
\inprod{\iNuisance}{\EstNuisancePar + r_1 r_2(\TrueNuisancePar -
  \EstNuisancePar )}\big)\notag\\&\qquad\qquad~~~~~~~~\cdot|\inprod{\iRegressor}{\TrueTargetPar -
  \EstTargetPar } +\inprod{ \iNuisance}{\TrueNuisancePar -
  \EstNuisancePar }|^2 dr_1 d r_2.
\end{align*}
\end{subequations}
We complete the proof by establishing the following three results.
\begin{subequations}
\begin{align}
&\sqrt{\NumIndexTwo }\EmpMean
  \InvSqrtCovarwhat_i(\iRegressor - \GlmMeanVecwhat_i)\iNoise \convdist
  \Normal(0, \IdMat_{\TargetDim})
\label{eqn:glm_clt_temp} 
\end{align}
If $\SoluTargetPar - \TrueTargetPar=\liloh_p(1)$, then we have
\begin{align}
&\sqrt{\NumIndexTwo }\EmpMean
  \InvSqrtCovarwhat_i(\iRegressor - \GlmMeanVecwhat_i)(Q_1 +
  Q_2)\overset{p}{\to}0
\label{eqn:glm_neyman_temp} 
\end{align}
If $\SoluTargetPar - \TrueTargetPar=\liloh_p(\Numobs^{-t})$, then
\begin{align}
& \quad\sqrt{\NumIndexTwo }\EmpMean
  \InvSqrtCovarwhat_i(\iRegressor - \GlmMeanVecwhat_i)(Q_3 +
  Q_4)\notag\\
    &=\liloh_p(1) + \liloh_p(\|\sqrt{\NumIndexTwo }\EmpMean
  \InvSqrtCovarwhat_i(\iRegressor -
  \GlmMeanVecwhat_i)g'\big(\inprod{\iRegressor}{\EstTargetPar} +
  \inprod{\iNuisance}{\EstNuisancePar} \big)(\iRegressor -
  \GlmMeanVecwhat_i)^\top(\SoluTargetPar - \TrueTargetPar )\|_2).
\label{eqn:glm_second_order_temp}
\end{align}
\end{subequations}
We prove the
claims~\eqref{eqn:glm_clt_temp},~\eqref{eqn:glm_neyman_temp}
and~\eqref{eqn:glm_second_order_temp} in
Lemma~\ref{lm:glm_clt_temp},~\ref{lm:glm_neyman_temp}~and~\ref{lm:glm_second_order_temp},
respectively (see details in Supplement).
  We
also verify in a moment that
\begin{align}
\label{eq:n_t_glm_consistency}
\SoluTargetPar - \TrueTargetPar = \liloh_p(\Numobs^{-t}).
\end{align}
With the last four results at hand, the proof
of Theorem~\ref{thm:glm_new1_temp} is immediate. Indeed, denoting
$\sqrt{\NumIndexTwo }\EmpMean \InvSqrtCovarwhat_i(\iRegressor -
\GlmMeanVecwhat_i)g'\big(\inprod{\iRegressor}{\EstTargetPar} +
\inprod{\iNuisance}{\EstNuisancePar} \big)(\iRegressor -
\GlmMeanVecwhat_i)^\top(\SoluTargetPar - \TrueTargetPar )$ by $Q_0$
and substituting results above into~\eqref{eq:glm_taylor_est} and
using Slusky's theorem yields
\begin{align}
Q_0 + \liloh_p(\opnorm{Q_0}) \convdist \Normal(0,
\IdMat_{\TargetDim})\label{eq:glm_asym_norm2_temp}.
\end{align}
Since $\Normal(0, \IdMat_{\TargetDim})=\bigoh_p(1)$, we have
$Q_0=\bigoh_p(1)$ and hence $\liloh_p(\opnorm{Q_0})=\liloh_p(1)$. This
together with Slusky's theorem yields the result as desired.  It
remains to prove the consistency
condition~\eqref{eq:n_t_glm_consistency}.

\subsubsection{Proof of consistency condition~\eqref{eq:n_t_glm_consistency}}

We use an inductive argument on $k$.  More precisely, we first
establish $\liloh_p(\Numobs^{-(\delta\wedge t)})$-consistency for the
base case $k = 1$.  In the inductive step, we assume that
$\liloh_p(\Numobs^{-(k\delta\wedge t)})$-consistency holds for some $k
\geq 1$, and then prove that it holds at step $(k+1)$---that is,
$\liloh_p(\Numobs^{-((k + 1)\delta\wedge t)})$-consistency holds.

\paragraph*{Base case}
We start by proving $\liloh_p(\Numobs^{-(\delta \wedge
  t)})$-consistency.

Introduce the shorthand $\AllNuisancePar \defn (\AuxiNuisancePar,
\NuisancePar)$, $\EstAllNuisancePar \defn (\EstTargetPar,
\EstNuisancePar )$ for the estimator computed in Step 3 of
\AlgShortGLM, and $\TrueAllNuisancePar \defn (\TrueTargetPar,
\TrueNuisancePar )$.

By the triangle inequality and the relation $\EmpMean
\ScoreFun_i(\SoluTargetPar, \EstAllNuisancePar) = 0$, we have
\begin{align}
\|\CondMean \ScoreFun_i(\SoluTargetPar, \TrueAllNuisancePar)\|_2 &\leq
\|(\CondMean \ScoreFun_i(\SoluTargetPar, \TrueAllNuisancePar) -
\CondMean \ScoreFun_i(\SoluTargetPar, \EstAllNuisancePar)\|_2 +
\|(\CondMean - \EmpMean ) \ScoreFun_i(\SoluTargetPar,
\EstAllNuisancePar)\|_2\notag\\
 & \leq \sup_{\TargetPar \in \TargetSpace}\|\CondMean
\ScoreFun_i(\TargetPar, \TrueAllNuisancePar) - \CondMean
\ScoreFun_i(\TargetPar, \EstAllNuisancePar)\|_2\notag +
\sup_{\TargetPar \in\TargetSpace} \|(\EmpMean - \CondMean )
\ScoreFun_i(\TargetPar, \EstAllNuisancePar)\|_2, \notag\\
\label{eq:glm_consistent_upperbound1}
& =: \; \Re_1 + \Re_2.
\end{align}

Since $\quad\|\CondMean \ScoreFun_i(\TargetPar, \TrueAllNuisancePar) -
\CondMean \ScoreFun_i(\TargetPar, \EstAllNuisancePar)\|_2\leq
L_{\ScoreFun,1}\|\EstAllNuisancePar - \TrueAllNuisancePar\|_2 $ for
some constant $L_{\ScoreFun,1}>0$ by Lemma~\ref{lm:glm_lip_cont_temp},
it follows that $\Re_1=\liloh_p(\Numobs^{-1/4}).$ For $\Re_2$, it
follows from Lemma~\ref{lm:glm_emp_bern} that $\Re_2=\bigoh_p(\log
\Numobs /\sqrt{\Numobs})$.

Combining the results above, we obtain $\|\CondMean
\ScoreFun_i(\SoluTargetPar, \TrueAllNuisancePar)\|_2=\liloh_p(\Numobs^{-1/4})$. On
the other hand,
\begin{align*}
&\|\CondMean \ScoreFun_i(\SoluTargetPar, \TrueAllNuisancePar)\|_2\\ &=
  \|\CondMean
  (\ScoreFun_i(\SoluTargetPar, \TrueAllNuisancePar) - \ScoreFun_i(\TrueTargetPar,
  \TrueAllNuisancePar))\|_2
  \\ &\geq 
  c_\ScoreFun\|\CondMean
  \partial_{\TargetPar}\ScoreFun_i(\TrueTargetPar,
  \TrueAllNuisancePar)(\SoluTargetPar  - \TrueTargetPar )\|_2\wedge
  c_\ScoreFun \Numobs^{-1/4} 
  \\ &\geq
  c_\ScoreFun\sigma_{\min}(\CondMean
  \partial_{\TargetPar}\ScoreFun_i(\TrueTargetPar,
  \TrueAllNuisancePar))\|\SoluTargetPar  - \TrueTargetPar \|_2\wedge
  c_\ScoreFun \Numobs^{-1/4}\\ 
  &\geq
  c_\ScoreFun{m_{\ScoreFun}}\Numobs^{\delta-t}\|\SoluTargetPar
   - \TrueTargetPar \|_2\wedge c_\ScoreFun \Numobs^{-1/4}
\end{align*}
with probability converging to one. Here the inequalities follows from
the identifiability assumptions in Theorem~\ref{thm:glm_new1_temp}. Therefore,
 \begin{align}
    \liloh_p(\Numobs^{-1/4})= \|\CondMean
    \ScoreFun(\SoluTargetPar, \TrueAllNuisancePar)\|_2\geq
    c_\ScoreFun
    m_{\ScoreFun}\Numobs^{\delta-t}\|\SoluTargetPar  - \TrueTargetPar
    \|_2\wedge c_\ScoreFun
    \Numobs^{-1/4} \label{eq:glm_consistent_upperlower1}
 \end{align}
 with probability converging to one. Since $t\leq 1/4$, it follows
 directly that $\|\SoluTargetPar  - \TrueTargetPar
 \|_2=\liloh_p(\Numobs^{ - \delta})=\liloh_p(\Numobs^{-(\delta\wedge
   t)})$.

\paragraph*{Inductive step}

Next we show that given $\SoluTargetPar - \TrueTargetPar =
\liloh_p(\Numobs^{-(k\delta\wedge t)})$, we have $\SoluTargetPar -
\TrueTargetPar =\liloh_p(\Numobs^{-((k + 1)\delta\wedge t)})$. In
fact, it suffices to show that $\|\CondMean
\ScoreFun_i(\SoluTargetPar, \TrueAllNuisancePar)\|_2 = \liloh_p(
\Numobs^{-1/4-(k\delta\wedge t)})$ for any $t \in (k \delta,
1/4]$. This is because combining it with the high probability lower
bound $m_{\ScoreFun}\Numobs^{-t}\|\SoluTargetPar  - \TrueTargetPar \|_2$
from equation~\eqref{eq:glm_consistent_upperlower1} directly gives
$\|\SoluTargetPar  - \TrueTargetPar \|_2 = \liloh_p(\Numobs^{-((k + 1)
  \delta \wedge t)})$ as desired.

Following the same steps as in the upper
bound~\eqref{eq:glm_consistent_upperbound1} and using the result that
$\Re_2=\bigoh_p(\log \Numobs /\sqrt{\Numobs})$, we obtain
\begin{align}
\|\CondMean \ScoreFun_i(\SoluTargetPar, \TrueAllNuisancePar)\|_2 & \leq
\|(\CondMean \ScoreFun_i(\SoluTargetPar, \TrueAllNuisancePar) - \CondMean
\ScoreFun_i(\SoluTargetPar, \EstAllNuisancePar)\| + \|(\CondMean
 - \EmpMean
)\ScoreFun_i(\SoluTargetPar, \EstAllNuisancePar)\|_2\notag\\ 
&\leq \sup_{\TargetPar \in\TargetSpace, \|\TargetPar - \TrueTargetPar
  \|\leq \Numobs^{-(k\delta \wedge t)}}\|\CondMean
\ScoreFun_i(\TargetPar, \TrueAllNuisancePar) - \CondMean
\ScoreFun_i(\TargetPar, \EstAllNuisancePar)\|\notag\\ &\quad +
\sup_{\TargetPar \in\TargetSpace} \|(\EmpMean  - \CondMean
)\ScoreFun_i(\TargetPar, \EstAllNuisancePar)\|_2, \notag\\ &=
\sup_{\TargetPar \in\TargetSpace, \|\TargetPar - \TrueTargetPar \|_2\leq
  \Numobs^{-(k\delta \wedge t)}}\|\CondMean
\ScoreFun_i(\TargetPar, \TrueAllNuisancePar) - \CondMean
\ScoreFun_i(\TargetPar, \EstAllNuisancePar)\|_2 + \bigoh_p(\log \Numobs
/\sqrt{\Numobs}).
 \label{eq:glm_consistent_upperbound2}
\end{align}
Denote $\{\TargetPar \in\TargetSpace, \|\TargetPar - \TrueTargetPar
\|_2\leq \Numobs^{-(k\delta \wedge t)}\}$ by $\mathcal C_k$. Then
\begin{align}
    &\sup_{\mathcal C_k}\|\CondMean \ScoreFun_i(\TargetPar,
  \TrueAllNuisancePar) - \CondMean \ScoreFun_i(\TargetPar,
  \EstAllNuisancePar)\|_2\notag\\ &\leq \sup_{\mathcal C_k}\|\CondMean
  \partial_{\AllNuisancePar}\ScoreFun_i(\TargetPar,
  \TrueAllNuisancePar)(\EstAllNuisancePar -
  \TrueAllNuisancePar)\notag\\ & + \frac{1}{2}(\EstAllNuisancePar -
  \TrueAllNuisancePar)^\top\Big[\int_{0}^{1}\int_{0}^{1} \CondMean
    \partial^2_{\AllNuisancePar}\ScoreFun_i(\TargetPar,
    \TrueAllNuisancePar + r_1 r_2(\EstAllNuisancePar -
    \TrueAllNuisancePar))dr_1 dr_2\Big](\EstAllNuisancePar -
  \TrueAllNuisancePar)\|_2\notag\\ &\lesssim \sup_{\mathcal
    C_k}\|[\CondMean \partial_{\AllNuisancePar}\ScoreFun_i(\TargetPar,
    \TrueAllNuisancePar) - \CondMean
    \partial_{\AllNuisancePar}\ScoreFun_i(\TrueTargetPar,
    \TrueAllNuisancePar)](\EstAllNuisancePar -
  \TrueAllNuisancePar)\|_2 + L_{\ScoreFun,2}\|\EstAllNuisancePar -
  \TrueAllNuisancePar\|_2^2\notag\\ &\leq \sup_{\mathcal
    C_k}L_{\ScoreFun,1}\|\SoluTargetPar - \TrueTargetPar
  \|_2\|\EstAllNuisancePar - \TrueAllNuisancePar\|_2 +
  L_{\ScoreFun,2}\|\EstAllNuisancePar -
  \TrueAllNuisancePar\|_2^2=\liloh_p(\Numobs^{-(k\delta \wedge
    t)-1/4}), \label{eq:glm_consistent_upperbound3}
\end{align}
where $L_{\ScoreFun,1},L_{\ScoreFun,2}>0$ are some constants
introduced in Lemma~\ref{lm:glm_lip_cont_temp}. In the second line we use
Taylor expansion, the second inequality follows from Neyman
orthogonality
$\CondMean\partial_\AllNuisancePar\ScoreFun_i(\TrueTargetPar,
\TrueAllNuisancePar)=0$ and Lemma~\ref{lm:glm_lip_cont_temp}, and the last
line is also due to Lemma~\ref{lm:glm_lip_cont_temp} and the assumption
that $\EstAllNuisancePar -
\TrueAllNuisancePar=\liloh_p(\Numobs^{-1/4})$ and $\SoluTargetPar -
\TrueTargetPar =\liloh_p(\Numobs^{-(k\delta\wedge t)})$.  Since
$(k\delta \wedge t) + 1/4<1/2$, we conclude by
combining~\eqref{eq:glm_consistent_upperbound2}
and~\eqref{eq:glm_consistent_upperbound3} that $\|\CondMean
\ScoreFun_i(\SoluTargetPar,
\TrueAllNuisancePar)\|_2=\liloh_p(\Numobs^{-(k\delta \wedge t)-1/4})$.

\subsubsection{Proof for general nuisance function $\FunPar$}
\label{sec:gen-nuisance-function-proof}
We remark that the proof for general nuisance function $\FunPar$ is
essentially the same with, $\Nuisance^\top\NuisancePar$ replaced by
$\FunPar(\Nuisance)$. This is because in the proof we only invoke our
assumption on the estimation error
$\|\EstFunPar(\Nuisance) - \TrueFunPar(\Nuisance)\|_\infty$ and does not
exploit the specific form of $\FunPar$. However, by assuming a linear
parameterization on $\FunPar$, we can avoid the usage of Gateaux
derivative and hence simplify our notation of the gradient of the
score function.

\subsection{Proof of Theorem~\ref{thm:glm_new1_temp_one_arm}}
\label{proof:thm:glm_new1_temp_one_arm}

The proof of this theorem is similar to the proof
of Theorem~\ref{thm:glm_new1_temp}, and we only prove it for linear nuisance,
i.e.,
$\TrueFunPar(\Nuisance)=\inprod{\Nuisance}{\TrueNuisancePar}$. We only
provide a proof sketch for brevity.

Substituting the definition of $\ScoreFun_{i1}$ into the estimating
equation~\eqref{eq:glm_one_arm_estimating_equation} yields
\begin{align*}
\quad &\sqrt{\NumIndexTwo }\EmpMean
\ScalePreCondVecwhat_{i1}(\iNuisance, \iHistory)(\iRegressor -
\GlmMeanVecwhat_i(\iNuisance, \iHistory))\iNoise \\
& = \sqrt{\NumIndexTwo }\EmpMean \ScalePreCondVecwhat_{i1}(\iRegressor
- \GlmMeanVecwhat_i)\cdot\notag\\
& \quad \Big[g\big(\inprod{\iRegressor}{\OneDirect}\SoluOneDirectPar +
  \iRegressor^\top (\IdMat_{\TargetDim} -
  \OneDirect\OneDirect^\top)\EstTargetPar +
  \inprod{\iNuisance}{\EstNuisancePar}
  \big)-g\big(\inprod{\iRegressor}{\OneDirect}\TrueOneDirectPar +
  \iRegressor^\top (\IdMat_{\TargetDim} -
  \OneDirect\OneDirect^\top)\TrueTargetPar +
  \inprod{\iNuisance}{\TrueNuisancePar} \big)\Big].
\end{align*}
Throughout, we use the shorthand $\EstOneDirectPar \defn
\inprod{\OneDirect}{\EstTargetPar}$.  Performing a second order Taylor
series expansion of $g$ on the right-hand side of the last equation at
$(\EstOneDirectPar, \EstTargetPar, \EstNuisancePar )$, we obtain
\begin{align}
 &\quad\sqrt{\NumIndexTwo }\EmpMean \ScalePreCondVecwhat_{i1}(\iNuisance, \iHistory)(\iRegressor - \GlmMeanVecwhat_i(\iNuisance, \iHistory))\iNoise\\
&=
\sqrt{\NumIndexTwo }\EmpMean \ScalePreCondVecwhat_{i1}(\iRegressor - \GlmMeanVecwhat_i)  g'\big(\inprod{\iRegressor}{\EstTargetPar}  + \inprod{\iNuisance}{\EstNuisancePar}\big)(\iRegressor - \GlmMeanVecwhat_{i})^\top\OneDirect(\SoluOneDirectPar - \TrueOneDirectPar) \notag\\
&\quad+
\sqrt{\NumIndexTwo }\EmpMean \ScalePreCondVecwhat_{i1}(\iRegressor - \GlmMeanVecwhat_i)  \Big[ \Qwtil_1 + \Qwtil_2 + \Qwtil_3 + \Qwtil_4 \Big], \label{eq:glm_taylor_est_one_arm}
\end{align}

where
\begin{align*}
\Qwtil_1 &=
 g'\big(\inprod{\iRegressor}{\EstTargetPar}  + \inprod{\iNuisance}{\EstNuisancePar}\big)
          [\iRegressor^\top(\IdMat_{\TargetDim} - \OneDirect\OneDirect^\top)(\EstTargetPar - \TrueTargetPar) + \inprod{\iNuisance}{\EstNuisancePar
             - \TrueNuisancePar }], \\
\Qwtil_2 &= g'\big(\inprod{\iRegressor}{\EstTargetPar}  + \inprod{\iNuisance}{\EstNuisancePar}\big)\inprod{
\GlmMeanVecwhat_{i}}{\OneDirect}(\SoluOneDirectPar - \TrueOneDirectPar) \\
\Qwtil_3 &= \frac{1}{2}\int^1_{0}\int^1_{0}
g''\big(\inprod{\iRegressor}{\OneDirect}(\EstOneDirectPar + r_1r_2(\SoluOneDirectPar - \EstOneDirectPar)) + \iRegressor^\top(\IdMat_{\TargetDim} - \OneDirect\OneDirect^\top)\EstTargetPar + \inprod{\iNuisance}{\EstNuisancePar}
\big)\notag\\
 &\qquad\qquad~~~~~~~~\cdot\big|\inprod{\iRegressor}{\OneDirect}(\SoluOneDirectPar - \EstOneDirectPar)\big|^2
dr_1d r_2 \\
\Qwtil_4 &=  - \frac{1}{2} \int^1_{0} \int^1_{0}
g''\Big(\inprod{\iRegressor}{\OneDirect}(\EstOneDirectPar  +  r_1
r_2(\TrueOneDirectPar -
\EstOneDirectPar)) 
+
 \iRegressor^\top(\IdMat_{\TargetDim} - \OneDirect\OneDirect^\top)(\EstTargetPar + r_1
r_2(\TrueTargetPar - \EstTargetPar)) 
\notag\\
 &\qquad\qquad~~~~~~~~+\inprod{\iNuisance}{\EstNuisancePar
 + r_1 r_2(\TrueNuisancePar  - \EstNuisancePar )}\Big)\cdot|\inprod{\iRegressor}{\TrueTargetPar -
  \EstTargetPar } +\inprod{ \iNuisance}{\TrueNuisancePar -
  \EstNuisancePar }|^2 dr_1 d r_2.
\end{align*}
Following an argument similar to Lemma~\ref{lm:glm_clt_temp},~\ref{lm:glm_neyman_temp}~and~\ref{lm:glm_second_order_temp},
it can be shown that
\begin{subequations}
\begin{align}
\label{eq:glm_clt_temp_one_arm}
\sqrt{\NumIndexTwo} \EmpMean 
\ScalePreCondVecwhat_{i1}(\iNuisance, \iHistory)(\iRegressor-
\GlmMeanVecwhat_i(\iNuisance, \iHistory))\iNoise& \convdist \Normal(0,1),
\\
\label{eq:glm_neyman_temp_one_arm} 
\sqrt{\NumIndexTwo} \EmpMean \ScalePreCondVecwhat_{i1}
(\iRegressor - \GlmMeanVecwhat_i)(\Qwtil_1 + \Qwtil_2) &
\overset{p}{\to} 0 \\
\label{eq:glm_second_order_temp_one_arm}
\sqrt{\NumIndexTwo }\EmpMean \ScalePreCondVecwhat_{i1}(\iRegressor -
\GlmMeanVecwhat_i) (\Qwtil_3 + \Qwtil_4) & \overset{p}{\to} 0
\end{align}
\end{subequations}
Here the claim~\eqref{eq:glm_second_order_temp_one_arm} requires the
assumption $\SoluOneDirectPar - \TrueOneDirectPar =
\liloh_p(\Numobs^{-t})$.  We can prove this
$\liloh_p(\Numobs^{-t})$-consistency following arguments that are
similar to those used in the proof of Theorem~\ref{thm:glm_new1_temp}.
Concretely, the proof is essentially the same except for replacing
$\ScoreFun_{i}$ with $\ScoreFun_{i1}$ and showing $\CondMean
\ScoreFun_{i1}, \CondMean \partial_\AllNuisancePar\ScoreFun_{i1}$ are
bounded and Lipschitz continuous in $(\OneDirectPar,
\AllNuisancePar)$. This completes the proof
of Theorem~\ref{thm:glm_new1_temp_one_arm}.

\section{Proofs of the corollaries}

Our proofs of the corollaries depend on the following technical lemma.
In stating it, we make use of the shorthand notation of $\iCovariate
\defn (\iRegressor^\top, \iNuisance^\top)^\top$, and define
$\CovariateMa \equiv \sum_{i=1}^{\NumIndexOne} \iCovariate
\iCovariate^\top$.

\begin{lems}
\label{lm:linear_model_exm}
Suppose that
Assumptions~\ref{assn-lin-noise}-- \ref{assn-lin-nuisance-est} hold for
some $t \in [0, 1/4)$, and moreover that
\begin{align*}
 \|\iNuisance\|_2 \leq B \quad \mbox{and} \quad \inf_{P\in\Pclass}
 \sigma_{\min}(\E_{\iNuisance \sim \Pclass} \iNuisance
 \iNuisance^\top) \geq c_P \qquad \mbox{for some $B < \infty$ and $c_P
   > 0$.}
\end{align*}
Then there exists some constant $c_{\CovariateMa} > 0$ such that the
minimum eigenvalue satisfies
\begin{align}
\lim_{n \to \infty} \prob(\sigma_{\min}(\CovariateMa)\geq
c_{\CovariateMa} \NumIndexOne^{1-2t}) \to 1.
  \end{align}
\end{lems}
\noindent We return to prove Lemma~\ref{lm:linear_model_exm} in
Section~\ref{proof:lm:linear_model_exm}.  Here we complete the proofs
of the corollaries using Lemma~\ref{lm:linear_model_exm}.


\subsection{Proof of Corollary~\ref{corr:linear_model_exm}}
\label{proof:corr:linear_model_exm}

In light of Theorem~\ref{thm:linear_new1}, it is sufficient to show
that \mbox{$\|\NuisanceParwhat_{ols} - \TrueNuisancePar\|_2 =
  \liloh_p(1)$.}  In order to do so, we invoke results due to Lai and
Wei~\cite{lai1982least}.  Specifically, denote the vector
$(\iRegressor^\top, \iNuisance^\top)^\top$ by $\iCovariate$ and let
$\CovariateMa\equiv\sum_{i=1}^{\NumIndexOne} \iCovariate
\iCovariate^\top$. By Theorem 1 in Lai and Wei~\cite{lai1982least} it
suffices to show that \mbox{$(\log
  \sigma_{\max}(\CovariateMa)/\sigma_{\min}(\CovariateMa))^{1/2} =
  \liloh_p(1)$.} Since both vectors $\iRegressor$ and $\iNuisance$ are
bounded in $\ell_2-$norm, we have $\log
\sigma_{\max}(\CovariateMa)=\bigoh(\log n)$.  Thus,
Lemma~\ref{lm:linear_model_exm} ensures that for any $t \in [0, 1/2)$,
  we have
  $1/\sigma_{\min}(\CovariateMa)=\bigoh_p(n^{2t-1})=\liloh_p(n^{ -
    \eps})$ for some small $\eps>0$.  Putting together the pieces, we
  conclude that $\|\NuisanceParwhat_{ols} -
  \TrueNuisancePar\|_2=\liloh_p(1)$, as claimed in
  Corollary~\ref{corr:linear_model_exm}.

  
\subsection{Proof of Corollary~\ref{corr:linear_model_exm_2}}
\label{proof:corr:linear_model_exm_2}
In light of Theorem~\ref{thm:linear_new1}, it only remains to show
that \mbox{$\|\NuisanceParwhat_{lasso} - \TrueNuisancePar\|_2 =
  \liloh_p(1)$.}  In order to do so, we exploit results due to Oh et
al.~\cite{oh2021sparsity}.  Define the index set
\begin{align*}
  \SparseSet
  \defn \{1,2, \ldots, \TargetDim\} \cup \{i + \TargetDim \mid
  \TrueNuisancePar_i\neq 0\},
\end{align*}
and introduce the shorthand notation $\iCovariate \defn
(\iRegressor^\top, \iNuisance^\top)^\top$, along with
\begin{align*}
  \CovariateMa \defn
  \sum_{i=1}^{\NumIndexOne}\iCovariate\iCovariate^\top,
  \AllParameter^* \defn (\TargetPar^{*\top}, \qquad
  \NuisancePar^{*\top})^\top, \quad \mbox{and} \quad
  \AllParameterwhat_{lasso} \defn (\EstTargetPar^{\top}_{lasso},
  \EstNuisancePar^{\top}_{lasso})^\top.
\end{align*}
For any vector $\AllParameter\in\R^{\TargetDim + \NuisanceDim}$, we
define the vector $\AllParameter_\SparseSet$ with $j$-th entry
\mbox{$\AllParameter_{j, \SparseSet} \defn \AllParameter_j 1_{j \in
    \SparseSet}$.}  Invoking Lemma~\ref{lm:linear_model_exm} yields
\begin{align*}
  \|\AllParameter_\SparseSet\|_1^2/\NumIndexOne^{2t} \leq |\SparseSet|
  \|\AllParameter_\SparseSet\|_2^2/\NumIndexOne^{2t}\lesssim
  \frac{|\SparseSet|}{\NumIndexOne} \cdot \AllParameter^\top
  \CovariateMa\AllParameter
\end{align*}
for all $\AllParameter$. Consequently, the compatibility condition in
Assumption 3 of the paper~\cite{oh2021sparsity} is satisfied with
$\phi^2_{\NumIndexOne}=c\NumIndexOne^{-2t}$ for some constant $c>0$
with probability converging to one. Thus, we may apply Lemma 1 in the
paper~\cite{oh2021sparsity} (note that the lemma remains true with $B'$ being the upper bound of  $\vecnorm{\iNuisance}{\infty}$  instead of $\vecnorm{\iNuisance}{2}$) to assert that
\begin{align*}
\|\NuisanceParwhat_{lasso} - \TrueNuisancePar\|_2 \leq
\|\AllParameterwhat_{lasso} - \AllParameter^*\|_1 & \leq
\frac{4(\SparseLevel +
  \TargetDim)\lambda_{\NumIndexOne}}{\phi^2_{\NumIndexOne}} \\
&= \frac{8(\SparseLevel + \TargetDim)\NumIndexOne^{2t}\subgauss(B' +
  1)}{c} \sqrt{\frac{2[\log (2/\delta_{\NumIndexOne}) +
      \log(\TargetDim + \NuisanceDim)]}{\NumIndexOne}}\\
& \lesssim (\SparseLevel + \TargetDim) \NumIndexOne^{2t-1/2}
\sqrt{\log (2/\delta_{\NumIndexOne}) + \log(\TargetDim +
  \NuisanceDim)}
\end{align*}
with probability $1 - \delta_{\NumIndexOne} -
\prob(\sigma_{\min}(\CovariateMa)<c\NumIndexOne^{1-2t})$.  Plugging in
$\delta_{\NumIndexOne}=\min\{(\SparseLevel +
\TargetDim)\NumIndexOne^{2t-1/2},\\1/(\TargetDim + \NuisanceDim)\}$ and
$(\SparseLevel + \TargetDim)\sqrt{\log (\TargetDim +
  \NuisanceDim)}=\liloh_p(\NumIndexOne^{1/2-2t})$, and noting that
$\TargetDim$ is fixed, we obtain
\begin{align*}
(\SparseLevel + \TargetDim) \NumIndexOne^{2t-1/2} \sqrt{\log (2 /
    \delta_{\NumIndexOne}) + \log( \TargetDim + \NuisanceDim)}
= \liloh_p( 1).
\end{align*}
From our choice of $\delta_{\NumIndexOne}$
and Lemma~\ref{lm:linear_model_exm}, it follows that $1 -
\delta_{\NumIndexOne} - \prob(\sigma_{\min}(\CovariateMa) < c
\NumIndexOne^{1-2t}) \to 1$. Thus, we conclude that
$\|\NuisanceParwhat_{lasso} - \TrueNuisancePar\|_2 = \liloh_p(1)$, and
this completes the proof of Corollary~\ref{corr:linear_model_exm_2}.

\subsection{Proof of Corollary~\ref{corr:linear_model_exm_3}}
\label{proof:corr:linear_model_exm_3}
The proof is essentially the same as the proof
of Corollary~\ref{corr:linear_model_exm_2}. Recall our notation from the proof
of Corollary~\ref{corr:linear_model_exm_2}.
Invoking Lemma~\ref{lm:linear_model_exm} yields
\begin{align*}
\|\AllParameter_\SparseSet\|_1^2/\NumIndexOne^{2t}
\leq |\SparseSet|
\|\AllParameter_\SparseSet\|_2^2/\NumIndexOne^{2t}\lesssim \frac{
  |\SparseSet|}{\NumIndexOne} \cdot \AllParameter^\top \CovariateMa
\AllParameter
\end{align*}
for all $\AllParameter$. Thus, the compatibility condition in Oh et
al.~\cite{oh2021sparsity} is satisfied with
$\phi^2_{\NumIndexOne}=c\NumIndexOne^{-2t}$ for some constant $c>0$
with probability converging to one. Invoking Lemma 1 in Oh et
al.~\cite{oh2021sparsity} we deduce that
\begin{align*}
    \max\left\{ \|\TargetParwhat_{lasso} - \TrueTargetPar\|_2,
    \|\NuisanceParwhat_{lasso} - \TrueNuisancePar\|_2 \right\}
    & \leq \|\AllParameterwhat_{lasso} - \AllParameter^*\|_1 \leq
    \frac{4 (\SparseLevel + \TargetDim)
      \lambda_{\NumIndexOne}}{l_g\phi^2_{\NumIndexOne}}\\
     & = \frac{8(\SparseLevel + \TargetDim)\NumIndexOne^{2t}\subgauss
      D_x}{c}\sqrt{\frac{2[\log (2/\delta_{\NumIndexOne}) +
          \log(\TargetDim + \NuisanceDim)]}{\NumIndexOne}}\\ &\lesssim
    (\SparseLevel + \TargetDim) \NumIndexOne^{2t-1/2}\sqrt{\log
      (2/\delta_{\NumIndexOne}) + \log(\TargetDim + \NuisanceDim)}
  \end{align*}
with probability at least $1 - \delta_{\NumIndexOne} -
\Prob(\sigma_{\min}(\CovariateMa) < c \NumIndexOne^{1-2t})$.  Making
the substitution $\delta_{\NumIndexOne} \defn \min \{(\SparseLevel +
\TargetDim) \NumIndexOne^{2t-1/4}, 1/(\TargetDim + \NuisanceDim)\}$
and $(\SparseLevel + \TargetDim)\sqrt{\log (\TargetDim +
  \NuisanceDim)} = \liloh_p(\NumIndexOne^{1/4-2t})$, and noting that
$\TargetDim$ is fixed, we obtain
\begin{align*}
 (\SparseLevel + \TargetDim)
  \NumIndexOne^{2t-1/2}\sqrt{\log
    (2/\delta_{\NumIndexOne}) + \log(\TargetDim + \NuisanceDim)}
    =\liloh_p(
  \NumIndexOne^{-1/4}).
\end{align*}   From our choice of $\delta_{\NumIndexOne}$ and Lemma~\ref{lm:linear_model_exm}, it follows that $1 - \delta_{\NumIndexOne} - \prob(\sigma_{\min}(\CovariateMa)<c\NumIndexOne^{1-2t})\to 1$. Putting together the pieces, we conclude that $ \max\left\{ \|\TargetParwhat_{lasso} - \TrueTargetPar\|_2, \|\NuisanceParwhat_{lasso} - \TrueNuisancePar\|_2 \right\} =\liloh_p(\NumIndexOne^{-1/4})$; this completes the proof of Corollary~\ref{corr:linear_model_exm_3}.

{
\subsection{Proof of Corollary~\ref{corr:linear_model_exm_4}}
\label{proof:corr:linear_model_exm_4}
Given Theorem~\ref{thm:linear_new1}, it suffices to prove that \begin{align}
\E_{\iNuisance\sim P}(\EstFunPar(\iNuisance)- \TrueFunPar(\iNuisance
))^2=o_p(1).\label{eq:proof_cor_linear_4_eq1}
\end{align} 
Since $\SelectProb_{i0}\geq c_0 i^{-2t}$ for some $t<1/2$ and $\NumIndexOne=\Numobs/K$ for some constant  $K\geq2$, using Freedman's inequality (see e.g., Lemma 9 in~\cite{agarwal2014taming}),  we have \begin{align}\sum_{i=1}^{\NumIndexOne} \mathbf{1}_{\{\iRegressor=\mathbf{0}\}}\geq \sum_{i=1}^{\NumIndexOne}\SelectProb_{i0}/2\geq c \Numobs^{1-2t}\label{eq:freedman_1}
\end{align} 
with probability converging to $1$ as $\Numobs\to \infty$ for some
constant $c>0$. Since we assume the selection probabilities
$\SelectProb_i$ depend only on $\iHistory$ and $\iNuisance$ are
i.i.d., letting
$(\widetilde\Nuisance_1,\widetilde\Regressor_1,\widetilde\Response_1),\ldots,(\widetilde\Nuisance_{c\Numobs^{1-2t}},
\\ \widetilde\Regressor_{c\Numobs^{1-2t}},\widetilde\Response_{c\Numobs^{1-2t}})$
denote the first $c\Numobs^{1-2t}$ samples in first $\NumIndexOne$
samples such that the corresponding regressor
$\iRegressor=\mathbf{0}$\footnote{We generate additional independent
samples if there are less than $c\Numobs^{1-2t}$ such samples.}, it
can be verified by induction that
$\{\widetilde\iNuisance\}_{i=1}^{c\Numobs^{1-2t}}$ are i.i.d. samples
from $P$
and \begin{align*}\widetilde\iResponse=\TrueFunPar(\widetilde\iNuisance)+\widetilde\iNoise
\end{align*}
for some  i.i.d. noise $\widetilde\iNoise\sim Q$. 
Therefore, as shown in Theorem 6.2 of~\cite{gyorfi2002distribution}, the $k$-nearest neighbor estimator  $\EstFunPar$ with $k\to\infty,k/n^{1-2t}\to 0$ 
based on the samples $\{(\widetilde\iNuisance,\widetilde\iResponse)\}_{i=1}^{c\Numobs^{1-2t}}$
satisfies $\E_{\Nuisance\sim P}(\EstFunPar(\iNuisance)- \TrueFunPar(\iNuisance))^2=o_p(1)$. Equation~\eqref{eq:proof_cor_linear_4_eq1} follows immediately since we can find such i.i.d. samples in the first $\NumIndexOne$ observed samples with probability converging to one as shown in  equation~\eqref{eq:freedman_1}. 
}



\subsection{Proof of Lemma~\ref{lm:linear_model_exm}}
\label{proof:lm:linear_model_exm}
It suffices to show that
\begin{align*}
\lim_{\NumIndexOne \to \infty} \prob(\sigma_{\max}(\CovariateMa^{-1})
\leq \NumIndexOne^{2t-1}/c_{\CovariateMa} ) \to 1 \qquad \mbox{for
  some constant $c_{\CovariateMa}$.}
\end{align*}
Using the Sherman-Woodbury formula for block-partitioned matrix
inverses, we have
\begin{align*}
\CovariateMa^{-1} \equiv
\begin{bmatrix}
\CovariateMa_1 & \CovariateMa_2 \\ \CovariateMa_3 & \CovariateMa_4
\end{bmatrix}
=
\begin{bmatrix}
\IdMat_{\TargetDim} & 0 \\  - \CovariateMa_4^{-1} \CovariateMa_3 &
\IdMat_{\NuisanceDim}
\end{bmatrix}
\begin{bmatrix}
\Big(\CovariateMa_1 - \CovariateMa_2 \CovariateMa_4^{-1}
\CovariateMa_3\Big)^{-1} & 0 \\ 0 & \CovariateMa_4^{-1}
\end{bmatrix}
\begin{bmatrix}
\IdMat_{\TargetDim} &  - \CovariateMa_2 \CovariateMa_4^{-1} \\ 0 &
\IdMat_{\NuisanceDim}
\end{bmatrix},
\end{align*}
where $\CovariateMa_1 = \sum_{i=1}^{\NumIndexOne} \iRegressor
\iRegressor^\top, \CovariateMa_2 = \CovariateMa_3^\top =
\sum_{i=1}^{\NumIndexOne} \iRegressor \iNuisance^\top$ and
$\CovariateMa_4 = \sum_{i=1}^{\NumIndexOne}\iNuisance\iNuisance^\top$.
Since the vectors $\iNuisance's$ are i.i.d. with bounded second
moment, it follows from the boundedness of $\iNuisance$ and
Lemma~\ref{lm:tech_md} that $\opnorm{\CovariateMa_4/\NumIndexOne -
  \PreCondMean \iNuisance \iNuisance^\top}
=\bigoh_p(\NumIndexOne^{-1/2})$. Combining this fact with the lower
bound
\begin{align*}
\sigma_{\min}(\PreCondMean \iNuisance\iNuisance^\top)
\geq 
\inf_{P\in\Pclass}\sigma_{\min}(\E_{\iNuisance\sim
  \Pclass}\iNuisance\iNuisance^\top)\geq c_p>0,
\end{align*}
we obtain $\lim_{\NumIndexOne\to\infty}
\prob(\opnorm{\CovariateMa_4/\NumIndexOne}\geq c_P/2)\to 1$ and hence
$\opnorm{\CovariateMa_4^{-1}}=\bigoh_p(n^{-1})$. Also, it follows from the
boundedness of $\iRegressor$ and $\iNuisance$ that
$\opnorm{\CovariateMa_3}, \opnorm{\CovariateMa_2}=\bigoh_p(n)$.  Combining the results above
we obtain
$\opnorm{\CovariateMa_4^{-1}\CovariateMa_3}=\opnorm{(\CovariateMa_2\CovariateMa_4^{-1})^\top}=\bigoh_p(1)$
and thus \begin{align*}\Big|\mns\Big|\mns\Big|\begin{bmatrix} \IdMat_{\TargetDim} & 0 \\  - \CovariateMa_4^{-1}
  \CovariateMa_3 & \IdMat_{\NuisanceDim}
\end{bmatrix}\Big|\mns\Big|\mns\Big|_{\mathrm{op}}, \Big|\mns\Big|\mns\Big|\begin{bmatrix}
\IdMat_{\TargetDim} &  - \CovariateMa_2 \CovariateMa_4^{-1} \\ 0 &
\IdMat_{\NuisanceDim}
\end{bmatrix}\Big|\mns\Big|\mns\Big|_{\mathrm{op}}=\bigoh_p(1).\end{align*}
Now, by the submultiplicativity of spectral norm and the fact that
$\opnorm{\CovariateMa^{-1}_4}=\bigoh_p(n^{-1})=\liloh_p(n^{2t-1})$, it remains
to show
$\lim_{\Numobs\to\infty}\prob(\opnorm{\Big(\CovariateMa_1 - \CovariateMa_2
\CovariateMa_4^{-1} \CovariateMa_3\Big)^{-1}}\lesssim n^{2t-1})\to
1$, or equivalently,
$\lim_{\Numobs\to\infty}\prob(\sigma_{\min}\Big(\CovariateMa_1 - \CovariateMa_2
\CovariateMa_4^{-1} \CovariateMa_3\Big)\gtrsim n^{1-2t})\to 1$.

For $\CovariateMa_1$, note that by the one-hot property of
$\iRegressor$ we have
$\iRegressor\iRegressor^\top - \diag\{\SelectProb_{i1}, \ldots, \SelectProb_{i\TargetDim}\}$
forms a matrix-valued Martingale difference sequence. Since
$\iRegressor, \SelectProb_i$ are bounded, it follows
from Lemma~\ref{lm:tech_md} that
$\|\CovariateMa_1 - \diag\{\sum_{i=1}^{\NumIndexOne}\SelectProb_{i1}, \ldots, \sum_{i=1}^{\NumIndexOne}
\SelectProb_{i\TargetDim} \} \|_F = \bigoh_p(\NumIndexOne^{1/2})$.

With slight abuse of notation, we denote $(\Regressor_1, \ldots,
\Regressor_{\NumIndexOne})^\top$ by $\RegressorMa$, $(\Nuisance_1,
\ldots, \Nuisance_{\NumIndexOne})^\top$ by $\NuisanceMa$ and
$(\SelectProb_1, \ldots, \SelectProb_{\NumIndexOne})^\top$ by
$\AllSelecProbMa$. Then
$\CovariateMa_2\CovariateMa_4^{-1}\CovariateMa_3=\RegressorMa^\top\HatMa
\RegressorMa$ where the projection matrix $\HatMa \defn
\NuisanceMa(\NuisanceMa^\top\NuisanceMa)^{-1}\NuisanceMa^\top$. Similarly,
$(\iRegressor - \SelectProb_i)\iNuisance^\top$ forms a matrix-valued
martingale difference sequence and $\E \|(\iRegressor -
\SelectProb_i)\iNuisance^\top\|^2_F$ is bounded. It then follows
from Lemma~\ref{lm:tech_md} that $\frobnorm{(\RegressorMa -
  \AllSelecProbMa)^\top\NuisanceMa}=\bigoh_p(\NumIndexOne^{1/2})$. Substituting
this into $\RegressorMa^\top\HatMa\RegressorMa$, we obtain
\begin{align*}
&\quad \opnorm{\RegressorMa^\top \HatMa \RegressorMa -
    \AllSelecProbMa^\top \NuisanceMa(\NuisanceMa^\top
    \NuisanceMa)^{-1} \NuisanceMa^\top \AllSelecProbMa}\\
  &= \opnorm{\RegressorMa^\top\NuisanceMa(\NuisanceMa^\top
    \NuisanceMa)^{-1} \NuisanceMa^\top \RegressorMa -
    \AllSelecProbMa^\top \NuisanceMa(\NuisanceMa^\top\NuisanceMa)^{-1}
    \NuisanceMa^\top \AllSelecProbMa} \\
& = \opnorm{ \AllSelecProbMa^\top
    \NuisanceMa(\NuisanceMa^\top\NuisanceMa)^{-1}\NuisanceMa^\top(\RegressorMa
    - \AllSelecProbMa) + (\RegressorMa -
    \AllSelecProbMa)^\top\NuisanceMa(\NuisanceMa^\top
    \NuisanceMa)^{-1}\NuisanceMa^\top \AllSelecProbMa + (\RegressorMa
    - \AllSelecProbMa)^\top\NuisanceMa(\NuisanceMa^\top
    \NuisanceMa)^{-1} \NuisanceMa^\top (\RegressorMa -
    \AllSelecProbMa)} \\
  & \leq \opnorm{ \AllSelecProbMa^\top
    \NuisanceMa(\NuisanceMa^\top\NuisanceMa)^{-1}
    \NuisanceMa^\top(\RegressorMa - \AllSelecProbMa)} +
  \opnorm{(\RegressorMa - \AllSelecProbMa)^\top \NuisanceMa
    (\NuisanceMa^\top\NuisanceMa)^{-1}\NuisanceMa^\top
    \AllSelecProbMa}  \notag\\
    &\qquad+
    \opnorm{(\RegressorMa - \AllSelecProbMa)^\top
    \NuisanceMa (\NuisanceMa^\top \NuisanceMa)^{-1} \NuisanceMa^\top
    (\RegressorMa - \AllSelecProbMa)} \\
& =\bigoh_p \Big( \NumIndexOne^{1/2}\Big),
\end{align*}
where the last line uses the relations
\begin{align*}
\opnorm{({\NuisanceMa^\top \NuisanceMa})^{-1}} &=
\opnorm{{\CovariateMa}_4^{-1}}=\bigoh_p(n^{-1}), \quad
\opnorm{\NuisanceMa^\top\AllSelecProbMa} = \bigoh(\Numobs),  \text{ and } \\ 
\frobnorm{(\RegressorMa - \AllSelecProbMa)^\top
  \NuisanceMa} &= \bigoh_p(\NumIndexOne^{1/2}).
\end{align*}
Combining the pieces yields
\begin{align*}
 \opnorm{\CovariateMa_1 - \CovariateMa_2 \CovariateMa_4^{-1}
\CovariateMa_3} 
& =  \opnorm{\RegressorMa^\top(\IdMat - \HatMa)\RegressorMa}
\\
& =  \opnorm{\diag \big \{{\sum_{i=1}^{\NumIndexOne}\SelectProb_{i1}},
\ldots, {\sum_{i=1}^{\NumIndexOne} \SelectProb_{i\TargetDim}} \big\} -
\AllSelecProbMa^\top
\NuisanceMa(\NuisanceMa^\top\NuisanceMa)^{-1}\NuisanceMa^\top
\AllSelecProbMa} + \bigoh_p(\Numobs^{1/2})\\
& \geq  \opnorm{
\diag\Big\{{\sum_{i=1}^{\NumIndexOne}\SelectProb_{i1}}, \ldots,{\sum_{i=1}^{\NumIndexOne}\SelectProb_{i\TargetDim}}\Big\} - \AllSelecProbMa^\top
\AllSelecProbMa} +
\bigoh_p(\Numobs^{1/2})\\
& = \opnorm{\sum_{i=1}^{\NumIndexOne} \iCovar}+
\bigoh_p(\Numobs^{1/2})\geq c_0\Numobs^{1-2t} +
\bigoh_p(\Numobs^{1/2}),
\end{align*}
where the first inequality uses the bound $ \opnorm{\NuisanceMa (
  \NuisanceMa^\top\NuisanceMa)^{-1} \NuisanceMa^\top}\leq 1$ combined
with positive definiteness; the last line follows from
Assumption~\ref{assn-lin-selection-prob}. Since $t \in (0, 1/4)$ by
assumption, we have \mbox{$1 - 2 t > 1/2$,} so that the proof is
complete.

\section{Neyman orthogonality and other assumptions}
\label{neyman_ortho}

In this section, we verify several conditions on the  score functions  we construct, including the Neyman orthogonality, and Assumption~\ref{assn-glm-minimum-eig},~\ref{assn-glm-minimum-eig-weak},~\ref{assn-glm-identifiability}~and~~\ref{assn-glm-identifiability-weak} on logistic models.

\subsection{Linear model}
Recalling the 
definition of $\ScoreFun_i$ from~\eqref{eq:linear_score} we have
\begin{subequations}
\begin{align}
\label{eqn:lin_verify_condition_mean}
\E(\ScoreFun_i(\TrueTargetPar, \TrueFunPar)\mid\iHistory) & =
\E_{\iRegressor, \iNuisance, \iNoise}[\iCovar^{-1/2}(\iRegressor -
  \SelectProb_i(\iNuisance, \iHistory))\iNoise\mid\iHistory]=0,
\end{align}
where the second equality uses the fact that
$\E(\iNoise\mid\iRegressor, \iNuisance, \iHistory)=0$. 
Next note that
 \begin{align}
\label{eqn:lin_verify_gradient}
 & \E(\partial_{\FunPar}\ScoreFun_i(\TrueTargetPar, \TrueFunPar)
      [\FunParwbar - \TrueFunPar] \mid \iHistory) \notag \\
& = \E_{\iRegressor, \iNuisance}[ - \iCovar^{-1/2} (\iRegressor -
        \SelectProb_i(\iNuisance, \iHistory))
        (\FunParwbar(\iNuisance) - \TrueFunPar(\iNuisance)) \mid
        \iHistory ] \notag \\
& = \E_{\iNuisance} \E_{\iRegressor}[ - \iCovar^{-1/2} (\iRegressor -
        \SelectProb_i(\iNuisance, \iHistory)) \mid \iHistory,
        \iNuisance](\FunParwbar(\iNuisance) -
      \TrueFunPar(\iNuisance) ) = 0,
\end{align}
\end{subequations}
where the last line follows from $\E[\iRegressor -
  \SelectProb_i(\iNuisance, \iHistory) \mid \iHistory, \iNuisance] =
0$. This verifies the condition~\eqref{eqn:score-nuisance-gradient}.


\subsection{Generalized linear model}

Recalling the definition of the score function $\ScoreFun_i$ from
equation~\eqref{eq:glm_score}, we have
\begin{align}
 \label{eq:glm_verify_mean} 
\E(\ScoreFun_i(\TrueTargetPar, \TrueTargetPar, \TrueFunPar) \mid
\iHistory) & = \E_{\iRegressor, \iNuisance,
  \iNoise}[\InvSqrtCovar^*_{i}(\iRegressor -
  \GlmMeanVec^*_i(\iNuisance, \iHistory)) \iNoise \mid \iHistory] = 0
\end{align}
We write $\ScoreFunwtil_i(\TargetPar, \TrueFunPar, \GlmMeanVec^*_i,
\InvSqrtCovar^*_i)=\ScoreFun_i(\TargetPar , \TrueTargetPar,
\TrueFunPar)$ to represent the explicit dependency of $\ScoreFun_i$ on
$\InvSqrtCovar_i$ and $\GlmMeanVec_i$.  To verify the gradient
conditions~\eqref{eq:no_theta01} and~\eqref{eq:no_theta012} we first
compute the partial derivatives of $\ScoreFunwtil_i$ wrt
$\InvSqrtCovar_i, \GlmMeanVec_i$ and $\FunPar$.  Concretely, for any
$\InvSqrtCovarbar_{i}=\InvSqrtCovarbar_{i}(\iNuisance, \iHistory)$,
\begin{align}
\label{eq:glm_verify_omega}  
& \quad\E(\partial_{\InvSqrtCovar_{i}}\ScoreFunwtil_i(\TrueTargetPar,
\TrueFunPar, \GlmMeanVec^*_i, \InvSqrtCovar^*_i)[\InvSqrtCovarbar_{i}
  - \InvSqrtCovar_{i}^*]\mid\iNuisance, \iHistory)\nonumber \\
& = \E_{\iRegressor, \iNuisance}[(\InvSqrtCovarbar_{i} -
  \InvSqrtCovar_{i}^*)(\iRegressor - \GlmMeanVec^*_i(\iNuisance,
  \iHistory))(\iResponse-g\big(\inprod{\iRegressor}{ \TrueTargetPar} +
  \TrueFunPar(\iNuisance) \big)) \mid\iNuisance, \iHistory]
\nonumber \\
& = \E_{\iRegressor, \iNuisance}[(\InvSqrtCovarbar_{i} -
  \InvSqrtCovar_{i}^*)(\iRegressor - \GlmMeanVec^*_i(\iNuisance,
  \iHistory))\iNoise \mid\iNuisance, \iHistory] = 0.
 \end{align}
Similarly, for any $\GlmMeanVecbar_i=\GlmMeanVecbar_i(\iNuisance,
\iHistory)$,
\begin{align}
& \quad\E(\partial_{\GlmMeanVec_{i}}\ScoreFunwtil_i(\TrueTargetPar,
  \TrueFunPar, \GlmMeanVec^*_i, \InvSqrtCovar^*_i)[\GlmMeanVecbar_{i}
    - \GlmMeanVec_{i}^*]\mid\iNuisance, \iHistory)\nonumber\\ &=
  \E_{\iRegressor, \iNuisance}[ - \InvSqrtCovar_{i}^*(\GlmMeanVecbar_i
    - \GlmMeanVec^*_i)(\iResponse-g\big(\inprod{\iRegressor}{
      \TrueTargetPar} + \TrueFunPar(\iNuisance) \big)) \mid\iNuisance,
    \iHistory] \nonumber\\
\label{eq:glm_verify_m}   
& = \E_{\iRegressor, \iNuisance}[ -
  \InvSqrtCovar_{i}^*(\GlmMeanVecbar_i - \GlmMeanVec^*_i)\iNoise
  \mid\iNuisance, \iHistory]=0.
 \end{align}
 Moreover, holding $\InvSqrtCovar_i, \GlmMeanVec_i$ as fixed, for any
 $\FunParbar = \FunParbar(\iNuisance)$
\begin{align}
& \quad\E(\partial_{\FunPar}\ScoreFunwtil_i(\TrueTargetPar,
  \TrueFunPar, \GlmMeanVec^*_i, \InvSqrtCovar^*_i)[\FunParbar -
    \TrueFunPar]\mid\iNuisance, \iHistory)\nonumber\\ &=
  \E[\InvSqrtCovar^*_{i}(\iRegressor - \GlmMeanVec^*_i(\iNuisance,
    \iHistory))g'\big(\inprod{\iRegressor}{ \TrueTargetPar} +
    \TrueFunPar(\iNuisance) \big)[\FunParbar -
      \TrueFunPar]\mid\iNuisance, \iHistory] \nonumber\\ &=
  \InvSqrtCovar^*_{i}\E[(\iRegressor - \GlmMeanVec^*_i(\iNuisance,
    \iHistory))g'\big(\inprod{\iRegressor}{ \TrueTargetPar} +
    \TrueFunPar(\iNuisance) \big)\mid\iHistory,
    \iNuisance](\FunParbar(\iNuisance) -
  \TrueFunPar(\iNuisance))\nonumber \\
\label{eq:glm_verify_theta1}  
& = \InvSqrtCovar^*_{i} \E_{\iNuisance}( 0 \mid\iNuisance,
\iHistory)(\FunParbar(\iNuisance) - \TrueFunPar(\iNuisance)) = 0.
 \end{align}
Putting the pieces together and applying the chain rule, we obtain
\begin{align}
& \E( \partial_{\AuxiNuisancePar}\ScoreFunwtil_i(\TrueTargetPar,
  \TrueTargetPar, \TrueFunPar ) \mid \iHistory) \notag \\
& = \E( \E( \partial_{\InvSqrtCovar_i} \ScoreFunwtil_i
  \partial_{\AuxiNuisancePar} \InvSqrtCovar_i +
  \partial_{\GlmMeanVec_i} \ScoreFunwtil_i \partial_{\AuxiNuisancePar}
  \GlmMeanVec_i + \partial_{\FunPar} \ScoreFunwtil_i
  \partial_{\AuxiNuisancePar} \FunPar \mid \iNuisance, \iHistory) \mid
  \iHistory)\notag \\
  & =\E (\E( \partial_{\InvSqrtCovar_i} \ScoreFunwtil_i
  \mid\iNuisance, \iHistory) \partial_{\AuxiNuisancePar}
  \InvSqrtCovar_i + \E(\partial_{\GlmMeanVec_i} \ScoreFunwtil_i \mid
  \iNuisance, \iHistory) \partial_{\AuxiNuisancePar} \GlmMeanVec_i +
  \E(\partial_{\FunPar}\ScoreFunwtil_i \mid \iNuisance, \iHistory)
  \partial_{\AuxiNuisancePar} \FunPar \mid \iHistory)\notag \\
\label{eq:glm_verfy_theta01}
  & =0.
 \end{align}
 Similarly, for any $\FunParbar=\FunParbar(\iNuisance)$, we have
 \begin{align} 
\label{eq:glm_verify_theta012}  
 \E(\partial_{\FunPar}\ScoreFun_i(\TrueTargetPar, \TrueTargetPar,
 \TrueFunPar )[\FunParbar - \TrueFunPar]\mid\iHistory) & = 0.
 \end{align}
Therefore, we conclude that $\ScoreFun_i(\TargetPar, \AuxiNuisancePar,
\TrueFunPar)$ is a Neyman orthogonal score function at
$(\TrueTargetPar, \TrueTargetPar, \TrueFunPar )$ with nuisance
$(\AuxiNuisancePar, \FunPar)$.

\subsection{Comments on the assumptions of logistic regression}
\label{sec:logistic-regression}

In this section, we show that
Assumptions~\ref{assn-glm-minimum-eig},~\ref{assn-glm-minimum-eig-weak}~and~\ref{assn-glm-identifiability-weak}
are satisfied in the setting of logistic regression. Moreover,
Assumption~\ref{assn-glm-identifiability} is satisfied in the special
case $\TargetDim=1$.

Let us first verify Assumption~\ref{assn-glm-minimum-eig}.  For
logistic regression, the inverse link function is given by $g(x) =
e^x/(1 + e^x)$, and we have $g'(x)=e^x/(1 +
e^x)^2=\GlmVar(x)$. Therefore, using the definition of
$\InvSqrtCovar_i^*$, we have
\begin{align*}
&\quad\CondMean \InvSqrtCovar^*_{i}(\iRegressor -
  \GlmMeanVec^*_i)g'\big(\inprod{\iRegressor}{ \TrueTargetPar} +
  \TrueFunPar(\iNuisance) \big)(\iRegressor - \GlmMeanVec^*_i)^\top\\
&=\CondMean \InvSqrtCovar^*_{i}(\iRegressor -
\GlmMeanVec^*_i)\GlmVar\big(\inprod{\iRegressor}{ \TrueTargetPar} +
\TrueFunPar(\iNuisance) \big)(\iRegressor - \GlmMeanVec^*_i)^\top 
 = \CondMean \InvSqrtCovar^{*-1}_i\succeq c
 \Numobs^{\delta-t}\IdMat_{\TargetDim}\end{align*} for some $c>0$,
where the last inequality follows from
Lemma~\ref{lm:equiv_glm}. Setting $m_{\ScoreFun,2} = c$, we see that
Assumption~\ref{assn-glm-minimum-eig} on the minimum singular value
holds.

Similarly, for Assumption~\ref{assn-glm-minimum-eig-weak}, it follows
from the definition of $\ScalePreCondVec_{i1}^*$ that
\begin{align*}
&\quad \CondMean \ScalePreCondVec^*_{i1}(\iRegressor -
  \GlmMeanVec^*_i) g'\big(\inprod{\iRegressor}{ \TrueTargetPar} +
  \TrueFunPar(\iNuisance) \big) (\iRegressor -
  \GlmMeanVec^*_{i})^\top\OneDirect\\
 &= \CondMean \ScalePreCondVec^*_{i1}(\iRegressor - \GlmMeanVec^*_i)
  \GlmVar \big(\inprod{\iRegressor}{ \TrueTargetPar} +
  \TrueFunPar(\iNuisance) \big) (\iRegressor - \GlmMeanVec^*_{i})^\top
  \OneDirect
  = \CondMean \tfrac{1}{\sqrt{\OneDirect^\top \InvSqrtCovar^{*,2}_i
      \OneDirect}} \geq c \Numobs^{\delta-t}
\end{align*}
for some $c > 0$, where the last inequality follows from the explicit
formula of $\Covar^{-1}_i$ in equation~\eqref{eq:glm_sigma_inverse}
and Assumption~\ref{assn-glm-prob-weak}.  Choosing $m_{\ScoreFun,2}=c$
yields Assumption~\ref{assn-glm-minimum-eig-weak}.

\newcommand{\GlmDiff}{{{\Delta}}}
\newcommand{\GlmDiffVec}{{\mathbf{\Delta}^{\mathrm{vec}}_i}}
\newcommand{\GlmDiffBar}{{\bar{\Delta}}}
\newcommand{\Indicator}{e}
\newcommand{\InvCompDiagMod}{{\widetilde{\InvCompDiag}}}
To verify Assumption~\ref{assn-glm-identifiability}, we first claim that
\begin{align}
 {\CondMean (\ScoreFun_i(\TargetPar, \TrueTargetPar, \TrueFunPar
 ) - \ScoreFun_i(\TrueTargetPar, \TrueTargetPar, \TrueFunPar ))}={[\CondMean\InvSqrtCovar^{*,-1}_i\InvCompDiagMod_i](\TrueTargetPar- \TargetPar)}
 \label{eq:verify_ide_claim1},
\end{align} 
where $\InvCompDiagMod_i$ are some diagonal matrices satisfying
$c_{\InvCompDiag,1}\IdMat_{\TargetDim}\preceq\InvCompDiagMod_i\preceq
c_{\InvCompDiag,2}\IdMat_{\TargetDim}$ for some constants
$c_{\InvCompDiag,1},c_{\InvCompDiag,2}>0$ that may depend on the
problem parameters. We return to establish this claim at the end of
the proof.  On the other hand, we have
\begin{align*}
\CondMean \partial_{\TargetPar}
 \ScoreFun_i(\TrueTargetPar, \TrueTargetPar, \TrueFunPar
 )(\TrueTargetPar- \TargetPar )
 &=\CondMean
\InvSqrtCovar^*_{i}(\iRegressor - \GlmMeanVec^*_i)g'\big(\inprod{\iRegressor}{
  \TrueTargetPar} + \TrueFunPar(\iNuisance) \big)\iRegressor^\top(\TrueTargetPar- \TargetPar)\\
  &=
  \CondMean
\InvSqrtCovar^*_{i}(\iRegressor - \GlmMeanVec^*_i)g'\big(\inprod{\iRegressor}{
  \TrueTargetPar} + \TrueFunPar(\iNuisance) \big)(\iRegressor -
\GlmMeanVec^*_i)^\top(\TrueTargetPar- \TargetPar)\\
&=[\CondMean
\InvSqrtCovar^{*,-1}_{i}](\TrueTargetPar- \TargetPar).
\end{align*} 
Therefore, it remains to show 
\begin{align}\label{eq:verify_ide}
\opnorm{[\CondMean
\InvSqrtCovar^{*,-1}_{i}][\CondMean\InvSqrtCovar^{*,-1}_i\InvCompDiagMod_i]^{-1}}=O_p(1).
\end{align}
When $\TargetDim=1$, since $\InvSqrtCovar^{*,-1}_i>0$, we have
\begin{align*}
\opnorm{[\CondMean
\InvSqrtCovar^{*,-1}_{i}][\CondMean\InvSqrtCovar^{*,-1}_i\InvCompDiagMod_i]^{-1}}
&=|{[\CondMean
\InvSqrtCovar^{*,-1}_{i}]/[\CondMean\InvSqrtCovar^{*,-1}_i\InvCompDiagMod_i]|}
\leq
|{[\CondMean
\InvSqrtCovar^{*,-1}_{i}]/[\CondMean\InvSqrtCovar^{*,-1}_i\cdot c_{\InvCompDiag,1}]}|\\
&=1/c_{\InvCompDiag,1}=O_p(1). 
\end{align*}Therefore by choosing $c_\ScoreFun=c_{\InvCompDiag,1}$ we have verified Assumption~\ref{assn-glm-identifiability} for logistic models with $\TargetDim=1$.

Lastly, we verify
Assumption~\ref{assn-glm-identifiability-weak}. Through similar
calculations, we find that
\begin{align*} 
\CondMean (\ScoreFun_{i1}(\OneDirectPar, \TrueTargetPar, \TrueFunPar )
- \ScoreFun_{i1}(\TrueOneDirectPar, \TrueTargetPar, \TrueFunPar ))
=\Big[\CondMean
  \tfrac{\OneDirect^\top\InvCompDiagMod_i\OneDirect}{\sqrt{\OneDirect^\top
      \InvSqrtCovar_i^{*,2}
      \OneDirect}}\Big](\TrueOneDirectPar- \OneDirectPar),
\end{align*}
where $\InvCompDiagMod_i$ are diagonal matrices satisfying
$c'_{\InvCompDiag,1}\IdMat_{\TargetDim}\preceq\InvCompDiagMod_i\preceq
c'_{\InvCompDiag,2}\IdMat_{\TargetDim}$ for some constants
$c_{\InvCompDiag,1},c_{\InvCompDiag,2}>0$ that may depend on the
problem parameters. Moreover, we have
\begin{align*}
\CondMean
  \partial_{\OneDirectPar}\ScoreFun_{i1}(\TrueOneDirectPar, \TrueTargetPar,
  \TrueFunPar )(\OneDirectPar - \TrueOneDirectPar)=\Big[\CondMean \frac{1}{\sqrt{\OneDirect^\top \InvSqrtCovar_i^{*,2} \OneDirect}}\Big](\TrueOneDirectPar- \OneDirectPar).
\end{align*}
Since 
\begin{align*}
\Big|\Big[\CondMean \frac{\OneDirect^\top\InvCompDiagMod_i\OneDirect}{\sqrt{\OneDirect^\top \InvSqrtCovar_i^{*,2} \OneDirect}}\Big](\TrueOneDirectPar- \OneDirectPar)\Big|
\geq 
\Big|\Big[\CondMean \frac{c'_{\InvCompDiag,1}\vecnorm{\OneDirect}{2}^2}{\sqrt{\OneDirect^\top \InvSqrtCovar_i^{*,2} \OneDirect}}\Big](\TrueOneDirectPar- \OneDirectPar)\Big|
=
c'_{\InvCompDiag,1}\Big|\Big[\CondMean \frac{1}{\sqrt{\OneDirect^\top \InvSqrtCovar_i^{*,2} \OneDirect}}\Big](\TrueOneDirectPar- \OneDirectPar)\Big|,
\end{align*}
Assumption~\ref{assn-glm-identifiability-weak} follows immediately by choosing $c_\ScoreFun=c'_{\InvCompDiag,1}$.

\paragraph*{Proof of claim~\eqref{eq:verify_ide_claim1}}
Note that we have 
\begin{align*}
 &\qquad{\CondMean (\ScoreFun_i(\TargetPar, \TrueTargetPar,
    \TrueFunPar ) - \ScoreFun_i(\TrueTargetPar, \TrueTargetPar,
    \TrueFunPar ))} \\
& = {\CondMean \InvSqrtCovar_i^{*} (\iRegressor- \GlmMeanVec_i^*) \big(
    g(\inprod{\iRegressor}{\TrueTargetPar}+\TrueFunPar(\iNuisance)) -
    g(\inprod{\iRegressor}{\TargetPar} +
    \TrueFunPar(\iNuisance))\big)} \\
& = {\CondMean\InvSqrtCovar_i^{*,-1}\InvSqrtCovar_i^{*,2}(\iRegressor
    - \GlmMeanVec_i^*) \big(g(\inprod{\iRegressor}{\TrueTargetPar} +
    \TrueFunPar(\iNuisance)) - g(\inprod{\iRegressor}{\TargetPar} +
    \TrueFunPar(\iNuisance))\big)}.
\end{align*}
Define $\GlmDiff_k = \GlmDiff_k(\TargetPar) \defn
g(\inprod{\Indicator_k}{\TrueTargetPar} + \TrueFunPar(\iNuisance)) -
g(\inprod{\Indicator_k}{\TargetPar} + \TrueFunPar(\iNuisance))$ for $k
\in [\TargetDim]$, $\GlmDiff_0:=0$ and write $\GlmDiffVec \defn
(\GlmDiff_1, \ldots, \GlmDiff_{\TargetDim})^\top$. Then we have
\begin{align*}
& \qquad \CondMean \InvSqrtCovar_i^{*,-1} \InvSqrtCovar_i^{*,2}
  (\iRegressor - \GlmMeanVec_i^*) \big(
  g(\inprod{\iRegressor}{\TrueTargetPar} + \TrueFunPar(\iNuisance)) -
  g(\inprod{\iRegressor}{\TargetPar}+\TrueFunPar(\iNuisance)) \big) \\
& = \CondMean \InvSqrtCovar_i^{*,-1} \E(\InvSqrtCovar_i^{*,2}
  (\iRegressor- \GlmMeanVec_i^*) \big(
  g(\inprod{\iRegressor}{\TrueTargetPar}+\TrueFunPar(\iNuisance)) -
  g(\inprod{\iRegressor}{\TargetPar} + \TrueFunPar(\iNuisance)) \big)
  | \iNuisance,\iHistory) \\
& =
  \CondMean\InvSqrtCovar_i^{*,-1}\InvSqrtCovar_i^{*,2}\DiagProbMatrix(\GlmDiffVec- \GlmMeanVecbar^*\GlmDiffBar),
\end{align*}
where $\DiagProbMatrix \defn \diag\{\SelectProb_{i1}, \ldots,
\SelectProb_{i{\TargetDim}}\}, \GlmMeanVecbar^*:=\DiagProbMatrix^{-1}
\GlmMeanVec_i^*$ and $\GlmDiffBar \defn \sum_{j=0}^{\TargetDim}
\SelectProb_{ij} \GlmDiff_j$. Here the last line follows from taking
the conditional expectation over $\iRegressor$. We omit the dependence
on time $i$ in $\SelectProb_{ij}$ for notational simplicity. Moreover,
from equation~\eqref{eq:glm_sigma_inverse} in the proof of
Lemma~\ref{lm:glm_lip_cont_temp2}, we have
\begin{align*}
\InvSqrtCovar_i^{*,2}\DiagProbMatrix(\GlmDiffVec- \GlmMeanVecbar^*\GlmDiffBar)
&=
(\InvCompOne_i+\InvCompTwo_i)\DiagProbMatrix(\GlmDiffVec- \GlmMeanVecbar^*\GlmDiffBar)\\ &=
(\InvCompDiag_i+\InvCompTwo_i\DiagProbMatrix)(\GlmDiffVec- \GlmMeanVecbar^*\GlmDiffBar),
\end{align*}
where $\InvCompDiag_i=\diag\{1/\epsbar^*_1, \ldots,1/\epsbar^*_{\TargetDim}\}$,
$\InvCompOne_i \defn \DiagProbMatrix^{-1}\InvCompDiag_i=\diag\{1/(\SelectProb_1\epsbar^*_1), \ldots,1/(\SelectProb_{\TargetDim
}\epsbar^*_{\TargetDim})\}$, 
\begin{align*}
\InvCompTwo_i:=\InvCompDiag_i\InvCompRow_i\frac{\begin{pmatrix} - \SelectProb_0
    \epsbar^*_0 & \GlmMeanVecbar^*_0 \SelectProb_0\\ 
    \GlmMeanVecbar^*_0 \SelectProb_0 & \sum_{k=1}^{\TargetDim}
    \SelectProb_k 
    \GlmMeanVecbar_k^{*2}/\epsbar^*_k\end{pmatrix}}{(\sum_{k=1}^{\TargetDim
  }\SelectProb_k\GlmMeanVecbar_k^{*2}/
  \epsbar^*_k)\SelectProb_0\epsbar^*_0 + \GlmMeanVecbar_0^{*2}
  \SelectProb_0^2}\InvCompRow_i^\top \InvCompDiag_i
  =
  \frac{1_{\TargetDim}1_{\TargetDim}^\top}{\SelectProb_0\epsbar^*_0}+\frac{\InvCompDiag_i\InvCompRow_i\begin{pmatrix} - \SelectProb_0
    \epsbar^*_0 & \GlmMeanVecbar^*_0 \SelectProb_0\\ 
    \GlmMeanVecbar^*_0 \SelectProb_0 & -
    \SelectProb_0 
    \GlmMeanVecbar_0^{*2}/\epsbar^*_0\end{pmatrix}\InvCompRow_i^\top \InvCompDiag_i}{(\sum_{k=1}^{\TargetDim
  }\SelectProb_k\GlmMeanVecbar_k^{*2}/
  \epsbar^*_k)\SelectProb_0\epsbar^*_0 + \GlmMeanVecbar_0^{*2}
  \SelectProb_0^2},
\end{align*}
$\InvCompRow_i \defn \begin{pmatrix} 
\GlmMeanVecbar_1&\GlmMeanVecbar_2 &\cdots &\GlmMeanVecbar_{\TargetDim
}\\ \epsbar^*_1&\epsbar^*_2 &\cdots &\epsbar^*_{\TargetDim} \\
\end{pmatrix}^\top$, 
 $\epsbar^*_j:=\GlmVar(g\big(\TrueTargetPar_{j} +
\TrueFunPar(\iNuisance)\big)$, and \begin{align*}
\GlmMeanVecbar_0^*:=g'(\TrueFunPar(\iNuisance))/\sum_{k=0}^{\TargetDim}\SelectProb_{ik}g'(\TrueTargetPar_k+\TrueFunPar(\iNuisance)).
\end{align*}
Since for logistic models $\GlmVar(g(s))=g'(s)$ for all $s\in\R$, it
follows
that
\begin{align*}
\epsbar^*_j/\GlmMeanVecbar^*_j =
\sum_{k=0}^{\TargetDim}\SelectProb_{ik} g'(\TrueTargetPar_k +
\TrueFunPar(\iNuisance))
\end{align*}
for $0 \leq j \leq \TargetDim$. Therefore, it can be verified that
$\InvCompTwo_i={1_{\TargetDim}1_{\TargetDim}^\top}/{(\SelectProb_0\epsbar^*_0)}$
and hence
\begin{align*}
&\qquad(\InvCompDiag_i+\InvCompTwo_i\DiagProbMatrix)(\GlmDiffVec-
  \GlmMeanVecbar^*\GlmDiffBar) \\
& = \InvCompDiag_i \GlmDiffVec +
  \frac{\GlmDiffBar}{\SelectProb_0\epsbar^*_0} 1_{\TargetDim}-
  \frac{\GlmDiffBar}{\sum_{k=0}^{\TargetDim}
    \SelectProb_{ik}g'(\TrueTargetPar_k+\TrueFunPar(\iNuisance))}1_{\TargetDim}-
  \frac{(1- \SelectProb_0
    \GlmMeanVecbar_0^*)\GlmDiffBar}{\SelectProb_0\epsbar_0^*}1_{\TargetDim}\\
  & = \InvCompDiag_i\GlmDiffVec.
\end{align*}
Putting the pieces together yields
\begin{align*}
{\CondMean (\ScoreFun_i(\TargetPar, \TrueTargetPar, \TrueFunPar ) -
  \ScoreFun_i(\TrueTargetPar, \TrueTargetPar, \TrueFunPar
  ))}={\CondMean\InvSqrtCovar^{*,-1}_i\InvCompDiag_i\GlmDiffVec}.
\end{align*}
Note that
$\GlmDiff_k=g'(\inprod{\Indicator_k}{\AuxiNuisancePar}+\TrueFunPar(\iNuisance))(\TrueTargetPar_k-
\TargetPar_k)$ for some $\AuxiNuisancePar$ by Taylor expansion.  By
the boundedness assumption on $g',\GlmVar$, we can further write
\begin{align*}
  {\CondMean\InvSqrtCovar^{*,-1}_i \InvCompDiag_i \GlmDiffVec} =
  {\CondMean\InvSqrtCovar^{*,-1}_i \InvCompDiagMod_i(\TrueTargetPar-
    \TargetPar)} =
  {[\CondMean\InvSqrtCovar^{*,-1}_i\InvCompDiagMod_i](\TrueTargetPar -
    \TargetPar)},
\end{align*} 
where $\InvCompDiagMod_i$ are diagonal matrices satisfying
$c_{\InvCompDiag,1}\IdMat_{\TargetDim}\preceq\InvCompDiagMod_i\preceq
c_{\InvCompDiag,2}\IdMat_{\TargetDim}$ for some constants
$c_{\InvCompDiag,1},c_{\InvCompDiag,2}>0$ that may depend on the
problem parameters.

\color{black}

\section{Adaptive estimation of the nuisance function}
 \label{SecAdaNuisance}

In this section, we discuss an alternative construction of an
estimator $\SoluTargetPar$ with potentially better sample efficiency.
Our original procedure is based on splitting the dataset and use the
first $\NumIndexOne$ data points to obtain the nuisance estimate
$\EstFunPar$ (and the target estimate $\EstTargetPar$ for GLMs).
Instead, suppose that at each time $i\in[\Numobs]$, we construct an
estimate $\EstFunPar_i$ (and the target estimate $\EstTargetPar_i$ for
GLMs) using the data collected up to time $i-1$ (starting with
$\EstTargetPar_1=0_{\TargetDim},\EstFunPar_1\equiv0$). For partial
linear models, we then solve
\begin{align}
\label{eq:partial_linear_alter}
\frac{1}\Numobs\sum_{i=1}^\Numobs
\ScoreFun_i(\TargetPar,\EstFunPar_i)=0
\end{align}to compute the estimate $\SoluTargetPar$. For generalized linear
models, we then compute the estimate $\SoluTargetPar$ from the system
 \begin{align}
\label{eq:generalize_linear_alter}
\frac{1}\Numobs\sum_{i=1}^\Numobs
\ScoreFun_i(\TargetPar,\EstTargetPar_i,\EstFunPar_i)=0
\end{align}
Given the use of adaptively updated nuisance estimates, it can be
shown that the estimates $\SoluTargetPar$ exhibit sample efficiency
superior to those obtained in
Algorithm~\ref{algo:DML-linear}~and~\ref{algo:DML-glm}. Namely, we
have the following results\footnote{Similar results can also be proved
for fixed direction inference.} (in contrast to
Theorem~\ref{thm:linear_new1}~and~\ref{thm:glm_new1_temp}).
\begin{cors}
\label{cor:linear_new1_alter} 
Suppose that the Assumptions in Theorem~\ref{thm:linear_new1} are in
force, with Assumption~\ref{assn-lin-nuisance-est} replaced by
\myassumption{{$\mathbf{NUI_{ada}}$}}{assn-lin-nuisance-est-alter}{The
  sequence of estimators $\EstFunPar_i$ obtained from
  equation~\eqref{eq:partial_linear_alter} satisfies
  \begin{align}
\label{eq:assn-lin-nuisance-est_alter}    
\frac{1}{\Numobs} \sum_{i=1}^\Numobs
\E_{\EstFunPar_i,\Nuisance_i}(\EstFunPar_i(\iNuisance)-
\TrueFunPar(\iNuisance))^2 \to 0,
\end{align}
where the expectation is over $(\iHistory,\iNuisance)$.}  Then
estimate~$\SoluTargetPar$ obtained from
equation~\eqref{eq:partial_linear_alter} satisfies
\begin{align}
\label{eq:linear_new1_result_alter}
(\sqrt{\Numobs} {\widehat\E_{\Numobs}} \iCovar^{1/2}) (\SoluTargetPar
- \TrueTargetPar) & \convdist \Normal(0, \LinearNoiseVar
\Id_{\TargetDim}).
\end{align}  
\end{cors}
See the proof in Section~\ref{sec:pf_cor:linear_new1_alter}.
\begin{cors}
\label{cor:glm_new1_alter} 
Suppose that the Assumptions in Theorem~\ref{thm:glm_new1_temp} are in
force, with Assumption~\ref{assn-glm-nuisance-est} replaced by
\myassumption{{$\mathbf{NUI^\prime_{ada}}$}}{assn-glm-nuisance-est-alter}{Suppose
  that all distributions in $\Pclass$ are supported on a set
  $\mathrm{dom}(\Pclass)$ (can be $\R^{\NuisanceDim}$). The estimators
  $\EstTargetPar_i, \EstFunPar_i$ obtained in
  equation~\eqref{eq:generalize_linear_alter} satisfy
\begin{align*}
\frac{1}{\Numobs} \sum_{i=1}^\Numobs\vecnorm{\EstTargetPar_i-
  \TrueTargetPar}{2}^2&=o_p(\Numobs^{-1/2}),~~~ \text{and}
\\ \frac1\Numobs\sum_{i=1}^\Numobs\sup
\limits_{\dummyRV\in\mathrm{dom}(\Pclass)}|\EstFunPar_i(\dummyRV) -
\TrueFunPar(\dummyRV)|^2&=o_p(\Numobs^{-1/2}).
\end{align*}
}
 Then estimate~$\SoluTargetPar$ obtained from equation~\eqref{eq:generalize_linear_alter} satisfies
\begin{align}
\label{eq:glm_new1_result_alter}
({\widehat\E_{\Numobs}} \InvSqrtCovarwhat_{i}(\iRegressor - \GlmMeanVecwhat_i)
g'\big( \inprod{\iRegressor}{\EstTargetPar_i} + \EstFunPar_i(\iNuisance)
\big) (\iRegressor - \GlmMeanVecwhat_i)^\top) \sqrt{\Numobs}
(\SoluTargetPar - \TrueTargetPar) \convdist \Normal(0,
\IdMat_{\TargetDim}).
\end{align}  
\end{cors}
See the proof in Section~\ref{sec:pf_cor:glm_new1_alter}.

  It can be verified that a sufficient condition for~\ref{assn-glm-nuisance-est-alter} is 
\begin{align*}
\lim_{i\to\infty}\E[i^{1+\delta_0}\vecnorm{\EstTargetPar_i- \TrueTargetPar}{2}^4]\to0,~~~\lim_{i\to\infty}\E\big[i^{1+\delta_0}\sup \limits_{\dummyRV\in\mathrm{dom}(\Pclass)}|\EstFunPar_i(\dummyRV) - \TrueFunPar(\dummyRV)|^4\big]\to0
\end{align*} for some constant $\delta_0>0$. Moreover, we remark that the conditions~\ref{assn-lin-nuisance-est-alter},~\ref{assn-glm-nuisance-est-alter} are stronger than~\ref{assn-lin-nuisance-est},~\ref{assn-glm-nuisance-est} since they are made on a sequence of estimators instead of a single estimator obtained from sample splitting.

Compared with the estimators from
equations~\eqref{eq:partial_linear_alter} and
\eqref{eq:generalize_linear_alter}, the estimators described in
Algorithm~\ref{algo:DML-linear}~and~\ref{algo:DML-glm} may have larger
asymptotic variances when using a fixed proportion (instead of a
decreasing proportion) of the data points to compute the prior
estimate $\EstFunPar$ (i.e., $\liminf\NumIndexOne/\Numobs>0$). On the
other hand, the estimators~\eqref{eq:partial_linear_alter}
and~\eqref{eq:generalize_linear_alter} require the calculation of the
nuisance estimate $\EstFunPar_i$ at every time step $i$.  This can be
computationally inefficient when a simple update rule of the nuisance
estimate does not exist.


\subsection{Proof of Corollary~\ref{cor:linear_new1_alter}}
\label{sec:pf_cor:linear_new1_alter}

Corollary~\ref{cor:linear_new1_alter} follows from the same arguments
used to prove Theorem~\ref{thm:linear_new1}, with the objects
$\EstFunPar$, $\NumIndexTwo$ and $\EmpMean$ in all formulas replaced,
respectively by $\EstFunPar_i$, $\Numobs$, and $\widehat{\E_\Numobs}$.
The main difference is to show a counterpart of
equation~\eqref{eqn:zero-nuisance-term}, namely, given
Assumption~\ref{assn-lin-nuisance-est-alter}, we have
\begin{align}
\EmpMeanAll\sqrt{\Numobs}\LinScaleVec_{i}
(\EstFunPar_i(\iNuisance) - \TrueFunPar(\iNuisance)
)\overset{p}{\to}0.\label{eqn:zero-nuisance-term_alter}
\end{align}
Since the remainder of the proofs are largely identical, we only prove
equation~\eqref{eqn:zero-nuisance-term_alter} here.

\paragraph*{Proof of equation~\eqref{eqn:zero-nuisance-term_alter}}
We begin by observing that
\begin{align*}
\E[\LinScaleVec_{i} (\EstFunPar_i(\iNuisance) -
  \TrueFunPar(\iNuisance))|\iHistory,\iNuisance]=\E[\LinScaleVec_{i}|\iHistory,\iNuisance](\EstFunPar_i(\iNuisance)
- \TrueFunPar(\iNuisance)) = 0.
\end{align*}
Consequently, it follows that $\{\LinScaleVec_{i}
(\EstFunPar_i(\iNuisance) - \TrueFunPar(\iNuisance))\}_{i=1}^\Numobs$
forms a martingale difference sequence. Moreover, we have
\begin{align*}
\frac{1}{\Numobs} \sum_{i=1}^\Numobs\E\vecnorm{\LinScaleVec_{i}
  (\EstFunPar_i(\iNuisance) - \TrueFunPar(\iNuisance))}{2}^2 &=
\frac1\Numobs\sum_{i=1}^\Numobs\E[(\EstFunPar_i(\iNuisance) -
  \TrueFunPar(\iNuisance))^2\cdot\E[\vecnorm{\LinScaleVec_{i}
    }{2}^2|\iHistory,\iNuisance]] \\
& = \frac{1}{\Numobs} \sum_{i=1}^\Numobs \E[(\EstFunPar_i(\iNuisance)
  - \TrueFunPar(\iNuisance))^2] \to 0,
\end{align*}
where the expectation in the last line is over $(\iHistory,
\iNuisance)$ and the convergence is due to
Assumption~\ref{assn-lin-nuisance-est-alter}. Therefore,
equation~\eqref{eqn:zero-nuisance-term_alter} follows immediately from
Lemma~\ref{lm:tech_md}.

\subsection{Proof of Corollary~\ref{cor:glm_new1_alter}}
\label{sec:pf_cor:glm_new1_alter}

The proof of Corollary~\ref{cor:glm_new1_alter} is largely identical
to that of Theorem~\ref{thm:glm_new1_temp}, but with the prior
estimate $\EstAllNuisancePar$ replaced by adaptive estimates
$\EstAllNuisancePar_i$. Again, we consider the simple case where the
nuisance component is linear, i.e.,
$\TrueFunPar(\iNuisance)=\inprod{\iNuisance}{\TrueNuisancePar}$, and
write $\AllNuisancePar=(\TargetPar,\NuisancePar)$ (similarly for
$\EstAllNuisancePar_i$ and $\TrueAllNuisancePar$). Note that
Assumption~\ref{assn-glm-nuisance-est-alter} implies that
\begin{align*}
\frac{1}{\Numobs} \sum_{i=1}^\Numobs
\vecnorm{\EstTargetPar_i- \TrueTargetPar}{2}=o_p(\Numobs^{-1/4}),~~
\frac{1}{\Numobs} \sum_{i=1}^\Numobs
\vecnorm{\EstNuisancePar_i- \TrueNuisancePar}{2}=o_p(\Numobs^{-1/4})
\end{align*}
by Cauchy-Schwartz inequality. Moreover, from the proof of
Lemma~\ref{lm:glm_lip_cont_temp} we see that
$\E(\ScoreFun_i(\TargetPar, \AllNuisancePar)\mid\iHistory)$ and
$\E(\partial_\AllNuisancePar\ScoreFun_i(\TargetPar,
\AllNuisancePar)\mid\iHistory)$ are uniformly Lipschitz across all
$i$.

Therefore, it can be verified that one can establish the same results
as in the proof of Theorem~\ref{thm:glm_new1_temp} (and the related
lemmas) but with $\vecnorm{\EstAllNuisancePar -
  \TrueAllNuisancePar}{2}$ and $\vecnorm{\EstAllNuisancePar -
  \TrueAllNuisancePar}{2}^2$ replaced by
$\sum_{i=1}^\Numobs\vecnorm{\EstAllNuisancePar_i -
  \TrueAllNuisancePar}{2}/\Numobs$ and
$\sum_{i=1}^\Numobs\vecnorm{\EstAllNuisancePar_i -
  \TrueAllNuisancePar}{2}^2/\Numobs$,
respectively. Corollary~\ref{cor:glm_new1_alter} then follows
immediately from Assumption~\ref{assn-glm-nuisance-est-alter}. Since
the proofs are essentially the same, we omit them here for simplicity.

\color{black}



\newcommand{\iAugHistory}{{\mathcal{G}_{i-1}}}
\newcommand{\AugHistory}{{\mathcal{G}}}

\section{Inference when $\SelectProb_i$ are unknown}
\label{SecEstSelectProb}
In this section, we study the inference problem when the selection
probabilities $\{\SelectProb_i\}_{i=1}^\Numobs$ are unknown, but we
have access to a sequence of consistent estimators
$\{\EstSelectProb_i\}_{i=1}^\Numobs$.  In this setting, one can
similarly obtain the estimates $\SoluTargetPar$ (or
$\SoluOneDirectPar$) by substituting $\SelectProb_i$ with
$\EstSelectProb_i$ in the calculation of the score functions
$\ScoreFun_i$. We demonstrate that a modified version of
Theorem~\ref{thm:linear_new1} remains valid when
$\{\EstSelectProb_i\}_{i=1}^\Numobs$ closely approximates
$\{\SelectProb_i\}_{i=1}^\Numobs$.  Define $\iEstCovar \defn
\E((\iRegressor - \EstSelectProb_i) (\iRegressor -
\EstSelectProb_i)^\top | \iHistory,\iNuisance)$ and $ \vecnorm{v}{\iCovar^{-1}}\defn \sqrt{v^\top\iCovar^{-1} v}$ for any vector $v\in\R^{\TargetDim}$.  We assume the
sequence of estimators $\{\EstSelectProb_i\}_{i=1}^\Numobs$ satisfy
the following set of convergence assumptions:

\myassumption{{\bf CON}}{assn-select-prob}{
\begin{itemize}
  \item[(a)] 
  { $\EstSelectProb_i\in[0,1]$ and $\EstSelectProb_i\in\sigma(\iNuisance,\iHistory)=:\iAugHistory$ for all $i\in[\Numobs]$,
    i.e., $\EstSelectProb_i$ is calculated using $\iNuisance$, the first $i-1$
    samples and any prior knowledge independent of the collected
    samples. }
  \item[(b)] $\EmpMean
    \vecnorm{\EstSelectProb_i- \SelectProb_i}{\iCovar^{-1}}^2\overset{p}{\to}0$,
    and $(\EmpMean
    \vecnorm{\EstSelectProb_i- \SelectProb_i}{\iCovar^{-1}}^2)\cdot
    (\EmpMean (\EstFunPar(\iNuisance) - \TrueFunPar(\iNuisance)
    )^2)=o_p({1}/{\NumIndexTwo}). $
\item [(c)] $\opnorm{\iEstCovar^{-1/2}\iCovar^{1/2}}\leq\EstCovBound$
  for some $\EstCovBound>0$, and $\EmpMean
  \opnorm{\iEstCovar^{-1/2}\iCovar^{1/2}- \IdMat_{\TargetDim}}^2\overset{p}{\to}0$.
\end{itemize}
}

\begin{cors}
\label{cor:linear_new1_est_select} 
Suppose that
Assumptions~\ref{assn-lin-noise},~\ref{assn-lin-selection-prob}~and~\ref{assn-select-prob}
are in force.  When the selection probabilities $\SelectProb_i$ are
unknown, the estimate~$\SoluTargetPar$ obtained from
equation~\eqref{eq:partial_linear_alter} with
$\{\EstSelectProb_i\}_{i=1}^\Numobs$ replacing
$\{\SelectProb_i\}_{i=1}^\Numobs$ satisfies
\begin{align}
\label{eq:linear_new1_est_select}
(\sqrt{\NumIndexTwo} \EmpMean \EstLinScaleVec_i
\iRegressor^\top)
(\SoluTargetPar - \TrueTargetPar) &  \convdist 
\Normal(0, \LinearNoiseVar \Id_{\TargetDim}). 
\end{align}  
\end{cors}
\noindent See the proof in
Section~\ref{sec:pf_cor:linear_new1_est_select}. \\

\noindent We note that the sole distinction between
equation~\eqref{eq:linear_new1_est_select}
and~\eqref{eq:linear_new1_result} in Theorem~\ref{thm:linear_new1}
lies in the preconditioning matrix on the left-hand
side. Specifically, equation~\eqref{eq:linear_new1_est_select}
substitutes $\EmpMean\iCovar^{1/2}$ with $\EmpMean \EstLinScaleVec_i
\iRegressor^\top$, a matrix computable in the absence of known
$\SelectProb_i$. 
We conjecture that similar conclusions might hold fixed direction
inference and generalized linear models when
$\{\EstSelectProb_i\}_{i=1}^{\Numobs}$ closely approximates
$\{\SelectProb_i\}_{i=1}^{\Numobs}$. We view this as a fertile
direction for future research.

In practice, finding such a sequence of consistent estimators $\{\EstSelectProb_i\}_{i=1}^\Numobs$ is difficult in general. 
Theoretically, it is impossible to do so if without any prior knowledge on the selection probabilities, as in the worst case the dependence of $\SelectProb_i$ on $(\iNuisance,\iHistory)$ can be arbitrarily different across $i\in[\Numobs]$ and we only have one sample $\iRegressor$ to estimate  $\SelectProb_i$ for each $i.$

Nevertheless, consistent estimation may be possible if additional prior knowledge is provided.
For example, if the selection probabilities remain constant over time, i.e., 
$\SelectProb_i(\iNuisance,\iHistory)=\SelectProb(\iNuisance)$ for all $i\in[\Numobs]$ and some function $\SelectProb$, then standard estimation methods such as empirical risk minimization could possibly find $\{\EstSelectProb_i\}_{i=1}^\Numobs$  that satisfies Assumption~\ref{assn-select-prob}, provided that 
 $\SelectProb$ has a benign parametric (or nonparametric) form.
Alternatively, 
if the entire set of selection
probability functions $\{\SelectProb_i(\cdot)\}_{i=1}^{\Numobs}$ (we call this set a selection algorithm) is chosen from
a known finite set of selection algorithms, it may be possible to identify the true selection algorithm based on the observed samples $\{(\iRegressor,\iNuisance)\}_{i=1}^\Numobs$ with probability converging to one as $\Numobs$  increases. 

Going beyond the setting of this work, consistent estimation of $\{\SelectProb_i\}_{i=1}^{\Numobs}$ may be possible if we observe a batch of $K$ i.i.d. trajectories $\{(\Response^{(k)}_i,\Regressor^{(k)}_i,\Nuisance^{(k)}_i)\}_{i=1}^\Numobs$, $k\in[K]$ for some sufficiently large $K$~\cite{zhang2020inference}. In this case, we have $K$ i.i.d. samples $\{(\Regressor^{(k)}_i,\Nuisance^{(k)}_i,\History_i^{(k)})\}_{k=1}^K$ to estimate each $\SelectProb_i.$ Therefore, consistent estimation may be achieved when the batch size $K\to\infty$.


\subsection{Proof of Corollary~\ref{cor:linear_new1_est_select}}
\label{sec:pf_cor:linear_new1_est_select} 

Recalling the vector $\LinScaleVec_i \defn \iCovar^{-1/2} (\iRegressor
- \EstSelectProb_i)$ from the proof of Theorem~\ref{thm:linear_new1},
we define the vector $\EstLinScaleVec_i\defn\iEstCovar^{-1/2}
(\iRegressor - \EstSelectProb_i)$.  Similar to the proof of
Theorem~\ref{thm:linear_new1}, we argue that the pair
$(\SoluTargetPar, \EstFunPar)$ satisfies the equation
\begin{align}
\label{linear_asymp1-2}
\sqrt{\NumIndexTwo } (\EmpMean \EstLinScaleVec_i
\iRegressor^\top)(\SoluTargetPar - \TrueTargetPar) & =
\sqrt{\NumIndexTwo } \big \{ \EmpMean \EstLinScaleVec_i \iNoise -
\EmpMean \EstLinScaleVec_i (\EstFunPar(\iNuisance) -
\TrueFunPar(\iNuisance)) \big \}.
\end{align}

Our proof is based on the following two auxiliary claims:
\begin{subequations}
\begin{align}
&\EmpMean \sqrt{\NumIndexTwo }\EstLinScaleVec_i\iNoise \convdist 
\Normal(0, \LinearNoiseVar \Id_{\TargetDim}),
\label{eqn:asymp-normality-consist} \\
&\EmpMean \sqrt{\NumIndexTwo }\EstLinScaleVec_i
(\EstFunPar(\iNuisance) - \TrueFunPar(\iNuisance) )\overset{p}{\to}0
\label{eqn:zero-nuisance-term-consist}. 
%
\end{align}
\end{subequations}Corollary~\ref{cor:linear_new1_est_select} follows immediately from combining equation~\eqref{eqn:asymp-normality-consist}~and~\eqref{eqn:zero-nuisance-term-consist}. 

\paragraph*{
 Proof of equation~\eqref{eqn:asymp-normality-consist}} 
 The proof is essentially the same as the proof of
Lemma~\ref{lm:asym_1}. We only highlight the differences here.

Recall that we define $\iAugHistory$ to be the $\sigma$-field $\sigma(\iNuisance,\iHistory)$. 
Note that $\{\EstLinScaleVec_i\iNoise\}_{i>\NumIndexOne}$ forms a martingale
difference sequence with repsect to $\{\AugHistory_i\}_{i>\NumIndexOne}$  as $\EstLinScaleVec_i\iNoise\in\AugHistory_i$ and  $\E(\EstLinScaleVec_i\iNoise\mid\iAugHistory)=0$. Therefore, we
may prove equation~\eqref{eqn:asymp-normality-consist} by applying the
central limit theorem for martingale difference sequences.

\paragraph*{Asymptotic covariance}
Observe that
\begin{align*}
  \E (\iNoise^2 \EstLinScaleVec_i \EstLinScaleVec_i ^\top\mid
  \iAugHistory)& =\E (\EstLinScaleVec_i \EstLinScaleVec_i
  ^\top\E(\iNoise^2\mid\iRegressor, \iNuisance,
  \iHistory)\mid\iAugHistory)) \\
& = \E (\LinearNoiseVar \EstLinScaleVec_i \EstLinScaleVec_i^\top \mid
  \iNuisance,\iHistory) = \LinearNoiseVar \iEstCovar^{-1/2}
  (\iCovar+(\SelectProb_i - \EstSelectProb_i) (\SelectProb_i -
  \EstSelectProb_i)^\top) \iEstCovar^{-1/2}.
\end{align*} 
Therefore, the asymptotic variance is given by
\begin{align*}
\frac{1}{\NumIndexTwo} \sum_{i=\NumIndexOne+1}^{\Numobs} \E (\iNoise^2
\EstLinScaleVec_i \EstLinScaleVec_i ^\top\mid \iAugHistory) & =
\LinearNoiseVar\EmpMean
\iEstCovar^{-1/2}(\iCovar+(\SelectProb_i- \EstSelectProb_i)(\SelectProb_i- \EstSelectProb_i)^\top)\iEstCovar^{-1/2} \\
& = \LinearNoiseVar\EmpMean \iEstCovar^{-1/2}\iCovar\iEstCovar^{-1/2}+
\LinearNoiseVar\EmpMean
\iEstCovar^{-1/2}(\SelectProb_i- \EstSelectProb_i)(\SelectProb_i- \EstSelectProb_i^\top)
\iEstCovar^{-1/2} \\
& \overset{p}{\to}\sigma^2\IdMat_{\TargetDim},
\end{align*}
where the last line uses Assumption~\ref{assn-select-prob}~(b) and
Lemma~\ref{lm:select_prob_est_1}.

\paragraph*{Lindeberg condition}
Note that by
Assumption~\ref{assn-lin-selection-prob-equi} we have $\iCovar\succeq c_i \IdMat_{\TargetDim}$, and 
\begin{align}
\label{eq:vi_bound_est}
  \|\EstLinScaleVec_i \|_2^2&\leq \opnorm{\iEstCovar^{-1/2}}^2 \cdot
  \| \iRegressor - \EstSelectProb_i(\iNuisance,
  \iHistory)\|_2^2\notag\\ &\leq
  \opnorm{\iEstCovar^{-1/2}\iCovar^{1/2}}^2 \cdot
  \opnorm{\iCovar^{-1/2}}^2\cdot \| \iRegressor -
  \EstSelectProb_i(\iNuisance, \iHistory)\|_2^2 \leq
  4\EstCovBound^2/c_i,
\end{align}
where the last inequality follows from $\|\iRegressor -
\EstSelectProb_i(\iNuisance, \iHistory)\|_2\leq \|\iRegressor\|_2 + \|
\EstSelectProb_i(\iNuisance, \iHistory)\|_2\leq 2$ and
Assumption~\ref{assn-select-prob}. Thus, it can be verified that
$\{\EstLinScaleVec_i \iNoise\}^\Numobs_{i=\NumIndexOne + 1}$ satisfies
Lindeberg's condition following a similar argument as in the proof of
Lemma~\ref{lm:asym_1}.

Putting together the pieces and invoking the martingale central limit
theorem, we conclude $\EmpMean \sqrt{\NumIndexTwo }\EstLinScaleVec_i
\iNoise \convdist \Normal(0, \LinearNoiseVar \IdMat_{\TargetDim})$.

\paragraph*{Proof of equation~\eqref{eqn:zero-nuisance-term-consist}}
Substituting the relation $\EstLinScaleVec_i = \iEstCovar^{-1/2}
(\iRegressor - \SelectProb_i) + \iEstCovar^{-1/2} (\SelectProb_i -
\EstSelectProb_i)$ into the LHS of
equation~\eqref{eqn:zero-nuisance-term-consist} yields the
decomposition $\EmpMean \sqrt{\NumIndexTwo }\EstLinScaleVec_i
(\EstFunPar(\iNuisance) - \TrueFunPar(\iNuisance) ) \equiv \Term_1 +
\Term_2$, where
\begin{align*}
\Term_1 & \defn \EmpMean \sqrt{\NumIndexTwo
}\iEstCovar^{-1/2}(\iRegressor- \SelectProb_i) (\EstFunPar(\iNuisance)
- \TrueFunPar(\iNuisance) ), \quad \mbox{and} \\
\Term_2 & \defn \EmpMean \sqrt{\NumIndexTwo
}\iEstCovar^{-1/2}(\SelectProb_i- \EstSelectProb_i)
(\EstFunPar(\iNuisance) - \TrueFunPar(\iNuisance) ).
\end{align*}
Note that $\{\iEstCovar^{-1/2}(\iRegressor - \SelectProb_i)
(\EstFunPar(\iNuisance) - \TrueFunPar(\iNuisance)
)\}_{i=\NumIndexOne+1}^{\Numobs}$ is a martingale difference sequence with respect to $\{\AugHistory_i\}_{i=\NumIndexOne+1}^\Numobs$.
Combined with the bound $\E(\vecnorm{\iEstCovar^{-1/2}(\iRegressor - \SelectProb_i)}{2}^2|\iAugHistory)=\tr(\iEstCovar^{-1/2}\iCovar\iEstCovar^{-1/2}) \leq \TargetDim
\EstCovBound^2$, it follows from a similar argument as in the proof of
Lemma~\ref{lm:asym_2} that $\Term_1 \overset{p}{\to} 0$.

Turning to the second term $\Term_2,$ observe that
\begin{multline*}
\EmpMean \sqrt{\NumIndexTwo } \iEstCovar^{-1/2}
(\SelectProb_i- \EstSelectProb_i) (\EstFunPar(\iNuisance) -
\TrueFunPar(\iNuisance) ) \\ \leq \sqrt{\NumIndexTwo } (\EmpMean
\vecnorm{\iEstCovar^{-1/2} (\SelectProb_i- \EstSelectProb_i)}{2}^2
)^{1/2}\cdot (\EmpMean | \EstFunPar(\iNuisance) -
\TrueFunPar(\iNuisance)|^2)^{1/2}.
\end{multline*}
By using Assumption~\ref{assn-select-prob}, we see that $\Term_2
\overset{p}{\to} 0$.  Putting together the results for $\Term_1$ and
$\Term_2$ concludes the proof.


\begin{lems}\label{lm:select_prob_est_1}
Under Assumption~\ref{assn-select-prob}, we have the following result
\begin{subequations}
\begin{align}
\opnorm{\EmpMean  \iEstCovar^{-1/2}\iCovar\iEstCovar^{-1/2}- \IdMat_{\TargetDim}}
&\overset{p}{\to}0, \label{eq:est_select_lm_1}\\
\EmpMean  \vecnorm{\iEstCovar^{-1/2}(\SelectProb_i- \EstSelectProb_i)}{2}^2\leq \EstCovBound^2\EmpMean  \vecnorm{\SelectProb_i- \EstSelectProb_i}{\iCovar^{-1}}^2&\overset{p}{\to}0
 \label{eq:est_select_lm_2}. 
\end{align}
\end{subequations}
\end{lems}

\begin{proof}
We start with the proof of equation~\eqref{eq:est_select_lm_1}. 
Let $s_{i1}\geq s_{i2}\geq\ldots \geq s_{i\TargetDim}$ be the singular values of $\iEstCovar^{-1/2}\iCovar^{1/2}$. Note that Assumption~\ref{assn-select-prob} implies 
\begin{align*}
\EmpMean |s_{ik}-1|^2 \overset{p}{\to} 0
\end{align*} for all $k\in[\TargetDim]$. Therefore,
\begin{align*}
\EmpMean |s^2_{ik}-1|\leq\EmpMean |s_{ik}-1|^2+2\EmpMean |s_{ik}-1|\overset{p}{\to} 0,
\end{align*}
where the last step uses Jensen's inequality. Since the eigenvalues of $\iEstCovar^{-1/2}\iCovar\iEstCovar^{-1/2}$ are $\{s^2_{ik}\}_{k=1}^{\TargetDim}$, it follows that
\begin{align*}
\opnorm{\EmpMean  \iEstCovar^{-1/2}\iCovar\iEstCovar^{-1/2}- \IdMat_{\TargetDim}}
&\leq
\EmpMean\opnorm{  \iEstCovar^{-1/2}\iCovar\iEstCovar^{-1/2}- \IdMat_{\TargetDim}}\\
&\leq \EmpMean \sum_{k=1}^{\TargetDim} |s^2_{ik}-1|
\overset{p}{\to}0.
\end{align*}
To prove the claim~\eqref{eq:est_select_lm_2}, we note that
\begin{align*}
\EmpMean
\vecnorm{\iEstCovar^{-1/2}(\SelectProb_i- \EstSelectProb_i)}{2}^2 &=
\EmpMean
\vecnorm{\iEstCovar^{-1/2}\iCovar^{1/2}\iCovar^{-1/2}(\SelectProb_i- \EstSelectProb_i)}{2}^2 \\
& \leq
\EmpMean\opnorm{\iCovar^{1/2}\iEstCovar^{-1}\iCovar^{1/2}}\cdot
\vecnorm{\SelectProb_i- \EstSelectProb_i}{\iCovar^{-1}}^2\\ 
&\leq
\EstCovBound^2\EmpMean
\vecnorm{\SelectProb_i- \EstSelectProb_i}{\iCovar^{-1}}^2\overset{p}{\to}0,
\end{align*}
where the last step uses Assumption~\ref{assn-select-prob}.

\end{proof}

\color{black}

 
\section{Auxiliary lemmas}
\label{SecAuxLemmas}

In this section, we collect the proofs of various lemmas that were
used in the proof of Theorem~\ref{thm:linear_new1}--
\ref{thm:glm_new1_temp}.


\subsection{Auxiliary lemmas for Theorem~\ref{thm:linear_new1}}
\label{proof:lm:thm:linear_new1}

In this section, we state and prove the auxiliary lemmas used in the
proof of Theorem~\ref{thm:linear_new1}.
\begin{lems}\label{lm:asym_1}
Under the assumptions of Theorem~\ref{thm:linear_new1} we have $\EmpMean \sqrt{\NumIndexTwo }\LinScaleVec_i \iNoise \convdist  \Normal(0, \LinearNoiseVar \IdMat_{\TargetDim})$.
\end{lems}
\begin{proof}
Recall that $\E(\iNoise\mid\iHistory, \iRegressor, \iNuisance)=0$ by
our assumption, and we have $\E(\LinScaleVec_i \iNoise\mid
\iHistory)=\E (\LinScaleVec_i \E(\iNoise\mid\iRegressor, \iNuisance,
\iHistory)\mid\iHistory)=0$, and consequently $\{\LinScaleVec_i
\iNoise\}_{i \geq \NumIndexOne}$ is a martingale difference
sequence. We prove Lemma~\ref{lm:asym_1} by applying the standard
martingale central limit theorem on the sequence $\{\LinScaleVec_i
\iNoise\}_{i \geq \NumIndexOne}$.

\paragraph*{Asymptotic covariance}
Observe that
$$\E (\iNoise^2 \LinScaleVec_i \LinScaleVec_i ^\top\mid \iHistory)=\E
(\LinScaleVec_i \LinScaleVec_i
^\top\E(\iNoise^2\mid\iRegressor, \iNuisance, \iHistory)\mid\iHistory)=\E
(\LinearNoiseVar \LinScaleVec_i \LinScaleVec_i
^\top\mid\iHistory)=\LinearNoiseVar \IdMat_d,$$ where the last equality
follows from
\begin{align*}
\E ( \LinScaleVec_i \LinScaleVec_i^\top\mid\iHistory)
 & = \E
(\iCovar^{-1/2}(\iRegressor - \SelectProb_i(\iNuisance, \iHistory))
(\iRegressor - \SelectProb_i(\iNuisance, \iHistory) )^\top
\iCovar^{-1/2}\mid\iHistory) \\
& = \E(\iCovar^{-1/2}\E ((\iRegressor -
\SelectProb_i(\iNuisance, \iHistory))(\iRegressor - \SelectProb_i(\iNuisance, \iHistory))^\top
\mid\iNuisance, \iHistory)\iCovar^{-1/2} \mid \iHistory) \\
& = \E(\iCovar^{-1/2}\iCovar\iCovar^{-1/2} \mid
\iHistory)=\IdMat_{\TargetDim}.
\end{align*}

\paragraph*{Lindeberg condition}
Note that by
Assumption~\ref{assn-lin-selection-prob-equi} we have $\iCovar\succeq c_i \IdMat_{\TargetDim}$, and 
\begin{align}
\label{eq:vi_bound}
  \|\LinScaleVec_i \|_2^2=\opnorm{\LinScaleVec_i \LinScaleVec_i ^\top} \leq
  \opnorm{\iCovar^{-1/2}} \cdot \|
  \iRegressor - \SelectProb_i(\iNuisance, \iHistory))\|_2^2\cdot
  \opnorm{\iCovar^{-1/2}} \leq 4/c_i,
\end{align}
where the second inequality follows from
$\|\iRegressor - \SelectProb_i(\iNuisance, \iHistory)\|_2\leq
\|\iRegressor\|_2 + \| \SelectProb_i(\iNuisance, \iHistory)\|_2\leq
2$. As a result we have $\LinScaleVec_i \LinScaleVec_i ^\top \preceq
4 \IdMat_{\TargetDim}/c_i$ and we deduce
 \begin{align}
\label{eqn:lindeberg-cond-linear}   
0 \preceq \lim_{\Numobs \to \infty} \frac{1}{\Numobs}
\sum_{i=1}^\Numobs\E(\iNoise^2 \LinScaleVec_i
\LinScaleVec_i^\top{1}_{\{\opnorm{\iNoise^2 \LinScaleVec_i
  \LinScaleVec_i^\top} > \eps{\Numobs}\}}\mid\iHistory) 
  &\preceq
\lim_{\Numobs \to \infty} \frac{4}{\Numobs} \sum_{i=1}^\Numobs
\frac{1}{c_i} \E(\iNoise^2 {1}_{\{\iNoise^2 \geq \eps{\Numobs}
  c_i/4\}} \mid\iHistory) \IdMat_{\TargetDim}\notag\\&=:\Term_0.
\end{align}
Since $\iNoise's$ are sub-Gaussian random variables with common
parameter $\subgauss$ almost surely, ${\iNoise^{2}} 's$ are subexponential
random variables with a common parameter. Therefore, there exists some
constant $K_1>0$ depending on $\subgauss$ such that $\prob(\iNoise^2\geq
s)\leq 2\exp(-s/K_1)$, and hence
\begin{align*}
\E(\iNoise^2 1_{\{\iNoise^2\geq\eps \Numobs c_i/4\}})=\int_{\eps
  \Numobs c_i/4}^\infty \prob(\iNoise^2\geq s)ds
  \leq2\int_{\eps
  \Numobs c_i/4}^\infty \exp(-s/K_1) ds =
2 K_1 \exp \Big(\frac{ - \eps \Numobs c_i}{4K_1} \Big).
\end{align*}
Substituting this into equation~\eqref{eqn:lindeberg-cond-linear},
for $t \in (0, 1/2)$, we have
\begin{align*}
\Term_0 & \lesssim \frac{c_0
  K_1}{\Numobs}\sum_{i=1}^{\Numobs}i^{2t}\exp\Big(\frac{ - \eps
v  \Numobs c_0}{4 K_1 i^{2t}} \Big) \leq
c_0K_1\Numobs^{2t}\exp\Big(\frac{ - \eps \Numobs^{1-2t}
  c_0}{4K_1}\Big)\to 0.
\end{align*}
Note that this implies that $\{\LinScaleVec_i
\iNoise\}^\Numobs_{i=\NumIndexOne + 1}$ satisfies Lindeberg's
condition.

Putting together the pieces and invoking the martingale central limit
theorem, we conclude $\EmpMean \sqrt{\NumIndexTwo }\LinScaleVec_i
\iNoise \convdist \Normal(0, \LinearNoiseVar \IdMat_{\TargetDim})$.

{
\paragraph*{Relaxation of Assumption~\ref{assn-lin-noise}}

Sub-Gaussianity of the noise variables $\{\iNoise\}_{i=1}^\Numobs$ is
not necessary for Theorem~\ref{thm:linear_new1} to hold.  The theorem
relies on Lindeberg's condition, which remains valid even when
Assumption~\ref{assn-lin-noise} is relaxed to the following:
\myassumption{\mbox{{$\mathbf{NOI_w}$}$(\powercoef,
    \LinearNoiseVar)$}}{assn-lin-noise-relax}{Conditioned upon
  $(\iRegressor,\iNuisance, \iHistory)$, each element of the zero-mean
  noise sequence has conditional variance \mbox{${\LinearNoiseVar
      \defn \E[\iNoise^2 \mid \iRegressor, \iNuisance, \iHistory]}$.}
  and satisfies
      \begin{align*}
      \Prob(|\iNoise|\geq s)\leq \frac{c}{s^\powercoef},~~\text{for all } s\geq0~\text{and some constant}~ c>0,
      \end{align*}
    }  
}
\noindent for some $\powercoef> 2/(1-2t)$. \\
 
Recall that the scalar $t
\in [0, 1/2)$ was defined in Assumption~\ref{assn-lin-selection-prob}.
  Note that this relaxed assumption allows for many heavy-tailed noise
  distributions that are not sub-Gaussian, including Cauchy
  distribution, (symmetric) Pareto distribution, etc.

Let us sketch the proof under the relaxed
Assumption~\ref{assn-lin-noise-relax}.  We have
\begin{align*}
\E(\iNoise^2 1_{\{\iNoise^2\geq\eps \Numobs c_i/4\}})=\int_{\eps
  \Numobs c_i/4}^\infty \prob(\iNoise^2\geq s)ds
  \leq2c\cdot \int_{\eps
  \Numobs c_i/4}^\infty s^{- \powercoef/2} ds \lesssim 
(\eps \Numobs c_i)^{1- \powercoef/2}
\end{align*}
Therefore, 
\begin{align*}
\Term_0 
& \lesssim 
\frac{
  1}{\Numobs}\sum_{i=1}^{\Numobs}\frac{1}{c_i}(\eps \Numobs c_i)^{1- \powercoef/2} 
  \lesssim
\sum_{i=1}^{\Numobs}( \Numobs c_i)^{- \powercoef/2}\eps^{1- \powercoef/2}\lesssim  \eps^{1- \powercoef/2}\frac{1}{\Numobs^{\powercoef/2}}\sum_{i=1}^{\Numobs} i^{\powercoef t}\lesssim \eps^{1- \powercoef/2} \Numobs^{\powercoef t+1- \powercoef/2}\to 0
\end{align*}
when $\powercoef> 2/(1-2t)$.  Lindeberg's condition is hence satisfied. 
\end{proof}

\begin{lems}\label{lm:asym_2}
Under the assumptions of Theorem~\ref{thm:linear_new1}  we have $\EmpMean
\sqrt{\NumIndexTwo }\LinScaleVec_i
(\EstFunPar(\iNuisance) - \TrueFunPar(\iNuisance) )\overset{p}{\to}0.$
\end{lems}
\begin{proof}
The proof follows from a standard application of Markov's inequality and utilizes Assumption~\ref{assn-lin-nuisance-est}. Note that 
 \begin{align*}
 \E(\LinScaleVec_i
   (\EstFunPar(\iNuisance) - \TrueFunPar(\iNuisance) )
   \mid\iHistory)=\E[\E(\LinScaleVec_i
     \mid\iHistory, \iNuisance)(\EstFunPar(\iNuisance) - \TrueFunPar(\iNuisance)
     )]=0,
 \end{align*}
 and it follows that $\LinScaleVec_i(\EstFunPar(\iNuisance) -
 \TrueFunPar(\iNuisance) )$ is a martingale difference sequence. Now,
 for any $\eps > 0$, define the event
\begin{align*}
\mathcal C_{\NumIndexOne, \eps} & \defn \{\sup_{P\in\Pclass
}(\E_{\Nuisance\sim P}{[\EstFunPar(\Nuisance) - \TrueFunPar( \Nuisance)]^2)^{1/2}} \leq
\eps \}.
\end{align*}
Note that $\mathcal C_{\NumIndexOne, \eps}\in
\History_{\NumIndexOne}$.  We have
\begin{align*}
\E \|\EmpMean \NumIndexTwo^{{1/2}} 1_{\mathcal C_{\NumIndexOne,
    \eps}}\LinScaleVec_i(\EstFunPar(\iNuisance) -
\TrueFunPar(\iNuisance))\|_2^2 & = \E 1_{\mathcal C_{\NumIndexOne,
    \eps}} \E(\EmpMean \|\LinScaleVec_i(\EstFunPar(\iNuisance) -
\TrueFunPar(\iNuisance))\|_2^2 \mid \History_{\NumIndexOne}) \\
& = \E 1_{\mathcal C_{\NumIndexOne, \eps}}\EmpMean\E(
\|\LinScaleVec_i\|_2^2 (\EstFunPar(\iNuisance) -
\TrueFunPar(\iNuisance))^2 \mid \History_{\NumIndexOne})\\ &=
\TargetDim \E 1_{\mathcal C_{\NumIndexOne,
    \eps}}\E(\EmpMean{[\EstFunPar(\iNuisance) -
    \TrueFunPar(\iNuisance)]^2} \mid \History_{\NumIndexOne})\\
& \leq \TargetDim \eps^2,
\end{align*}
where the third line uses the bound $\E(\|\LinScaleVec_i\|_2^2 \mid
\iNuisance, \iHistory)=\E(\tr(\LinScaleVec_i\LinScaleVec^\top_i)\mid
\iNuisance, \iHistory)=\TargetDim$, whereas the last line follows from
the definition of $\mathcal C_{\NumIndexOne, \eps}$.  Thus, for any
$\delta>0$, it follows from Markov's inequality that
\begin{align*}
\prob(\|\EmpMean \NumIndexTwo^{{1/2}}1_{\mathcal C_{\NumIndexOne,
    \eps}}\LinScaleVec_i(\EstFunPar(\iNuisance) -
\TrueFunPar(\iNuisance))\|_2\geq
\sqrt{\frac{\TargetDim}{\delta}}\eps)\leq \delta.
\end{align*}
Since $\prob({\mathcal C_{\NumIndexOne, \eps}})\to 1$ as
$\NumIndexOne\to\infty$ by Assumption~\ref{assn-lin-nuisance-est}, it
follows that $\prob(\|\EmpMean \NumIndexTwo^{{1/2}}
\LinScaleVec_i(\EstFunPar(\iNuisance) -
\TrueFunPar(\iNuisance))\|_2\geq
\sqrt{\frac{\TargetDim}{\delta}}\eps)\leq 2\delta$ for $\NumIndexTwo $
sufficiently large. Putting together the pieces, we conclude $\EmpMean
\sqrt{\NumIndexTwo}\LinScaleVec_i (\EstFunPar(\iNuisance) -
\TrueFunPar(\iNuisance))\overset{p}{\to}0$.
\end{proof}

\begin{lems}
\label{lm:asym_3}
Under the assumptions of Theorem~\ref{thm:linear_new1}, we have
\begin{align*}
\opnorm{\EmpMean \LinScaleVec_i \iRegressor^\top - \EmpMean \iCovar^{1/2}}
=\liloh_p(\sigma_{\min}(\EmpMean \iCovar^{1/2})).
\end{align*}
\end{lems}
\begin{proof}
Note that
\begin{align}
\label{eqn:triangle-decomp}
\opnorm{\EmpMean \LinScaleVec_i \iRegressor^\top - \EmpMean \iCovar^{1/2}}
 & \leq \opnorm{\EmpMean \LinScaleVec_i \iRegressor^\top - \CondMean
  \iCovar^{1/2}} + \opnorm{\CondMean \iCovar^{1/2} - \EmpMean
  \iCovar^{1/2}}.
\end{align}
We bound the two terms above by proving the following two bounds
\begin{subequations}
\begin{align}
\opnorm{\EmpMean \LinScaleVec_i \iRegressor^\top - \CondMean
  \iCovar^{1/2}} &= \liloh_p(\sigma_{\min}(\CondMean \iCovar^{1/2}))
\label{eqn:term-one} \\
\opnorm{\EmpMean \iCovar^{1/2} - \CondMean \iCovar^{1/2}}
&=\liloh_p(\sigma_{\min}(\EmpMean \iCovar^{1/2}))
\label{eqn:term-two}
\end{align}
\end{subequations}
Taking the last two bounds as given for the moment, we substitute them
into equation~\eqref{eqn:triangle-decomp}, thereby finding that
\begin{align*}
&\quad\opnorm{\EmpMean \LinScaleVec_i \iRegressor^\top - \EmpMean
  \iCovar^{1/2}}\notag\\
   &
    = \liloh_p(\sigma_{\min}(\CondMean \iCovar^{1/2}))
+ \liloh_p(\sigma_{\min}(\EmpMean \iCovar^{1/2}))
\\
& \stackrel{(i)}{\leq} \liloh_p(\sigma_{\min}(\EmpMean \iCovar^{1/2}))
+ \opnorm{\EmpMean \iCovar^{1/2} - \CondMean \iCovar^{1/2}} +
\liloh_p(\sigma_{\min}(\EmpMean \iCovar^{1/2})) \\
& \leq \liloh_p(\sigma_{\min}(\EmpMean \iCovar^{1/2})),
\end{align*}
where the inequality (i) follows from Weyl's theorem (see e.g.,
Theorem 4.3.1 in Horn and Johnson~\cite{horn2012matrix}), and the last
inequality follows from the bound~\eqref{eqn:term-two}. It remains to
prove the bounds~\eqref{eqn:term-one} and~\eqref{eqn:term-two}.

\paragraph*{Proof of the bound~\eqref{eqn:term-one}}

Since $\E(\LinScaleVec_i \SelectProb_i^\top \mid \iNuisance, \iHistory)
= \E (\iCovar^{-1/2} (\iRegressor - \SelectProb_i) \SelectProb_i^\top
\mid \iNuisance, \iHistory)=0$, it follows that $\{\LinScaleVec_i
\SelectProb_i^\top\}_{i=\NumIndexOne + 1}^\Numobs$ is a martingale
difference sequence with respect to the filtration
$\iHistory$. Moreover, note that $\E\frobnorm{\LinScaleVec_i
  \SelectProb_i^\top}^2= \E\|\LinScaleVec_i
\|_2^2\|\SelectProb_i\|_2^2\leq \E\|\LinScaleVec_i \|_2^2=\TargetDim
$. Therefore, we have from Lemma~\ref{lm:tech_md} that $\frobnorm{\EmpMean
  \LinScaleVec_i \SelectProb_i^\top}
=\bigoh(1)/\sqrt{\Numobs}=\liloh_p(\Numobs^{-t})$ for $0 < t <
\tfrac{1}{2}$.  Observe that
\begin{align}
\label{eqn:linear-replace-cov}  
\CondMean \LinScaleVec_i \LinScaleVec_i ^\top\iCovar^{1/2} = \CondMean
\iCovar^{1/2} \succeq \frac{c_0}{\Numobs^t} \IdMat_ {\TargetDim}.
\end{align}

Moreover, the random vectors $\LinScaleVec_i \LinScaleVec_i
^\top\iCovar^{1/2} - \E(\LinScaleVec_i \LinScaleVec_i
^\top\iCovar^{1/2} \mid \iHistory)$ define a martingale difference
sequence, and hence
\begin{align}
\label{eqn:matrix-diff-bound}  
\E \frobnorm{\LinScaleVec_i \LinScaleVec_i ^\top\iCovar^{1/2} -
  \E(\LinScaleVec_i \LinScaleVec_i ^\top\iCovar^{1/2} \mid
  \iHistory)}^2
& \leq \E\frobnorm{\LinScaleVec_i \LinScaleVec_i ^\top
  \iCovar^{1/2}}^2 = \E\frobnorm{\LinScaleVec_i (\iRegressor -
  \SelectProb_i)^\top}^2\notag \\
& = \E\|\LinScaleVec_i \|_2^2 \| \iRegressor - \SelectProb_i \|_2^2
\lesssim \E\|\LinScaleVec_i \|_2^2 = \TargetDim.
\end{align} 
Thus, it follows from Lemma~\ref{lm:tech_md} that $\frobnorm{(\EmpMean -
  \CondMean )\LinScaleVec_i \LinScaleVec_i ^\top\iCovar^{1/2}} =
\bigoh_p(\Numobs^{-1/2}) = \liloh_p(\Numobs^{-t})$ when
$t<1/2$. Combining this with $\frobnorm{ \EmpMean \LinScaleVec_i
  \SelectProb_i^\top}=\liloh_p(\Numobs^{-t})$,
equation~\eqref{eqn:linear-replace-cov}, and noting that $\CondMean
\LinScaleVec_i\LinScaleVec_i^\top\iCovar^{1/2}=\CondMean
\iCovar^{1/2}$, we find that
\begin{align*}
 \opnorm{\EmpMean \LinScaleVec_i \iRegressor^\top - \CondMean
   \iCovar^{1/2}} &\leq \opnorm{\EmpMean \LinScaleVec_i
   \LinScaleVec_i^\top\iCovar^{1/2} - \CondMean
   \iCovar^{1/2}}+\opnorm{\EmpMean \LinScaleVec_i
   \SelectProb_i^\top} \\
   &=\liloh_p(\sigma_{\min}(\CondMean
 \iCovar^{1/2})).
 \end{align*}

\paragraph*{Proof of bound~\eqref{eqn:term-two}}

Since $\iCovar^{1/2} - \E(\iCovar^{1/2}\mid\iHistory)$ is a martingale
difference sequence and
\begin{align*}
\E \frobnorm{\iCovar^{1/2} - \E(\iCovar^{1/2}\mid\iHistory)}^2
& \leq \E \frobnorm{\iCovar^{1/2}\|}^2=\E\tr(\iCovar)\leq \TargetDim,
\end{align*}
we have the bound $\E\frobnorm{(\EmpMean - \CondMean
  )\iCovar^{1/2}}^2\leq \frac{\TargetDim}{\NumIndexTwo
}=\bigoh(\Numobs^{-1})$.

From equation~\eqref{eqn:linear-replace-cov}, we have
\begin{align*}
\opnorm{\EmpMean \iCovar^{1/2} - \CondMean
  \iCovar^{1/2}}=\liloh_p(\Numobs^{-1/2})=\liloh_p(\Numobs^{-t})=\liloh_p(\sigma_{\min}(\CondMean
\iCovar^{1/2})).
\end{align*}
It follows from Weyl's theorem that $\opnorm{\EmpMean \iCovar^{1/2} -
  \CondMean \iCovar^{1/2}} = \liloh_p (\sigma_{\min} (\EmpMean
\iCovar^{1/2}))$.
\end{proof}


\subsection{Auxiliary lemmas for Theorem~\ref{thm:linear_new3}}
\label{proof:lm:thm:linear_new3}

This section is devoted to the proofs of the auxiliary lemmas used in
the proof of Theorem~\ref{thm:linear_new3}.

\begin{lems}
\label{lm:asym_21}
Under the assumptions in Theorem~\ref{thm:linear_new3}, we have $\EmpMean
\sqrt{\NumIndexTwo }\LinScaleVecOneArm_{i1} \iNoise \convdist
\Normal(0, \LinearNoiseVar )$.
\end{lems}
\begin{proof}
We follow an argument very similar to that used in
proving Lemma~\ref{lm:asym_1}: we show $\{\LinScaleVecOneArm_{i1} \iNoise
\}_{i \geq 1}$ is a martingale difference sequence, so that a standard
martingale central limit theorem can be applied.

It follows from straightforward calculations that
$\LinScaleVecOneArm_{i1} \iNoise$ is a martingale difference sequence.
Moreover, we have
\begin{align*}
\E (\iNoise^2 {{\LinScaleVecOneArm_{i1}}^2}\mid \iHistory)=\E
({\LinScaleVecOneArm_{i1}}^2 \E(\iNoise^2\mid\iRegressor, \iNuisance,
\iHistory) \mid \iHistory)
= \E (\LinearNoiseVar {\LinScaleVecOneArm_{i1}}^2 \mid\iHistory) =
\LinearNoiseVar,
\end{align*}
where the last equality follows from the relation
\begin{align*}
\E{\LinScaleVecOneArm_{i1}}^2 & = \E\ \frac{\OneDirect^\top
  \iCovar^{-1}(\iRegressor - \SelectProb_{i})(\iRegressor -
  \SelectProb_{i})^\top\iCovar^{-1}\OneDirect}{{\OneDirect^\top
    \iCovar^{-1} \OneDirect}} = 1.
\end{align*}

We now verify the Lindeberg condition. First observe that
$\SelectProb_{ik}\gtrsim c_i$ for all $k\in\SuppDirect\cup\{0\}$, and
hence
\begin{align}
\label{eq:vi_bound_b}
\LinScaleVecOneArm_{i1}^2 & \leq \|\ScalePreCondVec_{i1}\|_2^2 \cdot
\| \iRegressor - \SelectProb_i(\iNuisance, \iHistory)\|_2^2
\notag \\
& \lesssim
\frac{\OneDirect^\top\iCovar^{-2}\OneDirect}{{\OneDirect^\top
    \iCovar^{-1} \OneDirect}}
\notag \\
& = \frac{\sum_{j,k=1}^{\TargetDim}\OneDirect_j\Covar^{-2}_{i,jk}
  \OneDirect_k}{\sum_{j,k=1}^{\TargetDim}\OneDirect_j\Covar^{-1}_{i,jk}\OneDirect_k}
\lesssim \frac{1}{c_i},
\end{align}
where the second inequality follows from $\| \iRegressor -
\SelectProb_i(\iNuisance, \iHistory)\|_2\leq \|\iRegressor\|_2 + \|
\SelectProb_i(\iNuisance, \iHistory)\|_2\leq 2$ and the definition of
$\ScalePreCondVec_{i1}$, the last inequality is due to the fact that
for any $j,k\in\SuppDirect$
\begin{align*}
\frac{\Covar^{-2}_{i,jk}}{\Covar^{-1}_{i,jk}} & =
\frac{\sum_{l=1}^{\TargetDim}(\gamma_i +
  1_{\{j=l\}}\frac{1}{\SelectProb_{j}})(\gamma_i +
  1_{\{l=k\}}\frac{1}{\SelectProb_{k}})}{\gamma_i +
  1_{\{j=k\}}\frac{1}{\SelectProb_{j}}}
 \leq \TargetDim (\gamma_i + \frac{1}{\SelectProb_{j}} +
 \frac{1}{\SelectProb_{k}})\lesssim \frac{1}{c_i}
 \end{align*} 
by the expression~\eqref{eq:cov_inverse}.  Therefore, we have the
bound $\iNoise^2 \LinScaleVecOneArm_{i1}^2 \leq \iNoise^2/c_i $, and
for any $\eps>0$,
\begin{align}
\label{eqn:lindeberg-cond-linear2}  
0 \leq
\lim_{\Numobs\to\infty}\frac{1}{\Numobs}\sum_{i=1}^\Numobs\E(\iNoise^2
\LinScaleVecOneArm_{i1}^2 {1}_{\{|\iNoise^2
  \LinScaleVecOneArm_{i1}^2|>\eps{\Numobs}\}}\mid\iHistory) &\lesssim
\lim_{\Numobs\to\infty} \frac{1}{\Numobs}
\sum_{i=1}^\Numobs\frac{1}{c_i}\E(\iNoise^2 {1}_{\{\iNoise^2 \geq
  \eps{\Numobs}c_i\}} \mid\iHistory) \to 0,
\end{align}
where the convergence follows from the sub-Gaussianity of $\iNoise$,
and the same argument used in proving
equation~\eqref{eqn:lindeberg-cond-linear} in Lemma~\ref{lm:asym_1}.  This
implies that $\{\LinScaleVecOneArm_{i1}
\iNoise\}^\Numobs_{i=\NumIndexOne + 1}$ satisfies Lindeberg's
condition.

Putting together the pieces and applying the martingale central limit
theorem, we conclude $\EmpMean \sqrt{\NumIndexTwo
}\LinScaleVecOneArm_{i1} \iNoise \convdist \Normal(0, \LinearNoiseVar
)$.
\end{proof}

\begin{lems}
\label{lm:asym_22}
Under the assumptions of Theorem~\ref{thm:linear_new3}, we have
\begin{align*}
\EmpMean \sqrt{\NumIndexTwo }\LinScaleVecOneArm_{i1}
(\EstFunPar(\iNuisance) - \TrueFunPar(\iNuisance) )\overset{p}{\to} 0.
\end{align*}
\end{lems}
\begin{proof}
  Note that the proof of Lemma~\ref{lm:asym_2} only exploits the
  boundedness condition $\E \|\LinScaleVec_i \|^2=\TargetDim$.
  Moreover, we have shown $\E \LinScaleVecOneArm_{i1} ^2=1$ in the
  proof of Lemma~\ref{lm:asym_21}. Thus, this lemma can be established by
  following exactly the same argument used to prove Lemma~\ref{lm:asym_2},
  with $\LinScaleVec_i$ replaced by $\LinScaleVecOneArm_{i1}$.
\end{proof}

\begin{lems}
\label{lm:asym_23}
Under the assumptions of Theorem~\ref{thm:linear_new3} we have
\begin{align*}
\Big|\EmpMean \LinScaleVecOneArm_{i1} \iRegressor^\top \OneDirect -
\EmpMean \frac{1}{\sqrt{\OneDirect^\top\iCovar^{-1} \OneDirect}} \Big|
=\liloh_p \Big( \Big |\EmpMean
\frac{1}{\sqrt{\OneDirect^\top\iCovar^{-1}\OneDirect}} \Big |\Big).
\end{align*}
\end{lems}
\begin{proof}  
We have
\begin{multline*}
\Big|\EmpMean \LinScaleVecOneArm_{i1} \iRegressor^\top\OneDirect -
\EmpMean \frac{1 }{\sqrt{\OneDirect^\top\iCovar^{-1}\OneDirect}}\Big|
\leq \Big|\EmpMean \LinScaleVecOneArm_{i1} \iRegressor^\top\OneDirect
- \CondMean \frac{1
}{\sqrt{\OneDirect^\top\iCovar^{-1}\OneDirect}}\Big| \\
+ |\EmpMean \frac{1 }{\sqrt{\OneDirect^\top\iCovar^{-1}\OneDirect}} -
\CondMean \frac{1 }{\sqrt{\OneDirect^\top\iCovar^{-1}\OneDirect}}| =
\Term_1 + \Term_2
\end{multline*}
We show that both $\Term_1$ and $\Term_2$ are bounded by
$\liloh_p\Big(\Big|\EmpMean \frac{1
}{\sqrt{\OneDirect^\top\iCovar^{-1}\OneDirect}}\Big|\Big)$.

\paragraph*{Bound on $\Term_1$}
Since $\E(\LinScaleVecOneArm_{i1} \SelectProb_{i}^\top\mid\iNuisance,
\iHistory)=\E(\ScalePreCondVec_{i1}(\iRegressor -
\SelectProb_i)\SelectProb_{i}^\top\mid\iNuisance, \iHistory)=0$,
$\{\LinScaleVecOneArm_{i1} \SelectProb_{i}^\top\}_{i=\NumIndexOne +
  1}^\Numobs$ is a martingale difference sequence w.r.t.
$\iHistory$. Since $\E|\LinScaleVecOneArm_{i1} \SelectProb_{i}^\top
|^2= \E|\LinScaleVecOneArm_{i1} |^2|\SelectProb_{i}^\top|^2\leq \E
|\LinScaleVecOneArm_{i1} |^2= 1 $, it follows directly
from Lemma~\ref{lm:tech_md} that $\EmpMean \LinScaleVecOneArm_{i1}
\SelectProb_{i}^\top = \bigoh_p(\Numobs^{-1/2})$. Under the assumption
$\SelectProb_{ik} \gtrsim i^{-2t}$ for all $k \in \SuppDirect \cup
\{0\}$, it follows from the expression of $\iCovar$ from
equation~\eqref{eq:cov_inverse} that
$\frac{1}{\sqrt{\OneDirect^\top\iCovar^{-1}\OneDirect}}\gtrsim
\Numobs^{-t}$ and thus
\begin{align}
\label{eq:linear-compare-scale}  
\bigoh_p(\Numobs^{-1/2}) = \liloh_p(\Numobs^{-t}) =
\liloh_p\Big(\Big|\EmpMean \frac{1}{\sqrt{\OneDirect^\top \iCovar^{-1}
    \OneDirect}} \Big| \Big).
\end{align}
Note that
\begin{align}
\label{eqn:linear-replace-cov2}  
\CondMean \LinScaleVecOneArm_{i1} (\iRegressor - \SelectProb_{i})^\top
\OneDirect= \CondMean \frac{\OneDirect^\top \iCovar^{-1}(\iRegressor -
  \SelectProb_{i})(\iRegressor -
  \SelectProb_{i})^\top\OneDirect}{\sqrt{\OneDirect^\top\iCovar^{-1}\OneDirect}}=
\CondMean \frac{1 }{\sqrt{\OneDirect^\top\iCovar^{-1}\OneDirect}}.
\end{align}
and $\LinScaleVecOneArm_{i1} (\Regressor_{i} - \SelectProb_{i})^\top -
\E(\LinScaleVecOneArm_{i1} (\Regressor_{i} -
\SelectProb_{i})^\top\mid\iHistory)$ is a martingale difference
sequence with
\begin{align}
\E\|\LinScaleVecOneArm_{i1}(\Regressor_{i} - \SelectProb_{i})^\top -
\E(\LinScaleVecOneArm_{i1} (\Regressor_{i} -
\SelectProb_{i})^\top\mid\iHistory)\|_2^2
\label{eqn:matrix-diff-bound2}
&\leq\E\|\LinScaleVecOneArm_{i1} (\Regressor_{i} -
\SelectProb_{i})^\top\|_2^2\notag\\ &= \E\|\Regressor_{i} -
\SelectProb_{i}\|_2^2\LinScaleVecOneArm_{i1}^2\lesssim 1,
\end{align} 
it follows that $\E\|(\EmpMean - \CondMean )\LinScaleVecOneArm_{i1}
(\Regressor_{i} - \SelectProb_{i})^\top\|_2^2=\bigoh(\Numobs^{-1})$
and hence $(\EmpMean - \CondMean )\LinScaleVecOneArm_{i1}
(\Regressor_{i} - \SelectProb_{i})^\top=\bigoh_p(\Numobs^{-1/2})$.
Combining this with $\EmpMean \LinScaleVecOneArm_{i1}
\SelectProb_{i}^\top=\bigoh_p(\Numobs^{-1/2})$
and~\eqref{eq:linear-compare-scale},~\eqref{eqn:linear-replace-cov2}
yields
\begin{align*}
\Big|\EmpMean \LinScaleVecOneArm_{i1} \Regressor_{i}^\top\OneDirect-
\CondMean \frac{1 }{\sqrt{\OneDirect^\top\iCovar^{-1}\OneDirect}}
\Big|
=\bigoh_p(\Numobs^{-1/2})=\liloh_p\Big(\EmpMean  \frac{1
}{\sqrt{\OneDirect^\top\iCovar^{-1}\OneDirect}}\Big).
\end{align*}

\paragraph*{Bound on $\Term_2$}
Since $ \frac{1}{\sqrt{\OneDirect^\top\iCovar^{-1}\OneDirect}} - \E
(\frac{1 }{\sqrt{\OneDirect^\top \iCovar^{-1}
    \OneDirect}}\mid\iHistory)$ is a martingale difference sequence
and
\begin{align*}
\E\Big|\frac{1 }{\sqrt{\OneDirect^\top\iCovar^{-1}\OneDirect}} -
\E(\frac{1
}{\sqrt{\OneDirect^\top\iCovar^{-1}\OneDirect}}\mid\iHistory)\Big|^2
& \leq\E\Big|\frac{1 }{\sqrt{\OneDirect^\top\iCovar^{-1}\OneDirect}}\Big|^2=\E
\|\OneDirect\|_2^2\|\iCovar\|_2\lesssim 1,
\end{align*}
we have $\E|(\EmpMean - \CondMean )\frac{1
}{\sqrt{\OneDirect^\top\iCovar^{-1}\OneDirect}}|^2\lesssim \frac{1
}{\NumIndexTwo }=\bigoh(\Numobs^{-1})$. Therefore,
\begin{align*}
\Big|\EmpMean\frac{1 }{\sqrt{\OneDirect^\top\iCovar^{-1}\OneDirect}} -
\CondMean \frac{1 }{\sqrt{\OneDirect^\top\iCovar^{-1}\OneDirect}}\Big| =
\bigoh_p(\Numobs^{-1/2})=\liloh_p\Big(\Big|\EmpMean \frac{1
}{\sqrt{\OneDirect^\top\iCovar^{-1}\OneDirect}}\Big|\Big),
\end{align*}
which concludes the proof.
 
\end{proof}


\subsection{Auxiliary lemmas for Theorem~\ref{thm:glm_new1_temp}}
\label{proof:lm:thm:glm_new1_temp}

\begin{lems}[Upper bound on $\opnorm{\InvSqrtCovar_i}$]
\label{lm:equiv_glm}
Under the assumptions of Theorem~\ref{thm:glm_new1_temp}, we have
\begin{align*}
\opnorm{\InvSqrtCovarwhat_{i}} \leq \frac{1}{\sqrt{\cwtil_i}} \qquad
\text{and} \qquad
\opnorm{\InvSqrtCovar^*_i} \leq \frac{1}{\sqrt{\cwtil_i}},
\end{align*}
where $\cwtil_i = \cwtil_0/i^{2t}$ and $\cwtil_0 = m_{\eps}
c_0/(\TargetDim + 2)$.
\end{lems}

\begin{proof}
We only prove the result for $\InvSqrtCovarwhat_i$. The result for
$\InvSqrtCovar_i^*$ can be shown similarly.  Recall
\begin{align*}
\InvSqrtCovarwhat_i & = [\E(\GlmVar
  \big(g\big(\inprod{\iRegressor}{\EstTargetPar} +
  \inprod{\iNuisance}{\EstNuisancePar} \big)\big)(\iRegressor -
  \GlmMeanVecwhat_i)(\iRegressor-
  \GlmMeanVecwhat_i)^\top\mid\iNuisance, \iHistory)]^{-1/2}
\end{align*}
by definition, and it suffices to show \begin{align*}\E(\GlmVar
  \big(g\big(\inprod{\iRegressor}{\EstTargetPar} +
  \inprod{\iNuisance}{\EstNuisancePar} \big)\big)(\iRegressor -
  \GlmMeanVecwhat_i)(\iRegressor -
  \GlmMeanVecwhat_i)^\top\mid\iNuisance, \iHistory)\succeq \cwtil_i
  \IdMat_{\TargetDim}.\end{align*}  Note that
\begin{align*}
\quad & \E(\GlmVar \big(g\big(\inprod{\iRegressor}{\EstTargetPar} +
\inprod{\iNuisance}{\EstNuisancePar} \big)\big)(\iRegressor -
\GlmMeanVecwhat_i)(\iRegressor -
\GlmMeanVecwhat_i)^\top\mid\iNuisance, \iHistory) \\
& \succeq m_\eps \E((\iRegressor - \GlmMeanVecwhat_i)(\iRegressor -
\GlmMeanVecwhat_i)^\top \mid \iNuisance, \iHistory) \\
& \succeq m_\eps \E((\iRegressor- \SelectProb_i)(\iRegressor -
\SelectProb_i)^\top\mid\iNuisance, \iHistory),
\end{align*}
where the first inequality follows from the assumption that
$\GlmVar(x)\geq m_\eps$ and the second inequality is due to the fact
that $\E(\iRegressor\mid\iNuisance,
\iHistory)=\SelectProb_i$. In Lemma~\ref{lm:equiv_ass} we show
$\E((\iRegressor- \SelectProb_i)(\iRegressor -
\SelectProb_i)^\top\mid\iNuisance, \iHistory)\succeq c_i/(\TargetDim +
2)\IdMat_{\TargetDim}$. Putting together the pieces yields
\begin{align*}
\E(\GlmVar \big(g\big(\inprod{\iRegressor}{\EstTargetPar} +
\inprod{\iNuisance}{\EstNuisancePar} \big)\big)(\iRegressor -
\GlmMeanVecwhat_i)(\iRegressor -
\GlmMeanVecwhat_i)^\top\mid\iNuisance, \iHistory) \succeq \cwtil_i
\IdMat_{\TargetDim},
\end{align*}
where $\cwtil_i=m_{\eps}c_i/(\TargetDim + 2)$. This completes the
proof.
\end{proof}

\begin{lems}\label{lm:glm_clt_temp}
Under the assumptions in Theorem~\ref{thm:glm_new1_temp}, we have
\begin{align*}
\sqrt{\NumIndexTwo }\EmpMean \InvSqrtCovarwhat_i(\iRegressor -
\GlmMeanVecwhat_i)\iNoise \convdist \Normal(0, \IdMat_{\TargetDim
}).\end{align*}
\end{lems}
\begin{proof}
Similar to Lemma~\ref{lm:asym_1}, the idea of this proof is to apply a
Martingale version of the central limit theorem on the sequence
$\InvSqrtCovarwhat_i(\iNuisance, \iHistory)(\iRegressor-
\GlmMeanVecwhat_i)\iNoise$.  By definition, in the generalized linear
model $Y=g(X^\top\AllNuisancePar) + \eps$, the distribution of $\eps$
depends on the value of $X^\top\AllNuisancePar$.  Since
$\E(\iNoise\mid\iRegressor, \iNuisance)=0$,
$\InvSqrtCovarwhat_i(\iNuisance, \iHistory)(\iRegressor-
\GlmMeanVecwhat_i)\iNoise$ is a martingale difference sequence.

\paragraph*{Asymptotic covariance}
Note that
\begin{align*}
&\quad\E(\|\InvSqrtCovarwhat_i(\iRegressor -
  \GlmMeanVecwhat_i)\iNoise\|_2^2\mid\iHistory)\\
&= \E(\InvSqrtCovarwhat_i(\iRegressor -
  \GlmMeanVecwhat_i)\GlmVar\big(g\big(\inprod{\iRegressor}{\TrueTargetPar}
  + \inprod{\iNuisance}{\TrueNuisancePar} \big)\big)(\iRegressor -
  \GlmMeanVecwhat_i)^\top\InvSqrtCovarwhat_i^\top\mid\iHistory)\\
&= \E(\InvSqrtCovarwhat_i\E((\iRegressor -
\GlmMeanVecwhat_i)\GlmVar\big(g\big(\inprod{\iRegressor}{\TrueTargetPar}
+ \inprod{\iNuisance}{\TrueNuisancePar} \big)\big)(\iRegressor -
\GlmMeanVecwhat_i)^\top\mid\iNuisance,
\iHistory)\InvSqrtCovarwhat_i^\top\mid\iHistory)\\
&= \E(\InvSqrtCovarwhat_i\E((\iRegressor -
\GlmMeanVecwhat_i)\GlmVar\big(g\big(\inprod{\iRegressor}{\EstTargetPar}
+ \inprod{\iNuisance}{\EstNuisancePar} \big)\big)(\iRegressor -
\GlmMeanVecwhat_i)^\top\mid\iNuisance,
\iHistory)\InvSqrtCovarwhat_i^\top\mid\iHistory)\\
& + \bigoh_p(L_\eps L_g D_{x} \|\TrueAllNuisancePar -
\EstAllNuisancePar\|_2\E(\|\InvSqrtCovarwhat_i(\iRegressor -
\GlmMeanVecwhat_i)\|_2^2\mid\iHistory))\\
&=\IdMat_{\TargetDim} + \bigoh_p\Big(\frac{L_\eps L_g
  D_{x}}{m_\eps}\|\TrueAllNuisancePar - \EstAllNuisancePar\|_2
\E(\|\InvSqrtCovarwhat_i(\iRegressor -
\GlmMeanVecwhat_i)\|_2^2\GlmVar\big(g\big(\inprod{\iRegressor}{\EstTargetPar}
+ \inprod{\iNuisance}{\EstNuisancePar} \big)\big)\mid\iHistory)\Big)\\
&=\IdMat_{\TargetDim} + \bigoh_p(\TargetDim\|\TrueAllNuisancePar -
\EstAllNuisancePar\|_2)=\IdMat_{\TargetDim} + \liloh_p(1).
\end{align*}
where the third equation follows from triangle inequality combined
with the Lipschitz continuity of $\GlmVar,g$, and the boundedness of
$\AllNuisancePar$. The fourth equation is due to the definition of
$\InvSqrtCovarwhat_i$ and the lower bound assumption, $\GlmVar(x) \geq
m_\eps.$ The last line uses the definition of $\InvSqrtCovarwhat_i$
and the consistency assumption of $\EstTargetPar, \EstNuisancePar
$. Note that $\liloh_p(1)$ in the last line are the same for all $i$.
It follows from properties of Martingale difference sequences that
$\E\|\EmpMean \InvSqrtCovarwhat_i(\iRegressor -
\GlmMeanVecwhat_i)\iNoise\|_2^2\to \IdMat_{\TargetDim}$. Thus, the
lemma is implied by the Martingale central limit theory for $\EmpMean
\sqrt{\NumIndexTwo }\InvSqrtCovarwhat_i(\iRegressor -
\GlmMeanVecwhat_i)\iNoise$ and it remains to verify Lindeberg's
condition.

\paragraph*{Lindeberg condition}
We proceed by first bounding
$\opnorm{\InvSqrtCovarwhat_i(\iRegressor-
\GlmMeanVecwhat_i)\iNoise^2(\iRegressor-
\GlmMeanVecwhat_i)^\top\InvSqrtCovarwhat_i^\top}$. Specifically,
\begin{align*}
\opnorm{\iNoise^2\InvSqrtCovarwhat_{i}(\iRegressor-
\GlmMeanVecwhat_i)(\iRegressor-
\GlmMeanVecwhat_i)^\top\InvSqrtCovarwhat^{\top}_{i}}
&\leq
\iNoise^2\opnorm{\InvSqrtCovarwhat_{i}}^2
\|(\iRegressor - \GlmMeanVec^*_i)\|_2^2\leq \frac{4}{\cwtil_i}
\iNoise^2,
\end{align*}
where $\cwtil_i=\cwtil_0/i^{2t}$ and $\cwtil_0=m_\eps c_0/(\TargetDim  + 2)$. The last inequality follows
from Lemma~\ref{lm:equiv_glm}, and the fact that $\|\iRegressor-
\GlmMeanVecwhat_i\|_2\leq \|\iRegressor\|_2 + \|\GlmMeanVecwhat_i\|_2\leq 2$.
Therefore $\InvSqrtCovarwhat_i(\iRegressor-
\GlmMeanVecwhat_i)\iNoise^2(\iRegressor-
\GlmMeanVecwhat_i)^\top\InvSqrtCovarwhat_i^\top \preceq 4
\iNoise^2\IdMat_{\TargetDim}/\cwtil_i $ and for any $\eps>0$,
\begin{align}
0&\preceq
\lim_{\Numobs\to\infty}\frac{1}{\Numobs}
\sum_{i=1}^\Numobs\E(\InvSqrtCovarwhat_i(\iRegressor
- \GlmMeanVecwhat_i) \iNoise^2(\iRegressor - \GlmMeanVecwhat_i)^\top \InvSqrtCovarwhat_i^\top
           {1}_{\{\opnorm{\InvSqrtCovarwhat_i(\iRegressor-
             \GlmMeanVecwhat_i)\iNoise^2(\iRegressor-
             \GlmMeanVecwhat_i)^\top\InvSqrtCovarwhat_i^\top}
          >\eps{\Numobs}\}}\mid\iHistory)\notag\\
\label{eqn:lindeberg-cond-glm_temp}           
& \preceq \lim_{\Numobs \to \infty} \frac{4}{\Numobs}
\sum_{i=1}^\Numobs \frac{1}{\cwtil_i} \E(\iNoise^2
    {1}_{\{\iNoise^2 \geq \eps{\Numobs} \cwtil_i/4 \}}
    \mid\iHistory) \IdMat_{\TargetDim}.
\end{align}

Since $\iNoise$ are sub-Gaussian random variables (conditioned on
$\iRegressor, \iNuisance, \iHistory$) with common parameter
$\subgauss$ almost surely, it follows that $\iNoise^2$ are
subexponential random variables with a common parameter. Therefore,
there exists some constant $K_1>0$ depending on $\subgauss$ such that
$\prob(\iNoise^2\geq s)\leq 2\exp(-s/K_1)$ and hence
\begin{align*}
\E(\iNoise^2 1_{\{\iNoise^2\geq\eps \Numobs \cwtil_i/4\}}) = \int_{\eps
  \Numobs \cwtil_i/4}^\infty \prob(\iNoise^2 \geq s) ds
\leq 2 \int_{\eps \Numobs \cwtil_i/4}^\infty
\exp(-s/K_1)ds=2K_1\exp\Big(\frac{ - \eps \Numobs
  \cwtil_i}{4K_1}\Big).
\end{align*}
Substituting this into equation~\eqref{eqn:lindeberg-cond-glm_temp},
for any $t \in (0, 1/2)$, we have
\begin{align*}
& \lim_{\Numobs\to\infty}
  \frac{4}{\Numobs}\sum_{i=1}^\Numobs\frac{1}{\cwtil_i}\E(\iNoise^2
       {1}_{\{\iNoise^2 \geq
         \eps{\Numobs}\cwtil_i/4\}}\mid\iHistory)\\
&\lesssim\frac{\cwtil_0K_1}{\Numobs} \sum_{i=1}^{\Numobs} i^{2t} \exp
       \Big(\frac{ - \eps \Numobs \cwtil_0}{4 K_1 i^{2t}} \Big)\\
& \leq \cwtil_0 K_1 \Numobs^{2t} \exp \Big( \frac{ - \eps
         \Numobs^{1-2t} \cwtil_0}{4K_1 }\Big) \to 0.
\end{align*}
Thus, Lindeberg's condition is satisfied, so that the proof is
complete.
\end{proof}

\begin{lems}\label{lm:glm_neyman_temp}
Under the assumptions in Theorem~\ref{thm:glm_new1_temp} and suppose $\|\SoluTargetPar  - \TrueTargetPar\|_2 =\liloh_p(1)$,  we have 
 \begin{align*}
 \|\sqrt{\NumIndexTwo }\EmpMean \InvSqrtCovarwhat_i(\iRegressor - \GlmMeanVecwhat_i) g'\big(\inprod{\iRegressor}{\EstTargetPar}  + \inprod{\iNuisance}{\EstNuisancePar} \big)\GlmMeanVecwhat_i^\top(\SoluTargetPar  - \TrueTargetPar )\|_2
 &\overset{p}{\to}0\\
 \|\sqrt{\NumIndexTwo }\EmpMean \InvSqrtCovarwhat_i(\iRegressor - \GlmMeanVecwhat_i) g'\big(\inprod{\iRegressor}{\EstTargetPar}  + \inprod{\iNuisance}{\EstNuisancePar} \big)\iNuisance^\top(\EstNuisancePar  - \TrueNuisancePar )\|_2
 &\overset{p}{\to}0
 \end{align*}
    
\end{lems}
\begin{proof}
Since $\|\EstNuisancePar - \TrueNuisancePar\|_2
=\liloh_p(\Numobs^{-1/4})$, $\|\SoluTargetPar - \TrueTargetPar\|_2
=\liloh_p(1)$ are consistent, it suffices to show that
\begin{align}
 \label{eq:glm_lm_neyma\NumIndexOne_temp}  
 \frobnorm{\sqrt{\NumIndexTwo }\EmpMean
   \InvSqrtCovarwhat_i(\iRegressor - \GlmMeanVecwhat_i)
   g'\big(\inprod{\iRegressor}{\EstTargetPar} +
   \inprod{\iNuisance}{\EstNuisancePar}
   \big)\GlmMeanVecwhat_i^\top}=\bigoh_p(1) \\
\label{eq:glm_lm_neyma\NumIndexOne1_temp} 
 \frobnorm{\sqrt{\NumIndexTwo }\EmpMean
   \InvSqrtCovarwhat_i(\iRegressor - \GlmMeanVecwhat_i)
   g'\big(\inprod{\iRegressor}{\EstTargetPar} +
   \inprod{\iNuisance}{\EstNuisancePar}
   \big)\iNuisance^\top}=\bigoh_p(1).
\end{align}
Since $\E(\InvSqrtCovarwhat_i(\iRegressor - \GlmMeanVecwhat_i)
g'\big(\inprod{\iRegressor}{\EstTargetPar} +
\inprod{\iNuisance}{\EstNuisancePar} \big) \mid\iNuisance,
\iHistory)=0$ by definition of $\GlmMeanVecwhat_i$, it follows
directly that $\{\InvSqrtCovarwhat_i(\iRegressor - \GlmMeanVecwhat_i)
g'\big(\inprod{\iRegressor}{\EstTargetPar} +
\inprod{\iNuisance}{\EstNuisancePar}
\big)\GlmMeanVecwhat_i^\top\}_{i=1}^\Numobs$ is a Martingale
difference sequence.  Note that
\begin{align}
    \E\frobnorm{\InvSqrtCovarwhat_i(\iRegressor - \GlmMeanVecwhat_i) g'\big(\inprod{\iRegressor}{\EstTargetPar}  + \inprod{\iNuisance}{\EstNuisancePar} \big)\GlmMeanVecwhat_i^\top}^2
    \leq
     L^2_{g} \E\|\InvSqrtCovarwhat_i(\iRegressor - \GlmMeanVecwhat_i)\|_2^2 \|\GlmMeanVecwhat_i^\top\|_2^2\lesssim \frac{L^2_g \TargetDim}{m_\eps}=\bigoh(1), \label{eq:glm_lm_neyma\NumIndexOne2_temp}
\end{align}
where the first inequality uses the fact that $|g'|\leq L_g$, which is
implied by the standard assumptions on GLM. The second inequality
follows from $\|\GlmMeanVecwhat_i\|_2, \|\iRegressor\|_2\leq 1$ and,
 \begin{align}
      \E \|\InvSqrtCovarwhat_i(\iRegressor - \GlmMeanVecwhat_i)\|_2^2
      &=  \E \|\InvSqrtCovarwhat_i(\iRegressor - \GlmMeanVecwhat_i)\GlmVar\big(\inprod{\iRegressor}{\EstTargetPar}  + \inprod{\iNuisance}{\EstNuisancePar} \big)^{1/2}/\GlmVar\big(\inprod{\iRegressor}{\EstTargetPar}  + \inprod{\iNuisance}{\EstNuisancePar} \big)^{1/2}\|_2^2\notag\\
      &\leq \frac{1}{m_\eps} \E \|\InvSqrtCovarwhat_i(\iRegressor - \GlmMeanVecwhat_i)\GlmVar\big(\inprod{\iRegressor}{\EstTargetPar}  + \inprod{\iNuisance}{\EstNuisancePar} \big)^{1/2}\|_2^2\notag\\
      &=
      \frac{1}{m_\eps} \E \tr(\InvSqrtCovarwhat_i(\iRegressor - \GlmMeanVecwhat_i)\GlmVar\big(\inprod{\iRegressor}{\EstTargetPar}  + \inprod{\iNuisance}{\EstNuisancePar} \big)(\iRegressor - \GlmMeanVecwhat_i)^\top\InvSqrtCovarwhat_i)
      =\frac{\TargetDim}{m_\eps}=\bigoh(1), \label{eq:glm_bounded_oxm}
    \end{align}  
    where the second line uses the definition of $\InvSqrtCovarwhat_i$
    and $\GlmMeanVecwhat_i$. The
    bound~\eqref{eq:glm_lm_neyma\NumIndexOne2_temp} immediately
    implies the bound~\eqref{eq:glm_lm_neyma\NumIndexOne_temp}. Since
    we assume $\|\iNuisance\|_2$ is bounded, the
    bound~\eqref{eq:glm_lm_neyma\NumIndexOne1_temp} follows from
    similar arguments as above with $\GlmMeanVecwhat_i^\top$ replaced
    by $\iNuisance^\top$.
\end{proof}

\begin{lems}\label{lm:glm_second_order_temp}
In addition to the assumptions of Theorem~\ref{thm:glm_new1_temp} suppose
that $\|\SoluTargetPar  - \TrueTargetPar\|_2 =\liloh_p(\Numobs^{-t})$. Then we
have
\begin{align*}
\quad &\sqrt{\NumIndexTwo }\EmpMean
\InvSqrtCovarwhat_i(\iRegressor - \GlmMeanVecwhat_i)(Q_3 +
Q_4) \\
& = \liloh_p(1) + \liloh_p(\|\sqrt{\NumIndexTwo }\EmpMean 
\InvSqrtCovarwhat_i(\iRegressor - \GlmMeanVecwhat_i) g'\big(\inprod{\iRegressor}{\EstTargetPar}  + \inprod{\iNuisance}{\EstNuisancePar} \big) (
\iRegressor - \GlmMeanVecwhat_i)^\top (\SoluTargetPar -
\TrueTargetPar ) \|_2).
\end{align*}
\end{lems}

\begin{proof}
Since $g'$ is $L_{g'}$-Lipschitz by assumption, it follows that
$|g^{''}|\leq L_{g'}$. Thus, we have
\begin{align*}
\quad & \sqrt{\NumIndexTwo }\EmpMean 
\InvSqrtCovarwhat_i(\iRegressor - \GlmMeanVecwhat_i)(Q_3 + Q_4) \\
& =\frac{\sqrt{\NumIndexTwo }}{2}\EmpMean \int^1_{0}\int^1_{0}
g''\big(\inprod{\iRegressor}{\EstTargetPar + r_1r_2(\SoluTargetPar
 - \EstTargetPar )} +\inprod{ \iNuisance}{\EstNuisancePar }\big) \big|\inprod{\iRegressor}{
\SoluTargetPar  - \EstTargetPar }|^2 dr_1 dr_2 \\
&  - \frac{\sqrt{\NumIndexTwo}}{2}\EmpMean
\int^1_{0}\int^1_{0} \left\{ g^{''}\big(\inprod{\iRegressor}{\EstTargetPar +
r_1r_2(\TrueTargetPar  - \EstTargetPar )} + \inprod{\iNuisance}{
\EstNuisancePar + r_1 r_2(\TrueNuisancePar  - \EstNuisancePar )}\big) \right. \\
& \left. \qquad \qquad \qquad \qquad \qquad \qquad  \big|\inprod{
\iRegressor}{\TrueTargetPar  - \EstTargetPar } +
\inprod{\iNuisance}{\TrueNuisancePar  - \EstNuisancePar }\big|^2 \right\} dr_1 dr_2 \\
& = \bigoh(L_{g'}\sqrt{\NumIndexTwo}
\sup_i|\inprod{\iRegressor}{\SoluTargetPar  - \EstTargetPar }|^2) +
\bigoh(L_{g'}|\sqrt{\NumIndexTwo} \sup_i\inprod{\iNuisance}{
\EstNuisancePar  - \TrueNuisancePar} |^2) \\
& \qquad \qquad  + 
\bigoh(L_{g'}|\sqrt{\NumIndexTwo }\sup_i\inprod{\iRegressor}{\EstTargetPar
- \TrueTargetPar }|^2) \\
& = \bigoh_p(\sqrt{\NumIndexTwo } \|\SoluTargetPar - \TrueTargetPar
\|_2^2) + \bigoh_p(\sqrt{\NumIndexTwo } \|\EstNuisancePar -
\TrueNuisancePar \|_2^2) + \bigoh_p(\sqrt{\NumIndexTwo }\|\EstTargetPar
 - \TrueTargetPar \|_2^2)\\
& =\liloh_p(\Numobs^{1/2-t} \|\SoluTargetPar  - \TrueTargetPar \|_2) +
\liloh_p(1),
\end{align*}
where the second equation uses $|g^{''}|\leq L_{g'}$, the third
equation uses the boundedness assumption of $\iRegressor, \iNuisance$
and the fact that $\|\SoluTargetPar - \EstTargetPar \|_2^2\leq 2
(\|\SoluTargetPar - \TrueTargetPar \|_2^2 + \|\EstTargetPar -
\TrueTargetPar \|_2^2$). The last line follows from the
$\Numobs^{-t}$-consistency of $\SoluTargetPar $ and
$\Numobs^{-1/4}$-consistency of $\EstTargetPar, \EstNuisancePar
$. Denote $\EmpMean \InvSqrtCovarwhat_i(\iRegressor-
\GlmMeanVecwhat_i)g'(\iRegressor^\top \EstTargetPar + \iNuisance^\top
\EstNuisancePar ) (\iRegressor - \GlmMeanVecwhat_i)^\top$ by $Z_0$,
and by Lemma~\ref{lm:glm_consistent_zero_temp} we have
$\prob(\sigma_{\min}(Z_0)\geq c_{\min}\Numobs^{\delta-t})\to 1$ for
some $c_{\min}>0$. Thus we have $\prob(\|\sqrt{\NumIndexTwo }
Z_0(\SoluTargetPar - \TrueTargetPar ) \|_2\geq c_{\min}\Numobs^{1/2 +
  \delta-t} \|\SoluTargetPar - \TrueTargetPar \|_2) \to 1$, and we
conclude
\begin{align*}
\liloh_p(\Numobs^{1/2-t} \|\SoluTargetPar  - \TrueTargetPar \|_2)
=
\liloh_p(c_{\min}\Numobs^{1/2 + \delta-t}\|\SoluTargetPar
 - \TrueTargetPar \|_2)= \liloh_p(\sqrt{\NumIndexTwo }\|Z_0(\SoluTargetPar
 - \TrueTargetPar ) \|_2),
\end{align*}
which completes the proof.
\end{proof}


\begin{lems}
\label{lm:glm_consistent_zero_temp}
Under the assumptions in Theorem~\ref{thm:glm_new1_temp}, we have for some
$c_{\min}>0$ that
\begin{align}
  \mathbb \prob( \sigma_{\min}(\EmpMean
  \InvSqrtCovarwhat_i(\iRegressor -
  \GlmMeanVecwhat_i)g'\big(\inprod{\iRegressor}{\EstTargetPar}
  +\inprod{ \iNuisance}{ \EstNuisancePar} \big)(\iRegressor-
  \GlmMeanVecwhat_i)^\top)&\geq c_{\min} \Numobs^{\delta-t})\to
  1, \label{eq:consistent_lm0_1}
\end{align} 
as $\Numobs\to\infty$.
\end{lems}
\begin{proof}
Since the vectors $\iResponse, \iRegressor, \iNuisance$ are
independent of $\EstTargetPar, \EstNuisancePar $ conditioned on
$\History_{\NumIndexOne}$, we can without loss of generality treat
$\EstTargetPar, \EstNuisancePar $ as nonrandom variables. Next note
that $U_i \defn \InvSqrtCovarwhat_i(\iRegressor -
\GlmMeanVecwhat_i)g'\big(\inprod{\iRegressor}{\EstTargetPar} +\inprod{
  \iNuisance}{ \EstNuisancePar} \big)(\iRegressor -
\GlmMeanVecwhat_i)^\top- \E(\InvSqrtCovarwhat_i(\iRegressor -
\GlmMeanVecwhat_i)g'\big(\inprod{\iRegressor}{\EstTargetPar} +\inprod{
  \iNuisance}{ \EstNuisancePar} \big)(\iRegressor -
\GlmMeanVecwhat_i)^\top\mid\iHistory)$ forms a martingale difference
sequence, and moreover, we have
\begin{align*}
\E \frobnorm{U_i}^2 & \leq\E\frobnorm{\InvSqrtCovarwhat_i(\iRegressor
  - \GlmMeanVecwhat_i)g'\big(\inprod{\iRegressor}{\EstTargetPar}
  +\inprod{ \iNuisance}{ \EstNuisancePar} \big)(\iRegressor -
  \GlmMeanVecwhat_i)^\top}^2
    \;\\
    & \leq 
     \; \TargetDim \E\|\InvSqrtCovarwhat_i(\iRegressor -
    \GlmMeanVecwhat_i)\|^2_2
    g'\big(\inprod{\iRegressor}{\EstTargetPar} +\inprod{ \iNuisance}{
      \EstNuisancePar} \big)^2\|(\iRegressor -
    \GlmMeanVecwhat_i)^\top\|_2^2 \\
& \lesssim L_g^2\TargetDim \E\|\InvSqrtCovarwhat_i(\iRegressor -
    \GlmMeanVecwhat_i)\|^2_2=\bigoh(1),
\end{align*}
where the third inequality is due to $|g'|\leq L_g$ and $\|\iRegressor -
\GlmMeanVecwhat_i\|_2\leq 2$, and the last equality uses
equation~\eqref{eq:glm_bounded_oxm}. Thus it follows
from Lemma~\ref{lm:tech_md} that $\frobnorm{\EmpMean U_i} =
\bigoh_p(1/\sqrt{\Numobs})$. Using Weyl's theorem, we have
\begin{align*}
    &\quad\sigma_{\min}(\EmpMean \InvSqrtCovarwhat_i(\iRegressor -
  \GlmMeanVecwhat_i)g'\big(\inprod{\iRegressor}{\EstTargetPar}
  +\inprod{ \iNuisance}{ \EstNuisancePar} \big)(\iRegressor-
  \GlmMeanVecwhat_i)^\top)\\
    &\geq \sigma_{\min}(\CondMean \InvSqrtCovarwhat_i(\iRegressor -
  \GlmMeanVecwhat_i)g'\big(\inprod{\iRegressor}{\EstTargetPar}
  +\inprod{ \iNuisance}{ \EstNuisancePar} \big)(\iRegressor-
  \GlmMeanVecwhat_i)^\top)- \opnorm{\EmpMean U_i}\\
    &=\sigma_{\min}(\CondMean \InvSqrtCovarwhat_i(\iRegressor -
  \GlmMeanVecwhat_i)g'\big(\inprod{\iRegressor}{\EstTargetPar}
  +\inprod{ \iNuisance}{ \EstNuisancePar} \big)(\iRegressor-
  \GlmMeanVecwhat_i)^\top) + \bigoh_p(\Numobs^{-1/2}).
\end{align*}
Since $\delta-t>-1/2$, it remains to show there exists some $\cwtil_{\min}$  such that 
\begin{align}
\label{eq:consistent_lm0_2}  
\prob( \sigma_{\min}(\CondMean \InvSqrtCovarwhat_i(\iRegressor-
\GlmMeanVecwhat_i)
g'\big(\inprod{\iRegressor}{\EstTargetPar} +\inprod{ \iNuisance}{
  \EstNuisancePar} \big)(\iRegressor- \GlmMeanVecwhat_i)^\top) & \geq
\cwtil_{\min} \Numobs^{\delta-t}) \to 1
\end{align}
Recall our notation $\AllNuisancePar=(\AuxiNuisancePar,
\NuisancePar)$, $\TrueAllNuisancePar=(\TrueTargetPar, \NuisancePar)$
and $\EstAllNuisancePar=(\EstTargetPar, \EstNuisancePar )$.  Let
$\DiagProbMatrix \defn \diag\{\SelectProb_{i1}, \ldots,
\SelectProb_{i\TargetDim}\}$,
\begin{align*}
U_i(\AllNuisancePar)
& \defn \E(\InvSqrtCovar_i(\iRegressor-
\GlmMeanVec_i)g'\big(\inprod{\iRegressor}{\AuxiNuisancePar} +
\inprod{\iNuisance}{ \NuisancePar}\big)(\iRegressor-
\GlmMeanVec_i)^\top\mid\iHistory).
\end{align*}
We claim that $\EmpMean U_i(\AllNuisancePar)$ is Lipschitz in
$\AllNuisancePar$ with some constant parameter $L_{U}$ for now, i.e.,
$\opnorm{\EmpMean U_i(\AllNuisancePar^a) - \EmpMean
  U_i(\AllNuisancePar^b)}\leq L_U\|\AllNuisancePar^a -
\AllNuisancePar^b\|_2$ for any $\AllNuisancePar^a,
\AllNuisancePar^b\in\TargetSpace\times\NuisanceSpace$. Then it follows
from Weyl's theorem (see e.g., Theorem 4.3.1 in Horn and
Johnson~\cite{horn2012matrix}) again that
\begin{align*}
& \quad \sigma_{\min}(\CondMean \InvSqrtCovarwhat_i(\iRegressor -
  \GlmMeanVecwhat_i) g'\big(\inprod{\iRegressor}{\EstTargetPar}
  +\inprod{ \iNuisance}{ \EstNuisancePar} \big)(\iRegressor-
  \GlmMeanVecwhat_i)^\top) \\
& = \sigma_{\min}(\EmpMean U_i(\EstAllNuisancePar))\\
& \geq \sigma_{\min}(\EmpMean U_i(\TrueAllNuisancePar)) - \opnorm{\EmpMean
  U_i(\EstAllNuisancePar) - \EmpMean U_i(\TrueAllNuisancePar)} \\
  & \geq \sigma_{\min}(\EmpMean
  U_i(\TrueAllNuisancePar))-L_U\|\EstAllNuisancePar -
  \TrueAllNuisancePar\|_2 \\
& \geq m_{\ScoreFun,2}\Numobs^{\delta-t} + \liloh_p(\Numobs^{-t})
\end{align*}
with probability converging to one. Here in the last line we use the
assumption on the gradient of the score function
(Assumption~\ref{assn-glm-minimum-eig}). Therefore,
equation~\eqref{eq:consistent_lm0_2} holds by choosing
$\cwtil_{\min}=m_{\ScoreFun,2}/2$ and hence concludes the proof.

Now, it remains to prove $\EmpMean U_i(\AllNuisancePar)$ is Lipschitz
in $\AllNuisancePar$. By definition
\begin{align*}
\EmpMean U_i(\AllNuisancePar)
&= \InvSqrtCovar_i\E((\iRegressor- \GlmMeanVec_i)
g'\big(\inprod{\iRegressor}{\AuxiNuisancePar} +\inprod{ \iNuisance}{
  \NuisancePar}\big) (\iRegressor- \GlmMeanVec_i)^\top\mid\iHistory)
\\ &= \InvSqrtCovar_i
\DiagProbMatrix\E(\E(\DiagProbMatrix^{-1}(\iRegressor- \GlmMeanVec_i)
g'\big(\inprod{\iRegressor}{\AuxiNuisancePar} +\inprod{ \iNuisance}{
  \NuisancePar}\big)(\iRegressor- \GlmMeanVec_i)^\top\mid\iNuisance,
\iHistory)\mid\iHistory).
\end{align*}
In Lemma~\ref{lm:glm_lip_cont_temp2} we will show that
$\InvSqrtCovar_i\DiagProbMatrix$ is bounded and Lipschitz in
$\AllNuisancePar$. Thus, it suffices to show
$\E(\DiagProbMatrix^{-1}(\iRegressor- \GlmMeanVec_i)
g'\big(\inprod{\iRegressor}{\AuxiNuisancePar} +\inprod{ \iNuisance}{
  \NuisancePar}\big) (\iRegressor- \GlmMeanVec_i)^\top\mid\iNuisance,
\iHistory)$ is bounded and Lipschitz in $\AllNuisancePar$ since the
multiplication of two bounded Lipschitz functions is bounded and
Lipschitz. In fact, this quantity can be computed
directly. Concretely, we have
\begin{align}
& \quad \E(\DiagProbMatrix^{-1}(\iRegressor-
  \GlmMeanVec_i)g'\big(\inprod{\iRegressor}{\AuxiNuisancePar}
  +\inprod{ \iNuisance}{ \NuisancePar}\big)(\iRegressor-
  \GlmMeanVec_i)^\top\mid\iNuisance, \iHistory)\notag \\
  & =\begin{pmatrix} g'_1 & 0 & 0 & \cdots & 0 \\ 0 & g'_2 & 0 &
  \cdots & 0 \\ 0 & 0 & g'_3 & \cdots & 0 \\ \vdots & \vdots & \vdots
  & \ddots & \vdots \\ 0 & 0 & 0 & \cdots & g'_{\TargetDim}
   \end{pmatrix} - \begin{pmatrix} g'_1 \\ g'_2 \\ \vdots
      \\ g'_{\TargetDim}
  \end{pmatrix}
  \GlmMeanVec_i^\top - \GlmMeanVecbar_i\begin{pmatrix}
  \SelectProb_1g^{'}_1 \\ \SelectProb_2g^{'}_2 \\ \vdots
  \\ \SelectProb_{\TargetDim}g^{'}_{\TargetDim}\end{pmatrix}^\top + (
  \sum_{j=0}^{\TargetDim}\SelectProb_jg'_j) \GlmMeanVecbar_i
  \GlmMeanVec_i^\top, \label{eq:glm_consist_matrix}
\end{align}
where $ \GlmMeanVecbar_i \defn \DiagProbMatrix^{-1} \GlmMeanVec_i$,
$g'_j \defn g'(\AuxiNuisancePar_{j} + \iNuisance^\top\NuisancePar)$
(here we additionally define $\AuxiNuisancePar_{0}=0$) and
$\SelectProb_j \defn \SelectProb_{ij}$ for $j=0, \ldots, \TargetDim
$. It follows immediately from our assumptions on $g'$, definition of
$\GlmMeanVec_i$ and the proof of Lemma~\ref{lm:glm_lip_cont_temp} that the
matrix in equation~\eqref{eq:glm_consist_matrix} is bounded and
Lipschitz in $\AllNuisancePar$. This completes the proof.
\end{proof}


\begin{lems}[Empirical error]\label{lm:glm_emp_bern}
Under the assumptions of Theorem~\ref{thm:glm_new1_temp}, we have
\begin{align*}
    \sup_{\TargetPar \in\TargetSpace} \|(\EmpMean - \CondMean
    )\ScoreFun(\TargetPar, \EstAllNuisancePar)\|_2 =
    \bigoh_p(\frac{\log \Numobs }{\sqrt{\Numobs}}).
\end{align*}
\end{lems}

\begin{proof}
Since $\EstAllNuisancePar\in \History_{\NumIndexOne}\in\iHistory$, it is
independent of $\iResponse, \iRegressor, \iNuisance$ conditioned on
$\History_{\NumIndexOne}$. Therefore, we can view $\EstAllNuisancePar$
as fixed and prove the desired result for all
$\EstAllNuisancePar$. Since $g$ is Lipschitz and $\iRegressor,
\iNuisance, \TargetPar, \NuisancePar$ are all bounded, it follows that
$g$ is also bounded. We denote $\sup_{x\in D_x,(\TargetPar,
  \NuisancePar)\in \TargetSpace\times \NuisanceSpace}|g|$ by $M_g$.

Define $\iMartdiff(\TargetPar) \defn \ScoreFun(\TargetPar,
\EstAllNuisancePar) - \E (\ScoreFun(\TargetPar,
\EstAllNuisancePar)\mid\iHistory)$ and decompose
$\iMartdiff(\TargetPar)$ into $\iMartdiffa + \iMartdiffb(\TargetPar)$,
where
\begin{align*}
\iMartdiffa
& \defn \InvSqrtCovarwhat_i(\iRegressor - \GlmMeanVecwhat_i)\iNoise,
\\ \iMartdiffb(\TargetPar)& \defn \InvSqrtCovarwhat_i(\iRegressor -
\GlmMeanVecwhat_i)[g\big(\inprod{\iRegressor}{\TrueTargetPar}
  +\inprod{ \iNuisance}{
    \TrueNuisancePar}\big)-g\big(\inprod{\iRegressor}{\TargetPar}
  +\inprod{ \iNuisance}{ \EstNuisancePar} \big)], \\& -
\E(\InvSqrtCovarwhat_i(\iRegressor -
\GlmMeanVecwhat_i)[g\big(\inprod{\iRegressor}{\TrueTargetPar}
  +\inprod{ \iNuisance}{
    \TrueNuisancePar}\big)-g\big(\inprod{\iRegressor}{\TargetPar}
  +\inprod{ \iNuisance}{ \EstNuisancePar} \big)]\mid\iHistory).
\end{align*} 
It suffices to show $\|\EmpMean \iMartdiffa\|_2=\bigoh_p(\log \Numobs
/\sqrt{\Numobs})$ and $\sup_{\TargetPar \in\TargetSpace}\|\EmpMean
\iMartdiffb (\TargetPar)\|_2=\bigoh_p(\log \Numobs /\sqrt{\Numobs})$.
Note that $\iMartdiff(\TargetPar), \iMartdiffa,
\iMartdiffb(\TargetPar)$ are all martingale difference sequences for
any $\TargetPar \in \TargetSpace$.  Moreover,
\begin{align*}
    \E (\|\iMartdiffa\|_2^2\mid\iHistory)
    &=\E (\|\InvSqrtCovarwhat_i(\iRegressor -
    \GlmMeanVecwhat_i)\iNoise\|_2^2\mid\iHistory)\\
      &= \E (\|\InvSqrtCovarwhat_i(\iRegressor -
    \GlmMeanVecwhat_i)\|_2^2\GlmVar\big(g\big(\inprod{\iRegressor}{\TrueTargetPar}
    +\inprod{ \iNuisance}{ \TrueNuisancePar}\big)\big)\mid\iHistory)\\
      & \leq \E (\|\InvSqrtCovarwhat_i(\iRegressor -
    \GlmMeanVecwhat_i)\|_2^2\mid\iHistory) M_\eps=\bigoh(1),
\end{align*}
where the second line follows from calculation of the expectation
conditional on $\iRegressor, \iNuisance$, and the last equality uses
equation~\eqref{eq:glm_bounded_oxm}. Thus, it follows immediately from
Lemma~\ref{lm:tech_md} that $\|\EmpMean
\iMartdiffa\|_2=\bigoh_p(1/\sqrt{\Numobs})=\bigoh_p(\log \Numobs
/\sqrt{\Numobs})$.

Similarly, we have
\begin{align*}
       \E (\|\iMartdiffb(\TargetPar)\|_2^2\mid\iHistory)
    &\leq\E (\|\InvSqrtCovarwhat_i(\iRegressor -
       \GlmMeanVecwhat_i)[g\big(\inprod{\iRegressor}{\TrueTargetPar}
         +\inprod{ \iNuisance}{
           \TrueNuisancePar}\big)-g\big(\inprod{\iRegressor}{\TargetPar}
         +\inprod{ \iNuisance}{
           \EstNuisancePar}\big)]\|_2^2\mid\iHistory)\\
      &\lesssim \E (\|\InvSqrtCovarwhat_i(\iRegressor -
       \GlmMeanVecwhat_i)\|_2^2\mid\iHistory)M_g^2=\bigoh(1),
\end{align*}
and
 \begin{align*}
     \|\iMartdiffb(\TargetPar)\|_2 &\leq
     \opnorm{\InvSqrtCovarwhat_i}\|(\iRegressor -
     \GlmMeanVecwhat_i)\|_2
     |g\big(\inprod{\iRegressor}{\TrueTargetPar} +\inprod{
       \iNuisance}{
       \TrueNuisancePar}\big)-g\big(\inprod{\iRegressor}{\TargetPar}
     +\inprod{ \iNuisance}{ \EstNuisancePar}\big)|\\
     & + \E(\opnorm{\InvSqrtCovarwhat_i}\|(\iRegressor -
      \GlmMeanVecwhat_i)\|_2\big|g\big(\inprod{\iRegressor}{\TrueTargetPar}
      +\inprod{ \iNuisance}{
        \TrueNuisancePar}\big)-g\big(\inprod{\iRegressor}{\TargetPar}
      +\inprod{ \iNuisance}{
        \EstNuisancePar}\big)\big|\mid\iHistory)\\
     &\lesssim M_g i^{t - \delta},
 \end{align*}
 where in the last line we used $|g|\leq M_g$, $\|\iRegressor -
 \GlmMeanVecwhat_i\|_2\leq 2$ and Lemma~\ref{lm:equiv_glm}.  Since $
 \|\iMartdiffb(\TargetPar)\|_2$ is bounded by $i^{t - \delta}$ and
 have variance bounded by some constant, there exist some constants
 $\BernVarb, \BernBoundb$ such that $\iMartdiffb^j(\TargetPar)$ is a
 Bernstein type random variable with parameter $(\BernVarb,
 \BernBoundb \Numobs^{t - \delta})$ for each entry
 $\iMartdiffb^j(j=1,2, \ldots, \TargetDim )$. Therefore, it follows
 from (for example Proposition 2.10 in
 Wainwright~\cite{wainwright2019high}) that
 $\E(e^{\lambda\iMartdiffb^j(\TargetPar)}\mid\iHistory)\leq
 e^{\frac{\lambda^2\BernVarb/2}{1 - \BernBoundb \Numobs^{t -
       \delta}|\lambda|}}$ for all $|\lambda|< 1/(\BernBoundb
 \Numobs^{t - \delta})$. This implies
 \begin{align*}
 \E e^{\lambda\EmpMean \iMartdiffb^j(\TargetPar)}
 \leq \E \prod_{i=\NumIndexOne + 1}^{\Numobs}\E
 (e^{\lambda\iMartdiffb^j(\TargetPar)/\NumIndexTwo }\mid\iHistory)
 \leq e^{\frac{\lambda^2\BernVarb/2}{\NumIndexTwo (1 - \Bernbwtil
     \Numobs^{t - \delta-1}|\lambda|)}}
 \end{align*}
for all $|\lambda|< \NumIndexTwo ^{1 + \delta-t}/\Bernbwtil$, where
$\Bernbwtil$ is some constant depending on $\BernBoundb$ and the ratio
between $\NumIndexTwo $ and $\Numobs$.

Let $\CoverSet(\eps)$ be a $\eps$-covering of $\TargetSpace$ in
$\|\cdot\|_2$. From standard results, we can find such a set with
$|\CoverSet (\eps)|\lesssim \frac{1}{\eps^{\TargetDim}}$. Choosing
$\eps=1/\NumIndexTwo $, we get $|\CoverSet (1/\NumIndexTwo )|\lesssim
\NumIndexTwo ^{\TargetDim}$.  For any $\TargetPar \in\TargetSpace$,
let $\pi(\TargetPar)$ denote a point in $\CoverSet(1/\NumIndexTwo )$
such that $\|\TargetPar - \pi(\TargetPar)\|_2< 1/\NumIndexTwo $. Using
a discretization argument, we get
  \begin{align}
      \sup_{\TargetPar \in \TargetSpace}|\EmpMean
      \iMartdiffb^j(\TargetPar)|&\leq \sup_{\TargetPar
        \in \TargetSpace}[|\EmpMean \iMartdiffb^j(\pi(\TargetPar))| +
        |\EmpMean [\iMartdiffb^j(\pi(\TargetPar)) -
          \iMartdiffb^j(\TargetPar)] | ] \notag \\
& \leq \sup_{\TargetPar \in \CoverSet(1/\NumIndexTwo )}|\EmpMean
      \iMartdiffb^j(\TargetPar)| + \sup_{\|\TargetPar^a -
        \TargetPar^b\|_2\leq 1/\NumIndexTwo }| \EmpMean
                 [\iMartdiffb^j( \TargetPar^a) -
                   \iMartdiffb^j(\TargetPar^b)]| \label{eq:one_step_discretization}
  \end{align}
For the first term in equation~\eqref{eq:one_step_discretization}, we
have
 \begin{align*}
      \E \sup_{\TargetPar \in \CoverSet(1/\NumIndexTwo )}|\EmpMean
      \iMartdiffb^j(\TargetPar)|
      & \leq \log (\E e^{\lambda\sup_{\TargetPar \in
          \CoverSet(1/\NumIndexTwo )}|\EmpMean
        \iMartdiffb^j(\TargetPar)|})/\lambda\\ & \leq \log (\E
      e^{\lambda\sup_{\TargetPar \in \CoverSet(1/\NumIndexTwo
          )}\EmpMean \iMartdiffb^j(\TargetPar)} + e^{ -
        \lambda\sup_{\TargetPar \in \CoverSet(1/\NumIndexTwo
          )}\EmpMean \iMartdiffb^j(\TargetPar)})/\lambda\\
        &\leq \log \Big(2|\CoverSet(1/\NumIndexTwo
      )|e^{\frac{\lambda^2\BernVarb/2}{\NumIndexTwo (1 - \Bernbwtil
          \Numobs^{t - \delta-1}|\lambda|)}} \Big) /\lambda \\
       &\lesssim \TargetDim \log \NumIndexTwo /\lambda +
          {\frac{\lambda\BernVarb/2}{\NumIndexTwo (1 - \Bernbwtil
              \Numobs^{t - \delta-1}|\lambda|)}}
 \end{align*}
 for all $|\lambda|< \NumIndexTwo ^{1 +
   \delta-t}/\Bernbwtil$. Choosing $\lambda =\sqrt{\NumIndexTwo }
 <\NumIndexTwo ^{1 + \delta-t}/\Bernbwtil$ yields
  \begin{align*}
     \E \sup_{\TargetPar \in \CoverSet(1/\NumIndexTwo )}|\EmpMean
     \iMartdiffb^j(\TargetPar)|
     & \lesssim \TargetDim \frac{\log \NumIndexTwo
     }{\sqrt{\NumIndexTwo }} + \frac{1}{\sqrt{\NumIndexTwo }}\lesssim
     \TargetDim \frac{\log \NumIndexTwo }{\sqrt{\NumIndexTwo }}.
  \end{align*}
Thus, we have shown that $\sup_{\TargetPar \in
  \CoverSet(1/\NumIndexTwo )}|\EmpMean
\iMartdiffb^j(\TargetPar)|=\bigoh_p(\log \Numobs
/\sqrt{\Numobs})$. For the discretization error (the second term in
equation~\eqref{eq:one_step_discretization}), using the definition of
$\iMartdiffb$ we obtain
 \begin{align*}
 &\quad\big|\EmpMean [\iMartdiffb^j(\TargetPar^a) -
   \iMartdiffb^j(\TargetPar^b)]\big|
   \\&= \Big|(\EmpMean - \CondMean
 )\InvSqrtCovarwhat_{ij\cdot}(\iRegressor -
 \GlmMeanVecwhat_i)\big[g\big(\inprod{\iRegressor}{\TargetPar^b}
   + \inprod{\iNuisance}{\EstNuisancePar
   }\big)-g\big(\inprod{\iRegressor}{\TargetPar^a} +\inprod{
     \iNuisance}{\EstNuisancePar} \big)\big]\Big|\\
     &\leq (\EmpMean + \CondMean
     )\opnorm{\InvSqrtCovarwhat_{i}}\|(\iRegressor -
     \GlmMeanVecwhat_i)\|_2\big|g\big(\inprod{\iRegressor}{\TargetPar^b}
     + \inprod{\iNuisance}{\EstNuisancePar
     }\big)-g\big(\inprod{\iRegressor}{\TargetPar^a} +\inprod{
       \iNuisance}{\EstNuisancePar} \big)\big|\\ &\leq (\EmpMean +
     \CondMean )\opnorm{\InvSqrtCovarwhat_{i}}\|(\iRegressor -
     \GlmMeanVecwhat_i)\|_2L_gD_x\|\TargetPar^a -
     \TargetPar^b\|_2\\ &\lesssim (\EmpMean + \CondMean ) i^{t -
       \delta}\|\TargetPar^a - \TargetPar^b\|_2\leq \Numobs^{t -
       \delta}\|\TargetPar^a - \TargetPar^b\|_2,
 \end{align*}
 where the fourth line uses the Lipschitz continuity of $g$, the fact
 that $\|\iRegressor- \GlmMeanVecwhat_i\|_2\leq 2$ and
 $\|\InvSqrtCovarwhat_{i}\|_2 \lesssim i^{t - \delta}$. Thus, we have
 the bound
\begin{align*} 
 \sup_{\|\TargetPar^a - \TargetPar^b\|_2\leq 1/\NumIndexTwo}
 \big|\EmpMean [\iMartdiffb^j(\TargetPar^a) -
   \iMartdiffb^j(\TargetPar^b)]\big|\lesssim \Numobs^{t -
   \delta}\|\TargetPar^a - \TargetPar^b\|_2\leq \Numobs^{t - \delta}
 \NumIndexTwo^{-1}=\bigoh(\Numobs^{t - \delta-1})=o(\Numobs^{-1/2}).
\end{align*}
Putting together the pieces, we find that $\sup_{\TargetPar
  \in \TargetSpace}|\EmpMean \iMartdiffb^j(\TargetPar)|=\bigoh_p(\log
\Numobs /\sqrt{\Numobs})$, and thus $ \sup_{\TargetPar
  \in \TargetSpace}|\EmpMean \iMartdiffb(\TargetPar)|=\bigoh_p(\log
\Numobs /\sqrt{\Numobs})$.

\end{proof}


\begin{lems}[Lipschitz continuity of $\CondMean \ScoreFun, \CondMean \partial_{\AllNuisancePar}\ScoreFun$]
\label{lm:glm_lip_cont_temp}
Under the assumptions given in Theorem~\ref{thm:glm_new1_temp}, $\CondMean
\ScoreFun_i(\TargetPar, \AllNuisancePar)$ and $\CondMean
\partial_{\AllNuisancePar}\ScoreFun_i(\TargetPar, \AllNuisancePar)$
are Lipschitz in $(\TargetPar, \AllNuisancePar)$ with parameters
$L_{\ScoreFun,1},L_{\ScoreFun,2}>0$ that depend only on the constants
from Theorem~\ref{thm:glm_new1_temp}.
\end{lems}
\begin{proof}
By definition of $\CondMean $, it suffices to show
$\E(\ScoreFun_i(\TargetPar, \AllNuisancePar)\mid\iHistory)$ and
$\E(\partial_\AllNuisancePar\ScoreFun_i(\TargetPar,
\AllNuisancePar)\mid\iHistory)$ are uniformly Lipschitz across all
$i$.

\subsubsection*{Lipschitz continuity of $\E(\ScoreFun_i(\TargetPar, \AllNuisancePar) \mid \iHistory)$}

Plugging the definition of $\ScoreFun_i$ into
$\E(\ScoreFun_i(\TargetPar, \AllNuisancePar)\mid\iHistory)$, we
obtain,
\begin{align*}
  \E(\ScoreFun_i(\TargetPar, \AllNuisancePar)\mid\iHistory)&=
  \E(\InvSqrtCovar_i(\iRegressor -
  \GlmMeanVec_i)\big(\iResponse-g\big(\inprod{\iRegressor}{\TargetPar}
  + \inprod{\iNuisance}{\NuisancePar}\big)\big)\mid\iHistory),
\end{align*}
where
\begin{align*} \GlmMeanVec_i&=\GlmMeanVec_i(\iNuisance, \iHistory)
 \equiv \E(\iRegressor g'\big(\inprod{\iRegressor}{\AuxiNuisancePar}
  + \inprod{\iNuisance}{\NuisancePar}\big)\mid\iNuisance,
  \iHistory)[\E( g'\big(\inprod{\iRegressor}{\AuxiNuisancePar} +
    \inprod{\iNuisance}{\NuisancePar}\big)\mid\iNuisance,
    \iHistory)]^{-1}, \\
\InvSqrtCovar_i &= \InvSqrtCovar_i(\iNuisance, \iHistory)  \equiv {}
             [\E(\iNoise^2(\iRegressor - \GlmMeanVec_i(\iNuisance,
               \iHistory))(\iRegressor - \GlmMeanVec_i(\iNuisance,
               \iHistory))^\top\mid\iNuisance, \iHistory)]^{-1/2} \\
& = [\E(\GlmVar \big(g\big(\inprod{\iRegressor}{\AuxiNuisancePar} +
               \inprod{\iNuisance}{\NuisancePar}\big)\big)(\iRegressor
               - \GlmMeanVec_i(\iNuisance, \iHistory))(\iRegressor -
               \GlmMeanVec_i(\iNuisance,
               \iHistory))^\top\mid\iNuisance, \iHistory)]^{-1/2}.
\end{align*}
We remark here that $\GlmMeanVec_i, \InvSqrtCovar_i$ both depend on
$\AllNuisancePar=(\AuxiNuisancePar, \NuisancePar)$.  Due to the fact
that the expectation of $L$-Lipschitz functions is still
$L$-Lipschitz, it remains to show $\E(\ScoreFun_i(\TargetPar,
\AllNuisancePar)\mid\iNuisance, \iHistory)$ is Lipschitz in
$(\TargetPar, \AllNuisancePar)$ with parameter independent of $i$ and
$\iNuisance$. From now on in this proof, we use Lipschitz in
$(\TargetPar, \AllNuisancePar)$ to refer to Lipschitz in $(\TargetPar,
\AllNuisancePar)$ with parameter which does not depend on $i$.
Equivalently, it remains to show
\begin{align*}
& \E_{\iRegressor, \iNoise}\InvSqrtCovar_i(\iRegressor -
  \GlmMeanVec_i(\iNuisance, \iHistory)) \big(\iResponse-
  g\big(\inprod{\iRegressor}{\TargetPar} +
  \inprod{\iNuisance}{\NuisancePar}\big)\big)\\
& = \InvSqrtCovar_i\E_{\iRegressor}(\iRegressor -
  \GlmMeanVec_i(\iNuisance, \iHistory))
  \big(g\big(\inprod{\iRegressor}{\TrueTargetPar} +
  \inprod{\iNuisance}{\TrueNuisancePar}\big) -
  g\big(\inprod{\iRegressor}{\TargetPar} +
  \inprod{\iNuisance}{\NuisancePar}\big)\big)
\end{align*}
is Lipschitz. Here we abuse the notation $\E_{\iRegressor, \iNoise},
\E_{\iRegressor}$ to denote the expectation conditioned on
$\iNuisance, \iHistory$.  Adopt the shorthand notation $\SelectProb_j,
\GlmMeanVec_j, \epsbar_j$ for $\SelectProb_{ij}(\iNuisance,
\iHistory)$, $\GlmMeanVec_{ij}(\iNuisance, \iHistory)$,
$\GlmVar(g\big(\AuxiNuisancePar_{j} +
\inprod{\iNuisance}{\NuisancePar}\big)$ $j=0,1, \ldots, \TargetDim $
respectively (we additionally define $\AuxiNuisancePar_{0} \defn
0$). Since the conditional expectation is over $\iRegressor$ and
$\iNoise$, it follows that $\SelectProb_j, \GlmMeanVec_j, \epsbar_j$
can be viewed as fixed quantities conditioned on $\iNuisance,
\iHistory$. Also, $\SelectProb_j$ does not depend on the parameters
$(\TargetPar, \AllNuisancePar)$ while $\GlmMeanVec_j, \epsbar_j$ are
functions of $\AllNuisancePar$. Define
$\DiagProbMatrix=\diag\{\SelectProb_{1}, \SelectProb_{2}, \ldots,
\SelectProb_{\TargetDim}\}$.  By some algebraic calculations, we
obtain that $\GlmMeanVec_i(\iNuisance, \iHistory)$ is a vector with
the $j$--th entry equals
\begin{align}
\label{eq:m_j_value}  
\SelectProb_{j} g'\big(\AuxiNuisancePar_{j} +
\inprod{\iNuisance}{\NuisancePar}\big)/\sum_{k=0}^{\TargetDim}
\SelectProb_{k} g'\big(\AuxiNuisancePar_{k} +
\inprod{\iNuisance}{\NuisancePar}\big).
\end{align}
Define $\GlmMeanVecbar_i\equiv \DiagProbMatrix^{-1}\GlmMeanVec_{i}$ be
the normalized version of $\GlmMeanVec_i$.  Since we have assumed
$L_{g}\geq|g'|\geq l_g>0$, $g'$ is $L_{g'}$ Lipschitz and
$\|(\iRegressor^\top, \iNuisance^\top)\|_2\leq D_x$, it follows that
$g'\big(\AuxiNuisancePar_{j} + \inprod{\iNuisance}{\NuisancePar}\big)$, 
$\sum_{k=0}^{\TargetDim} \SelectProb_{k} g'\big(\AuxiNuisancePar_{k} +
\inprod{\iNuisance}{\NuisancePar}\big)$ are both Lipschitz and the
second term is also bounded between $l_g$ and $L_g$. Therefore, it
follows that both $\GlmMeanVec_j$ and $\GlmMeanVecbar_j$ are bounded
and Lipschitz in $(\TargetPar, \AllNuisancePar)$.

Moreover, it can be verified that the $j$-th entry of
$\E_{\iRegressor}(\iRegressor - \GlmMeanVec_i(\iNuisance, \iHistory))
\big[g\big(\inprod{\iRegressor}{\TrueTargetPar} + \inprod{
    \iNuisance}{\TrueNuisancePar}
  \big)-g\big(\inprod{\iRegressor}{\TargetPar} + \inprod{
    \iNuisance}{\TrueNuisancePar} \big)\big]$ equals
\begin{align*}
 & \SelectProb_j \big(g\big(\TrueTargetPar_{j} +
  \inprod{\iNuisance}{\TrueNuisancePar}\big) - g\big(\TargetPar_{j} +
  \inprod{\iNuisance}{\NuisancePar}\big)\big) - \GlmMeanVec_j
  \E_{\iRegressor} (g\big(\inprod{\iRegressor}{\TrueTargetPar} +
  \inprod{\iNuisance}{\TrueNuisancePar}\big)-g\big(\inprod{\iRegressor}{\TargetPar}
  + \inprod{\iNuisance}{\NuisancePar}\big)) \\
  & = \SelectProb_j\Big [g\big(\TrueTargetPar_{j} +
       \inprod{\iNuisance}{\TrueNuisancePar}
       \big)-g\big(\TargetPar_{j} + \inprod{
         \iNuisance}{\NuisancePar}\big) -
       \GlmMeanVecbar_j\E_{\iRegressor}[g\big(\inprod{\iRegressor}{\TrueTargetPar}
         +
         \inprod{\iNuisance}{\TrueNuisancePar}\big)-g\big(\inprod{\iRegressor}{\TargetPar}
         + \inprod{\iNuisance}{\NuisancePar}\big)]\Big].
\end{align*}
Since $g, \GlmMeanVecbar_j$ are both bounded (the boundedness of $g$
follows from the Lipschitz continuity of $g$ and boundedness of
$\TargetSpace\times\NuisanceSpace$, $(\iRegressor, \iNuisance)$) and
Lipschitz in $(\TargetPar, \AllNuisancePar)$, it follows directly that
the quantity inside the bracket in the second line is bounded and
Lipschitz in $(\TargetPar, \AllNuisancePar)$.  Therefore,
$\DiagProbMatrix^{-1}\E_{\iRegressor}(\iRegressor -
\GlmMeanVec_i(\iNuisance, \iHistory))\big [ g
  \big(\inprod{\iRegressor}{\TrueTargetPar} +
  \inprod{\iNuisance}{\TrueNuisancePar}\big) -
  g\big(\inprod{\iRegressor}{\TargetPar} +
  \inprod{\iNuisance}{\NuisancePar}\big)\big]$ is bounded and
Lipschitz in $(\TargetPar,
\AllNuisancePar)$. Since Lemma~\ref{lm:glm_lip_cont_temp2} shows
$\InvSqrtCovar_i \DiagProbMatrix$ is bounded and Lipschitz in
$(\TargetPar, \AllNuisancePar)$, the desired result follows as the
multiplication of two bounded Lipschitz functions is bounded and
Lipschitz.


\subsubsection*{Lipschitz continuity of $\E(\partial_{\AllNuisancePar}\ScoreFun_i(\TargetPar, \AllNuisancePar)\mid\iHistory)$}

Define
\begin{align*}
\Term_1 & \defn \DiagProbMatrix^{-1} \E_{\iRegressor}(\iRegressor -
\GlmMeanVec_i)\big(g\big(\inprod{\iRegressor}{\TrueTargetPar } +
\inprod{\iNuisance}{\TrueNuisancePar} \big) - g
\big(\inprod{\iRegressor}{\TargetPar } +
\inprod{\iNuisance}{\NuisancePar} \big)\big) \\
\Term_2 & \defn \DiagProbMatrix^{-1}\partial_\AllNuisancePar
\GlmMeanVec_i \E_{\iRegressor} \big( g
\big(\inprod{\iRegressor}{\TrueTargetPar } +
\inprod{\iNuisance}{\TrueNuisancePar} \big) - g
\big(\inprod{\iRegressor}{\TargetPar } +
\inprod{\iNuisance}{\NuisancePar} \big)\big) \\
\Term_3 & \defn \DiagProbMatrix^{-1}\E_{\iRegressor}(\iRegressor -
\GlmMeanVec_i) g'(\inprod{\iRegressor}{\TargetPar } +
\inprod{\iNuisance}{\NuisancePar})(0_{\TargetDim},
\iNuisance^\top)\partial_{\AllNuisancePar}\NuisancePar.
\end{align*}
Substituting the expression of the partial derivative into
$\E(\partial_{\AllNuisancePar}\ScoreFun_i( \TargetPar,
\AllNuisancePar) \mid \iNuisance, \iHistory)$, we obtain
\begin{align}
&\qquad \E(\partial_\AllNuisancePar\ScoreFun_i(\TargetPar,
  \AllNuisancePar) \mid \iNuisance , \iHistory)\notag\\
  & = \E(\partial_\AllNuisancePar\InvSqrtCovar_i(\iRegressor -
  \GlmMeanVec_i)\big(\iResponse - g \big(
  \inprod{\iRegressor}{\TargetPar} + \inprod{\iNuisance}{\NuisancePar}
  \big) \big)\mid\iNuisance , \iHistory)\notag\\
    & \quad - \E(\InvSqrtCovar_i\partial_\AllNuisancePar \GlmMeanVec_i
  \big( \iResponse - g\big(\inprod{\iRegressor}{\TargetPar} +
  \inprod{\iNuisance}{\NuisancePar}\big)\big)\mid\iNuisance ,
  \iHistory) \notag\\
& \quad- \E(\InvSqrtCovar_i(\iRegressor -
  \GlmMeanVec_i)g'\big(\inprod{\iRegressor}{\TargetPar} +
  \inprod{\iNuisance}{\NuisancePar}\big)(\iNuisance^\top,0_{\TargetDim})\partial_{\AllNuisancePar}\NuisancePar\mid\iNuisance
  , \iHistory) \notag\\
& = \partial_\AllNuisancePar\InvSqrtCovar_i\DiagProbMatrix\Term_1 -
  \InvSqrtCovar_i\DiagProbMatrix\Term_2 -
  \InvSqrtCovar_i\DiagProbMatrix\Term_3.    \label{eq:lip_decomp_review}
\end{align}
Since Lemma~\ref{lm:glm_lip_cont_temp2} shows that
$\InvSqrtCovar_i\DiagProbMatrix,
\partial_\AllNuisancePar\InvSqrtCovar_i\DiagProbMatrix$ are bounded
and Lipschitz in $(\TargetPar, \AllNuisancePar)$, it remains to show
that $\Term_1, \Term_2, \Term_3$ are all bounded and Lipschitz in
$(\TargetPar, \AllNuisancePar)$. For $\Term_1, \Term_3$, after some
basic algebraic calculations we obtain the $j$-th entry of each term
\begin{align*}
   \Term_{1j}
   &= g\big(\TrueTargetPar_{j} + \inprod{\iNuisance}{\TrueNuisancePar}
   \big)-g\big(\TargetPar_{j} + \inprod{\iNuisance}{\NuisancePar}\big)
   - \GlmMeanVecbar_{ij}\E_{\iRegressor}(g\big(\TrueTargetPar_{j} +
   \inprod{\iNuisance}{\TrueNuisancePar} \big)-g\big(\TargetPar_{j} +
   \inprod{\iNuisance}{\NuisancePar}\big)), \\ \Term_{3j}
   &= \big[g'\big(\TargetPar_{j} +
     \inprod{\iNuisance}{\NuisancePar}\big) - \GlmMeanVecbar_{ij}
     \E_{\iRegressor}g'\big(\inprod{\iRegressor}{\TargetPar} +
     \inprod{\iNuisance}{\NuisancePar})(0_{\TargetDim},
     \iNuisance^\top)\big]\partial_{\AllNuisancePar}\NuisancePar.
\end{align*}
Since the functions $g'$ and $\GlmMeanVecbar_i$ are bounded and
Lipschitz, $\iNuisance$ is bounded and
$\partial_\AllNuisancePar\NuisancePar=({\bf 0}_{\NuisanceDim \times
  \TargetDim}, \IdMat_{\NuisanceDim })^\top$, it follows directly that
$\Term_{1j}, \Term_{3j}$ are bounded and Lipschitz in $(\TargetPar,
\AllNuisancePar)$. For $\Term_2$, we also consider the $j$-th entry
$\Term_{2j}$. Use shorthand $g'_k,g^{''}_k$ for
$g'\big(\AuxiNuisancePar_{k} +
\inprod{\iNuisance}{\NuisancePar}\big),g^{''}\big(\AuxiNuisancePar_{k}
+ \inprod{\iNuisance}{\NuisancePar}\big)$ respectively. We have from
equation~\eqref{eq:m_j_value} and some derivative calculations that
the $j$-th entry of
$\DiagProbMatrix^{-1}\partial_{\NuisancePar}\GlmMeanVec_i$
\begin{align*}
\partial_{\NuisancePar}\GlmMeanVec_{ij}/\SelectProb_j&=\partial_{\NuisancePar}\GlmMeanVecbar_{ij}
= [(\sum_{k=0}^{\TargetDim}\SelectProb_kg'_k)g^{''}_j-g'_j(\sum_{k=0}^{\TargetDim}\SelectProb_kg^{''}_k)]\iNuisance^\top/(\sum_{k=0}^{\TargetDim}\SelectProb_kg'_k)^2.
\end{align*}
Since $g_k' \geq l_g>0$ for all $k$, it follows that $(\sum_{k=0}^{\TargetDim}\SelectProb_kg'_k)^2\geq l_g^2$. Combining this with the assumption that $g'_k$ is Lipschitz, we have  $1/(\sum_{k=0}^{\TargetDim}\SelectProb_kg'_k)^2$ is bounded and Lipschitz. Moreover, since $g_k^{''}$ is bounded and Lipschitz and $\iNuisance$ is bounded by our assumption, it follows that $\partial_{\NuisancePar}\GlmMeanVecbar_i$ is bounded and Lipschitz in $(\TargetPar, \AllNuisancePar)$. Since $\E_{\iRegressor}\big(g\big(\inprod{\iRegressor}{\TrueTargetPar}  + \inprod{\iNuisance}{\TrueNuisancePar}\big)-g\big(\inprod{\iRegressor}{\TargetPar} + \inprod{\iNuisance}{\NuisancePar}
\big)\big)$ is also bounded and Lipschitz due to the boundedness and Lipschitz continuity of g, it follows that $\Term_2$ is bounded and Lipschitz in $(\TargetPar, \AllNuisancePar)$. The proof is hence completed.

\end{proof}

\begin{lems}[Lipschitz continuity of $\InvSqrtCovar_i\DiagProbMatrix, \partial_\AllNuisancePar\InvSqrtCovar_i\DiagProbMatrix$]\label{lm:glm_lip_cont_temp2} Under the assumption in Theorem~\ref{thm:glm_new1_temp} and notations in Lemma~\ref{lm:glm_lip_cont_temp}, we have $\InvSqrtCovar_i\DiagProbMatrix$ and $\partial_\AllNuisancePar\InvSqrtCovar_i\DiagProbMatrix$ are both bounded and Lipschitz continuous in $(\TargetPar, \AllNuisancePar)$.
\end{lems}

\begin{proof}
The Lipschitz continuity w.r.t. $\TargetPar$ is obvious, since $\InvSqrtCovar_i$ only depends on $\AllNuisancePar=(\AuxiNuisancePar, \NuisancePar)$. It remains to show Lipschitz continuity in $\AllNuisancePar$. Likewise, we say a function is Lipschitz in $\AllNuisancePar$ if the Lipschitz parameter is some constant depending only on the constants defined in Theorem~\ref{thm:glm_new1_temp} but not depending on $i$.
Define
\begin{align*}
    \Covar_i
& \defn 
\E(\iNoise^2(\iRegressor - \GlmMeanVec_i)(\iRegressor - \GlmMeanVec_i)^\top\mid\iNuisance, \iHistory)\\
&= 
\E(\GlmVar \big(g\big(\inprod{\iRegressor}{\AuxiNuisancePar} + \inprod{\iNuisance}{\NuisancePar}\big)\big)(\iRegressor - \GlmMeanVec_i)(\iRegressor - \GlmMeanVec_i)^\top\mid\iNuisance, \iHistory).
\end{align*}
Again, we remark that $\Covar_i$ is implicitly depending on $\AllNuisancePar$. Since $\GlmVar (g(\iRegressor^\top\AuxiNuisancePar + \iNuisance^\top\NuisancePar))\leq M_\eps$, $\|\iRegressor\|_2\leq 1$, $\|\GlmMeanVec_i\|_2\leq 1$, it follows that \begin{align*}
\opnorm{\Covar_i}\leq \E(\GlmVar \big(g\big(\inprod{\iRegressor}{\AuxiNuisancePar} + \inprod{\iNuisance}{\NuisancePar}\big)\big)\cdot\|\iRegressor - \GlmMeanVec_i\|_2^2\mid\iNuisance, \iHistory)\leq 4M_\eps.\end{align*}
By some algebraic calculations, we obtain
\begin{align*}
   \Covar_i=\begin{pmatrix}\SelectProb_1 {\epsbar}_1&0&0&\cdots&0\\
   0&\SelectProb_2\epsbar_2&0&\cdots&0\\
   0&0&\SelectProb_3\epsbar_3&\cdots&0\\
  \vdots&\vdots&\vdots&\ddots&\vdots\\
   0&0&0&\cdots&\SelectProb_{\TargetDim}\epsbar_{\TargetDim}\\\end{pmatrix}  - \begin{pmatrix}
   \SelectProb_1\epsbar_1\\
   \SelectProb_2\epsbar_2\\\vdots\\\SelectProb_{\TargetDim}\epsbar_{\TargetDim}\end{pmatrix}\GlmMeanVec_i^\top - \GlmMeanVec_i\begin{pmatrix}
   \SelectProb_1\epsbar_1\\
   \SelectProb_2\epsbar_2\\\vdots\\\SelectProb_{\TargetDim}\epsbar_{\TargetDim}\end{pmatrix}^\top + 
  (\sum_{j=0}^{\TargetDim}\SelectProb_j\epsbar_j) \GlmMeanVec_i\GlmMeanVec_i^\top .
\end{align*}
Moreover, calculating the inverse of $\Covar_i$ using Woodbury's identity, we obtain
\begin{align}
\InvSqrtCovar_i^2=\Covar_i^{-1}&=\InvCompOne_i + \InvCompDiag_i\InvCompRow_i\frac{\begin{pmatrix} - \SelectProb_0
    \epsbar_0 & \GlmMeanVecbar_0 \SelectProb_0\\ 
    \GlmMeanVecbar_0 \SelectProb_0 & \sum_{k=1}^{\TargetDim}
    \SelectProb_k 
    \GlmMeanVecbar_k^2/\epsbar_k\end{pmatrix}}{(\sum_{k=1}^{\TargetDim
  }\SelectProb_k\GlmMeanVecbar_k^2/
  \epsbar_k)\SelectProb_0\epsbar_0 + \GlmMeanVecbar_0^2
  \SelectProb_0^2}\InvCompRow_i^\top \InvCompDiag_i\label{eq:glm_sigma_inverse}\\
& =:\InvCompOne_i(\AllNuisancePar) + \InvCompTwo_i(\AllNuisancePar)\notag
\end{align}
where $\InvCompDiag_i=\diag\{1/\epsbar_1, \ldots,1/\epsbar_{\TargetDim}\}$,
$\InvCompOne_i \defn \DiagProbMatrix^{-1}\InvCompDiag_i=\diag\{1/(\SelectProb_1\epsbar_1), \ldots,1/(\SelectProb_{\TargetDim
}\epsbar_{\TargetDim})\}$, and $\InvCompRow_i \defn \begin{pmatrix} 
\GlmMeanVecbar_1&\GlmMeanVecbar_2 &\cdots &\GlmMeanVecbar_{\TargetDim
}\\ \epsbar_1&\epsbar_2 &\cdots &\epsbar_{\TargetDim} \\
\end{pmatrix}^\top$.
Since we assume $\SelectProb_0 \geq \cwtil_0>0$, it follows that
\begin{align*}(\sum_{k=1}^{\TargetDim}\SelectProb_k\GlmMeanVecbar_k^2/
\epsbar_k)\SelectProb_0\epsbar_0 + \GlmMeanVecbar_0^2
\SelectProb_0^2\geq \GlmMeanVecbar_0^2 \SelectProb_0^2\geq
(l_g/L_g)^2\czerotil^2.\end{align*} Therefore, $1/[(\sum_{k=1}^{\TargetDim
  }\SelectProb_k\GlmMeanVecbar_k^2/
  \epsbar_k)\SelectProb_0\epsbar_0 + \GlmMeanVecbar_0^2
  \SelectProb_0^2]$ is bounded and Lipschitz in
$\AllNuisancePar$. Similarly, we can verify that
$\InvCompDiag_i, \InvCompRow_i, \SelectProb_0\epsbar_0, \GlmMeanVecbar_0 \SelectProb_0,
\sum_{k=1}^{\TargetDim}\SelectProb_k\GlmMeanVecbar_k^2/\epsbar_k$
are all bounded and Lipschitz. It then follows that
$\InvCompTwo_i(\AllNuisancePar)$ is bounded and Lipschitz in
$\AllNuisancePar$ . Unfortunately, $\InvCompOne_i(\AllNuisancePar)$ is not
necessarily Lipschitz in $\AllNuisancePar$ since $\SelectProb_{i}$ may
not be lower bounded by some constant. However, it follows from Lemma~\ref{lm:glm_lip_cont_temp3} that
$\sqrt{\InvCompOne_i(\AllNuisancePar) + \InvCompTwo_i(\AllNuisancePar)} - \sqrt{\InvCompOne_i(\AllNuisancePar)}$
is bounded and Lipschitz in $\AllNuisancePar$. Since
$\InvCompOne_i(\AllNuisancePar)\DiagProbMatrix=\InvCompDiag_i(\AllNuisancePar)$ is bounded
and Lipschitz in $\AllNuisancePar$ and $\opnorm{\DiagProbMatrix}\leq 1$, it
follows that
$\InvSqrtCovar_i\DiagProbMatrix=\sqrt{\InvCompOne_i(\AllNuisancePar) + \InvCompTwo_i(\AllNuisancePar)}\DiagProbMatrix=(\sqrt{\InvCompOne_i(\AllNuisancePar) + \InvCompTwo_i(\AllNuisancePar)} - \sqrt{\InvCompOne_i(\AllNuisancePar)})\DiagProbMatrix + \sqrt{\InvCompOne_i(\AllNuisancePar)}\DiagProbMatrix$
is bounded and Lipschitz in $\AllNuisancePar$. Similarly, Lemma~\ref{lm:glm_lip_cont_temp3} shows
$\partial_{\AllNuisancePar}[\sqrt{\InvCompOne_i(\AllNuisancePar) + \InvCompTwo_i(\AllNuisancePar)} - \sqrt{\InvCompOne_i(\AllNuisancePar)}]$
is bounded and Lipschitz in
$\AllNuisancePar$. Moreover, \begin{align*}
\partial_\AllNuisancePar\sqrt{\InvCompOne_i(\AllNuisancePar)}\DiagProbMatrix=\diag\{ - \frac{\sqrt{\SelectProb_1}\epsbar_1^{'}(\AllNuisancePar)}{2\sqrt{\epsbar_1^3}}, \cdots, - \frac{\sqrt{\SelectProb_{\TargetDim
    }}\epsbar_{\TargetDim
    }^{'}(\AllNuisancePar)}{2\sqrt{\epsbar_{\TargetDim
      }^3}}\}.\end{align*} Since $\SelectProb_k\leq 1$,
$m_\eps\leq\epsbar_k\leq M_\eps$ and $\epsbar_k'$ is Lipschitz in
$\AllNuisancePar$ for all $k$, it follows that
$\partial_{\AllNuisancePar}\sqrt{\InvCompOne_i(\AllNuisancePar)}
\DiagProbMatrix$ is bounded and Lipschitz in
$\AllNuisancePar$. Therefore, we obtain
$\partial_\AllNuisancePar\InvSqrtCovar_i
\DiagProbMatrix=\partial_{\AllNuisancePar}[\sqrt{\InvCompOne_i(\AllNuisancePar) + \InvCompTwo_i(\AllNuisancePar)} - \sqrt{\InvCompOne_i(\AllNuisancePar)}]\DiagProbMatrix + \partial_{\AllNuisancePar}\sqrt{\InvCompOne_i(\AllNuisancePar)}\DiagProbMatrix$
is bounded and Lipschitz in $\AllNuisancePar$.\end{proof}

The following result uses the notation previously introduced in Lemma~\ref{lm:glm_lip_cont_temp}~and~\ref{lm:glm_lip_cont_temp2}.
\begin{lems}[Lipschitz continuity]
\label{lm:glm_lip_cont_temp3} 
Under the assumptions of Theorem~\ref{thm:glm_new1_temp}, the
quantities
\begin{align*}
\sqrt{\InvCompOne_i(\AllNuisancePar)  +  \InvCompTwo_i(\AllNuisancePar)} -
\sqrt{\InvCompOne_i(\AllNuisancePar)}, \quad
\partial_\AllNuisancePar[\sqrt{\InvCompOne_i(\AllNuisancePar)  + 
    \InvCompTwo_i(\AllNuisancePar)} - \sqrt{\InvCompOne_i(\AllNuisancePar)} ]
\end{align*}
are both bounded and Lipschitz in $\AllNuisancePar$.
\end{lems}
\begin{proof}
For notational simplicity, we drop the dependence of each quantity on
$i$. In this proof, we say a quantity is bounded if it is bounded by
some constant only depends on the constants defined
in Theorem~\ref{thm:glm_new1_temp} but not on $i$. Similarly, we use
$\lesssim$ to denote $\leq$ up to some constant (may or may not)
depend on the quantities defined in Theorem~\ref{thm:glm_new1_temp}. Also, we
say a function is Lipschitz in $\AllNuisancePar$ if the Lipschitz
parameter only depends on the constants defined
in Theorem~\ref{thm:glm_new1_temp}.


\subsubsection{Boundedness of $\sqrt{\InvCompOne_i(\AllNuisancePar)  +  \InvCompTwo_i(\AllNuisancePar)} -
\sqrt{\InvCompOne_i(\AllNuisancePar)}$ and
$\partial_\AllNuisancePar[\sqrt{\InvCompOne_i(\AllNuisancePar) +
    \InvCompTwo_i(\AllNuisancePar)} -
  \sqrt{\InvCompOne_i(\AllNuisancePar)}]$}

By definition,
$\sigma_{\min}(\InvSqrtCovar^2)=\opnorm{\Covar^{-1}}\geq
1/(4M_\eps)$. Combining this with the fact that $\InvSqrtCovar^2 =
\InvCompOne(\AllNuisancePar) + \InvCompTwo(\AllNuisancePar)$,
$\InvCompTwo(\AllNuisancePar)$ is bounded, the diagonal matrix
$\InvCompOne(\AllNuisancePar)$ has minimum eigenvalue lower bounded by
some constant, it follows that there exists some sufficient large
constant $c_T > 0$ such that $c_t \IdMat_{\TargetDim
}\preceq\InvSqrtCovar_{i, \text{trun}}^2(c_T)$ for some constant $c_t
> 0$, where $\InvSqrtCovar_{\text{trun}}^2(c_T)$ is a matrix the same
as $\InvSqrtCovar_{}^2$ except for replacing each diagonal term
$\InvSqrtCovar^2_{kk}$ with $\InvSqrtCovar^2_{kk}\wedge
c_T$. W.l.o.g., since the off-diagonal terms of $\InvSqrtCovar^2$ are
bounded, we can choose $c_T$ sufficiently large such that
$\opnorm{\InvSqrtCovar_{\text{trun}}^2(c_T)}\leq \tfrac{3}{2} c_T$.

Now, define $\Ctil(\AllNuisancePar) \defn \diag\{\InvSqrtCovar_{11}^2
\vee c_T, \ldots, \InvSqrtCovar_{\TargetDim \TargetDim}^2\vee c_T\}$ and
$\DelTil(\AllNuisancePar) \defn \InvSqrtCovar^2_{\text{trun}}(c_T) -
c_T \IdMat_{\TargetDim}$. Then we have $\InvSqrtCovar^2 =
\Ctil(\AllNuisancePar)  +  \DelTil(\AllNuisancePar)$, and
\begin{align*}
\opnorm{ \InvCompTwo(\AllNuisancePar)}\leq \max \{ 0.5 c_T, c_T - c_t\} \leq
\max\{0.5, (c_T - c_t)/c_T\} \sigma_{\min}(\Ctil(\AllNuisancePar))=:
\gamma \sigma_{\min}(\Ctil(\AllNuisancePar))
\end{align*}
for some constant $\gamma<1$. Moreover,
$\Ctil(\AllNuisancePar)\DiagProbMatrix, \DelTil(\AllNuisancePar)$ are
bounded and Lipschitz in $\AllNuisancePar$.

Expanding $\sqrt{\Ctil(\AllNuisancePar)  +  \DelTil(\AllNuisancePar)}$
at $\Ctil(\AllNuisancePar)$ using Taylor expansion (this can be done
since $\sigma_{\min}(\InvCompOne(\AllNuisancePar)) >
\opnorm{\DelTil(\AllNuisancePar)}$), we obtain
\begin{align*}
 \sqrt{\Ctil(\AllNuisancePar)  +  \DelTil(\AllNuisancePar)} -
 \sqrt{\Ctil(\AllNuisancePar)} 
 &
 =\sum_{k=1}^\infty\frac{1}{k!}[\nabla^k \Ctil(\AllNuisancePar)\cdot
   \DelTil_i(\AllNuisancePar)],
\end{align*}
where $[\nabla \Ctil(\AllNuisancePar)\cdot
  \DelTil(\AllNuisancePar)]=\int_{0}^\infty e^{-t
  \sqrt{\Ctil(\AllNuisancePar)}}\DelTil(\AllNuisancePar)e^{-t
  \sqrt{\Ctil(\AllNuisancePar)}}dt$, and the higher order derivatives
are defined iteratively via
\begin{align*}
    &\quad[\nabla^k \Ctil(\AllNuisancePar)\cdot \DelTil(\AllNuisancePar)]\\&= -
    \Big[\nabla \Ctil(\AllNuisancePar)\cdot \Big(\sum_{p +
        q=k-2}\frac{k!}{(p + 1)!(q + 1)!}[\nabla^{p + 1}
        \Ctil(\AllNuisancePar)\cdot
        \DelTil(\AllNuisancePar)][\nabla^{q + 1}
        \Ctil(\AllNuisancePar)\cdot
        \DelTil(\AllNuisancePar)]\Big)\Big].
\end{align*}
From results due to Moral and Niclas~\cite{del2018taylor} (see, in
particular, their equation (4) and the proof of Theorem 1.1), we
establish $\opnorm{\nabla^{k + 1} \Ctil(\AllNuisancePar)\cdot
  \DelTil(\AllNuisancePar)}\leq c_T^{1/2}k!{2k\choose k} 2^{-(2k +
  1)}\gamma^{k + 1/2}$ for $k\geq 0$.  Moreover, define
\begin{align*}
    \HighOrdMa_{k + 1}
    & \defn \sum_{p + q=k-1}\frac{(k + 1)!}{(p + 1)!(q + 1)!}[\nabla^{p + 1} \Ctil(\AllNuisancePar)\cdot  \DelTil(\AllNuisancePar)][\nabla^{q + 1} \Ctil(\AllNuisancePar)\cdot  \DelTil(\AllNuisancePar)].\end{align*} 
    Then
\begin{align*}
    \opnorm{\HighOrdMa_{k + 1}} &\leq c_T(k + 1)! \sum_{p + q=k-1} {2p\choose p}
       {2q\choose q} 2^{-2k}\gamma^{k}/[(p + 1)(q + 1)]\\ &= c_T
       {2k\choose k} 2^{-2k}\gamma^{k}k!,
\end{align*}
where the second line follows from Segner's Recurrence Formula of Catalan
numbers~\cite{koshy2008catalan}.  
Since ${2k\choose k} 2^{-(2k + 1)}\asymp 1/\sqrt{k}$ by Stirling's
formula and $\gamma<1$, we have
\begin{align*}
\opnorm{\sum_{k=1}^\infty\frac{1}{k!}[\nabla^k \Ctil(\AllNuisancePar)\cdot
  \DelTil(\AllNuisancePar)]}\leq
\sum_{k=0}^{\infty}\frac{1}{(k + 1)!}\opnorm{[\nabla^{k + 1}
  \Ctil(\AllNuisancePar)\cdot \DelTil(\AllNuisancePar)]} \lesssim
c_T^{1/2}\sum_{k=0}^{\infty}k^{-3/2} \gamma^{k + 1/2}
\end{align*}
is bounded by some constant which does not depend on $i$ and hence
$\sqrt{\Ctil(\AllNuisancePar)  +  \DelTil(\AllNuisancePar)} -
\sqrt{\Ctil(\AllNuisancePar)}$ is also bounded.
In fact, we have a stronger result. Note that
\begin{align*}
\opnorm{\frac{1}{k!}[\nabla^k \Ctil(\AllNuisancePar)\cdot
  \DelTil(\AllNuisancePar)] \DiagProbMatrix^{-1/2}}
  &=
   \opnorm{\frac{1}{k!}
\Big [\nabla \Ctil(\AllNuisancePar)\cdot \HighOrdMa_k\Big]
\DiagProbMatrix^{-1/2}}\\
& \leq
\frac{1}{k!}  \int_{0}^\infty \opnorm{e^{-t
  \sqrt{\Ctil(\AllNuisancePar)}}}\opnorm{\HighOrdMa_k}\opnorm{e^{-t
  \sqrt{\Ctil(\AllNuisancePar)}} \DiagProbMatrix^{-1/2}}dt \\
& \leq \sum_{j=1}^{\TargetDim} \frac{1}{k!}  \int_{0}^\infty \opnorm{\HighOrdMa_k}
e^{-t \sqrt{1/(\SelectProb_j\epsbar_j)}} (1/\SelectProb_j)^{-1/2} dt
\\
& \lesssim \frac{\TargetDim \opnorm{\HighOrdMa_k}}{2k!}\lesssim \TargetDim c_T
\frac{\gamma^{k-1}}{k^{3/2}}.
\end{align*}

It follows directly from Taylor expansion that
$[\sqrt{\Ctil(\AllNuisancePar)  + 
    \DelTil(\AllNuisancePar)} - \sqrt{\Ctil(\AllNuisancePar)}]
\DiagProbMatrix^{-1/2}$ is bounded. Thus,
$[\sqrt{\Ctil(\AllNuisancePar)  +  \DelTil(\AllNuisancePar)} -
  \sqrt{\Ctil(\AllNuisancePar)}] \DiagProbMatrix^{-1/2}
F(\AllNuisancePar)$ is bounded for any bounded function $F$.

The boundedness of
$\partial_\AllNuisancePar[\sqrt{\InvCompOne_i(\AllNuisancePar) +
    \InvCompTwo_i(\AllNuisancePar)} -
  \sqrt{\InvCompOne_i(\AllNuisancePar)}]$ follows directly from the
boundedness of $\frac{\partial \AllNuisancePar}{\partial x}$ and from
the Lipschitz continuity of $\sqrt{\InvCompOne_i(\AllNuisancePar) +
  \InvCompTwo_i(\AllNuisancePar)} -
\sqrt{\InvCompOne_i(\AllNuisancePar)}$ which we prove next.


\subsubsection{Lipschitz continuity of $\sqrt{\InvCompOne_i(\AllNuisancePar)  +  \InvCompTwo_i(\AllNuisancePar)} -
\sqrt{\InvCompOne_i(\AllNuisancePar)}$ and
$\partial_\AllNuisancePar[\sqrt{\InvCompOne_i(\AllNuisancePar) +
    \InvCompTwo_i(\AllNuisancePar)} -
  \sqrt{\InvCompOne_i(\AllNuisancePar)}]$} Note that
$\Ctil(\AllNuisancePar), \DelTil(\AllNuisancePar)]$ depend on
$\AllNuisancePar$ through $\iRegressor^\top \AuxiNuisancePar +
\iNuisance^\top\NuisancePar$ and we assume $\|(\iRegressor^\top,
\iNuisance^\top)\|_2\leq D_x$. With an abuse of notation, we use $x$
to denote the scalar $\inprod{\iRegressor}{\AuxiNuisancePar} +
\inprod{\iNuisance}{\NuisancePar}$, and define the function
\begin{align}
\label{eqn:dx-defn}
d(x)
\defn \sqrt{\Ctil(x)  +  \DelTil(x)} - \sqrt{\Ctil(x)}
\end{align}
In order to prove the claimed Lipschitz properties it now suffices to show that the functions $d'(x),d^{''}(x)$ are both bounded by some
constant. (Note that we still have $\Ctil(x) \DiagProbMatrix,
\DelTil(x)$ are Lipschitz in $x$ and $|x|\leq M_\AllNuisancePar D_x$.)

\subsubsection{Boundedness of $d'(x)$}

Using the formula of the first order derivative, we obtain
\begin{align*}
& \Big|\mns\Big|\mns\Big|[\sqrt{\Ctil(x) + \DelTil(x)} -
    \sqrt{\Ctil(x)}]' \Big|\mns\Big|\mns\Big|_{\mathrm{op}} \\
& = \Big|\mns\Big|\mns\Big|\int e^{-t \sqrt{\Ctil(x) + \DelTil(x)}
  }(C'(x) + \DelTil'(x)) e^{-t \sqrt{\Ctil(x) + \DelTil(x)} } dt -
  \int e^{-t \sqrt{\Ctil(x)}} C'(x)e^{-t \sqrt{\Ctil(x)}}
  dt\Big|\mns\Big|\mns\Big|_{\mathrm{op}}\\
& \leq \Big|\mns\Big|\mns\Big|\int e^{-t \sqrt{\Ctil(x) + \DelTil(x)}
  } \DelTil'(x) e^{-t \sqrt{\Ctil(x) + \DelTil(x)} }
  dt\Big|\mns\Big|\mns\Big|_{\mathrm{op}} \\
  & +
  2\Big|\mns\Big|\mns\Big|\int (e^{-t \sqrt{\Ctil(x) +
      \DelTil(x)}}-e^{-t \sqrt{\Ctil(x)}}) C'(x)e^{-t \sqrt{\Ctil(x)}}
  dt\Big|\mns\Big|\mns\Big|_{\mathrm{op}} \\
  & + \Big|\mns\Big|\mns\Big|\int (e^{-t \sqrt{\Ctil(x) +
      \DelTil(x)}}-e^{-t \sqrt{\Ctil(x)}}) C'(x)(e^{-t \sqrt{\Ctil(x)
      + \DelTil(x)}}-e^{-t \sqrt{\Ctil(x)}})
  dt\Big|\mns\Big|\mns\Big|_{\mathrm{op}} \\
  & \: = : \: \Term_1 + 2\Term_2
  + \Term_3.
\end{align*}
We now bound the terms $\Term_1, \Term_2$ and $\Term_3$ individually.
For $\Term_1$, we have,
\begin{align*}
\Term_1\leq \int_{0}^\infty \opnorm{e^{-t \sqrt{\Ctil(x) + \DelTil(x)}
}}\opnorm{\DelTil'(x)}\opnorm{e^{-t \sqrt{\Ctil(x) + \DelTil(x)} }} dt
\lesssim \frac{\opnorm{\DelTil'(x)}}{\sigma_{\min}(\InvSqrtCovar^2)},
\end{align*}
which is bounded by our assumption.

For $\Term_2$ and $\Term_3$, note that
\begin{align}
   &\opnorm{(e^{-t \sqrt{\Ctil(x)  +  \DelTil(x)}}-e^{-t
    \sqrt{\Ctil(x)}})\Ctil'(x)^{1/2}}\notag \\
& = \opnorm{\int_0^1 e^{-st \sqrt{\Ctil(x) + \DelTil(x)}}
  [\sqrt{\Ctil(x) + \DelTil(x)} - \sqrt{\Ctil(x)}]e^{-(1-s)t
    \sqrt{\Ctil(x)}})ds \Ctil'(x)^{1/2}}\notag\\ 
    &\leq \int_0^1 \left\{
  \opnorm{e^{-st \sqrt{\Ctil(x) + \DelTil(x)}}} \opnorm{ [\sqrt{\Ctil(x)  + 
      \DelTil(x)} - \sqrt{ \Ctil(x)}] \DiagProbMatrix^{-1/2}} \right. \\
  & \qquad \qquad \qquad \qquad \qquad \qquad \times   \left.  \opnorm{\DiagProbMatrix^{1/2}e^{-(1-s)t \sqrt{\Ctil(x)}}
  \Ctil'(x)^{1/2}} \right\} ds \notag \\
\label{eq:glm_lip_cont_error1}  
& \lesssim \opnorm{ [\sqrt{\Ctil(x)  +  \DelTil(x)} - \sqrt{\Ctil(x)}]
\DiagProbMatrix^{-1/2}} e^{-t \min \{\sqrt{c_T},
  \sqrt{\sigma_{\min}(\InvSqrtCovar^2)}\}} =: v_1 e^{-v_2t},
\end{align}
where the first equation is due to the decomposition
$e^{-A}-e^{-B}=\int_{0}^{1} e^{-s A}(B-A) e^{-(1-s) B}d s$ and the
last line follows from the fact that $\DiagProbMatrix^{1/2}
\Ctil'(x)^{1/2}$ is bounded. Also,
\begin{align}
  &\opnorm{e^{-t \sqrt{\Ctil(x)}} \Ctil'(x)^{1/2}} \notag\\
& \leq \opnorm{e^{-t \sqrt{\Ctil(x)}}\DiagProbMatrix^{-1/2}}
  \opnorm{\DiagProbMatrix^{1/2}\Ctil'(x)^{1/2}}\notag\\
  & \lesssim \opnorm{e^{-t \sqrt{\Ctil(x)}}\DiagProbMatrix^{-1/2}}\notag
  \\
\label{eq:glm_lip_cont_error2}  
& \lesssim \sum_{j=1}^{\TargetDim}
e^{-t(1/\sqrt{\SelectProb_j\epsbar_j})}/\sqrt{\SelectProb_j}.
\end{align}
\textbf{Remark.}  From the derivations we see that results in
equations~\eqref{eq:glm_lip_cont_error1}
and~\eqref{eq:glm_lip_cont_error2} hold in general with
$\Ctil(x)^{1/2}$ replaced by some diagonal matrix function $F(x)$
which satisfies the property that $\DiagProbMatrix^{1/2}F(x)$ is
bounded. For example, we can let $F(x)=[\sqrt {\Ctil(x)}]'$.

Combining the above two results, we obtain
\begin{align*}
 \Term_2
 &\leq \int_0^{\infty} v_1 e^{-v_2t} \sum_{j=1}^{\TargetDim}
 e^{-t (1/(\SelectProb_j\epsbar_j))}/\SelectProb_j dt \\ 
 &\lesssim
 \int_0^{\infty} \sum_{j=1}^{\TargetDim} e^{-t
   (1/\sqrt{\SelectProb_j\epsbar_j})}/\sqrt{\SelectProb_j} dt
   \leq
 \TargetDim \sqrt{M_\eps}=\bigoh(1) 
 \\ 
 \Term_3
 &\leq
  \int_0^{\infty} v_1^2
 e^{-2v_2t}dt= v_1^2/(2v_2)=\bigoh(1).
\end{align*}
Therefore, we conclude that $\opnorm{[\sqrt{\Ctil(x)  +  \DelTil(x)} -
  \sqrt{\Ctil(x)}]'}$ is bounded, and therefore $[\sqrt{\Ctil(x)  + 
    \DelTil(x)} - \sqrt{\Ctil(x)}]$ is Lipschitz.


\subsubsection{Boundedness of $d''(x)$}
Next, we show that $d^{''}(x)$ is also bounded. First, for any matrix
function $\MatFun(x) \in \mathcal S_{\TargetDim}^{ + }$, we have
\begin{align*}
&\quad\sqrt{\MatFun(x)}^{''} \\
&=
 \int_0^{\infty} e^{-t\sqrt{\MatFun(x)}} \MatFun^{''}(x)
e^{-t\sqrt{\MatFun(x)}} dt \\
& \qquad \qquad \qquad \qquad   - 2 \int_{0}^\infty e^{-t\sqrt{\MatFun(x)}} \Big(
\int_0^{\infty} e^{-t\sqrt{\MatFun(x)}} \MatFun'(x) e^{-t\sqrt{\MatFun(x)}} dt \Big)^2
e^{-t \sqrt{\MatFun(x)}} dt
\end{align*}
Therefore,
\begin{align*}
&\Big|\mns\Big|\mns\Big|[\sqrt{\Ctil(x)  +  \DelTil(x)} - \sqrt{\Ctil(x)}]^{''}\Big|\mns\Big|\mns\Big|_{\mathrm{op}}\\
& \leq \Big|\mns\Big|\mns\Big| \int_0^{\infty} e^{-t\sqrt{\Ctil(x)  + 
      \DelTil(x)}} [ \Ctil^{''}(x)  + 
    \DelTil^{''}(x) ] e^{-t\sqrt{\Ctil(x) + \DelTil(x)}} dt \\
    & \qquad \qquad \qquad \qquad\qquad \qquad  -
  \int_0^{\infty} e^{-t\sqrt{\Ctil(x)}}[
    \Ctil^{''}(x)]e^{-t\sqrt{\Ctil(x)}} dt\Big|\mns\Big|\mns\Big|_{\mathrm{op}} \\
&  +  2 \Big|\mns\Big|\mns\Big|\int_{0}^\infty \left\{ e^{-t\sqrt{\Ctil(x)  +  \DelTil(x)}}
  \Big(\int_0^{\infty} e^{-t\sqrt{\Ctil(x) +  \DelTil(x)}} [\Ctil'(x)
     +  \DelTil'(x)]e^{-t\sqrt{\Ctil(x) + \DelTil(x)}} dt\Big)^2 \right. 
\\     
  & \qquad \qquad \qquad \qquad \qquad  \left.e^{-t\sqrt{\Ctil(x) + \DelTil(x)}} \right\} dt\Big.\\
  & \Big.   - \int_{0}^\infty e^{-t \sqrt{\Ctil(x)}}
  \Big(\int_0^{\infty} e^{-t\sqrt{\Ctil(x)}}\Ctil'(x)
  e^{-t\sqrt{\Ctil(x)}} dt\Big)^2 e^{-t\sqrt{\Ctil(x)}} dt
  \Big|\mns\Big|\mns\Big|_{\mathrm{op}} \; =: \; \Term_4 + 2\Term_5.
\end{align*}
For $\Term_4$, we can prove its boundedness using the same argument we
used to show the boundedness of $[\sqrt{\Ctil(x)  +  \DelTil(x)} -
  \sqrt{\Ctil(x)}]^{'}$. The only difference is that we replace
$\Ctil^{'}(x), \DelTil^{'}(x)$ with $\Ctil^{''}(x), \DelTil^{''}(x)$
respectively. Note that in our proof, we only used the property that
$\DelTil^{'}(x)$ and $\DiagProbMatrix \Ctil^{'}(x)$ are bounded. Thus,
the same lines follow because both $\DiagProbMatrix\Ctil^{''}(x)$ and
$\DelTil^{''}(x)$ are bounded.

For simplicity, we drop the dependence of $\Ctil$ and $\DelTil$ on $x$
sometimes when the meaning is clear.
Define 
\begin{align*}
\Term_6&\defn 
\Big|\mns\Big|\mns\Big|\int_{0}^\infty e^{-t\sqrt{\Ctil  +  \DelTil}}
\sqrt{\Ctil}'(\sqrt{\Ctil  +  \DelTil}' - \sqrt{\Ctil}')
e^{-t\sqrt{\Ctil + \DelTil}} dt \Big|\mns\Big|\mns\Big|_{\mathrm{op}}  \\
\Term_7&\defn 
 \Big|\mns\Big|\mns\Big|\int_{0}^\infty
e^{-t\sqrt{\Ctil  +  \DelTil}} (\sqrt{\Ctil  + \DelTil}' - \sqrt{\Ctil}')^2
e^{-t\sqrt{\Ctil + \DelTil}} dt\Big|\mns\Big|\mns\Big|_{\mathrm{op}}\\
\Term_8&\defn
 \Big|\mns\Big|\mns\Big|\int_{0}^\infty (e^{-t\sqrt{\Ctil +
    \DelTil}}-e^{-t\sqrt{\Ctil}}) [\sqrt{\Ctil}']^2 e^{-t\sqrt{\Ctil}}
dt\Big|\mns\Big|\mns\Big|_{\mathrm{op}} \\
\Term_9&\defn 
 \Big|\mns\Big|\mns\Big|\int_{0}^\infty (e^{-t \sqrt{\Ctil + \DelTil}} -
e^{-t\sqrt{\Ctil}}) [\sqrt{\Ctil}']^2 (e^{-t\sqrt{\Ctil +
    \DelTil}}-e^{-t\sqrt{\Ctil}}) dt\Big|\mns\Big|\mns\Big|_{\mathrm{op}}. \\
\end{align*} For $\Term_5$, we have from the
triangle inequality that
\begin{align*}
{\Term_5} 
& = \Big|\mns\Big|\mns\Big| \int_{0}^\infty e^{-t \sqrt{\Ctil  +  \DelTil}}
\Big[\sqrt{\Ctil + \DelTil}'\Big]^2 e^{-t \sqrt{\Ctil  +  \DelTil}} dt
 - \int_{0}^\infty e^{-t \sqrt{\Ctil}} \Big[ \sqrt{\Ctil}'\Big]^2
e^{-t \sqrt{\Ctil}} dt\Big|\mns\Big|\mns\Big|_{\mathrm{op}}\\
& \leq \Big|\mns\Big|\mns\Big|\int_{0}^\infty e^{-t \sqrt{\Ctil  +  \DelTil}}
\sqrt{\Ctil + \DelTil}'(\sqrt{\Ctil  +  \DelTil}' - \sqrt{\Ctil}')
e^{-t\sqrt{\Ctil + \DelTil}} dt\Big|\mns\Big|\mns\Big|_{\mathrm{op}}\\
&  +  \Big|\mns\Big|\mns\Big|\int_{0}^\infty e^{-t \sqrt{\Ctil  +  \DelTil}} (\sqrt{\Ctil
   +  \DelTil}' - \sqrt{\Ctil}')\sqrt{\Ctil  +  \DelTil}' e^{-t\sqrt{\Ctil  + 
    \DelTil}} dt\Big|\mns\Big|\mns\Big|_{\mathrm{op}}\\
&  + \Big|\mns\Big|\mns\Big|\int_{0}^\infty e^{-t\sqrt{\Ctil + \DelTil}}
[\sqrt{\Ctil}']^2 e^{-t\sqrt{\Ctil + \DelTil}} dt - \int_{0}^\infty e^{-t
  \sqrt{\Ctil}} [\sqrt{\Ctil}']^2 e^{-t\sqrt{\Ctil}} dt \Big|\mns\Big|\mns\Big|_{\mathrm {op}} \\
&\lesssim\Term_6 + \Term_7 + \Term_8 + \Term_9.
\end{align*}
We now turn to bounding $\Term_i$ for $i \in \{6, 7, 8, 9 \}$.

\paragraph*{Bounds on $\Term_6$ and $\Term_7$}
Beginning with $\Term_6$, we have
\begin{align*}
 \Term_6 & \leq \int_{0}^\infty \opnorm{e^{-t\sqrt{\Ctil + \DelTil}}
 \sqrt{\Ctil}'} \opnorm{(\sqrt{\Ctil + \DelTil}' - \sqrt{\Ctil}')}
 \opnorm{e^{-t\sqrt{\Ctil + \DelTil}} }dt \\
 & \lesssim \int_{0}^\infty \opnorm{ (e^{-t \sqrt{\Ctil + 
     \DelTil}}-e^{-t\sqrt{\Ctil}})\sqrt{\Ctil}'} + \opnorm{e^{-t
   \sqrt{\Ctil}} \sqrt{\Ctil}'} dt \\
& \lesssim \int_0^\infty v_1 e^{-v_2t} dt  +  \sum_{j=1}^{\TargetDim}
 \int_{0}^\infty e^{-t (1/\sqrt{\SelectProb_j 
     \epsbar_j})}/\sqrt{\SelectProb_j} dt \\
& = \TargetDim \sqrt{M_\eps} + \frac{v_1}{v_2} = \bigoh(1),
\end{align*}
where in the second line we used the fact that $\opnorm{(\sqrt{\Ctil  + 
  \DelTil}' - \sqrt{\Ctil}')}$ is bounded,
$\opnorm{e^{-t\sqrt{\Ctil + \DelTil}} }\leq
e^{-t\sigma_{\min}({\InvSqrtCovar^2})^{1/2}}\leq 1$ and the last line
follows from equation~\eqref{eq:glm_lip_cont_error1}
and~\eqref{eq:glm_lip_cont_error2}, along with the subsequent remarks.
Similarly, we have
\begin{align*}
 \Term_7 & \leq \int_{0}^\infty \opnorm{e^{-t\sqrt{\Ctil + \DelTil}}} \opnorm{
 (\sqrt{\Ctil  +  \DelTil}' - \sqrt{\Ctil}')}^2
 \opnorm{e^{-t\sqrt{\Ctil + \DelTil}}} dt \\
& \lesssim \int_0^\infty e^{-2 t \sigma_{\min}
   ({\InvSqrtCovar^2})^{1/2}} dt =\bigoh(1),
\end{align*}
where the second line follows from $\opnorm{e^{-t \sqrt{\Ctil  +  \DelTil}}
}\leq e^{-t \sigma_{\min}({\InvSqrtCovar^2})^{1/2}}\leq 1$ and the
boundedness of $\opnorm{(\sqrt{\Ctil  +  \DelTil}' - \sqrt{\Ctil}')}$.  

\paragraph*{Bounds on $\Term_8$ and $\Term_9$}
For
$\Term_8$, we have
\begin{align*}
 \Term_8 & \leq \int_{0}^\infty \opnorm{ (e^{-t\sqrt{\Ctil + \DelTil}} - e^{-t
   \sqrt{\Ctil}}) \sqrt{\Ctil}'} \opnorm{\sqrt{\Ctil}' e^{-t\sqrt{\Ctil}}
 }dt \\
& \lesssim \int_{0}^\infty (v_1e^{-v_2t}) (\sum_{j=1}^{\TargetDim}
 e^{-t (1/\sqrt{\SelectProb_j\epsbar_j})}/\sqrt{\SelectProb_j}) dt \\
 & \lesssim \sum_{j=1}^{\TargetDim} \int_{0}^\infty e^{-t
   (1/\sqrt{\SelectProb_j\epsbar_j})}/\sqrt{\SelectProb_j} dt =
 \bigoh(1),
\end{align*}
where the second line follows from
equation~\eqref{eq:glm_lip_cont_error1}
and~\eqref{eq:glm_lip_cont_error2}. Finally, we bound $\Term_9$ as
\begin{align*}
\Term_9 & \leq \int_{0}^\infty \opnorm{ (e^{-t\sqrt{\Ctil +
    \DelTil}}-e^{-t\sqrt{\Ctil}}) \sqrt{\Ctil}'}\opnorm{ \sqrt{\Ctil}'
(e^{-t\sqrt{\Ctil + \DelTil}} - e^{-t\sqrt{\Ctil}})}dt\\ &\lesssim
\int_{0}^\infty (v_1 e^{-v_2t})^2 dt=\frac{v_1^2}{v_2}=\bigoh(1),
\end{align*}
where we again use equation~\eqref{eq:glm_lip_cont_error1} in the
second line. Therefore, we have shown that both $\Term_4$ and
$\Term_5$ are bounded and hence $\opnorm{d^{''}(x)}=\opnorm{[\sqrt{\Ctil(x) +
    \DelTil(x)} - \sqrt{\Ctil(x)}]^{''}}$ is bounded, i.e.,
$[\sqrt{\Ctil(x) + \DelTil(x)} - \sqrt{\Ctil(x)}]^{'}$ is Lipschitz in
$x$. This completes the proof.
\end{proof}


\section{Technical lemmas and their proofs}
\label{SecTechnicalLemmas}

This section is devoted several technical lemmas used in our
proofs.

\subsection{Martingale difference sequence}

We begin with an auxiliary result on martingale difference sequences.
It applies to either vectors or matrices, and we use $\|\cdot\|_F$ to
indicate the Frobenius norm in either case, equivalent to the
Euclidean norm in the vector case.
\begin{lems}
\label{lm:tech_md}
Let $\{D_i\}_{i \geq 1}$ be a martingale difference sequence with
respect to the filtration $\{\History_i \}_{i \geq 1}$ (i.e., $\E (D_i
\mid \iHistory) = 0$ for all $i \geq 1$). If $\frac{1}{\Numobs}
\sum_{i= 1}^\Numobs \E \|D_i\|_F^2 \stackrel{(*)}{=} \bigoh(1)$, then
\begin{align*}
  \frac{1}{\sqrt \Numobs } \sum_{i=1}^{\Numobs} D_i = \bigoh_p(1).
\end{align*}
\end{lems}
\noindent As a special case, the assumption (*) in the above statement
holds, for example, when the second moments $\E\|D_i\|_F^2$ are
uniformly bounded.

\begin{proof}
By properties of boundedness in probability, it suffices to
prove\ that
\begin{align*}
  \frac{1}{\Numobs} \E \| \sum_{i=1}^\Numobs D_i\|_F^2 = \bigoh(1).
\end{align*}
Since $\{D_i\}_{i \geq 1}$ is a martingale difference sequence, we
have
\begin{align*}
\E \tr(D_i D_j^\top)= \E \E(\tr(D_i D_j^\top\mid\History_{j-1}) = \E
\tr(D_i\E( D_j^\top\mid\History_{j-1}))=0. \qquad \mbox{for all $i <
  j$,}
\end{align*}
and as a consequence,
\begin{align*}
\E\|\sum_{i=1}^\Numobs D_i\|_F^2/n=\E \sum_{i,j=1}^\Numobs
\tr(D_iD_j^\top)/n=\E \sum_{i=1}^\Numobs \|D_i\|_F^2/n=\bigoh(1).
\end{align*}
\end{proof}

\subsection{Equivalent condition of Assumption~\ref{assn-lin-selection-prob}}
\label{sec:Equiv-cond-of-SELt}

 Assumption~\ref{assn-lin-selection-prob} is equivalent to the following
assumption on the minimum singular value of the covariance matrix
$\iCovar$.
\begin{enumerate}[label=\myasslabel{(A2b)}]
\item \label{assn-lin-selection-prob-equi} There exists constants $c_0
  > 0$ and $t \in [0, \tfrac{1}{2})$ such that the conditional
    covariance matrix
    \begin{align*}
\iCovar \defn \E \big[ (\iRegressor -
  \SelectProb_i(\iNuisance, \iHistory)) (\iRegressor -
  \SelectProb_i(\iNuisance, \iHistory))^\top \mid \iHistory, \iNuisance
  \big]
    \end{align*}
 satisfies
 \begin{align}
\label{eq:a3a}      
\iCovar \succeq c_i \IdMat_{\TargetDim} = \frac{c_0}{i^{2t}}
\IdMat_{\TargetDim} \quad \mbox{for all $i = 1, 2, \ldots$.}
\end{align} 
\end{enumerate}
Specifically, we have
\begin{lems}[Equivalence of Assumption~\ref{assn-lin-selection-prob} and~\ref{assn-lin-selection-prob-equi}]
\label{lm:equiv_ass}
Given $\SelectProb_0, \SelectProb_1, \ldots, \SelectProb_{\TargetDim}>0$ such
that $\SelectProb_0 + \SelectProb_1 + \cdots + \SelectProb_{\TargetDim}=1$.  Let,
$\Covar\in\R^{{\TargetDim}\times {\TargetDim}}$ with $\Covar_{jj}=(1 -
\SelectProb_j)\SelectProb_j$ and $\Covar_{jk}= -
\SelectProb_j\SelectProb_k$ for $j\neq k$.
\begin{enumerate}
\item[(a)] If there exists some constant $c_0>0$ such that $\Covar\succeq
  c_0\IdMat_{\TargetDim}$, then $\SelectProb_{j}\geq c_0$ for all $j=0,1,
  \ldots,\TargetDim$.
 \item[(b)] If there exists some constant $c_0>0$ such that
   $\SelectProb_{j}\geq c_0$ for $j= 0, 1, \ldots,\TargetDim$, then
   $\Covar\succeq c_0\IdMat_{\TargetDim}/({\TargetDim} + 2).$
\end{enumerate}
\end{lems}
The equivalence of Assumption~\ref{assn-lin-selection-prob}
and~\ref{assn-lin-selection-prob-equi} follows directly from
Lemma~\ref{lm:equiv_ass}. Later in the proofs of auxiliary lemmas, we
also invoke Assumption~\ref{assn-lin-selection-prob-equi} instead
of~\ref{assn-lin-selection-prob}.

\begin{proof}
We split our proof into the two parts of the lemma.

\paragraph*{Proof of part (a)}
Since $\Covar_{jj}\geq\lambda_{\min}(\Covar)\geq c_0$, it follows that
$\SelectProb_j(1 - \SelectProb_j)\geq c_0$ and therefore
$\SelectProb_j>c_0$ for $j\geq1$.  Moreover, since $\TargetDim c_0=c_0\|\mathbf
1\|_2^2\leq \mathbf 1^\top\Covar \mathbf 1=\SelectProb_0(1 -
\SelectProb_0)$, we have $\SelectProb_0>{\TargetDim}c_0>c_0$.

\paragraph*{Proof of part (b)}
Note that
\begin{align*}
\lambda_{\min}(\Covar) = \big(\opnorm{\Covar^{-1}} \big)^{-1} & \geq
\big( \frobnorm{\Covar^{-1}} \big)^{-1} \\
& \stackrel{(j)}{=} \big ( \sqrt{{\TargetDim}(\TargetDim-1) \frac{1}{\SelectProb_0^2}} +
\sum_{j=1}^{\TargetDim} (\frac{1}{\SelectProb_0} + \frac{1}{\SelectProb_j})^2
\big)^{-1} \\
& > 1/\sqrt{\TargetDim(\TargetDim + 1) \frac{1}{\SelectProb_0^2} + 2\sum_{j=1}^{\TargetDim}
  \frac{1}{\SelectProb_j^2}}\\
& > \frac{c_0}{\TargetDim + 2},
\end{align*}
where step (i) follows from the explicit expression of
$\Covar^{-1}$~\eqref{eq:cov_inverse}. It then follows that
$\Covar\succeq c_1\IdMat_{\TargetDim}$ for $c_1=c_0/(\TargetDim + 2)$.
\end{proof}

\end{document}